%% file: Type2ConeSingMain2.tex
\documentclass{amsart}

\usepackage{amssymb}
\usepackage{upgreek} 
\usepackage{hyperref}	\hypersetup{colorlinks, linkcolor=blue, linktoc=all} 
\usepackage{subfiles}
\usepackage{cancel} 

\newtheorem{thm}{Theorem}[section]
\newtheorem{prop}[thm]{Proposition}
\newtheorem{lem}[thm]{Lemma}
\newtheorem{cor}[thm]{Corollary}

\theoremstyle{definition}
\newtheorem{definition}[thm]{Definition}

\theoremstyle{remark}

\newtheorem{remark}[thm]{Remark}

\numberwithin{equation}{section}

\newcommand{\tl}[1]{\tilde{#1}}
\newcommand{\ol}[1]{\overline{#1}}
\newcommand{\ul}[1]{\underline{#1}}
\newcommand{\clZ}{\mathcal{Z}}
\newcommand{\clU}{\mathcal{U}}
\newcommand{\cl}[1]{\mathcal{#1}}

\title[Curv. Blow-up in Doubly-warped Products Evolving by Ricci Flow]{Curvature Blow-up in Doubly-warped Product Metrics Evolving by Ricci Flow}
\author{Maxwell Stolarski}
\address{Department of Mathematics, The University of Texas at Austin}
\email{mstolarski@math.utexas.edu}
\urladdr{www.ma.utexas.edu/users/mstolarski}

\begin{document}

\maketitle

\begin{abstract}
	For any manifold $N^p$ admitting an Einstein metric with positive Einstein constant, 
	we study the behavior of the Ricci flow on high-dimensional products $M =  N^p \times S^{q+1}$
	with doubly-warped product metrics.
	In particular, we provide a rigorous construction of local, type II, conical singularity formation on such spaces.
	It is shown that for any $k > 1$ there exists a solution with curvature blow-up rate $\| Rm \|_{\infty} (t) \gtrsim (T-t)^{-k}$ with singularity modeled on a Ricci-flat cone at parabolic scales.
\end{abstract}

\tableofcontents

\section{Introduction} \label{Intro}
\subfile{Intro.tex}

\section{Setup and Preliminaries} \label{Setup}

\subfile{Setup.tex}

\section{The Initial Data and the Topological Argument} \label{BoxDefn}

\subfile{BoxDefn.tex}

\section{Pointwise Estimates} \label{PointwiseEst}

\subfile{PointwiseEst.tex}

\section{No Inner Region Blow-Up} \label{NoInnBlowup}
\subfile{NoInnBlowup.tex}

\section{Coefficient Estimate} \label{CoeffEst}

\subfile{CoeffEst.tex}

\section{Short-Time Estimates} \label{ShortTimeEsts}

\subfile{APrioriShortTimeEsts2.tex}

\section{Long-Time Estimates} \label{LongTimeEsts}
\subfile{APrioriLongTimeEsts2.tex}

\section{Scalar Curvature Behavior} \label{ScalarCurv}

\subfile{ScalarCurv2.tex}

\appendix
\section{Analytic Facts} \label{AnalyticFacts}
\subfile{AnalyticFacts2.tex}

\newpage
\section{Constants} \label{AppendixOfConstants}
\subfile{AppendixOfConstants.tex}

\bibliographystyle{alpha}
\bibliography{Type2ConeSingBib}

\end{document}

%% file: Intro.tex
A collection of Riemannian metrics $\{ g(t) \}_{t \in [0, T)}$ on a closed, smooth manifold $M$ evolves by Ricci flow if
	$$\partial_t g = - 2 Rc		\qquad \text{for all } t \in (0, T)$$
Solutions often form singularities in finite time and 
Hamilton ~\cite{H95} showed that the Riemann curvature tensor blows up at a finite-time singularity, that is
	$$\limsup_{t \nearrow T} \sup_{x \in M} | Rm |_{g(t)} ( x, t) = + \infty$$
In fact,  
	$$\limsup_{t \nearrow T} ( T - t) \sup_{x \in M} | Rm  |_{g(t)} ( x, t) > 0 $$
as the parabolic scaling invariance of the Ricci flow suggests.

Finite-time singularities of the Ricci flow are classified as
	\begin{equation*} \begin{aligned}
		\text{\textit{Type I} if }  & \limsup_{t \nearrow T} ( T - t) \sup_{x \in M} | Rm  |_{g(t)} ( x, t)  < + \infty \text{, or }\\
		\text{\textit{Type II} if } &  \limsup_{t \nearrow T} ( T - t) \sup_{x \in M} | Rm  |_{g(t)}( x, t)  = + \infty 
	\end{aligned} \end{equation*}
Due in part to results of Enders-M{\"u}ller-Topping ~\cite{EMT11}, type I singularities are better understood than type II singularities in many ways.
The first construction of type II singularities for the Ricci flow appeared in ~\cite{DP06}.
Later, Gu and Zhu ~\cite{GZ08} constructed type II singularities for the Ricci flow by considering rotationally symmetric metrics on $S^{n+1}$.
Angenent, Isenberg, and Knopf ~\cite{AIK15} then provided an alternate construction that in particular allowed for curvature blow-up rates of
	$$\sup_{x \in M} | Rm|_{g(t)}( x, t) \sim (T-t)^{-2 + \frac{2}{k} }$$		
for any $k \in \mathbb{N}$ with $k \ge 3$.
Examples of type II singularities for the Ricci flow on $\mathbb{R}^n$ with blow-up rates of $(T-t)^{-\lambda}$ for any $\lambda \ge 2$ appeared in in ~\cite{W14}. 
The first examples of type II singularities for the K{\"a}hler-Ricci flow recently appeared in ~\cite{LTZ18}.

Here, we construct Ricci flow solutions on certain product manifolds that form type II singularities with curvature blow-up rates given by arbitrarily large powers of $(T-t)$.
The main theorem is the following:
\begin{thm} \label{mainThmAbridged}
	Let $N^p$ be a closed $p$-dimensional manifold that admits an Einstein metric $g_N$ with positive Einstein constant.
	Let $q \ge 10$ and $k > 1$.
	Then there exists a smooth solution $\{ g_k(t) \}_{t \in [0, T)}$ to the Ricci flow on $M = N^p \times S^{q+1}$ that forms a local type II singularity at time $T < \infty$ such that
		$$0 < \limsup_{t \nearrow T} (T-t)^{k} \sup_{x \in M} | Rm|_{g_k(t)} (x,t) \le \infty$$
	and, for some $x \in  S^{q+1}$,
		$$\left( N^p \times S^{q+1}, \frac{1}{T-t} g_k( t ) , x \right) \xrightarrow[t \nearrow T]{ C^2_{loc} \left( N^p \times S^{q+1} \setminus( N^p \times \{ x \}) \right) } \big( C( N^p \times S^q), g_{RFC}, x_* \big)$$
	where $g_{RFC}$ is the Ricci-flat cone metric
		$$g_{RFC} = dr^2 + \frac{p-1}{p+q-1} r^2 g_{N} + \frac{q-1}{p+q-1} r^2 g_{\mathbb{S}^q}$$
	on the cone $C(N^p \times S^q)$ with vertex $x_*$ (see section \ref{ricciFlatConeMetric} for additional details)
	and $g_N$ is normalized such that $Ric_{g_N} = (p-1) g_N$.
\end{thm}
\noindent More precise asymptotics on the singularity formation will be obtained in the course of the proof
that, in particular, imply convergence of the parabolically rescaled \textit{flows}, as opposed to simply the convergence of \textit{time-slices}.
When $N$ is homogeneous with respect to $G$, the metrics $g_k(t)$ are cohomogeneity one with respect to the $G \times SO(q+1)$ action on $N^p \times S^{q+1}$.
The proof of theorem \ref{mainThmAbridged} thereby provides insight into the dynamics of the Ricci flow of metrics of low cohomogeneity and on multiply-warped product metrics.
Indeed, these rigorous examples of conical singularity formation for the Ricci flow (cf. ~\cite{M14, IKS16, Appleton19}) qualitatively differ from the examples of Ricci flow singularities modeled on generalized cylinders $\mathbb{R}^p \times S^q$ (see e.g. ~\cite{S00, AK04, AIK15, Carson17}).

Theorem \ref{mainThmAbridged} has some precedence in mean curvature flow.
In ~\cite{V94}, Vel{\'a}zquez examined the mean curvature flow starting from smooth, non-compact, $O(n) \times O(n)$-invariant hypersurfaces in $\mathbb{R}^{2n}$.
When $n \ge 4$, he shows that there exist such mean curvature flow solutions $\{ \Sigma_t \}_{t \in [0, T)}$ that form a finite-time singularity at the origin whose parabolic dilations $\{ ( T- t)^{-1/2} \Sigma_{ t} \}$ about the origin converge to the Simons cone in a precise sense.
Moreover, for any $l \ge 2$, there are such solutions with blow-up rates for the second fundamental form given by
	$$\sup | A| \sim ( T-t)^{-\frac{ 1}{2} - \sigma_l }	$$
	$$ \text{ where } \quad \sigma_l = \frac{\frac{\alpha - 1}{2} + l}{1 + | \alpha| } \quad \text{ and } \quad \alpha = - \frac{2n-3}{2} + \frac{1}{2} \sqrt{ (2n - 1)^2 - 8(2n-2)}$$
After rescaling by the blow-up rate of the second fundamental form, the rescaled hypersurfaces $\{ (T- t)^{ -\frac{1}{2} - \sigma_l } \Sigma_{t} \}$ converge locally uniformly to a minimal hypersurface that is asymptotic to the Simons cone.
These solutions with $l = 2$ were further analyzed in ~\cite{GS18} where it is shown that the mean curvature blows up but at a rate strictly less than that of the second fundamental form.


While the Riemann curvature tensor and the Ricci tensor ~\cite{S05} blow up at a finite-time singularity of the Ricci flow, it is unknown if the scalar curvature necessarily blows up at a finite-time singularity.
This problem of scalar curvature blow-up is related to the possible blow-up rates of the Riemann curvature tensor.
Indeed, B. Wang ~\cite{W12} showed that 
	$$\limsup_{t \nearrow T} (T-t) \sqrt{ \sup_{x \in M} | Rm | } \sqrt{ \sup_{x \in M } R } > 0$$
which in particular implies that the scalar curvature can remain bounded only if $ |Rm|$ blows up at a rate of at least $(T- t)^{-2}$.

Theorem \ref{mainThmAbridged} gives the existence of Ricci flow solutions with blow-up rates greater than $(T-t)^{-2}$ and singularities modeled on a Ricci-flat cone at parabolic scales.
Moreover, a formal matched asymptotic argument indicates that the solutions converge to the Ricci-flat B{\" o}hm metric on $\mathbb{R}^{q+1} \times N^p$ (see section \ref{bohmMetric}) at the curvature blow-up scale.
These asymptotics suggest that the solutions in theorem \ref{mainThmAbridged} may include examples of finite-time singularities of the Ricci flow with bounded scalar curvature.
We do not obtain the scalar curvature blow-up behavior in this current paper but believe it merits further investigation.
To this end, we include a formal and numerical argument that suggests the scalar curvature satisfies a type I bound, namely
	$$\limsup_{t \nearrow T} (T-t) \sup_{x \in S^p \times S^{q+1} } |R|(x,t) < \infty$$
Additionally, because a Ricci-flat cone may be considered as a shrinking Ricci soliton, we note that the singularity formation exhibited in theorem \ref{mainThmAbridged} is consistent with the convergence results for Ricci flows with uniform scalar curvature bounds proved in ~\cite{Bamler18}.

The outline of the current work is as follows:
Section \ref{Setup} establishes some preliminary results for the Ricci flow of doubly-warped product metrics and collects notation for the coordinate systems used throughout the paper.
In section \ref{BoxDefn}, we set up the topological argument used to construct the solutions in theorem \ref{mainThmAbridged}.
The topological argument follows the general strategy employed in ~\cite{V94, HV}.
In our case, however, we must consider a system of partial differential equations (\ref{psRF}), rather than a scalar equation, and so the use of maximum principles is more delicate.
The remainder of the proof relies on technical estimates contained in sections \ref{PointwiseEst}, \ref{CoeffEst}, \ref{ShortTimeEsts}, and \ref{LongTimeEsts}, as well as an argument in section \ref{NoInnBlowup} that ensures the Ricci flow solutions remain smooth up to time $T$.
Section \ref{ScalarCurv} provides the formal argument for the scalar curvature behavior at the singularity time.
Finally, appendix \ref{AnalyticFacts} collects several facts about the special functions and weighted $L^2$-spaces used throughout the paper.
Appendix \ref{AppendixOfConstants} lists the parameters and constants used in the topological argument for the readers' convenience
and is ordered such that each constant depends only on those above it in the list.

\addtocontents{toc}{\protect\setcounter{tocdepth}{0}} 
\subsection*{Acknowledgements}
The author would like to thank Dan Knopf for his mentorship, guidance, and initial suggestion of the problem.
Without his supervision, the completion of this project would not have been possible.
The author was partially supported by NSF RTG grant DMS-1148490.
\addtocontents{toc}{\protect\setcounter{tocdepth}{2}}

%% file: Setup.tex
\subsection{Coordinate Systems for Doubly-Warped Product Metrics} \label{coordSystems}
\subsubsection{Non-Geometric Coordinates} \label{nongeomtricCoords}
Throughout, we will consider \textit{doubly-warped product} metrics
on $(S^{q+1} \times N^p, g(x,t))$ of the form
		$$g(x,t) = \chi(x,t)^2 dx^2 + \phi(x,t)^2 g_{N} + \psi(x,t)^2 g_{S^q} \qquad (x \in (-1,1))$$
where 
 $g_{N}$ is an Einstein metric on $N^p$ with positive Einstein constant normalized such that $Ric_{g_N} = (p-1) g_N$
 and $g_{S^q}$ is the round metric on the sphere $S^q$ of radius one.
It is straightforward to see from O'Neill's formulas for curvature that the Ricci flow preserves the doubly-warped product structure (see e.g. Chapter 9, Section J of ~\cite{Besse87}). 
In fact, the solution $\chi, \phi, \psi$ is independent of the choice of $N$, or equivalently
	\begin{equation*} \begin{aligned}
		&g(x,t) = \chi(x,t)^2 dx^2 + \phi(x,t)^2 g_{N} + \psi(x,t)^2 g_{S^q}  \\
		& \qquad \text{ is a Ricci flow solution on $S^{q+1} \times N^p$} \\
		\iff& \tl{g}(x,t) = \chi(x,t)^2 dx^2 + \phi(x,t)^2 g_{S^p} + \psi(x,t)^2 g_{S^q}  \\
		& \qquad \text{ is a Ricci flow solution on $S^{q+1} \times S^p$} 
	\end{aligned} \end{equation*}
Therefore, we will assume without loss of generality that $N = S^p$ with $p \ge 2$ for the remainder of the paper.

Note that smoothness of the metric implies that, for all $x \in (-1,1)$, $\chi, \phi, \psi$ are positive without loss of generality. 
These metrics are cohomogeneity one with respect to the natural $SO(p+1) \times SO(q+1)$ action.
Often, it will be additionally assumed that these metrics are reflection symmetric, that is
	$$\chi(-x, t) = \chi(x,t) \qquad \phi(-x, t) = \phi(x,t) \qquad \psi(-x, t) = \psi(x,t)$$
Because the Ricci flow preserves isometries, this reflection symmetry will also be preserved under the flow.

The $x$ coordinate can be reparametrized by arc length by defining $s(x,t)$ such that $ds = \chi dx$.
In terms of the $s$ coordinate, the metric takes the form
	$$g(s,t) = ds^2 + \phi(s,t)^2 g_{S^p} + \psi(s,t)^2 g_{S^q}.$$

Note that smoothness of the metric implies that for all $ t$
	\begin{equation*} \begin{aligned}
		\phi( s(-1,t), t) \ge 0		&\qquad&	\phi_s (s(-1,t), t) \equiv 0	\\
		\psi( s(-1,t), t) \equiv 0 	&\qquad&	\psi_s (s(-1,t),t) \equiv 1	\\
	\end{aligned} \end{equation*}

The $x$ and $s$ derivatives are related by
	\begin{equation*} \begin{aligned}
		\frac{\partial}{\partial s} &= \frac{1}{\chi} \frac{\partial}{\partial x}\\
		\frac{\partial^2}{\partial s^2} &= \frac{1}{\chi^2} \frac{\partial^2}{\partial x^2} - \frac{\chi_x}{\chi^3} \frac{\partial}{\partial x}\\
		\left. \frac{\partial}{\partial t} \right|_s &= \left. \frac{\partial}{\partial t} \right|_x - \left( \int p \frac{\phi_{ss}}{\phi} + q \frac{\psi_{ss}}{\psi} ds \right)\frac{\partial}{\partial s}\\
	\end{aligned} \end{equation*}
Here $\partial_t |_x, \partial_t|_s$ denote the partial derivatives with respect to $t$ that fix $x$ and $s$ respectively.
$\partial_t|_x$ and $\partial_s$ have a nontrival commutator
	$$\left[ \left. \frac{\partial}{\partial t} \right|_x , \frac{\partial}{\partial s} \right] = -\left( p \frac{\phi_{ss}}{\phi} + q \frac{\psi_{ss}}{\psi} \right) \frac{\partial}{\partial s}$$
	
In terms of $s$, all sectional curvatures of $(S^p \times S^{q+1}, g)$ are convex linear combinations of 
	\begin{equation} \label{sectionalCurvatures}
		- \frac{ \phi_{ss}}{\phi} , - \frac{ \psi_{ss}}{\psi} , \frac{1- \phi_s^2}{\phi^2}, \frac{ 1 - \psi_s^2}{\psi^2}, - \frac{ \phi_s \psi_s}{\phi \psi}
	\end{equation}
and the Ricci tensor of a doubly-warped product metric is given by
\begin{equation*} \begin{aligned}
	Ric =& - \left( p \frac{\phi_{ss}}{\phi} + q \frac{\psi_{ss}}{\psi} \right) ds^2 \\
	&+ \left( -\frac{\phi_{ss}}{\phi} + (p-1) \frac{1- \phi_s^2}{\phi^2} - q \frac{\phi_s \psi_s}{\phi \psi} \right) \phi^2 g_{S^p} \\
	&+ \left( -\frac{\psi_{ss}}{\psi} + (q-1) \frac{1- \psi_s^2}{\psi^2} - p \frac{\phi_s \psi_s}{\phi \psi} \right) \psi^2 g_{S^q} \\
\end{aligned} \end{equation*}
(see e.g. Chapter 3, Section 2.4 of ~\cite{Petersen06} for reference).
It now follows that the Ricci flow $\partial_t g = -2Rc$ of a doubly-warped product metric is equivalent to the following PDE system
\begin{equation} \label{RF} \begin{aligned}
	\left. \frac{\partial}{\partial t} \right|_x \chi &= p \frac{\phi_{xx}}{\chi \phi} - p \frac{\chi_x \phi_x}{\chi^2 \phi} +  q \frac{\psi_{xx}}{\chi \psi} - q \frac{\chi_x \psi_x}{\chi^2 \psi} = \chi \left( p \frac{\phi_{ss}}{\phi} + q \frac{\psi_{ss}}{\psi} \right)\\ 
	\left. \frac{\partial}{\partial t} \right|_x\phi &= \phi_{ss} - (p-1) \frac{1- \phi_s^2}{\phi} + q \frac{\phi_s \psi_s}{ \psi} \\
	\left. \frac{\partial}{\partial t} \right|_x \psi &= \psi_{ss}- (q-1) \frac{1- \psi_s^2}{\psi} + p \frac{\phi_s \psi_s}{\phi} \\
\end{aligned} \end{equation}
Due to the form of the equation for $\chi(x,t)$, it suffices to estimate only $\phi, \psi$ and their $s$-derivatives to guarantee regularity of the metric $g$.

\subsubsection{Singly-Warped Product on General Base} \label{singlyWarpedCoords}
By letting $g_B = ds^2 + \psi^2 g_{S^q}$ denote a rotationally symmetric metric on $S^{q+1}$, a doubly-warped product metric $g$ on $S^{q+1} \times S^p$ can be regarded as a singly-warped product metric, namely
	$$g = g_B + \phi^2 g_{S^p}$$
In this setup, the Ricci flow for $g$ is equivalent to the system
	\begin{equation} \label{swRF} \left\{ \begin{aligned} 
		\partial_t g_B &= - 2 Ric_B + 2p \frac{ \nabla^2 \phi}{\phi} \\
		\partial_t \phi &= \Delta \phi - (p-1) \frac{ 1 - | \nabla \phi |^2}{\phi} \\
	\end{aligned} \right. \end{equation}
where here the covariant derivatives are with respect to the metric $g_B$ on $S^{q+1}$.

An application of the maximum principle for $\phi$ immediately yields the following:
\begin{prop} \label{cylinderComparePhi}
	Assume $g = g_B + \phi^2 g_{S^p}$ is a smooth Ricci flow on $S^{q+1} \times S^p$ for $t \in [0, T)$.
	Then
		$$\sqrt{ \min_{x \in S^{q+1}} \phi(x, 0)^2 - 2(p-1) t } \le \phi \le \sqrt{ \max_{x \in S^{q+1}} \phi(x, 0 )^2 - 2(p-1) t}$$ 
	In particular,
		$$T \le \frac{1}{2(p-1)}  \max_{x \in S^{q+1}} \phi(x, 0 )^2$$
\end{prop}

An application of the maximum principle also implies the following gradient bound which was also observed in ~\cite{Carson17}.
\begin{prop} \label{gradBoundPhi}
	Assume $g = g_B + \phi^2 g_{S^p}$ is a smooth Ricci flow on $S^{q+1} \times S^p$ for $t \in [0, T)$.
	Then
	$$\max_{x \in S^{q+1} } | \nabla \phi|^2(x, t) \le \max\left( 1, \max_{x \in S^{q+1}} | \nabla \phi|^2(x, 0) \right)$$
\end{prop}
\begin{proof}
	The evolution equation for $| \nabla \phi|^2$ can be computed from the system (\ref{swRF}) as follows
	\begin{equation*} \begin{aligned}
		\partial_t | \nabla \phi|^2 = & \partial_t \left( g_B^{ij} \partial_i \phi \partial_j \phi \right) \\
		=&2 g_B^{ij} ( \partial_i \partial_t \phi) (\partial_j \phi) + \partial_t g_B^{ij} \partial_i \phi \partial_j \phi \\
		=& 2 \left\langle \nabla \phi, \nabla \left( \Delta \phi - (p-1) \frac{1- | \nabla \phi|^2}{\phi} \right) \right\rangle \\
		&+ 2 Rc_B(\nabla \phi, \nabla \phi) -2p \frac{ \nabla^i \nabla^j \phi}{\phi} \nabla_i \phi \nabla_j \phi\\
		=& \Delta | \nabla \phi|^2 - 2 | \nabla^2 \phi|^2 -2(p-1)\left\langle \nabla \phi, - \frac{1}{\phi^2} \nabla \phi - \nabla\left( \frac{ | \nabla \phi |^2}{\phi} \right) \right\rangle	 \\
		& -2p \frac{ \nabla^i \nabla^j \phi}{\phi} \nabla_i \phi \nabla_j \phi\\
	\end{aligned} \end{equation*}
	where the last equality follows from the Bochner formula.
	Using the equality
		$$-2p \frac{ \nabla^i \nabla^j \phi}{\phi} \nabla_i \phi \nabla_j \phi = -p \left\langle \nabla \phi , \nabla \left( \frac{ | \nabla \phi|^2}{\phi} \right) \right\rangle - p \frac{| \nabla \phi|^4}{\phi^2}$$
	it then follows that
	\begin{equation*} \begin{aligned}
		\partial_t | \nabla \phi|^2  
		=& \Delta | \nabla \phi|^2 - 2 | \nabla^2 \phi |^2 + 2(p-1) \frac{ | \nabla \phi |^2}{\phi^2} \\
		&+ (p-2) \left \langle \nabla \phi , \nabla \left( \frac{ | \nabla \phi|^2}{\phi } \right) \right \rangle - p \frac{ | \nabla \phi |^4}{\phi^2} \\
		=& \Delta | \nabla \phi|^2 - 2 | \nabla^2 \phi |^2 + (p-2) \left\langle \frac{ \nabla \phi}{\phi}, \nabla \left( | \nabla \phi|^2 \right) \right \rangle \\
		&+ 2(p-1) \frac{ | \nabla \phi|^2}{\phi^2} \left( 1 - | \nabla \phi|^2 \right) \\
		\le& \Delta | \nabla \phi|^2 + (p-2) \left\langle \frac{ \nabla \phi}{\phi}, \nabla \left( | \nabla \phi|^2 \right) \right \rangle 
		+ 2(p-1) \frac{ | \nabla \phi|^2}{\phi^2} \left( 1 - | \nabla \phi|^2 \right) \\
	\end{aligned} \end{equation*}
	The upper bound for $| \nabla \phi |^2$ then follows from the maximum principle.
\end{proof}

\subsubsection{Singly-Warped Product on $(-1,1) \times S^p$}
If instead we let $g_B = \chi^2 dx^2 + \phi^2 g_{S^p}$ denote an incomplete, rotationally symmetric metric on $(-1, 1) \times S^p$, 
then a doubly-warped product metric $g$ on $S^{q+1} \times S^p$ can similarly be regarded as a singly-warped product metric over $(-1, 1) \times S^p$, namely
	$$g = g_B + \psi^2 g_{S^p}$$
In this setup, if $g$ evolves by the Ricci flow then $g_B$ and $\psi$ satisfy
	\begin{equation} \label{swRF'} \left\{ \begin{aligned} 
		\partial_t g_B &= - 2 Ric_B + 2q \frac{ \nabla^2 \psi}{\psi} \\
		\partial_t \psi &= \Delta \psi - (q-1) \frac{ 1 - | \nabla \psi |^2}{\psi} \\
	\end{aligned} \right. \end{equation}
Again, the covariant derivatives here are with respect to the metric $g_B$.
Moreover, smoothness of the metric $g$ implies the following boundary behavior for $\psi$
\begin{equation*} \begin{aligned}
	\psi(x,t) =& 0 \qquad &\text{for all } x \in \{ \pm 1 \} \times S^p\\
	| \nabla \psi |^2(x,t) =& 1 		&\text{for all } x \in \{ \pm 1 \} \times S^p
\end{aligned} \end{equation*}

Again, an application of the maximum principle yields the following propositions:
\begin{prop} \label{cylinderComparePsi}
	Assume $g = g_B + \psi^2 g_{S^p}$ is a smooth Ricci flow on $S^{q+1} \times S^p$ for $t \in [0, T)$.
	Then
		$$\psi \le \sqrt{ \max_{x \in (-1,1) \times S^p} \psi(x, 0 )^2 - 2(q-1) t}$$
\end{prop}

\begin{prop} \label{gradBoundPsi}
	Assume $g = g_B + \psi^2 g_{S^p}$ is a smooth Ricci flow on $S^{q+1} \times S^p$ for $t \in [0, T)$.
	Then
	$$\max_{x \in (-1,1) \times S^p} | \nabla \psi|^2(x, t) \le \max\left( 1, \max_{x \in (-1,1) \times S^p} | \nabla \psi|^2(x, 0) \right)$$
\end{prop}

\subsubsection{Sideways Coordinates} \label{sidewaysCoords}
Aside from the above estimates, it is difficult to estimate the dynamics of Ricci flow solutions near the Ricci-flat cone in the above coordinate systems.
For this reason, we introduce another coordinate system in which most of the remaining analysis will be carried out.

Namely, in a connected domain where $\psi_s$ is bounded away from zero, the $S^q$ radius $\psi$ can itself be used as a coordinate and the metric can be written as
	$$g = \frac{1}{z} d \psi^2 + e^{2u} g_{S^p} + \psi^2 g_{S^q}$$
	$$z = \psi_s^2 \qquad z = z(\psi, t) \qquad u = u(\psi, t)$$
Following (\ref{sectionalCurvatures}), all sectional curvatures are convex linear combinations of 
\begin{equation*} \begin{aligned}
		-\frac{z_\psi}{2 \psi} =& -\frac{\psi_{ss}}{\psi} \qquad
		& -z u_{\psi \psi} - z u_\psi^2 - u_\psi \frac{z_\psi}{2} =& -\frac{ \phi_{ss}}{\phi}  \\
		\frac{1-z}{\psi^2} =& \frac{1 - \psi_s^2}{\psi^2}
		& e^{-2u} - zu_\psi^2 =& \frac{1- \phi_s^2}{\phi^2} 
	\end{aligned} \end{equation*}	
\begin{gather*}
	-\frac{z}{\psi} u_\psi =	-\frac{\psi_s \phi_s}{\psi \phi} 
\end{gather*}
In these coordinates, the Ricci flow for $g$ yields the following PDE system for $z$ and $u$ (cf. ~\cite{AIK15})
	\begin{equation} \label{sRF} \begin{aligned}
		z_t =& \mathcal{F}_\psi[z, u_\psi] \doteqdot z z_{\psi \psi} + z_\psi \left( \frac{q-1}{\psi} - \frac{z}{\psi} \right) - \frac{1}{2} z_\psi^2 + \frac{2(q-1)}{\psi^2} z(1-z) - 2pz^2 u_\psi^2 \\
		u_t =& \mathcal{E}_\psi[z, u] \doteqdot z u_{\psi \psi} + u_\psi \frac{z + q-1}{\psi} - (p-1) e^{-2u}
	\end{aligned} \end{equation}	
Here, $\partial_t = \partial_t|_\psi$, and so the $\partial_\psi$ and $\partial_t$ operators commute.
In particular, differentiating the $u$ equation with respect to $\psi$ yields the following \textit{linear} equation for $u_\psi$
\begin{equation*} \begin{aligned}
	\partial_t u_\psi =& \mathcal{L}_\psi[z,u]u_\psi \\
	\doteqdot& z u_{\psi \psi \psi} + u_{\psi \psi} \left[ z_\psi + \frac{z}{\psi} + \frac{q-1}{\psi} \right] + u_\psi \left[ \frac{z_\psi}{\psi} - \frac{z}{\psi^2} - \frac{q-1}{\psi^2} + 2 (p-1) e^{-2u} \right]
\end{aligned} \end{equation*} 

\subsubsection{Parabolically Rescaled Sideways Coordinates} \label{rescaledSidewaysCoords}
In order to investigate singularity formation at a chosen time $T$, it will be relevant to consider the parabolically rescaled sideways coordinates defined by
	$$\gamma = e^{\tau/2} \psi  \qquad \tau = -\log(T-t)$$
	$$Z(\gamma, \tau) = \psi_s^2(s,t) \qquad U(\gamma, \tau)  = \log \left( e^{\tau/2} \phi(s,t) \right)$$
	In these coordinates, the Ricci flow system (\ref{swRF}) is equivalent to
		\begin{equation} \label{psRF} \begin{aligned}
			\partial_\tau|_\gamma Z + \frac{\gamma}{2} Z_\gamma &= \mathcal{F}_\gamma[Z, U_\gamma] = \mathcal{F}^l_\gamma [Z] + \mathcal{F}^q_\gamma[Z] - 2pZ^2 U_\gamma^2 \\ \\
 			\partial_\tau |_\gamma U  + \frac{\gamma}{2} U_\gamma - \frac{1}{2} &= \mathcal{E}_\gamma[ Z, U] = Z U_{\gamma \gamma} +  \frac{Z+q-1}{\gamma} U_\gamma - (p-1) e^{-2U} \\
		\end{aligned} \end{equation}
	where
		\begin{equation*}
			\mathcal{F}^l_\gamma[Z] \doteqdot \frac{q-1}{\gamma} Z_\gamma + 2 \frac{q-1}{\gamma^2} Z	\qquad \text{and} \qquad
			\mathcal{F}^q_\gamma[Z] \doteqdot Z Z_{\gamma \gamma} - \frac{1}{2} Z_\gamma^2 - \frac{1}{\gamma} Z Z_\gamma - 2 \frac{q-1}{\gamma^2} Z^2
		\end{equation*}
	\\
	Observe that $\mathcal{F}^q_\gamma [Z] = \mathcal{Q}_\gamma [ Z, Z]$, where $\mathcal{Q}_\gamma$ is the symmetric bilinear operator defined by
		$$\mathcal{Q}_\gamma [ Z, W] \doteqdot  \frac{1}{2} Z W_{\gamma \gamma} + \frac{1}{2} W Z_{\gamma \gamma} - \frac{1}{2} Z_\gamma W_\gamma - \frac{1}{2 \gamma} Z W_\gamma - \frac{1}{2 \gamma} W Z_\gamma - \frac{2(q-1)}{\gamma^2} Z W$$
We also note that
	$$\partial_\tau U_\gamma + \frac{1}{2} U_\gamma + \frac{\gamma}{2} U_{\gamma \gamma} = \mathcal{L}_\gamma[Z, U] U_\gamma$$
		
\subsection{Special Solutions} \label{specialSolutions}
Before continuing, we collect some known examples of doubly-warped product Ricci flow solutions.
The following solutions are all self-similar though they may not be smooth. 

\subsubsection{Products of Shrinking Solitons}
For $p, q \ge 2$, there are several doubly-warped shrinking Ricci solitons given by products of shrinking Ricci solitons
\begin{itemize}
	\item Einstein products of spheres: $S^{q+1} \times S^p$, $S^{p+1} \times S^q$
	\item generalized cylinders: $\mathbb{R} \times S^p \times S^q$, $\mathbb{R}^{q+1} \times S^p$, $\mathbb{R}^{p+1} \times S^q$
\end{itemize}
Note that ~\cite{B98} contains examples of Einstein metrics on products of spheres in low dimensions that are doubly-warped products but not homogeneous.

\subsubsection{Ricci Flat Cone} \label{ricciFlatConeMetric}
On the cone $C(S^p \times S^q)$ over $S^p \times S^q$, there is a Ricci flat metric given by
	$$g = dx^2 + A^2 x^2 g_{S^p} + B^2 x^2 g_{S^q} \qquad x \in [0, \infty)$$
	$$\text{where } A \doteqdot \sqrt{ \frac{p-1}{p+q-1} } \quad \text{ and } \quad B \doteqdot \sqrt{ \frac{q-1}{p+q-1} }$$
This metric can be regarded as a non-smooth stationary solution to the Ricci flow.
In the sideways coordinates and parabolically rescaled sideways coordinates, this metric is given by
\begin{equation*} \begin{aligned}
	z_{RFC}(\psi) =& B^2  \qquad	&	u_{RFC}(\psi) =& \log \left( \frac{A}{B} \psi \right) \\
	Z_{RFC}(\gamma) =& B^2 		&	U_{RFC}(\gamma) =& \log \left( \frac{A}{B} \gamma \right) \\
\end{aligned} \end{equation*}

\subsubsection{Sine Cone} \label{sineCone}
There is a family of singular, reflection symmetric, Einstein metrics given by
	$$g_{SC, \lambda} = dx^2 +  \lambda^2 \sin^2 (x/ \lambda) A^2 g_{S^p} + \lambda^2 \sin^2 (x/\lambda) B^2 g_{S^q} 	\qquad	x \in [0, \lambda \pi]$$	
	$$\text{with } Rc[ g_{SC, \lambda} ] = \frac{ p+q}{\lambda^2} g_{SC, \lambda}$$
As such, 
	$$g_{SC, \lambda(t)} \qquad \text{ with } \qquad \lambda(t) = \sqrt{ \lambda_0^2 - 2(p+q)t}$$
can be regarded as a singular Ricci flow solution for $t \in \left(-\infty, \frac{ \lambda_0^2}{2(p+q)} \right)$.

In the sideways coordinates and parabolically rescaled sideways coordinates, this metric is given by
\begin{equation*} \begin{aligned}
	z_{SC, \lambda}(\psi) =& B^2 - \frac{ \psi^2}{\lambda^2} \qquad	&	u_{SC, \lambda}(\psi) =& \log \left( \frac{A}{B} \psi \right) \\
	Z_{SC, \lambda}(\gamma, \tau) =& B^2 - \frac{ \gamma^2 }{e^{\tau} \lambda^2}		&	U_{SC, \lambda}(\gamma) =& \log \left( \frac{A}{B} \gamma \right) \\
\end{aligned} \end{equation*}
In particular, the corresponding singular Ricci flow solution is given by
\begin{equation*} \begin{aligned}
	z_{SC, \lambda}(\psi, t) =& B^2 - \frac{ \psi^2}{\lambda_0^2 - 2(p+q)t} \qquad	&	u_{SC, \lambda}(\psi) =& \log \left( \frac{A}{B} \psi \right) \\
	Z_{SC, \lambda}(\gamma) =& B^2 - \frac{ \gamma^2 }{2(p+q)}		&	U_{SC, \lambda}(\gamma) =& \log \left( \frac{A}{B} \gamma \right) \\
\end{aligned} \end{equation*}
Here, the domain of $(z,u)(\psi, t)$ is 
$$\left\{ (\psi, t) \in [0, \infty) \times \left(-\infty,  \frac{ \lambda_0^2}{2(p+q)} \right) : \psi \le B \sqrt{ \lambda_0^2 - 2(p+q)t} \right\}$$
 and $\tau = - \log \left(  \frac{ \lambda_0^2}{2(p+q)} - t \right)$ is implicitly used to define $Z, U$ locally.

\subsubsection{B{\"o}hm's Ricci Flat $\mathbb{R}^{q+1} \times S^p$} \label{bohmMetric}
In ~\cite{B99}, B{\"o}hm shows that, for any $p,q \ge 2$, there exists a smooth, complete, doubly-warped, Ricci-flat metric $g_{BGK}$ on $\mathbb{R}^{q+1} \times S^p$ which is unique up to rescaling.
The construction of this metric was also carried out by Gastel-Kronz ~\cite{GK04}.
The warping functions $\phi_{BGK}(s), \psi_{BGK}(s)$ used to define the metric $g_{BGK}$ are not known explicitly but are given as solutions to a system of ordinary differential equations.
Nonetheless, it is known that this metric is asymptotically conical with asymptotic tangent cone given by the Ricci-flat cone $( C(S^p \times S^q), g_{RFC} )$.

\subsection{Sturmian Theorem}
Informally, the ``Sturmian theorem," sometimes known as the ``first Sturmian theorem," states that the the number of zeros of a solution to a linear parabolic PDE in one space variable is nonincreasing in time.
In this subsection, we show that the Sturmian theorem holds for the space derivatives of the warping functions of a doubly-warped product metric.
The reader is advised to consult ~\cite{Angenent88, GH05} for additional background on the (first) Sturmian theorem.

\begin{definition}
	Let $f(x,t)$ be a function defined on $(-1,1) \times [0,T']$.
	For any $t \in [0,T']$, the \textit{number of sign changes of $f$ at time $t$} is defined by
	\begin{equation*}
		\mathfrak{Z}(f, t) \doteqdot \sup \left\{ k \in \mathbb{N} : 
		\begin{array}{c}
			\text{there exists } -1 < x_0 < x_1 < ... < x_k <1 \\
			\text{ such that } f(x_{i-1},t) \cdot f(x_{i},t) <0  \text{ for all } 1 \le  i \le k
		\end{array} \right\}
	\end{equation*}
\end{definition}

\begin{prop} \label{SturmianThm}
	Let $g = \chi(x,t)^2 dx^2 + \phi(x,t)^2 g_{S^p} + \psi(x,t)^2 g_{S^q}$ be a smooth doubly-warped Ricci flow solution on $S^{q+1} \times S^p$ defined for $t \in [0, T']$.
	$\mathfrak{Z}(\psi_s, T')$, the number of sign changes of $\psi_s( x, T')$, is no larger than $\mathfrak{Z}(\psi_s, 0)$, the number of sign changes of $\psi_s(x,0)$. 
	
	
\end{prop}
\begin{proof}
	$g$ is a smooth Ricci flow solution so say 
		$$\sup_{(x,t) \in S^{q+1} \times S^p \times [0,T']} |Rm|(x,t) \le K$$
	Moreover, $\phi$ is uniformly bounded below, since, at a spacetime minimum of $\phi$,
		$$\frac{1}{\phi^2} = \frac{1 -  \phi_s^2}{\phi^2} \le K$$
	Smoothness of the metric also implies $\chi(\cdot , 0 ) >c > 0$ for some $c$.
	Because $\chi$ satisfies
		$$\partial_t |_x \chi = \chi \left( p \frac{\phi_{ss}}{\phi} + q \frac{\psi_{ss}}{\psi} \right) \qquad \text{and} \qquad \left| p \frac{\phi_{ss} }{\phi} + q \frac{\psi_{ss}}{\psi} \right| \lesssim_{p,q} K \footnote{
		Throughout $``A \lesssim B"$ means that there exists a constant $C$ such that $A \le C B$.
		$``A \lesssim_{a,b} B"$ indicates that the constant $C = C(a,b)$ depends on $a,b$.
		$``A \sim B"$ means that $A \lesssim B \lesssim A$.
	}	$$
	it follows that $\chi$ is uniformly bounded below by a constant $c' > 0$.

	It can be computed from the system (\ref{RF}) and the commutator formula for $\left[ \left. \frac{\partial}{\partial t} \right|_x , \frac{\partial}{\partial s} \right]$ that
	$$\partial_t |_x \psi_s  = \psi_{sss} + \left( p \frac{\phi_s}{\phi}  \right)  \psi_{ss} +  \left( -p \frac{\phi_s^2}{\phi^2} + (q-1) \frac{1-\psi_s^2}{\psi^2}  + (q-2) \frac{ \psi_{ss}}{\psi} \right) \psi_s$$
	Considering $\psi_s$ as a function of $(x,t)$, $\psi_s$ satisfies an equation of the form
	\begin{equation*} \begin{aligned}
		\partial_t|_x \psi_s 
		=& \frac{1}{\chi^2} \partial_{xx} \psi_s + \frac{1}{\chi} \left(  p \frac{\phi_s}{\phi}  - \frac{ \chi_x}{\chi^2} \right)  \partial_x \psi_s \\
		&+ \left( -p \frac{\phi_s^2}{\phi^2} + (q-1) \frac{1-\psi_s^2}{\psi^2}  + (q-2) \frac{ \psi_{ss}}{\psi} \right) \psi_s
	\end{aligned} \end{equation*}
	The uniform bounds on $\phi, \phi_s, \chi$, and curvature therefore imply that $\psi_s(x,t)$ satisfies a linear parabolic equation for which the Sturmian theorem (theorem 2.1 of ~\cite{GH05}) applies.
		
\end{proof}

\begin{cor} \label{monotonePsi}
	Let $g = \chi(x,t)^2 dx^2 + \phi(x,t)^2 g_{S^p} + \psi(x,t)^2 g_{S^q}$ be a smooth doubly-warped Ricci flow solution on $S^{q+1} \times S^p$ defined for $t \in [0, T)$ which is additionally reflection symmetric.
	Then $\mathfrak{Z}(\psi_s, t) \ge 1$ for all $t \in [0,T)$.
	If, additionally, $\mathfrak{Z}(\psi_s, 0) = 1$, then, for any $t \in (0,T)$, $\psi(s, t)$ is increasing in $s$ for all $s$ in $\big( s(-1,t), s(0,t) \big)$ and decreasing in $s$ for all $s$ in $\big( s(0,t), s(-1,t) \big)$.
\end{cor}
\begin{proof}
	The first statement follows from observing that, for a reflection symmetric metric, $\psi_s$ is an odd function of $x$.

	To prove the second statement, observe from the proof of proposition \ref{SturmianThm} that $\psi_s(x,t)$ satisfies a linear evolution equation with smooth coefficients that are, in particular, bounded on any time interval $[0, T']$ with $T' < T$.
	In particular, a \textit{strong} maximum principle holds.
	Note that smoothness of the metric $g$ and reflection symmetry forces $\psi_s(-1, t) = - \psi_s(1, t) = 1$ and $\psi_s(0, t) = 0$ for all $t \in [0, T)$.
	The additional assumption of $\mathfrak{Z}(\psi_s, 0) = 1$ also implies that
	\begin{equation*}
		\psi_s(x, 0) \left\{ \begin{array}{ccl}
			\ge 0, & \quad & \text{if } x \in [-1, 0] \\
			\le 0, & & \text{if } x \in [0, 1] \\
		\end{array} \right.
	\end{equation*}
	Therefore, the strong maximum principle implies that for any $t \in (0,T)$
		\begin{equation*}
		\psi_s(x, t) \left\{ \begin{array}{ccl}
			> 0, & \quad & \text{if } x \in [-1, 0) \\
			< 0, & & \text{if } x \in (0, 1] \\
		\end{array} \right.
	\end{equation*}
	The second statement of the corollary then follows immediately.
\end{proof}

\subsection{Initial Simplifying Assumptions} \label{initSimplifyingAssumptions}
To simplify the analysis that follows, we shall make the following assumptions that will apply for the rest of the paper.
\begin{enumerate}
	\item The metrics are smooth (prior to the singularity time) doubly-warped product metrics on $S^{q+1} \times S^p$ which are additionally reflection symmetric, and
	\item $\psi_s(s, 0)$ is positive for $s \in \big( s(-1, 0), s(0,0) \big)$ and negative for $s \in \big( s(0,0), s(1,0) \big)$.
	\item $| \nabla \phi |$ and $| \nabla \psi |$ are bounded above by $1$ at $t = 0$
\end{enumerate}
In light of the second assumption and corollary \ref{monotonePsi}, the sideways coordinates $(z,u)(\psi,t)$ are valid coordinates outside the submanifold of reflection symmetry $\{ s = s(0,t) \} \subset S^{q+1} \times S^p$ which we refer to as \textit{the equator}.
Moreover, $z  = \psi_s^2$ is positive everywhere but this submanifold.

In terms of the parabolically rescaled sideways coordinates, define
\begin{equation*} \begin{aligned}
	\tl{Z}(\gamma, \tau) \doteqdot& Z(\gamma, \tau) - Z_{RFC} (\gamma) = Z(\gamma, \tau) - B^2\\
	\tl{U}(\gamma, \tau) \doteqdot& U(\gamma, \tau) - U_{RFC} (\gamma) = U(\gamma, \tau) - \log \left( \frac{A}{B} \gamma \right)
\end{aligned} \end{equation*}
Recall that $A$ and $B$ denote the constants associated to the Ricci flat cone metric
	$$A = \sqrt{ \frac{p-1}{p+q-1} } \quad \text{ and } \quad B = \sqrt{ \frac{q-1}{p+q-1} }$$
Because $(Z,U)(\gamma, \tau)$ is a solution of \ref{psRF}, $(\tl{Z}, \tl{U})(\gamma, \tau)$ solve
	\begin{equation} \label{ppsRF} 
	\begin{aligned}
		\partial_\tau|_\gamma \tl{Z} = & B^2 \mathcal{D}_Z \tl{Z} + B^2 \mathcal{N} \tl{U} + Err_Z[ \tl{Z}, \tl{U} ]\\
		\partial_\tau|_\gamma \tl{U} = & B^2 \mathcal{D}_U \tl{U} + Err_U [ \tl{Z}, \tl{U}] \\
	\end{aligned} \end{equation}
	where
	\begin{equation*} \begin{aligned}
		\mathcal{D}_Z \doteqdot& \partial_{\gamma \gamma} + \left( \frac{n-2}{\gamma} - \frac{\gamma}{2B^2} \right) \partial_\gamma - \frac{2(n-1)}{\gamma^2} \\
		\mathcal{D}_U \doteqdot& \partial_{\gamma \gamma} + \left( \frac{n}{\gamma} - \frac{\gamma}{2B^2} \right) \partial_\gamma + \frac{2(n-1)}{\gamma^2} \\
		\mathcal{N} \doteqdot& - \frac{4p B^2}{\gamma} \partial_\gamma \\
		Err_Z  [ \tl{Z}, \tl{U}] (\gamma, \tau) \doteqdot& \tl{Z} \tl{Z}_{\gamma \gamma} - \frac{1}{2} \tl{Z}_\gamma^2 - \frac{ 1}{\gamma} \tl{Z} \tl{Z}_\gamma - 2 \frac{q-1}{\gamma^2 } \tl{Z}^2 	\\
		& -2p \left\{ B^4 \tl{U}_\gamma^2 + \frac{ 4B^2}{\gamma } \tl{Z} \tl{U}_\gamma + 2 B^2 \tl{Z} \tl{U}_\gamma^2 + \frac{1}{\gamma^2} \tl{Z}^2 + \frac{2}{\gamma} \tl{Z}^2 \tl{U}_\gamma + \tl{Z}^2 \tl{U}^2_\gamma \right\}	\\
		Err_U  [ \tl{Z}, \tl{U}] (\gamma, \tau) \doteqdot& \tl{Z} \tl{U}_{\gamma \gamma} + \frac{ \tl{Z} }{\gamma} \tl{U}_\gamma + \frac{n-1}{\gamma^2} \left( 1 - e^{- 2 \tl{U} } - 2 \tl{U} \right) \\
	\end{aligned} \end{equation*}

For $M, \beta > 0$ to be determined, let $\grave{\chi}(\gamma, \tau)$ denote a bump function such that
\begin{equation*} \begin{aligned}
	\grave{\chi}(\gamma, \tau) \equiv& 1		&\text{for } \gamma \le M e^{\beta \tau} \\
	0 \le \grave{\chi}(\gamma, \tau) \le& 1	&\text{for } M e^{\beta \tau} \le \gamma \le M e^{\beta \tau}+1\\
	\grave{\chi}(\gamma, \tau) \equiv& 0		&\text{for } \gamma \ge M e^{\beta \tau} +1 \\
	| \grave{\chi}_\gamma| + | \grave{\chi}_{\gamma \gamma} | \lesssim& 1		&\text{for } M e^{\beta \tau} < \gamma < M e^{\beta \tau} + 1 \\
	| \grave{\chi}_\tau | \lesssim& M \beta e^{\beta \tau}		\qquad &\text{for } M e^{\beta \tau} < \gamma < M e^{\beta \tau} + 1 \\
\end{aligned} \end{equation*}
Note that such a function can be shown to exist by taking $\grave{\chi}(\gamma, \tau) = f( \gamma - Me^{\beta \tau} )$ for a suitably chosen bump function $f : \mathbb{R} \to \mathbb{R}$.
We localize the above differential equation (\ref{ppsRF}) by considering
	$$\grave{Z} \doteqdot \grave{\chi} \cdot \tl{Z} \qquad		\grave{U} \doteqdot \grave{\chi} \cdot \tl{U}$$
These evolve by
\begin{equation*} \begin{aligned}
	\grave{Z}_\tau =& B^2 \mathcal{D}_Z \grave{Z}  + \grave{\chi} B^2 \mathcal{N} \tl{U} + \grave{\chi}  Err_Z (\tl{Z}, \tl{U} ) + B^2 [ \grave{\chi}, \mathcal{D}_Z ] \tl{Z} + \grave{\chi}_\tau \tl{Z} \\
	\grave{U}_\tau =& B^2 \mathcal{D}_U \grave{U}  + \grave{\chi}  Err_U (\tl{Z}, \tl{U} ) + B^2 [ \grave{\chi}, \mathcal{D}_U ] \tl{U} + \grave{\chi}_\tau \tl{U} \\
\end{aligned} \end{equation*}
where the commutators are given by
\begin{equation} \label{commutators} \begin{aligned}[]
	[\grave{\chi}, \mathcal{D}_Z ] \tl{Z} =& \grave{\chi}_\gamma \left\{ -2 \tl{Z}_\gamma - \tl{Z} \left( \frac{n-2}{\gamma} - \frac{\gamma}{2B^2} \right) \right\} - \grave{\chi}_{\gamma \gamma} \tl{Z} \\
	[\grave{\chi}, \mathcal{D}_U ] \tl{U} =& \grave{\chi}_\gamma \left\{ -2 \tl{U}_\gamma - \tl{U} \left( \frac{n}{\gamma} - \frac{\gamma}{2B^2} \right) \right\} - \grave{\chi}_{\gamma \gamma} \tl{U} \\
\end{aligned} \end{equation}
Note that the commutators and $\grave{\chi}_\tau$ are supported in the region $$\{ M e^{\beta \tau} < \gamma < M e^{\beta \tau} + 1 \}$$

%% file: BoxDefn.tex
In this section, we outline the choice of initial data and the degree argument used to prove the existence of Ricci flow solutions that form conical singularities in theorem \ref{mainThmAbridged}.
The general approach and topological degree argument mirrors that used in ~\cite{HV, M04Dec} for a semilinear heat equation, ~\cite{V94} for mean curvature flow, and ~\cite{BS17} for harmonic map heat flow.
Additionally, the results in ~\cite{HV, M04Dec} have been generalized to the non-radial setting in ~\cite{Collot17}
through the use of more robust energy methods in the continuity of ~\cite{Collot18, MRR14}.
Unlike the previous cases, however, we must consider a \textit{system} of differential equations which contributes added difficulties to the analysis.
This fact also motivates our use of the degree argument rather than the Wa{\.z}ewski retraction principle ~\cite{Wazewski47}
used for other existence theorems in geometric flows ~\cite{AV97, AIK15},
though the two approaches are generally similar.

Before delving into details, we provide a broad sketch of the argument used to prove theorem \ref{mainThmAbridged}.
The linearization of (\ref{ppsRF})
at the Ricci-flat cone is given by
\begin{equation*}
	\partial_\tau \left( \begin{array}{c}
		\tl{Z} \\
		\tl{U} \\
	\end{array} \right)
	= B^2
	\left( \begin{array}{cc}
	\mathcal{D}_Z & \cl{N} \\
	0 & \cl{D}_U \\
	\end{array} \right)
	\left( \begin{array}{c}
		\tl{Z} \\
		\tl{U} \\
	\end{array} \right)
\end{equation*}
The linearization has separable solutions given by $ e^{B^2 \lambda_k \tau} ( \tl{Z}_{\lambda_k}, \tl{U}_{\lambda_k} )(\gamma)$
where 
	\begin{equation*} 
		\lambda_k \in \sigma( \cl{D}_U) \setminus \sigma( \cl{D}_Z), \quad
		\lambda_k \tl{U}_{\lambda_k} = \cl{D}_U \tl{U}_{\lambda_k}, \quad 
		\text{ and }\quad  \tl{Z}_{\lambda_k} = ( \lambda_k - \cl{D}_Z)^{-1} \cl{N} \tl{U}_{\lambda_k} 
	\end{equation*}
(see appendices \ref{LinearizedOps} and \ref{Eigenfunctions} for additional details).	
If $\lambda_k < 0$, then this separable solution limits to the Ricci-flat cone as $\tau \nearrow \infty$.
For any $k \in \mathbb{N}$ with $\lambda_k < 0$, we establish the existence of solutions $(\tl{Z}, \tl{U})(\gamma, \tau)$ to the full equation (\ref{ppsRF})
that are quantitatively close to $e^{B^2 \lambda_k \tau} ( \tl{Z}_{\lambda_k}, \tl{U}_{\lambda_k} )(\gamma)$
on parabolics scales $\gamma \sim 1$.

The existence of these solutions is accomplished via a topological degree argument.
For fixed $k \in \mathbb{N}$ with $\lambda_k < 0$,
we consider a collection of initial data $( \tl{Z}_{\ul{p}, \ul{q}}, \tl{U}_{\ul{p}, \ul{q}} )( \gamma, \tau_0)$ parametrized by a closed ball $\ol{B} \ni (\ul{p}, \ul{q})$ in $\mathbb{R}^{k+K+1}$,
and the corresponding solutions $( \tl{Z}_{\ul{p}, \ul{q}}, \tl{U}_{\ul{p}, \ul{q}})$ to (\ref{ppsRF}).
For $\tau_1 > \tau_0$, we define a function
	\begin{gather*} 
		P_{\tau_0, \tau_1} : \mathcal{W}_{\tau_0, \tau_1} \to \mathbb{R}^{k+K+1} \\
		\text{on }\cl{W}_{\tau_0, \tau_1} = 
		\left\{  (\ul{p}, \ul{q}) \in \ol{B} \quad  \left|  
		\begin{array}{c}
			( \tl{Z}_{\ul{p}, \ul{q}}, \tl{U}_{\ul{p}, \ul{q}})(\gamma, \tau) 
			\text{ stays ``quantitatively close" to } \\
			e^{B^2 \lambda_k \tau} ( \tl{Z}_{\lambda_k}, \tl{U}_{\lambda_k} )(\gamma) 
			\text{ for all } \tau \in [\tau_0, \tau_1] 
		\end{array} \right.
		\right \}
	 \end{gather*}
Informally, $P_{\tau_0, \tau_1}$ gives the amount of ``slower" eigenmodes that 
$\tl{U}_{\ul{p}, \ul{q}} - e^{B^2 \lambda_k \tau} \tl{U}_{\lambda_k}$ and 
$\tl{Z}_{\ul{p}, \ul{q}} - e^{B^2 \lambda_k \tau} \tl{Z}_{\lambda_k}$ contain.
In sections \ref{PointwiseEst} - \ref{LongTimeEsts},
we develop a priori estimates for $( \tl{Z}_{\ul{p}, \ul{q}} , \tl{U}_{\ul{p}, \ul{q}} )$ on the inner, parabolic, and outer regions.
A notable difference with previous related works is the collection of estimates for the inner region in subsection \ref{InnerRegionPointwiseEsts}
that require a more delicate application of barrier arguments.
From these a priori estimates, it is shown that the degree of $P_{\tau_0, \tau_1}$ with respect to $0 \in \mathbb{R}^{k+K+1}$ is $1$ for all $\tau_1 > \tau_0$.
The existence of the sought after solution $( \tl{Z}_{\ul{p}, \ul{q}}, \tl{U}_{\ul{p}, \ul{q}})(\gamma, \tau)$ with $( \ul{p}, \ul{q} ) \in \cl{W}_{\tau_0, \infty}$ soon follows.

\subsection{Construction of Initial Data}

We begin by formalizing what is meant by ``quantitatively close" to the separable solution on parabolic scales.		
Our conditions in the parabolic region amount to saying that the profile functions $\tl{U}, \tl{Z}$ are weighted $C^2$-close to separable solutions $e^{B^2 \lambda_k \tau} \tl{U}_{\lambda_k}(\gamma), e^{B^2 \lambda_k \tau} \tl{Z}_{\lambda_k}(\gamma)$ of the flow linearized at the Ricci-flat cone.
The reader is advised to consult appendices \ref{LinearizedOps} and \ref{Eigenfunctions} for background material on the functions $\tl{Z}_{\lambda_k}(\gamma), \tl{U}_{\lambda_k}(\gamma)$.
\begin{definition}	
	$\mathcal{B}^{\ul{\eta}}_{\tau_0, \tau_1} = \mathcal{B}\left[ \ul{\eta} = \left( \eta_0^U, \eta_1^U, \eta_2^U, \eta_0^Z, \eta_1^Z, \eta_2^Z \right), \tau_0, \tau_1, \Upsilon_{U,Z}, M ,\alpha, \beta , \lambda_k \right]$ consists of the functions $( \tl{Z}, \tl{U} ) (\gamma , \tau)$ that satisfy the following bounds for all $\tau \in [\tau_0, \tau_1]$:

	\begin{equation*} \begin{aligned}
		| \tl{U}(\gamma, \tau) - e^{ B^2 \lambda \tau} \tl{U}_\lambda(\gamma)| \le& \eta_0^U \left( \gamma^{-2k - 2B^2\lambda} + \gamma^{- 2B^2\lambda} \right) e^{B^2 \lambda \tau}
		&& \Upsilon_U e^{-\alpha \tau} < \gamma < M e^{\beta \tau}\\
		| \tl{U}_\gamma - e^{ B^2 \lambda \tau} \tl{U}'_\lambda(\gamma)| \le& \eta_1^U \left( \gamma^{-2k - 2B^2\lambda-1} + \gamma^{-2B^2\lambda} \right) e^{B^2 \lambda \tau}
		&& \Upsilon_U e^{-\alpha \tau} < \gamma < M e^{\beta \tau}\\
		| \tl{U}_{\gamma \gamma} - e^{ B^2 \lambda \tau} \tl{U}''_\lambda(\gamma) | \le& \eta_2^U \left( \gamma^{-2k -2B^2 \lambda-2} + \gamma^{-2B^2\lambda} \right) e^{B^2 \lambda \tau}
		&& \Upsilon_U e^{-\alpha \tau} < \gamma < M e^{\beta \tau}\\
		| \tl{Z} - e^{B^2 \lambda \tau} \tl{Z}_\lambda(\gamma) | \le & \eta^Z_0 \left( \gamma^{-2k - 2B^2\lambda} + \gamma^{- 2B^2\lambda } \right) e^{B^2 \lambda \tau}
		&&\Upsilon_Z e^{-\alpha \tau} < \gamma < M e^{\beta \tau}\\
		| \tl{Z}_\gamma - e^{B^2 \lambda \tau} \tl{Z}'_\lambda(\gamma)| \le & \eta^Z_1 \left( \gamma^{-2k - 2B^2\lambda-1} + \gamma^{- 2B^2\lambda} \right) e^{B^2 \lambda \tau} 
		&& \Upsilon_Z e^{-\alpha \tau} < \gamma < M e^{\beta \tau}\\
		| \tl{Z}_{\gamma \gamma}- e^{B^2 \lambda \tau} \tl{Z}''_\lambda(\gamma) | \le & \eta^Z_2 \left( \gamma^{-2k -2B^2 \lambda-2} + \gamma^{- 2B^2\lambda } \right) e^{B^2 \lambda \tau}
		&& \Upsilon_Z e^{-\alpha \tau} < \gamma < M e^{\beta \tau}\\
	\end{aligned} \end{equation*}
\end{definition}
	

\begin{definition}
	(Outer Region Conditions)
	$\mathcal{O}[ \ul{D} = \left( D_0, D_1, D_2 \right), \Gamma(\tau) , \tau_0, \tau_1]$ consists of the functions $(\tl{Z}, \tl{U} )(\gamma, \tau)$ that satisfy the following bounds for all $\tau \in [\tau_0, \tau_1]$ and $\Gamma(\tau) \le \gamma$:
	\begin{gather*}
		0 \le \tl{U}, \qquad
		\tl{U}_\gamma \le \frac{ D_0}{\gamma},   \qquad
		| Rm | \le D_2 e^{\tau}, 			\text{ and }	\\
		z^-(\gamma e^{- \tau/2} , T - e^{- \tau} ; D_0+1, D_1) \le Z(\gamma, \tau) \le 1
	\end{gather*}
	where $z^-(\psi, t ; C, K_0)$ is defined as in lemma \ref{outerSineConeBarrierZ}.
\end{definition}

The following conditions for the profile functions in the inner region are motivated by an analysis of Ricci-flat B{\"o}hm metric on $\mathbb{R}^{q+1} \times S^p$ (see section \ref{bohmMetric}).
\begin{definition} \label{innerRegConds}
	(Inner Region Conditions)\\
	$\mathcal{I}[\ul{D} = \left( D_0^U, D_0^Z, D_1, D_2, D_3 \right),a,b,c, \kappa, \epsilon,\Upsilon, \Upsilon_{U}, \Upsilon_Z, \tau_0, \tau_1]$
	consists of the functions $(\tl{Z}, \tl{U})(\gamma, \tau)$ that satisfy the following bounds for all $\tau \in [\tau_0, \tau_1]$:
	
	(Inner Region Barriers I)
	\begin{equation*} \begin{aligned}
		0 \le \tl{U} & \qquad 
		&& \text{for all } 0 \le \gamma \le \Upsilon_U e^{-\alpha \tau}\\ 
		-\frac{1}{\gamma} \le \tl{U}_\gamma \le 0 &
		&& \text{for all } 0 \le \gamma \le \Upsilon_U e^{-\alpha \tau}\\ 
		\\
		0 \le \tl{U}_{\gamma \gamma}  + \frac{\tl{U}_\gamma}{\gamma}  &
		&& \text{for all } 0 \le \gamma \le \Upsilon_U e^{-\alpha \tau}\\ 
		\\
		0 \le \tl{Z}(\gamma, \tau) \le 1-B^2&
		&& \text{for all } 0 \le \gamma \le \Upsilon_Z e^{-\alpha \tau}\\ 
		\\
	\end{aligned} \end{equation*}
	
	(Weighted $C^2$ Bounds)
	\begin{equation*} \begin{aligned}
		| \gamma \tl{U}_\gamma | + | \gamma^2 \tl{U}_{\gamma \gamma} | \le D_0^U 	
		\qquad & \text{for all } 0 < \gamma \le \Upsilon_U e^{-\alpha \tau} \\
		| \tl{Z} | + | \gamma \tl{Z}_\gamma | + | \gamma^2 \tl{Z}_{\gamma \gamma} | \le D_0^Z 	
		\qquad & \text{for all } 0 < \gamma \le \Upsilon_Z e^{-\alpha \tau} \\
	\end{aligned} \end{equation*}
	
	(Inner Region Barriers II)
	\begin{equation*} \begin{aligned}
		\tl{U} &\le D_1 \Upsilon^{a- 2k -2B^2 \lambda_k} (\gamma e^{\alpha \tau})^{-a} 
		&&  \text{for all } 0 < \gamma \le \Upsilon_U e^{-\alpha \tau} \\
		\tl{Z} &\le D_2 \Upsilon^{\frac{b}{a}(a-2k-2B^2 \lambda_k) + \epsilon} (\gamma e^{\alpha \tau})^{-b}  
		&&  \text{for all } 0 < \gamma \le \Upsilon_Z e^{-\alpha \tau}\\
		\tl{U} &\le D_3 \Upsilon^{c} (\gamma e^{\alpha \tau})^{-2k-2B^2 \lambda_k - \kappa} 
		&&  \text{for all } 0 < \gamma \le \Upsilon_U e^{-\alpha \tau}\\
	\end{aligned} \end{equation*}
\end{definition}

\begin{definition}
	Define $\mathcal{G}$ to be the set of smooth Riemannian metrics $g$ on $S^{q+1} \times S^p$ such that 
	\begin{enumerate}
		\item $g$ is a doubly-warped product metric,
			$$g = ds^2 + \phi^2 g_{S^p} + \psi^2 g_{S^q}		\qquad (s = s(x))$$
		\item $g$ is reflection symmetric across the equator $s = s(0)$, and
		\item the warping function $\psi$ in monotone in the region $s(-1) < s < s(0)$ and so 
		the sideways coordinates $(z, u)(\psi)$ are valid coordinates away from the equator $s = s(0)$.
	\end{enumerate}
\end{definition}

We are now ready to define our choice of initial data.

\begin{definition} \label{assignmentOfInitData}
	For constants $k, a,b,c, \kappa, \epsilon, \epsilon_0, \ul{D}^{\mathcal{I}}= \left(D^{\mathcal{I}}\right)_{0,1,2,3}^{U, Z}, \ul{D}^{\mathcal{O}} = \left( {D}^{\mathcal{O}} \right)_{0,1,2},$ $ \ul{\Upsilon}, \Upsilon_U, M, \ol{\beta} , \tau_0$,
	define a collection of Riemannian metrics $g_{\ul{p}, \ul{q}} \in \mathcal{G}$ parametrized by 
	$$(\ul{p}, \ul{q}) \in B_{\epsilon_0 e^{B^2 \lambda_k \tau_0}} (0) \subset \mathbb{R}^k \times \mathbb{R}^{K+1}$$
	 with the following properties:
	\begin{enumerate}
		\item (Perturbation of Cone in Parabolic Region)
			$$\tl{U}_{\ul{p}, \ul{q}}(\gamma, \tau_0) = \sum_{j=0}^{k-1} p_j \tl{U}_{\lambda_j} + e^{B^2 \lambda_k \tau_0} \tl{U}_{\lambda_k}		\qquad \text{for all } \ul{\Upsilon} e^{-\alpha \tau_0} < \gamma < M e^{\ol{\beta} \tau_0} $$
			$$\tl{Z}_{\ul{p}, \ul{q}}(\gamma, \tau_0) = \sum_{j=0}^{K} q_j \tl{Z}_{h_j} + e^{B^2 \lambda_k \tau_0} \tl{Z}_{\lambda_k}		\qquad \text{for all } \ul{\Upsilon} e^{-\alpha \tau_0} < \gamma < M e^{\ol{\beta} \tau_0}$$
			where $K = K(k)$ is the index of the eigenvalue such that $h_{K+1} < \lambda_k < h_K$.
			
		\item (Inner Region)
				$$\left( \tl{Z}_{\ul{p}, \ul{q}}, \tl{U}_{\ul{p}, \ul{q}}\right) (\gamma, \tau_0) \in \mathcal{I}[ \ul{D}^{\mathcal{I}}, a,b,c, \kappa, \epsilon,\Upsilon_U, \ul{\Upsilon}, \ul{\Upsilon}, \tau_0, \tau_0] $$

		\item (Outer Region)
			$$\left( \tl{Z}_{\ul{p}, \ul{q}}, \tl{U}_{\ul{p}, \ul{q}}\right) (\gamma, \tau_0) \in \mathcal{O}[ \ul{D}^{\mathcal{O}}, M e^{\ol{\beta} \tau_0}, \tau_0, \tau_0]$$
	
	\end{enumerate} 
	where 
	$$( \tl{Z}_{\ul{p}, \ul{q}}, \tl{U}_{\ul{p}, \ul{q}}) = ( {Z}_{\ul{p}, \ul{q}}, {U}_{\ul{p}, \ul{q}}) - (Z_{RFC}, U_{RFC}) $$ 
	denote the warping functions of the metric $g_{\ul{p}, \ul{q}}$ with respect to the parabolically rescaled sideways coordinates (see sections \ref{rescaledSidewaysCoords} and \ref{initSimplifyingAssumptions} for reference). 
\end{definition}

\begin{lem}
	Let $k$ be an even integer for which $\lambda_k < 0$, 
	let $0  < \ol{\beta} < \frac{1}{2}$, and
	let $a,b,c, \kappa, \epsilon$ be as in remark \ref{constantsToUse}.
	There exist $\ul{D} = ( \ul{D}^{\mathcal{I}}, \ul{D}^{\mathcal{O}}) $ depending on $p,q,k$ such that
	an assignment of initial metrics $g_{\ul{p}, \ul{q}}(\tau_0) \in \mathcal{G}$ as in the above definition \ref{assignmentOfInitData} exists if
	$\ul{\Upsilon} \gg 1$ is sufficiently large depending on $p,q,k$,
	$\Upsilon_U \gg 1$ is sufficiently large depending on $p,q,k, a, \ul{\Upsilon}$, 
	and $\tau_0 \gg 1$ is sufficiently large depending on $p,q,k, \ul{\Upsilon} , M , \ol{\beta}, \ul{D}$.
\end{lem}
\begin{proof}
	It suffices to show that the interior of the set of metrics satisfying the conditions in definition \ref{assignmentOfInitData} is nonempty.
	Begin by letting 
	\begin{equation*} \begin{aligned}
		\tl{U}(\gamma) = \tl{U}_{\ul{0}, \ul{0}}(\gamma, \tau_0) &=  e^{B^2 \lambda_k \tau_0} \tl{U}_{\lambda_k}\\
		\text{and } \tl{Z}(\gamma) = \tl{Z}_{\ul{0}, \ul{0}}(\gamma, \tau_0) &=  e^{B^2 \lambda_k \tau_0} \tl{Z}_{\lambda_k} \\
		\text{for all } &\ul{\Upsilon} e^{- \alpha \tau_0} \le \gamma \le M e^{\ol{\beta} \tau_0}
	\end{aligned} \end{equation*}
	We wish to show that, for suitably chosen parameters, we can extend these across the inner \& outer region in such a way that they satisfy the desired inequalities and define smooth metrics $g \in \mathcal{G}$.
	
	(Inner Region Extension):
	Let $f : \mathbb{R} \to \mathbb{R}$ be a $C^\infty$ function such that
	\begin{equation*} \begin{aligned}
		 &f(\gamma) \ge 0  
		 && \text{for all } \gamma \in \left [ 0 , \ul{\Upsilon} e^{- \alpha \tau_0} \right] \\
		 &0 = f'(0) = f'''(0) = ... = f^{(\text{odd})}(0) \\
		& f( \gamma ) = \partial_{\gamma \gamma} \tl{U} + \frac{1}{\gamma} \partial_\gamma \tl{U}
		&& \text{for all } \gamma \in \left[ \ul{\Upsilon} e^{ - \alpha \tau_0}, M e^{ \ol{\beta} \tau_0} \right] \\
		& -1 + \int_0^{\ul{\Upsilon} e^{- \alpha_k \tau_0}} \ol{\gamma} f ( \ol{\gamma} ) d \ol{\gamma} 
		= \ul{\Upsilon} e^{- \alpha_k \tau_0} \tl{U}_\gamma (  \ul{\Upsilon} e^{- \alpha_k \tau_0} ) \\
	\end{aligned} \end{equation*}
	By \ref{eigfuncUAsymps}, such functions exist if
	$\ul{\Upsilon} \gg 1$ is sufficiently large depending on $p,q,k$ and
	$\tau_0 \gg 1$ is sufficiently large depending on $p,q,k, \ul{\Upsilon}$.
	Define the extension of $\tl{U}$ to the inner region by letting $\tl{U}$ be the unique solution to the initial value problem
	\begin{equation*} \begin{aligned}
		\partial_{\gamma \gamma} u + \frac{1}{\gamma} u_\gamma &= f(\gamma)	
		&& \text{for all } \gamma \in \left( 0 , \ul{\Upsilon} e^{- \alpha_k \tau_0} \right)\\
		(u, u_\gamma)  &= (\tl{U}, \tl{U}_\gamma)
		&& \text{ at } \gamma = \ul{\Upsilon} e^{- \alpha \tau_0}
	\end{aligned} \end{equation*}
	Note that $\tl{U} : \left[ 0 , M e^{\ol{\beta} \tau_0} \right] \to \mathbb{R}$ is smooth and its odd order derivatives vanish at $\gamma = 0$. 
	
	\noindent We now confirm that this extension satisfies ``Inner Region Barriers I":
	\begin{enumerate}
		\item 
		$\partial_{\gamma \gamma} \tl{U} + \frac{1}{\gamma} \tl{U}_\gamma = f(\gamma) \ge 0$ for all 
		$\gamma \in \left[ 0, \ul{\Upsilon} e^{- \alpha_k \tau_0} \right]$.
		
		\item 
		For $\gamma \in \left[ 0, \ul{\Upsilon} e^{- \alpha_k \tau_0} \right]$, 
		$$\gamma \tl{U}_\gamma = -1 + \int_0^\gamma \ol{\gamma} f ( \ol{\gamma} ) d \ol{\gamma} \ge -1$$
		Hence, the bound
			$$\tl{U}_\gamma \ge - \frac{1}{\gamma}$$
		is automatic. Additionally, for $\gamma \le \ul{\Upsilon} e^{- \alpha \tau_0}$
		\begin{equation*} \begin{aligned}
			& \quad \int_0^\gamma \ol{\gamma} f(\ol{\gamma} ) d \ol{\gamma} \\
			& \le \int_0^{ \ul{\Upsilon} e^{- \alpha \tau_0} } \ol{\gamma } f( \ol{\gamma} ) d \ol{\gamma} \\
			&= 1 + \ul{\Upsilon} e^{- \alpha \tau_0} \tl{U}_\gamma (  \ul{\Upsilon} e^{ - \alpha \tau_0} ) \\
			& \le 1 + C_{p,q,k} ( - 2k - 2B^2 \lambda_k ) \ul{\Upsilon}^{-2k -2B^2 \lambda} \\
		\end{aligned} \end{equation*}
		if $\tau_0 \gg 1$ is sufficiently large depending on $p,q,k, \ul{\Upsilon}$.
		For $k$ even, we can take $C_{p,q,k}$ to be positive by \ref{eigfuncUAsymps}.
		Hence,
			$$C_{p,q,k} ( - 2k - 2B^2 \lambda_k ) \ul{\Upsilon}^{-2k -2B^2 \lambda} < 0$$
		and so $\gamma \tl{U}_\gamma \le 0$.
		
		\item 
		$\tl{U}_\gamma \le 0$ and $\tl{U}(\ul{\Upsilon} e^{- \alpha \tau_0}) \ge 0$ implies that $\tl{U} \ge 0$ for all $\gamma \le \ul{\Upsilon} e^{- \alpha \tau_0}$.
	\end{enumerate}
	
	\noindent Next, we confirm the ``Weighted $C^2$ Bounds":
	\begin{enumerate}
		\item 
		$| \gamma \tl{U}_\gamma| \le 1$ follows automatically from ``Inner Region Barriers I."
		
		\item 
		For the second derivative estimate, first note that 
		\begin{gather*}
			| \gamma^2 \tl{U}_{\gamma \gamma} | 
			\le | \gamma \tl{U}_\gamma | + | \gamma^2 f | 
			\le 1 + | \gamma^2 f| 
			\le 1 + (\ul{\Upsilon} e^{- \alpha \tau_0})^2 \sup_{\gamma \in \left[ 0, \ul{\Upsilon} e^{-\alpha \tau_0} \right] } | f( \gamma) |
		\end{gather*}
		Also, observe that
			$$1 + C_{p,q,k}' \ul{\Upsilon}^{-2k -2B^2\lambda} = \int_0^{\ul{\Upsilon} e^{- \alpha \tau_0} } \ol{\gamma} f(\ol{\gamma}) d \ol{\gamma} \le \frac{1}{2} \ul{\Upsilon}^2 e^{- 2\alpha \tau_0} \sup | f |$$
		and that
			$$f( \ul{\Upsilon} e^{- \alpha \tau_0} ) \sim_{p,q,k} \ul{\Upsilon}^{-2k-2B^2 \lambda-2} e^{2 \alpha \tau_0}$$
		if $\tau_0 \gg 1$ is sufficiently large depending on $p,q,k, \ul{\Upsilon}$.
		It follows that $f$ may be chosen so that
			$$(\ul{\Upsilon} e^{- \alpha \tau_0})^2 \sup_{\gamma \in \left[ 0, \ul{\Upsilon} e^{-\alpha \tau_0} \right] } | f( \gamma) | \lesssim_{p,q,k} 1$$
		if $\tau_0 \gg 1$ is sufficiently large depending on $p,q,k, \ul{\Upsilon}$.
	\end{enumerate}
	
	Finally, we check the ``Inner Region Barriers II" condition:\\
	Let $V(\gamma) \doteqdot \Upsilon_U^{a-2k-2B^2 \lambda} ( \gamma e^{\alpha \tau_0} )^{-a}$.
	By \ref{eigfuncUAsymps}, there exists $C = C(p,q,k) > 0$ such that 
		$$\tl{U}( \ul{\Upsilon} e^{- \alpha_k \tau_0} ) \le C \ul{\Upsilon}^{-2k -2B^2 \lambda_k}$$
	if $\tau_0 \gg 1$ is sufficiently large depending on $p,q,k, \ul{\Upsilon}$.
	Taking $D_1 = C(p,q,k)$, it follows that
		$$D_1 V \left( \ul{\Upsilon} e^{- \alpha_k \tau_0} \right) = D_1 \Upsilon_U^{a - 2k -2B^2 \lambda_k} \ul{\Upsilon}^{-a} \ge D_1 \ul{\Upsilon}^{-2k -2B^2 \lambda} \ge \tl{U}\left( \ul{\Upsilon} e^{- \alpha_k \tau_0} \right)$$
	if $0 < 2k + 2B^2 \lambda_k < a$ and $\ul{\Upsilon} \le \Upsilon_U$.
	
	Now, it suffices to show that 
		$$\gamma D_1V_\gamma < \gamma \tl{U}_\gamma = -1 + \int_0^\gamma \ol{\gamma} f d \ol{\gamma} \in [-1, 0] \qquad   \text{for all }\gamma \in \left( 0 , \ul{\Upsilon} e^{- \alpha \tau_0} \right]$$
	We have
			$\gamma D_1 V_\gamma = -a D_1 \Upsilon_U^{a-2k-2B^2 \lambda} ( \gamma e^{\alpha \tau_0} )^{-a}$
	which is negative and decreasing in $\gamma$.
	At $\gamma = \ul{\Upsilon} e^{- \alpha \tau_0}$, 
			$$|\gamma D_1 V_\gamma| = a D_1 \ul{\Upsilon}^{-a} \Upsilon_U^{a-2k -2B^2 \lambda} > 1$$
	if $2k + 2B^2 \lambda_k < a$ and
	$\Upsilon_U \gg 1$ is sufficiently large depending on $p,q,k, a, \ul{\Upsilon}$.
	It follows that $\gamma D_1V_\gamma < \gamma \tl{U}_\gamma$
	and thus
	$\tl{U}(\gamma) < D_1 V$ for all $\gamma \in \left[ 0, \ul{\Upsilon} e^{-\alpha_k \tau_0} \right]$.
	
	A similar argument can be used to show that there exists
	$D_3$ depending on $p,q,k$ such that
		$$\tl{U} < D_3 \Upsilon_U^c ( \gamma e^{\alpha_k \tau_0} )^{-2k-2B^2 \lambda_k -c}	\qquad \text{for all } \gamma \le \ul{\Upsilon} e^{- \alpha_k \tau_0}$$
	if $0 < \kappa < c$,
	$\ul{\Upsilon} \gg 1$ is sufficiently large depending on $p,q,k, c, \kappa$,
	$\Upsilon_U \gg 1$ is sufficiently large depending on $p,q,k, c, \kappa, \ul{\Upsilon}$,
	and $\tau_0 \gg 1$ is sufficiently large depending on $p,q,k, \ul{\Upsilon}$.
	
	A similar argument may be applied to define the extension of $\tl{Z}$ to the region $\gamma \in \left[ 0, \ul{\Upsilon} e^{- \alpha_k \tau_0} \right]$.\\
	
	(Outer Region Extension): 
	For the extension to the outer region, take $D_0 = 1$ 
	and $D_1$ 
	small enough depending on $p,q,k$ 
	such that
		$$\frac{D_1}{1 - D_1 ( q + (1+D_0)^2 p) T } > 0$$
	In particular, $\frac{D_1}{1 - D_1 ( q + 4 p) t } > 0$ for all $t < T$.
	
	Let $\psi_*$ be defined implicitly as
		$$z^- ( \psi_* , t; D_0+1, D_1 ) = 0$$
	Begin by extending $z(\psi, t) = Z(\gamma, \tau)$ in such a way so that
		$$0 \le z^- ( \psi, t; D_0+1, D_1) < z(\psi, t) \le 1		\qquad \text{for all } \psi \in \left[ M e^{\ol{\beta} \tau_0 - \tau_0/2} , \psi_* \right]$$
	Note that $M e^{\ol{\beta} \tau_0 - \tau_0/2} < \psi_*$ if $\tau_0 \gg 1$ is sufficiently large depending on $p,q,k, M ,\ol{\beta}, D_0, D_1$.

	Next, recover $\psi(s, \tau_0)$ from $z(\psi, \tau_0)$ via the formula
		$$\psi(s) = \int_0^s \sqrt{ z( \psi(\sigma) ) } d \sigma$$
	Note that $\psi$ is a strictly increasing function of $s$ and that this equality is valid so long as $\psi \le \psi_*$.
	Let $s_*$ be defined such that $\psi(s_*) = \psi_*$.
	Proceed to extend $\psi(s)$ to $ s \in [ s_*, s_* + 1]$ in such a way so that
		$$0 < \psi_s \le 1 	\qquad \text{for all } s \in [ s_* , s_* + 1] $$
		$$0 = \psi_s(s_{*}+1) = \psi_{sss} (s_*+1) =... = \partial_s^{(\text{odd})} \psi ( s_* + 1) $$
		$$\psi_{ss}(s) \text{ is uniformly bounded for } s \in [ s_*, s_*+1]$$

	Now, let $s_0$ denote the unique value of $s$ such that
		$$\psi(s_0) = M e^{\ol{\beta} \tau_0 - \tau_0/2}$$
	Extend $\phi(s)$ by letting $\phi$ solve
		$$\frac{ \phi_s(s)}{\phi(s) } = f(s) \frac{ \psi_s(s)}{\psi(s)} 	\qquad \text{for all } s \in [ s_0, s_{*}+1]$$
	for some smooth function $f: \mathbb{R} \to \mathbb{R}$ (distinct from the inner region function $f$ above)
	satisfying 
	\begin{equation*} \begin{aligned}
	&1 \le f(s) \le 2 && \text{for all } s \in [s_0, s_*+1] \\
	 & \frac{ f(s) }{ \psi(s)} =  [\partial_\psi u]|_{\psi(s)}  && \text{for all } s \in [ 0, s_0 ] \\
	 &0 = f'(s_* + 1) = f'''(s_* + 1) = ... = f^{(\text{odd})}(s_* + 1), && \text{and} \\
	 &f'(s) \text{ is uniformly bounded for } s \in [s_0, s_* + 1] 
	\end{aligned} \end{equation*}
	It follows that
	\begin{equation*} \begin{aligned}
		&\partial_s \log(\phi) = \frac{ \phi_s}{\phi} \ge \frac{ \psi_s}{\psi} \ge \partial_s \log(\psi) \\
		\implies &\log( \phi(s) ) \ge \log( \psi(s) ) + \log \left( \frac{ \phi(s_0) }{\psi(s_0)} \right) 	\qquad \text{for all } s \ge s_0\\ 
	\end{aligned} \end{equation*}
	The parabolic region data imply that 
		$$ \log \left( \frac{ \phi(s_0) }{\psi(s_0)} \right) \ge \log \left(  \frac{A}{B} \right)$$
	if $k$ is even and $\tau_0 \gg 1$ is sufficiently large.
	Hence,
		$$\phi(s) \ge \frac{A}{B} \psi(s) 		\qquad \text{for all } s \ge s_0$$
	Additionally, $1 \le f(s) \le 2$ implies that $\partial_\gamma U \le \frac{2}{\gamma}$ for $\gamma \ge M e^{\ol{\beta} \tau_0}$. 
	
	The uniform bounds on $f_s$ and $\psi_{ss}$ together with \ref{eigfuncUAsymps} and \ref{ZlambdaAsymps} imply that 
	for some constant $D_2 = D_2(p,q,k)$ 
	the estimate $|Rm| \le D_2 e^{\tau_0}$ holds for all $\gamma \ge M e^{\ol{\beta} \tau_0}$ 
	when $\tau_0 \gg 1$ is sufficiently large depending on $p,q,k, M , \ol{\beta}$ (see also \ref{outerCurvatureBounds}).
	
	By construction, 
		$$\partial_s^{(\text{odd})} \phi (0) = 0 \qquad \partial_s^{(\text{even})} \psi(0) = 0$$
		$$\partial_s^{(\text{odd})} \phi (s_*+1) = \partial_s^{(\text{odd})} \psi(s_*+1) = 0$$
	Hence, the metric defined by the warping functions $\phi, \psi$ (evenly reflected about $s_{*}+1$) defines a smooth metric on $S^{q+1} \times S^p$.
\end{proof}

Let
	$$\grave{Z}_{\ul{p}, \ul{q}} = \grave{\chi} \tl{Z}_{\ul{p}, \ul{q} } \qquad 		\grave{U}_{\ul{p} , \ul{q} } = \grave{\chi} \tl{U}_{\ul{p}, \ul{q} }$$
denote the localized versions of the perturbations.

\begin{definition}
	For $\tl{Z}_{\ul{p}, \ul{q}} , \tl{U}_{\ul{p}, \ul{q}}$ as in definition \ref{assignmentOfInitData} and $\tau_1 > \tau_0$,
	let
	$$\mathcal{W}_{\tau_0, \tau_1} \subset B_{\epsilon_0 e^{B^2 \lambda_k \tau_0}} \subset \mathbb{R}^k \times \mathbb{R}^K$$
	denote the set of $(\ul{p}, \ul{q})$ such that $\tl{Z}_{\ul{p}, \ul{q}} , \tl{U}_{\ul{p}, \ul{q}} \in \mathcal{B}_{\tau_0, \tau_1}^{\ul{\eta}}$.
	Define 
	$$P_{\tau_0, \tau_1} : \mathcal{W}_{\tau_0, \tau_1} \to \mathbb{R}^{k} \times \mathbb{R}^K$$
	\begin{equation*}
	P_{\tau_0, \tau_1}( \ul{p}, \ul{q} ) \doteqdot 
	\left( \begin{array}{ccc}
		\langle \grave{U}_{\ul{p}, \ul{q}}(\tau_1), \tl{U}_{\lambda_0} \rangle_{L^2_{n,\frac{1}{2B^2} }}& ...&  \langle \grave{U}_{\ul{p}, \ul{q}}(\tau_1), \tl{U}_{\lambda_{k-1}} \rangle_{L^2_{n,\frac{1}{2B^2}}} \\
		\\
	\langle \grave{Z}_{\ul{p}, \ul{q}}(\tau_1) - e^{B^2 \lambda_k \tau_1} \tl{Z}_{\lambda_k}, \tl{Z}_{h_0} \rangle_{L^2_{n-2,\frac{1}{2B^2}}} &...&  \langle \grave{Z}_{\ul{p}, \ul{q}}(\tau_1) - e^{B^2 \lambda_k \tau_1} \tl{Z}_{\lambda_k}, \tl{Z}_{h_K} \rangle_{L^2_{n-2,\frac{1}{2B^2}}}
	\end{array} \right)
	\end{equation*}
\end{definition}

\begin{definition}
	\begin{equation*} \begin{aligned}
	\check{\grave{Z}}(\gamma, \tau_0) 
	=& e^{-B^2 \lambda_k \tau_0} \left[ \grave{Z}(\gamma, \tau_0) - \sum_{j=0}^K q_j \tl{Z}_{h_j}(\gamma) \right] \\
	\check{\grave{U}}(\gamma, \tau_0) 
	=& e^{-B^2 \lambda_k \tau_0 }\left[ \grave{U}(\gamma, \tau_0) - \sum_{j=0}^{k-1} p_j \tl{U}_{\lambda_j}(\gamma) \right] \\ 
	\end{aligned} \end{equation*} 
	
	Observe that the choice of initial data implies
	\begin{equation*} \begin{aligned}
		\check{\grave{Z}}(\gamma, \tau_0) =& \tl{Z}_{\lambda_k}(\gamma)   \qquad &\ul{\Upsilon} e^{- \alpha \tau_0} < \gamma < M e^{\beta \tau_0}\\
		\check{\grave{U}}(\gamma, \tau_0) =& \tl{U}_{\lambda_k}(\gamma) 		&\ul{\Upsilon} e^{- \alpha \tau_0} < \gamma < M e^{\beta \tau_0}
	\end{aligned} \end{equation*}
\end{definition}

\subsection{The Degree Argument}

\begin{lem}  \label{zerosInTheInterior}
	Let $\ul{\Upsilon}, \Upsilon_U, \Upsilon_Z, \tau_0 \gg 1$ be sufficiently large.
	If $(\ul{p}, \ul{q}) \in \ol{\mathcal{W}}_{\tau_0, \tau_1} \cap P_{\tau_0, \tau_1}^{-1} (0)$ for some $\tau_1 > \tau_0$,
	then
	$( \tl{Z}_{\ul{p}, \ul{q}}, \tl{U}_{\ul{p}, \ul{q}} ) \in \mathcal{B}_{\tau_0, \tau_1}^{ \frac{1}{2} \ul{\eta} }$.
\end{lem}

\noindent The proof of this lemma relies on a collection of estimates that we postpone to the later sections of the paper.
Assuming these estimates, we provide the proof of lemma \ref{zerosInTheInterior}.
\begin{proof}
	Pick $\Gamma \gg 1$ sufficiently large depending on $p,q,k, \ul{\eta}$ such that 
	\begin{equation*} \begin{aligned}
		\tl{Z}^\pm \doteqdot \left( C^Z_{p,q,k} \pm \frac{1}{C_{p,q,k}'^Z} \eta_0^Z \right) e^{B^2 \lambda \tau} \gamma^{-2B^2 \lambda} + D_{\pm}^Z e^{B^2 \lambda \tau} \gamma^{-2B^2 \lambda -1} \\
		\tl{U}^\pm \doteqdot \left( C_{p,q,k}^U \pm \frac{1}{C_{p,q,k}'^U} \eta_0^U \right) e^{B^2 \lambda \tau} \gamma^{-2B^2 \lambda} + D_{\pm}^U e^{B^2 \lambda \tau} \gamma^{-2B^2 \lambda -1} \\
	\end{aligned} \end{equation*}
	are barriers for $\tl{Z}, \tl{U}$ respectively on the domain $\frac{1}{2} \Gamma \le \gamma \le 2 M e^{\beta \tau}$
	satisfying 
		$$\tl{Z}^-(\gamma, \tau) \le \tl{Z}^+(\gamma, \tau) \quad \text{ and } \quad \tl{U}^-( \gamma, \tau) \le \tl{U}^+(\gamma, \tau)	\qquad \text{for all } \frac{1}{2} \Gamma \le \gamma \le 2M e^{\beta \tau}$$
	and where the constants $C_{p,q,k}'^{Z,U}$ are chosen such that the interior estimates \ref{parabOuterInteriorEstU+}, \ref{parabOuterInteriorEstZ+} yield
	\begin{gather*}
		\tl{Z}^- \le \tl{Z} \le \tl{Z}^+ \quad \text{ and } \quad
		\tl{U}^- \le \tl{U} \le \tl{U}^+ 
		\qquad \text{for all } \frac{1}{2} \Gamma \le \gamma \le 2 M e^{\beta \tau}\\
		\text{imply}  \quad 
		( \tl{Z}, \tl{U} ) \in \mathcal{B} \left[ \frac{1}{2} \ul{\eta} , \tau_0 , \tau_1, \Gamma, M , \alpha, \beta, \lambda_k \right]
	\end{gather*}
	 
	 Next, 
	 if $\ul{\Upsilon}, \Upsilon_U, \Upsilon_Z, \tau_0 \gg 1$ are sufficiently large then the estimates in sections \ref{ShortTimeEsts}, \ref{LongTimeEsts} imply that
	 	$$( \tl{Z}, \tl{U} ) \in \mathcal{B} \left[ \left( \frac{1}{C_{p,q,k, \Gamma}''} \eta_0^U, \eta_{1,2}^U, \eta_{0,1,2}^Z \right), \frac{1}{C_\Gamma} \Upsilon_U, \frac{1}{C_\Gamma} \Upsilon_Z , 4 \Gamma, \alpha, 0, \lambda_k \right]$$ 
	 The interior estimate \ref{parabInteriorEstU} then imply that
	 	$$( \tl{Z}, \tl{U} ) \in \mathcal{B} \left[ \left( \frac{1}{C_{p,q,k, \Gamma}'''} \eta_{0,1,2}^U, \eta_{0,1,2}^Z \right), \frac{1}{C_\Gamma'} \Upsilon_U, \frac{1}{C_\Gamma} \Upsilon_Z , 2 \Gamma, \alpha, 0, \lambda_k \right]$$
	Here, the constant $C_{p,q,k,\Gamma}'''$ is such that the above estimate implies
		$$( \tl{Z}, \tl{U} ) \in \mathcal{B} \left[ \left( \frac{1}{C_{p,q,k, \Gamma}'''} \eta_{0,1,2}^U, \frac{1}{C_{p,q,k, \Gamma}''''} \eta_0^Z, \eta_{1,2}^Z \right), \frac{1}{C_\Gamma'} \Upsilon_U, \frac{1}{C_\Gamma} \Upsilon_Z , 2 \Gamma, \alpha, 0, \lambda_k \right]$$
	if $\ul{\Upsilon}, \Upsilon_U, \Upsilon_Z, \tau_0 \gg 1$ are sufficiently large.
	Moreover, the constant $C_{p,q,k, \Gamma}''''$ is such that
	the interior estimates \ref{parabInteriorEstZ} then yield that
		$$( \tl{Z}, \tl{U} ) \in \mathcal{B} \left[ \left( \frac{1}{C_{p,q,k, \Gamma}'''} \eta_{0,1,2}^U, \frac{1}{C_{p,q,k, \Gamma}'} \eta_{0,1,2}^Z \right), \frac{1}{C_\Gamma'} \Upsilon_U, \frac{1}{C_\Gamma'} \Upsilon_Z ,  \Gamma, \alpha, 0, \lambda_k \right]$$
	 
	 In particular, 
	 	$$( \tl{Z}, \tl{U} ) \in \mathcal{B} \left[ \left( \frac{1}{C_{p,q,k, \Gamma}'} \eta_0^{U, Z}, \eta_{1,2}^{U,Z} \right), \frac{1}{C_\Gamma'} \Upsilon_U, \frac{1}{C_\Gamma'} \Upsilon_Z , \frac{1}{2} \Gamma, \alpha, 0, \lambda_k \right]$$	
	so the barriers $\tl{Z}^{\pm}, \tl{U}^\pm$ imply that that these estimates hold in fact for
	$\frac{1}{C_\Gamma'} \Upsilon_{U, Z} e^{- \alpha \tau} \le \gamma \le 2 M e^{\beta \tau}$, that is
		$$( \tl{Z}, \tl{U} ) \in \mathcal{B} \left[ \left( \frac{1}{C_{p,q,k, \Gamma}'} \eta_0^{U, Z}, \eta_{1,2}^{U,Z} \right), \frac{1}{C_\Gamma'} \Upsilon_U, \frac{1}{C_\Gamma'} \Upsilon_Z , 2M, \alpha, \beta, \lambda_k \right]$$
	Interior estimates \ref{parabInteriorEstU}, \ref{parabInteriorEstZ}, \ref{parabOuterInteriorEstU+}, \ref{parabOuterInteriorEstZ+} then yield that
		$$(\tl{Z}, \tl{U}) \in \mathcal{B}\left[ \frac{1}{2} \ul{\eta}, \Upsilon_U, \Upsilon_Z, M, \alpha, \beta, \lambda_k \right]$$
\end{proof}

\begin{lem} \label{degree1}
	Let $\ul{\Upsilon}, \Upsilon_U, \Upsilon_Z, \tau_0 \gg 1$ be sufficiently large.
	If $\tau_1 > \tau_0$ and $\mathcal{W}_{\tau_0, \tau_1} \ne \emptyset$, 
	then the degree of $P_{\tau_0, \tau_1}$ in $\mathcal{W}_{\tau_0, \tau_1}$ with respect to $0 \in \mathbb{R}^k \times \mathbb{R}^K$ is 1.
\end{lem}
\begin{proof}
	Begin by noting that $\mathcal{W}_{\tau_0, \tau_0} = B_{\epsilon_0 e^{B^2 \lambda_k \tau_0}}$ and
	\begin{equation*} \begin{aligned}
		&P_{\tau_0, \tau_0} ( \ul{p}, \ul{q} )  = ( \ul{p}, \ul{q} ) \\
		&+ 
		e^{B^2 \lambda_k \tau_0} 
		\left( \begin{array}{ccc}
		\langle \check{\grave{U}}_{\ul{p}, \ul{q}}(\tau_1), \tl{U}_{\lambda_0} \rangle_{L^2_{n,\frac{1}{2B^2} }}& ...&  \langle \check{\grave{U}}_{\ul{p}, \ul{q}}(\tau_1), \tl{U}_{\lambda_{k-1}} \rangle_{L^2_{n,\frac{1}{2B^2}}} \\
		\\
	\langle \check{\grave{Z}}_{\ul{p}, \ul{q}}(\tau_1) - e^{B^2 \lambda_k \tau_1} \tl{Z}_{\lambda_k}, \tl{Z}_{h_0} \rangle_{L^2_{n-2,\frac{1}{2B^2}}} &...&  \langle \check{\grave{Z}}_{\ul{p}, \ul{q}}(\tau_1) - e^{B^2 \lambda_k \tau_1} \tl{Z}_{\lambda_k}, \tl{Z}_{h_K} \rangle_{L^2_{n-2,\frac{1}{2B^2}}}
	\end{array} \right)
	\end{aligned} \end{equation*}
	by construction.
	Moreover, if $\tau_0 \gg 1$ is sufficiently large, then 
	\begin{equation*} \begin{aligned}
	\left \| \left( \begin{array}{ccc}
		\langle \grave{U}_{\ul{p}, \ul{q}}(\tau_1), \tl{U}_{\lambda_0} \rangle_{L^2_{n,\frac{1}{2B^2} }}& ...&  \langle \grave{U}_{\ul{p}, \ul{q}}(\tau_1), \tl{U}_{\lambda_{k-1}} \rangle_{L^2_{n,\frac{1}{2B^2}}} \\
		\\
	\langle \grave{Z}_{\ul{p}, \ul{q}}(\tau_1) - e^{B^2 \lambda_k \tau_1} \tl{Z}_{\lambda_k}, \tl{Z}_{h_0} \rangle_{L^2_{n-2,\frac{1}{2B^2}}} &...&  \langle \grave{Z}_{\ul{p}, \ul{q}}(\tau_1) - e^{B^2 \lambda_k \tau_1} \tl{Z}_{\lambda_k}, \tl{Z}_{h_K} \rangle_{L^2_{n-2,\frac{1}{2B^2}}}
	\end{array} \right) \right\| \le \frac{\epsilon_0}{2}
	\end{aligned} \end{equation*}
	(see for example lemmas \ref{L^2EstInitU} and \ref{L^2EstInitZ}).
	Hence, the straight-line homotopy from $P_{\tau_0, \tau_0}$ to the identity map never vanishes on the boundary of $B_{\epsilon_0 e^{B^2 \lambda_k \tau_0} }$.
	It then follows from the homotopy invariance of degree that the degree of $P_{\tau_0,\tau_0}$ with respect to 0 is 1.
	
	By the previous lemma \ref{zerosInTheInterior},
	$$\partial \mathcal{W}_{\tau_0, \tau_1} \cap P^{-1}_{\tau_0, \tau_1}(0) = \emptyset 	\qquad \text{for all } \tau_1 > \tau_0$$
	Homotopy invariance of degree then implies that the degree is 1 for all $\tau_1 > \tau_0$.
\end{proof}

\begin{lem} \label{WNonempty}
	For $\ul{\Upsilon}, \Upsilon_U, \Upsilon_Z, \tau_0 \gg 1$ sufficiently large,
	$\mathcal{W}_{\tau_0, \tau_1}$ is nonempty for all $\tau_1 > \tau_0$.
\end{lem}
\begin{proof}
	Suppose otherwise and consider $\tau_*$ defined by
		$$\tau_* = \inf \{ \tau > \tau_0 : \mathcal{W}_{\tau_0, \tau} = \emptyset \}$$
	Take a sequence $\tau_n \nearrow \tau_*$.
	By the previous lemma \ref{degree1},
	there exists a sequence $(\ul{p}, \ul{q} )_n \in \ol{\mathcal{W}}_{\tau_0, \tau_n} \cap P_{\tau_0, \tau_n}^{-1}(0)$.
	Without loss of generality, say $\lim_{n \to \infty} (\ul{p}, \ul{q})_n = (\ul{p}, \ul{q} )_*$.
	It follows from the continuity of $P$ that $(\ul{p}, \ul{q})_* \in \ol{\mathcal{W}}_{\tau_0, \tau_*} \cap P_{\tau_0, \tau_*}^{-1}(0)$.
	Lemma \ref{zerosInTheInterior} then implies
		$$(\tl{Z}_{\ul{p}, \ul{q}} , \tl{U}_{\ul{p}, \ul{q}} )  \in \mathcal{B}^{ \frac{1}{2} \ul{\eta} }_{\tau_0, \tau_*}$$
	and continuity then yields that, for some $\delta > 0$,
		$$(\tl{Z}_{\ul{p}, \ul{q}} , \tl{U}_{\ul{p}, \ul{q}} )  \in \mathcal{B}^{ \ul{\eta} }_{\tau_0, \tau_* + \delta}$$
	However, this result contradicts the choice of $\tau_*$.
\end{proof}

\begin{lem} \label{desiredZUSolution}
	If $\ul{\Upsilon}, \Upsilon_U, \Upsilon_Z, \tau_0 \gg 1$ are sufficiently large,
	then there exists $(\ul{p}, \ul{q})$ such that 
		$$( \tl{Z}_{\ul{p}, \ul{q}}, \tl{U}_{\ul{p}, \ul{q}} )(\gamma, \tau) \in \mathcal{B}_{\tau_0, \infty}^{\ul{\eta}} \doteqdot  \bigcap_{\tau_1 > \tau_0} \mathcal{B}_{\tau_0, \tau_1}^{\ul{\eta}}$$
\end{lem}
\begin{proof}
	Assume $\ul{\Upsilon}, \Upsilon_U, \Upsilon_Z, \tau_0 \gg 1$ are sufficiently large such that lemmas \ref{zerosInTheInterior} - \ref{WNonempty} hold.
	Let $\{ \tau_n \}_{n \in \mathbb{N}}$ be an increasing sequence that limits to $+\infty$. 
	By lemma \ref{WNonempty}, there exists $(\ul{p}_n, \ul{q}_n) \in \mathcal{W}_{\tau_0, \tau_n}$ for all $n$.
	Since this sequence $(\ul{p}_n, \ul{q}_n)$ is bounded, there exists a convergent subsequence that limits to say $(\ul{p}_\infty, \ul{q}_\infty)$. It follows that 
		$$( \tl{Z}_{\ul{p}_\infty, \ul{q}_\infty}, \tl{U}_{\ul{p}_\infty, \ul{q}_\infty} )(\gamma, \tau) \in \mathcal{B}_{\tau_0, \infty}^{\ul{\eta}} $$
\end{proof}

\begin{proof} (of theorem \ref{mainThmAbridged})
	The profile functions 
		$$(\tl{Z}, \tl{U}) = ( Z_{\ul{p}, \ul{q}}, U_{\ul{p}, \ul{q}})(\gamma, \tau) = (Z_{RFC}, U_{RFC}) + ( \tl{Z}_{\ul{p}, \ul{q}}, \tl{U}_{\ul{p}, \ul{q}})$$ 
	provided by lemma \ref{desiredZUSolution} yield the desired Ricci flow solution.
	Indeed, 
		$$( \tl{Z}_{\ul{p}, \ul{q}}, \tl{U}_{\ul{p}, \ul{q}} )(\gamma, \tau) \in \mathcal{B}_{\tau_0, \infty}^{\ul{\eta}} $$
	implies the $C^2_{loc}$ convergence to the Ricci-flat cone metric when $x \in S^{q+1}$ is assumed to be the point where $\psi = 0$.
	
	It only remains to establish a lower bound on the curvature blow-up rate.
	By theorem \ref{ptwiseEstsZU}, we may assume that $(\tl{Z}, \tl{U}) \in \mathcal{P}[ \ul{D}, l , \kappa, \ul{\eta}, \Upsilon_U, \Upsilon_Z, M , \alpha_k , \beta, \lambda_k ]$
	after perhaps taking the parameters $\ul{\Upsilon}, \Upsilon_U, \Upsilon_Z, \tau_0$ larger.
	In particular, $U(\gamma, \tau)$ is an increasing function of $\gamma$ for $\gamma \in \left[ 0, \Upsilon_U e^{-\alpha_k \tau} \right]$.
	We proceed to estimate the sectional curvature $\frac{1- \phi_s^2}{\phi^2}$ at a point where $x = -1$
	\begin{equation*} \begin{aligned}
		\frac{1- \phi_s^2}{\phi^2} 
		&=e^\tau e^{-2U(0, \tau)}\\
		&\ge e^\tau e^{-2 U( \Upsilon_U e^{-\alpha_k \tau}, \tau)} \\
		&= e^\tau e^{-2U_{RFC}( \Upsilon_U e^{-\alpha_k \tau} ) } e^{- 2 \tl{U}( \Upsilon_U e^{-\alpha_k \tau}, \tau)} \\
		&= e^\tau \frac{B^2}{A^2 \Upsilon_U^2 e^{- 2 \alpha_k \tau} } e^{- 2 \tl{U}( \Upsilon_U e^{-\alpha_k \tau}, \tau)} \\
		&\gtrsim_{p,q,k, \Upsilon_U} e^{2 \alpha_k \tau + \tau} 
		&& \big( ( \tl{Z}, \tl{U} ) \in \mathcal{B}^{\ul{\eta}}_{\tau_0, \infty} \big) \\
		&= \frac{1}{( T-t)^{2 \alpha_k + 1} } \\
	\end{aligned} \end{equation*}
	Finally, by observing that
		$$ \alpha_k = \frac{ -2B^2 \lambda_k }{ n - 1 - \sqrt{( n-9)(n-1)}} = \frac{ 2k - \frac{n-1}{2} + \frac{1}{2} \sqrt{ (n-9)(n-1)} }{ n - 1 - \sqrt{( n-9)(n-1)}}$$
	grows linearly in $k$, it follows that we obtain blow-up rates given by arbitrarily large powers of $\frac{1}{(T-t)}$. 
\end{proof}

\noindent The remainder of the paper provides a series of estimates used to justify the proof of lemma \ref{zerosInTheInterior}.
The estimates also provide additional information on the behavior of the Ricci flow solutions corresponding to the profile functions $( \tl{Z}, \tl{U}) (\gamma, \tau)$.

%% file: PointwiseEst.tex
In this section, we collect pointwise estimates that will be used to streamline the proofs in other sections.
The main results of this section include theorem \ref{ptwiseEstsZU} and proposition \ref{ptwiseEstsErr}.
The remainder of the section will justify the lemmas invoked in the proofs of these estimates.

\begin{definition}
	The set
		$$\mathcal{P}[\ul{D}, l,\kappa, \ul{\eta}, \tau_0, \tau_1, \Upsilon_{U,Z}, M ,\alpha, \beta , \lambda_k]$$
	consists of 
$( \tl{Z}, \tl{U})(\gamma, \tau) \in \mathcal{B}[ \ul{\eta}, \tau_0, \tau_1, \Upsilon_{U,Z}, M ,\alpha, \beta , \lambda_k ]$ 	that also satisfy the following bounds:
\begin{equation*} \begin{aligned}
	0 &\le \tl{U}		&&  \text{for all }  0 < \gamma \le \Upsilon_U e^{- \alpha \tau} \\
	 | \tl{U}_\gamma | &\le \frac{1}{\gamma} 	&& \text{for all }  0 < \gamma \le \Upsilon_U e^{- \alpha \tau} \\
	\end{aligned} \end{equation*}
	
\begin{equation*} \begin{aligned}
	| \tl{U} |  
	&\le D_0^U \min \left\{  | \log \gamma | ,  \Upsilon_U^l (\gamma e^{- \alpha_k \tau} )^{-2k-2B^2 \lambda - \kappa} \right\}	
	&& \text{for all }  0 < \gamma \le e^{- \alpha_k \tau} \\
\end{aligned} \end{equation*}
\begin{equation*} \begin{aligned}
	| \gamma \tl{U}_\gamma  | + | \gamma^2 \tl{U}_{\gamma \gamma} | 
	&\le D_0^{U_\gamma} 
	&&\text{for all }  0 < \gamma \le e^{- \alpha_k \tau} \\
	|\tl{Z} | + | \gamma \tl{Z}_\gamma | + | \gamma^2 \tl{Z}_{\gamma \gamma} | 
	&\le D_0^Z 
	&&\text{for all }  0 < \gamma \le e^{- \alpha_k \tau} \\
\end{aligned} \end{equation*}

\begin{equation*} \begin{aligned}
	| \tl{U} | + | \gamma \tl{U}_\gamma  | + | \gamma^2 \tl{U}_{\gamma \gamma} | 
	&\le D_1^U \Upsilon_U^l \left( \gamma e^{\alpha_k \tau} \right)^{-2k-2B^2 \lambda_k} 
	&& \text{for all }  e^{- \alpha_k \tau} \le \gamma \le \Upsilon_U e^{- \alpha_k \tau} \\
	| \tl{Z} | + | \gamma \tl{Z}_{\gamma} | + | \gamma^2 \tl{Z}_{\gamma \gamma} | 
	&\le D_1^Z \Upsilon_U^l \left( \gamma e^{\alpha_k \tau} \right)^{-2k - 2B^2 \lambda_k } 
	&& \text{for all } e^{- \alpha_k \tau} \le \gamma \le \Upsilon_Z e^{- \alpha_k \tau} \\
\end{aligned} \end{equation*}

\begin{equation*} \begin{aligned}
	| \tl{U} | + |  \tl{U}_\gamma  | + | \tl{U}_{\gamma \gamma} | 
	&\le D_2^U \gamma^{-2B^2 \lambda_k} e^{B^2 \lambda_k \tau} 
	&& \text{for all } M e^{\beta \tau} \le \gamma \le M e^{\beta \tau} + 1\\
	| \tl{Z} | + |  \tl{Z}_{\gamma} | + |  \tl{Z}_{\gamma \gamma} | 
	&\le D_2^Z \gamma^{-2B^2 \lambda_k } e^{B^2 \lambda_k \tau} 
	&& \text{for all } M e^{\beta \tau} \le \gamma \le M e^{\beta \tau} + 1 \\
\end{aligned} \end{equation*}
\end{definition}


\begin{remark}
	For estimates in later sections, it will matter that $$0 < l < \frac{1}{2} (2k + 2B^2 \lambda_k)$$
	However, its exact value is not important.
\end{remark}

\begin{thm} \label{ptwiseEstsZU}
	There exist constants $a,b,c,\kappa , \epsilon, l, \beta, \ol{\beta}$ such that the following holds:
	If $(\tl{Z}, \tl{U})$ is in
	\begin{equation*} \begin{aligned} 
		& \cl{I} [ \ul{D}^{\cl{I}}, a,b,c,\kappa, \epsilon, \Upsilon_U, \ul{\Upsilon}, \ul{\Upsilon} , \tau_0, \tau_0 ] \\
		\cap & \cl{B} [ \ul{\eta}, \tau_0,\tau_0, \ul{\Upsilon}, \ul{\Upsilon} , M , \alpha_k , \ol{\beta} , \lambda_k ] \\
		\cap & \cl{O} [ \ul{D}^{\cl{O}}, M e^{\ol{\beta} \tau_0}, \tau_0, \tau_0 ] \\
		\cap &  \cl{B} [ \ul{\eta} , \tau_0, \tau_1, \Upsilon_U, \Upsilon_Z, M , \alpha_k , \beta, \lambda_k ]
	\end{aligned} \end{equation*}	
	then 
		$$( \tl{Z}, \tl{U} ) \in \cl{P} [ \hat{\ul{D}} , l , \kappa, \ul{\eta}, \Upsilon_{U, Z}, M , \alpha_k , \beta , \lambda_k ]$$
		$$\text{where } \hat{\ul{D}} = \hat{\ul{D}}(p,q,k, \ul{D}^{\cl{I}, \cl{O}})$$
	if 
	$1 \ll \ul{\Upsilon} \le \Upsilon_U \le \Upsilon_Z$ are sufficiently large depending on $p,q,k, a,b, c , \kappa, \epsilon, l$,
	and $\tau_0 \gg 1$ is sufficiently large depending on $p,q,k, \Upsilon_U, \Upsilon_Z, \ul{\Upsilon}, a, b, c, \kappa, \epsilon, l, M, \beta, \ol{\beta}$. 
\end{thm}
\begin{proof}
	Since $|\ul{\eta}|$ is small, we will assume throughout the proof that it is bounded above by a uniform constant.
	We also write
	\begin{equation*} \begin{aligned}
		\cl{I} \left[ \ul{D}, a, b, c, \kappa, \epsilon, \Upsilon, \Upsilon_U, \Upsilon_Z, \tau_0, \tau_1 \right]
		\doteqdot & \quad
		\cl{I}_{(\text{Barrier I})} [ \Upsilon_U, \Upsilon_Z, \tau_0, \tau_1 ] \\
		& \cap \cl{I}_{(\text{weighted $C^2$ bounds})} [ D_0^{U, Z}, \Upsilon_U, \Upsilon_Z, \tau_0, \tau_1 ] \\
		& \cap \cl{I}_{(\text{Barrier II})} [ D_{1,2,3}, a,b,c,\kappa, \epsilon, \Upsilon, \Upsilon_U, \Upsilon_Z, \tau_0, \tau_1 ]
	\end{aligned} \end{equation*}
	where
		$\cl{I}_{(\text{Barrier I})} [ \Upsilon_U, \Upsilon_Z, \tau_0, \tau_1 ]$
	consists of the $(\tl{Z}, \tl{U})$ that satisfy the ``Inner Region Barriers I" bounds in definition \ref{innerRegConds},
		$\cl{I}_{(\text{weighted $C^2$ bounds})} [ D_0^{U, Z}, \Upsilon_U, \Upsilon_Z, \tau_0, \tau_1 ] $
	consists of the $(\tl{Z}, \tl{U})$ that satisfy the ``Weighted $C^2$ Bounds" conditions in definition \ref{innerRegConds}, and
		$\cl{I}_{(\text{Barrier II})} [ D_{1,2,3}, a,b,c,\kappa, \epsilon, \Upsilon, \Upsilon_U, \Upsilon_Z, \tau_0, \tau_1 ]$
	consists of the $(\tl{Z}, \tl{U})$ that satisfy the ``Inner Region Barriers II" bounds in definition \ref{innerRegConds}.
	
	The general idea of the rest of the proof will be to use that $( \tl{Z}, \tl{U}) \in \cl{B}$ and the initial conditions to show that various barriers apply.
	Equipped with such estimates, the fact that $(\tl{Z}, \tl{U}) \in \cl{B}$, and the asymptotics of the functions $\tl{Z}_{\lambda_k},  \tl{U}_{\lambda_k}$ for large and small $\gamma$, we then invoke interior estimates to get bounds on the derivatives of $\tl{Z}, \tl{U}$ in the various regions.
	
	To begin, note that the condition 
		$$( \tl{Z}, \tl{U}) \in \cl{B}[ \ul{\eta} , \tau_0, \tau_1, \Upsilon_U, \Upsilon_Z, M , \alpha_k , \beta, \lambda_k ]$$
	is automatic by assumption.
	
	If
	\begin{equation*} \begin{aligned}
		(\tl{Z}, \tl{U}) \in \cl{B} [ \ul{\eta} , \tau_0, \tau_1, \Upsilon_U, \Upsilon_Z, M , \alpha_k , \beta, \lambda_k ] 
	\end{aligned} \end{equation*}
	and $\tau_0 \gg 1$ is sufficiently large depending on $\ul{\Upsilon}, \Upsilon_U, \Upsilon_Z, p,q, k$,
	then the hypotheses of lemmas \ref{innerConeBarrierU}-\ref{innerSecondDerivBarrierU} are satisfied with 
	$\Upsilon = 4 \Upsilon_Z \ge 4 \Upsilon_U$ throughout, 
	$c = \frac{A}{B}$ in lemma \ref{innerConeBarrierU}, and 
	$c = B^2$ in lemma \ref{innerConeBarrierZ}.
	Hence, 
	\begin{gather*}
		(\tl{Z}, \tl{U}) \in \cl{I}_{(\text{Barrier I})} [ 4\Upsilon_Z, 4\Upsilon_Z, \tau_0, \tau_1] 
		\subset \cl{I}_{(\text{Barrier I})} [ \Upsilon_U, \Upsilon_Z, \tau_0, \tau_1]
	\end{gather*}
	which in particular implies that
	\begin{equation*} \begin{aligned}
		0 \le \tl{U}		& \quad \text{for all } \quad 0 < \gamma \le \Upsilon_U e^{- \alpha \tau} \\
		| \tl{U}_\gamma | \le \frac{1}{\gamma} 	& \quad \text{for all } \quad 0 < \gamma \le \Upsilon_U e^{- \alpha \tau} \\	
	\end{aligned} \end{equation*}
	Moreover, integrating the $\tl{U}_\gamma$ bound on $\gamma \in [0, \Upsilon_U e^{-\alpha \tau} ]$ and using
		$$( \tl{Z}, \tl{U}) \in \cl{B} [ \ul{\eta} , \tau_0, \tau_1, \Upsilon_U, \Upsilon_Z, M , \alpha_k , \beta, \lambda_k ] $$
	yields the logarithmic bound for $\tl{U}$ on $\gamma \in ( 0 , e^{-\alpha \tau} ] \subset ( 0 , \Upsilon_U e^{-\alpha \tau} ]$.
	
	Now, for $\tau_0 \gg 1$ sufficiently large, 
		$$(\tl{Z}, \tl{U}) \in  \cl{I}_{(\text{Barrier I})} [ 4\Upsilon_Z, 4\Upsilon_Z, \tau_0, \tau_1] \cap \cl{B} [ \ul{\eta} , \tau_0, \tau_1, \Upsilon_U, \Upsilon_Z, M , \alpha_k , \beta, \lambda_k ]$$
	and
		$$(\tl{Z}, \tl{U}) \in \cl{I} [ \ul{D}^{\cl{I}}, a,b,c,\kappa, \epsilon, \Upsilon_U, \ul{\Upsilon}, \ul{\Upsilon} , \tau_0, \tau_0 ]$$
	imply that the interior estimates of subsubsection \ref{innerIntEsts} apply.
	Hence,
		$$(\tl{Z}, \tl{U}) \in \cl{I}_{\text{(weighted $C^2$ bounds)}} [ \hat{\ul{D}}^{\cl{I}} , \Upsilon_Z, \Upsilon_Z, \tau_0, \tau_1]$$
	$$\text{where } \hat{\ul{D}}^{\cl{I}} = \hat{\ul{D}}^{\cl{I}}( p,q,k, \ul{D}^{\cl{I}} )$$ 
	
	We now show the $C^0$-estimate for the region $\gamma \in \left[ e^{-\alpha \tau}, \Upsilon_{U,Z} e^{-\alpha \tau} \right]$.
	Begin by noting that lemma \ref{innerParabExtendZBounds} simplifies the situation to proving the $C^0$ bound on the region $e^{-\alpha \tau} \le \gamma \le \Upsilon_U e^{-\alpha \tau}$ for both $\tl{U}$ and $\tl{Z}$.
	This $C^0$ estimate then follows from lemmas \ref{innerFarBarrierUWeak} - \ref{innerFarBarrierU} with $\Upsilon = \Upsilon_U \gg 1$ sufficiently large and constants $a,b,c,\kappa, \epsilon$ specified in remark \ref{constantsToUse}.
	Indeed, 
		$$(\tl{Z}, \tl{U}) \in \cl{I}[ \ul{D}^{\cl{I}}, a,b,c,\kappa, \epsilon, \Upsilon_U, \ul{\Upsilon}, \ul{\Upsilon} , \tau_0, \tau_0 ] \\ \cap \cl{B}[ \ul{\eta} , \tau_0, \tau_1, \Upsilon_U, \Upsilon_Z, M , \alpha_k , \beta, \lambda_k ]$$
	implies that the hypotheses of lemma \ref{innerFarBarrierUWeak} hold with $\Upsilon = \Upsilon_U$
	and $C_0 \gtrsim \Upsilon_U^{a - 2k-2B^2 \lambda_k}$.
	This bound then implies that the hypotheses of lemma \ref{innerFarBarrierZ} and so we obtain the desired $C^0$-bound on $\tl{Z}$ for $e^{-\alpha \tau} \le \gamma \le \Upsilon_U e^{- \alpha \tau}$, i.e.
		$$| \tl{Z} | \le D_1^Z \Upsilon_U^l ( \gamma e^{-\alpha \tau} )^{-2k -2B^2 \lambda}$$
	This bound then implies that the hypotheses of lemma \ref{innerFarBarrierU} hold and so we obtain the corresponding $C^0$ bound on $\tl{U}$ for $\gamma \in \left[ e^{-\alpha \tau}, \Upsilon_U e^{-\alpha \tau} \right]$.
	The $C^2$ estimate then follows from \ref{innerInteriorEstU+}
	and \ref{innerInteriorEstZ+} 
	using additionally the fact that 
	$$(\tl{Z}, \tl{U}) \in \cl{B}[ \ul{\eta} , \tau_0, \tau_1, \Upsilon_U, \Upsilon_Z, M , \alpha_k , \beta, \lambda_k ]$$
	and the asymptotics of the eigenfunctions for small $\gamma$ 
	contained in propositions \ref{eigfuncUAsymps} and \ref{ZlambdaAsymps}.
	Note that the choice of constants specified by remark \ref{constantsToUse} ensures that
		$$0 <  l < \frac{1}{2} (2k + 2B^2 \lambda_k)$$
	
	For the outer region estimates, first
	observe that the choice of initial data 
	and 
		$$(\tl{Z}, \tl{U}) \in \cl{B} [ \ul{\eta} , \tau_0, \tau_1, \Upsilon_U, \Upsilon_Z, M , \alpha_k , \beta, \lambda_k ] $$
	imply that, if $\tau_0 \gg 1$ is sufficiently large, then the hypotheses of lemmas \ref{outerCurvatureBounds}-\ref{outerSineConeBarrierZ} apply
	with $c=\frac{A}{B}$ in lemma \ref{outerConeBarrierU}, $C = 1 + D_0^\cl{O}$ in lemmas \ref{outerConeBarrierGradU} and \ref{outerSineConeBarrierZ}, and $K_0 = D_1^\cl{O}$ in lemma \ref{outerSineConeBarrierZ}.

	In particular, for any $\epsilon' > 0$, $\tl{Z} > \epsilon'$ on the outer-parabolic interface if $\tau_0 \gg 1$ is sufficiently large. 
	This fact and the fact that	
		$$(\tl{Z}, \tl{U}) \in \cl{B} [ \ul{\eta} , \tau_0, \tau_1, \Upsilon_U, \Upsilon_Z, M , \alpha_k , \beta, \lambda_k ] $$
	then imply that the hypotheses of lemmas \ref{parabOuterBarriersU}-\ref{parabOuterInteriorEstZ} hold (with constants $C$ depending only on $p,q,k, \ul{D}^{\cl{O}}$) if $\tau_0 \gg 1$ is sufficiently large.
	Thus, if
		$$\beta < \frac{1}{2} \left( \frac{ 1}{ 1 + \frac{1}{-2B^2 \lambda_k } } \right) \le \ol{\beta} < \frac{1}{2}$$
	 and $\tau_0 \gg 1$ is sufficiently large, then 
	 \begin{equation*} \begin{aligned}
	 | \tl{U} | + |  \tl{U}_\gamma  | + | \tl{U}_{\gamma \gamma} | \le \hat{D}_2^U \gamma^{-2B^2 \lambda_k} e^{B^2 \lambda_k \tau} & \quad \text{for all } \quad M e^{\beta \tau} \le \gamma \le M e^{\beta \tau} + 1\\
	| \tl{Z} | + |  \tl{Z}_{\gamma} | + |  \tl{Z}_{\gamma \gamma} | \le \hat{D}_2^Z \gamma^{-2B^2 \lambda_k } e^{B^2 \lambda_k \tau} & \quad \text{for all } \quad M e^{\beta \tau} \le \gamma \le M e^{\beta \tau} + 1 \\
\end{aligned} \end{equation*}
	for $\hat{D}_2^{U,Z}$ depending on $p,q,k, \ul{D}^{\cl{O}}$.
\end{proof}

\begin{prop} \label{ptwiseEstsErr}
	If $(\tl{Z}, \tl{U}) \in \mathcal{P}[ \ul{D}, l, \kappa, \ul{\eta}, \tau_0, \tau_1, \Upsilon_{U,Z}, M, \alpha, \beta, \lambda_k]$ then there are the following pointwise estimates on the error terms
	\begin{equation*} \begin{aligned}
		| Err_U | 
		\lesssim_{p,q} &  \left( D_0^{U_\gamma} D_0^Z + \Upsilon_U^l \right) ( \gamma e^{\alpha \tau} )^{-2k-2B^2 \lambda - \kappa}	\gamma^{-2}
		&& \text{for  }  0 < \gamma \le e^{- \alpha \tau} \\
		| Err_U | 
		\lesssim_{p,q} & \left( D_1^Z D_1^U  + (D_1^U)^2 \right) \Upsilon_U^{2l}  \left( \gamma e^{\alpha_k \tau} \right)^{-4k-4B^2 \lambda_k} \gamma^{-2} 
		&& \text{for  }  e^{-\alpha_k \tau} < \gamma \le \Upsilon_U e^{- \alpha \tau} \\
		| Err_U| 
		\lesssim_{p,q,k} & \left[ D_1^Z  \Upsilon_U^l + 1  \right] 
		\cdot \gamma^{ -4k - 4B^2 \lambda_k  - 2} e^{2B^2 \lambda_k \tau } 
		&& \text{for } \Upsilon_U e^{- \alpha \tau} \le \gamma \le \Upsilon_Z e^{- \alpha \tau} \\
		| Err_U | 
		\lesssim_{p,q,k} & e^{2B^2 \lambda_k \tau} \left( \gamma^{ -4k - 4B^2 \lambda_k -2} + \gamma^{ -4B^2 \lambda_k } \right) 
		&& \text{for } \Upsilon_Z e^{- \alpha \tau} \le \gamma \le M e^{\beta \tau} \\
		| \grave{\chi} Err_U | \lesssim_{p,q} & \left[ D_2^Z D_2^U + \left(D_2^U \right)^2 \right] \gamma^{-4B^2 \lambda_k} e^{2B^2 \lambda_k \tau} 
		&& \text{for }  M e^{\beta \tau} \le \gamma \le M e^{\beta \tau}  + 1 \\
	\end{aligned} \end{equation*}
	
	\begin{equation*} \begin{aligned}
		| Err_Z | \lesssim_{p,q} & \left(D_0^Z D_0^{U_\gamma} \right)^2 \gamma^{-2} 
		&& \text{for } 0 \le \gamma \le e^{- \alpha_k \tau} \\
		| Err_Z | \lesssim_{p,q} & \left[ (D_1^Z)^2 + \left( D_1^{U_\gamma} \right)^2 \right] \Upsilon_U^{2l} \gamma^{-4k - 4B^2 \lambda_k  - 2} e^{ 2B^2 \lambda_k \tau }
		&& \text{for }  e^{- \alpha \tau} \le \gamma \le \Upsilon_U e^{- \alpha_k \tau} \\
		| Err_Z | \lesssim_{p,q, k} & \left[ (D_1^Z)^2\Upsilon_U^{2l} + 1 \right]  \gamma^{-4k - 4B^2 \lambda_k  - 2} e^{ 2B^2 \lambda_k \tau }
		&& \text{for }  \Upsilon_U e^{- \alpha \tau} \le \gamma \le \Upsilon_Z e^{- \alpha_k \tau} \\
		| Err_Z |  
		\lesssim_{p,q,k} & e^{2B^2 \lambda_k \tau} \left( \gamma^{ -4k - 4B^2 \lambda_k -2} + \gamma^{ -4B^2 \lambda_k } \right) 
		&& \text{for }  \Upsilon_Z e^{- \alpha \tau} \le \gamma \le M e^{\beta \tau} \\
		| \grave{\chi} Err_Z | \lesssim_{p,q} & \left[ (D_2^Z)^2 + (D_2^U)^2 \right] \gamma^{-4B^2 \lambda_k} e^{2B^2 \lambda_k \tau} 
		&& \text{for }  M e^{\beta \tau} \le \gamma \le M e^{\beta \tau}  + 1 \\
	\end{aligned} \end{equation*}
	
	\begin{equation*} \begin{aligned}
		\left| B^2 \cl{N} \left( \tl{U}(\gamma, \tau) - e^{B^2 \lambda_k \tau} \tl{U}_{\lambda_k}(\gamma) \right) \right|
		\lesssim_{p,q,k} & D_0^{U_\gamma} \gamma^{-2k - 2B^2 \lambda_k -2} e^{B^2 \lambda_k \tau} \\
		& \text{for }  0 < \gamma \le e^{-\alpha \tau} \\
		\\
		\left| B^2 \cl{N} \left( \tl{U}(\gamma, \tau) - e^{B^2 \lambda_k \tau} \tl{U}_{\lambda_k}(\gamma) \right) \right|
		\lesssim_{p,q,k} & D_1^U \Upsilon_U^l \gamma^{-2k - 2B^2 \lambda_k -2} e^{B^2 \lambda_k \tau} \\
		& \text{for } e^{- \alpha \tau}  < \gamma \le \Upsilon_U e^{-\alpha \tau} \\
		\\
		\left| B^2 \cl{N} \left( \tl{U}(\gamma, \tau) - e^{B^2 \lambda_k \tau} \tl{U}_{\lambda_k}(\gamma) \right) \right|
		\lesssim_{p,q} & \eta_1^U e^{B^2 \lambda_k \tau} \left( \gamma^{-2k - 2B^2 \lambda_k -2} + \gamma^{-2B^2 \lambda_k -1} \right) \\
		& \text{for } \Upsilon_U e^{- \alpha \tau}  < \gamma \le M e^{\beta \tau} \\
		\\
		\left| \grave{\chi} B^2 \mathcal{N}  \tl{U}(\gamma, \tau) - B^2 \mathcal{N} e^{B^2 \lambda_k \tau} \tl{U}_{\lambda_k}(\gamma)  \right|
		\lesssim_{p,q,k} &( D_2^U + 1) e^{B^2 \lambda_k \tau} \gamma^{-2B^2 \lambda_k-1}\\
		& \text{for } M e^{\beta \tau}  \le  \gamma \le M e^{\beta \tau} +1 \\
	\end{aligned} \end{equation*} 
	
	\begin{equation*} \begin{aligned}
		\left| B^2 [ \grave{\chi} , \mathcal{D}_U ] \tl{U} + \grave{\chi}_\tau \tl{U} \right| \lesssim_{p,q,k}& D_2^U  e^{B^2 \lambda_k \tau} \gamma^{-2B^2 \lambda_k +1}  \qquad
		&& \text{for } M e^{\beta \tau} \le \gamma \le M e^{\beta \tau} +1 \\
		\left| B^2 [ \grave{\chi} , \mathcal{D}_Z ] \tl{Z} + \grave{\chi}_\tau \tl{Z} \right| \lesssim_{p,q,k}& D_2^Z e^{B^2 \lambda_k \tau} \gamma^{-2B^2 \lambda_k +1} 
		&& \text{for } M e^{\beta \tau} \le \gamma \le M e^{\beta \tau} +1 \\
	\end{aligned} \end{equation*}
\end{prop}
\begin{proof}
	Nearly all these estimates follow immediately from the assumed pointwise estimates on $\tl{U}, \tl{Z}$ and their derivatives.
	We make a few remarks in the cases where some of the estimates are not immediately clear.

	($| Err_U |$ bound for $0 < \gamma \le  e^{-\alpha \tau}$)\\
	To estimate the $1 - e^{-2\tl{U}} - 2\tl{U}$ term, simply note that $0 \le \tl{U}$ implies
		$$\left| 1 - e^{- 2 \tl{U} }  - 2 \tl{U} \right| \le \left| 1 - e^{-2 \tl{U} } \right| + 2 \left| \tl{U} \right| \le 1 + 2 \left| \tl{U} \right|$$
	The estimate then follows from the pointwise estimate on $\tl{U}$ for $0 < \gamma \le e^{-\alpha \tau} $.

	($| Err_U |$ bound for $e^{-\alpha \tau}  \le \gamma \le \Upsilon_U e^{-\alpha \tau}$)\\
	There is some subtlety in estimating the contribution of the factor
		$$1 - e^{-2 \tl{U}} - 2 \tl{U}$$
	First, note that 
		$$0 \le \tl{U} \le D_1^U \Upsilon_U^l ( \gamma e^{- \alpha \tau} )^{-2k -2B^2 \lambda_k}		\qquad \text{for all } e^{-\alpha \tau} \le \gamma \le \Upsilon_U e^{- \alpha \tau}$$
		$$ \implies 0 \le \tl{U} \le D_1^U \Upsilon_U^l		 	\qquad \text{for all } e^{-\alpha \tau} \le \gamma \le \Upsilon_U e^{- \alpha \tau}$$
	Taylor's inequality for the function $x \mapsto 1 - e^{-2x} - 2 x$ then implies that
		$$\left| 1 - e^{-2 \tl{U}} - 2 \tl{U} \right| \le \frac{  \tl{U}^2 }{2} \sup_{x \in [0, D_1^U \Upsilon_U^l]} 4 e^{-2x} \le 2 \tl{U}^2 $$
	for all $e^{-\alpha \tau} \le \gamma \le \Upsilon_U e^{- \alpha \tau}$.
	Hence,
		$$\frac{n-1}{\gamma^2} \left| 1 - e^{-2 \tl{U}} - 2 \tl{U} \right| \lesssim_{p,q} (D_1^U)^2 \Upsilon_U^{2l} \gamma^{-4k - 4B^2 \lambda_k -2} e^{2B^2 \lambda_k \tau}$$
	for all $e^{-\alpha \tau} \le \gamma \le \Upsilon_U e^{- \alpha \tau}$.
\end{proof}

\subsection{Inner Region Pointwise Estimates} \label{InnerRegionPointwiseEsts}
\subsubsection{Barriers}
Throughout this subsection we let $\Omega$ denote the spacetime region
	$$\Omega = \Omega(\Upsilon, \tau_0, \tau_1) = \left\{ (\gamma, \tau) \in (0, \infty) \times (\tau_0, \tau_1) : \gamma < \Upsilon e^{-\alpha_k \tau} \right\}$$
its parabolic boundary
	$$\partial_P \Omega = \left(  [0, \Upsilon e^{ - \alpha_k \tau_0} ] \times \{ \tau_0 \}  \right) \cup \{ (\Upsilon e^{-\alpha_k \tau}, \tau) : \tau \in [\tau_0, \tau_1) \} \cup \left( \{ 0 \} \times [\tau_0, \tau_1) \right)$$
and its pseudo-parabolic boundary
	$$\partial_{PP} \Omega = \left(  [0, \Upsilon e^{ - \alpha_k \tau_0} ] \times \{ \tau_0 \}  \right) \cup \{ (\Upsilon e^{-\alpha_k \tau}, \tau) : \tau \in [\tau_0, \tau_1) \}$$

For the proofs in this subsection, it is assumed that
\begin{enumerate}
	\item $U, Z \in C^\infty \left( \ol{\Omega} \right)$, and
	\item $0 < Z( \gamma, \tau) \le 1$ for all $(\gamma, \tau) \in \ol{\Omega}$.
\end{enumerate}
The first assumption will be justified in section \ref{NoInnBlowup}.
The second assumption ensures that the differential equations involved in the proofs of the following barriers are parabolic equations for which comparison principles apply.
A quantitative lower bound on $\inf_\Omega Z$ will be obtained in proposition \ref{innerConeBarrierZ}.
	
\begin{lem} \label{innerConeBarrierU}
	Let $c \ge \frac{A}{B}$.
	If $U \ge \log( c \gamma)$ for all $(\gamma, \tau) \in \partial_{PP} \Omega$,
	then $U \ge \log( c \gamma)$ for all	$ (\gamma, \tau) \in \Omega$.
\end{lem}
\begin{proof}
	Since $U(0, \tau) > -\infty$ for all $\tau \in [\tau_0, \tau_1)$,
		$$U \ge \log( c \gamma) \quad \text{for all } (\gamma, \tau) \in \partial_{PP} \Omega \quad \implies \quad U \ge \log( c \gamma) \quad \text{for all } (\gamma, \tau) \in \partial_{P} \Omega$$
	By the comparison principle, it now suffices to show that for any $c \ge \frac{A}{B}$
		$$U_- = \log( c \gamma)$$
	is a subsolution in $\Omega$, i.e. that
		$$\partial_\tau U_- + \frac{\gamma}{2} \partial_\gamma U_- - \frac{1}{2} \le \mathcal{E}_\gamma [ Z, U_-]		\qquad \forall (\gamma, \tau) \in \Omega$$	
	Note that here $Z$ is simply regarded as a smooth coefficient valued in $(0,1]$ rather than as a solution to a differential equation depending on $U_-$.
	It is straightforward to check that
	if $c \ge \frac{A}{B} = \sqrt{ \frac{p-1}{q-1} }$ then for all $(\gamma, \tau) \in \Omega$
		$$\mathcal{E}_\gamma [ Z, U_-]- \frac{\gamma}{2} \partial_\gamma U_- + \frac{1}{2} = \frac{q-1}{\gamma^2} - \frac{p-1}{c^2 \gamma^2} \ge 0 = \partial_\tau U_-$$
		
\end{proof}

\begin{lem} \label{innerConeBarrierGradU}
	If $U \ge \log \left( \frac{A}{B} \gamma \right)$ for all $(\gamma, \tau) \in \Omega$ and $U_\gamma \le \frac{1}{\gamma}$ on $\partial_{PP} \Omega$, then $U_\gamma \le \frac{1}{\gamma}$ for all $(\gamma, \tau) \in \Omega$.
\end{lem}
\begin{proof}
	Smoothness implies that $U_\gamma(0, \tau) = 0$ for all $\tau \in [\tau_0, \tau_1)$.
	Hence, $U_\gamma \le \frac{1}{\gamma}$ for all $ (\gamma, \tau) \in \partial_{P} \Omega$.
	The remainder of the proof follows from the comparison principle once we confirm that $\frac{1}{\gamma}$ is a supersolution for the parabolic equation satisfied by $U_\gamma$.
	Indeed, if $U \ge \left( \frac{A}{B} \gamma \right)$ then
	\begin{equation*} \begin{aligned}
		& \mathcal{L}_\gamma[Z, U] \left( \frac{1}{\gamma} \right) 
		= - \frac{2(q-1)}{\gamma^3} + \frac{2(p-1)}{\gamma} e^{-2U} 
		\le 0 = \partial_\tau \left( \frac{1}{\gamma} \right) + \frac{1}{2} \left( \frac{1}{\gamma} \right)   + \frac{\gamma}{2} \partial_\gamma \left( \frac{1}{\gamma} \right) 
	\end{aligned} \end{equation*}
	
\end{proof}

\begin{lem} \label{innerConeBarrierZ}
	Let $0 < c \le B^2$.
	If $U_\gamma \le \frac{1}{\gamma}$ in $\Omega$ and $Z \ge c$ on $\partial_{PP} \Omega$ then $Z \ge c$ in $\Omega$.
\end{lem}
\begin{proof}
	Smoothness implies that $Z(0, \tau) = 1$ for all $\tau \in [\tau_0, \tau_1)$.
	Hence, $Z \ge c$ on $\partial_{PP} \Omega$ implies that $Z \ge c$ on $\partial_{P} \Omega$.
	It now suffices to check that $Z_- = c$ is a subsolution in $\Omega$ when $0<c \le B^2$ and $U_\gamma \le \frac{1}{\gamma}$.
	Indeed,
		$$\mathcal{F}_\gamma[ Z_-, U_\gamma] = \frac{2(q-1)}{\gamma^2} c (1-c) - 2p c^2 U_\gamma^2 \ge \frac{1}{\gamma^2} \big( 2(q-1)c(1-c) - 2pc^2 \big) $$
		$$\ge 0 = \partial_\tau Z_- + \frac{\gamma}{2} \partial_\gamma Z_-$$
\end{proof}

\begin{lem} \label{innerGradUPositive}
	If $U_\gamma \ge 0$ on $\Omega_{PP}$, then $U_\gamma \ge 0$ in $\Omega$.
\end{lem}
\begin{proof}
	Smoothness implies that $U_\gamma(0, \tau) = 0$ for all $\tau \in [\tau_0, \tau_1)$.
	The proposition then follows immediately from the comparison principle after noting that $U_\gamma$ satisfies a \textit{linear} parabolic equation (with coefficients depending on $Z, Z_\gamma,$ and $U$).
\end{proof}

\begin{lem} \label{innerSecondDerivBarrierU}
	If $0 \le U_\gamma \le \frac{1}{\gamma}$ in $\Omega$,
	$U \ge \log \left( \frac{A}{B} \gamma \right)$ in $\Omega$,
	and $U_{\gamma \gamma} + \frac{U_\gamma}{\gamma} \ge 0$ on $\partial_{PP} \Omega$,
	then $U_{\gamma \gamma} + \frac{U_\gamma}{\gamma} \ge 0$ in $\Omega$.
\end{lem}
\begin{proof}
	First note that $U_\gamma \ge 0$ in $\Omega$ and $U_\gamma(0, \tau) = 0$ imply that $U_{\gamma \gamma}(0, \tau) \ge 0$.
	 Additionally, smoothness and L'Hopital's rule imply that
		$$\lim_{\gamma \searrow 0 } U_{\gamma \gamma} + \frac{U_\gamma}{\gamma} = 2 U_{\gamma \gamma}(0, \tau) \ge 0 $$
	Hence,
		$$U_{\gamma \gamma} + \frac{U_\gamma}{\gamma} \ge 0 \text{ on } \partial_{PP} \Omega \implies U_{\gamma \gamma} + \frac{U_\gamma}{\gamma} \ge 0 \text{ on } \partial_{P} \Omega $$
	
	The result follows from a comparison principle once we check that $U_{\gamma \gamma} + \frac{U_\gamma}{\gamma}$ is a supersolution to a linear parabolic equation.
	Write the evolution equation for $V = U_\gamma$ as
		$$\partial_\tau V = a(Z, \gamma) V_{\gamma \gamma} + b(Z, \gamma) V_\gamma + c(Z, U, \gamma) V$$
	where
		$$a(Z) = Z, \qquad  \qquad b(Z, \gamma) = Z_\gamma + Z/\gamma + (q-1)/\gamma - \gamma/2$$
		$$c(Z, \gamma, U) = Z_\gamma / \gamma - Z / \gamma^2 - (q-1)/\gamma^2 + 2(p-1) e^{-2U} - \frac{1}{2}$$
	Observe that
		$$\partial_\gamma (V_\gamma /\gamma) = V_\gamma / \gamma - V/ \gamma^2 \qquad \qquad \partial_{\gamma \gamma} (V / \gamma) = V_{\gamma \gamma} / \gamma - 2 V_\gamma / \gamma^2 + 2 V / \gamma^3$$
	and so it follows that
	\begin{equation*} \begin{aligned}
		&\partial_\tau (V_\gamma + V/ \gamma) \\
		=& a V_{\gamma \gamma \gamma} + V_{\gamma \gamma } [ b + \partial_\gamma a ] + V_\gamma [ c + \partial_\gamma b] + V [\partial_\gamma c ] + a V_{\gamma \gamma} / \gamma + b V_\gamma / \gamma + c V/ \gamma \\
		=& a \partial_{\gamma \gamma} (V_\gamma + V/ \gamma ) + b \partial_\psi ( V_\gamma + V/ \gamma) + c(V_\gamma + V/ \gamma) \\
		&+ \partial_\gamma ( a) V_{\gamma \gamma} + \frac{2a}{\gamma} \partial_\gamma (V/ \gamma) + V_\gamma \partial_\gamma b + \frac{b}{\gamma} V/ \gamma + V \partial_\gamma (c) \\
	\end{aligned} \end{equation*}
	The nonlinear terms can be rewritten as follows
	\begin{equation*} \begin{aligned}
		& \partial_\gamma ( a) V_{\gamma \gamma} + \frac{2a}{\gamma} \partial_\gamma (V / \gamma) + V_\gamma \partial_\gamma b + \frac{b}{\gamma} V/ \gamma + V \partial_\gamma (c) \\
		=& Z_\gamma \partial_\gamma \left( V_\gamma + V/\gamma \right) - Z_\gamma \partial_\gamma ( V/\gamma) 
		+ \frac{2Z}{\gamma } \partial_\gamma ( V/\gamma) \\
		& + V_\gamma \left( Z_{\gamma \gamma} + \frac{Z_\gamma}{\gamma} - Z/\gamma^2 -(q-1)/\gamma^2 - 1/2 \right) \\
		& + \frac{V}{\gamma} \left( Z_\gamma/\gamma + Z/ \gamma^2 + (q-1)/\gamma^2 - 1/2 \right) \\
		& + V \left( Z_{\gamma \gamma}/\gamma - 2Z_\gamma / \gamma^2 + 2 Z/\gamma^3 + 2(q-1)/\gamma^3 -4(p-1) e^{-2U} V \right) \\
		= & Z_\gamma \partial_\gamma \left( V_\gamma + V/\gamma \right)  + Z_{\gamma \gamma} ( V_\gamma + V/\gamma) + (Z/\gamma^2) (V_\gamma + V/\gamma)\\
		& - \frac{1}{2} ( V_\gamma + V/\gamma) - \frac{q-1}{\gamma^2} (V_\gamma + V/\gamma) 
		 +4 \frac{q-1}{\gamma^2} \frac{V}{\gamma} - 4(p-1) e^{-2U} V^2 \\
	\end{aligned} \end{equation*}
	So in fact the only nonlinear terms are
		$$4 \frac{q-1}{\gamma^2} \frac{V}{\gamma} - 4(p-1) e^{-2U} V^2 = 4 \frac{q-1}{\gamma^2} \frac{U_\gamma}{\gamma} - 4(p-1) e^{-2U} U_\gamma^2 $$
	This quantity is nonnegative if $0 \le U_\gamma \le \frac{1}{\gamma}$ and $U \ge \log \left( \frac{A}{B} \gamma \right)$, and so the result follows from the comparison principle.
\end{proof}

\subsubsection{Interior Estimates} \label{innerIntEsts}

In the next few propositions, it will easier to work in terms of the following variables adapted to the inner region
	$$\xi = \gamma e^{\alpha_k \tau} 	\qquad r = \frac{1}{2 \alpha} e^{2 \alpha \tau}$$
	$$\mathcal{Z}(\xi, r) = Z(\gamma, \tau) \qquad	 \mathcal{U}(\xi, r) = U(\gamma, \tau) + \alpha \tau$$
	$$\tl{\cl{Z}} = \cl{Z} - B^2 \qquad 		\tl{\cl{U}} = \cl{U} - \log \left( \frac{A}{B} \xi \right)$$
If $Z,U$ satisfy the rescaled Ricci flow equations (\ref{psRF}) then $\mathcal{Z}, \mathcal{U}$ satisfy
\begin{equation} \label{isRF} \begin{aligned}
	 \mathcal{Z}_r + \frac{ \alpha + \frac{1}{2} }{2 \alpha r} \xi \mathcal{Z}_\xi  =& \mathcal{F}_\xi [ \mathcal{Z}, \mathcal{U}_\xi ] \\
	 \mathcal{U}_r + \frac{ \alpha + \frac{1}{2} }{2 \alpha r} ( \xi \mathcal{U}_\xi - 1 )  =& \mathcal{E}_\xi [ \mathcal{Z}, \mathcal{U} ]
\end{aligned} \end{equation}
Note $\Omega = \left\{ (\xi, \tau) \in (0, \Upsilon) \times (\tau_0, \tau_1) \right\} =  \left\{ (\xi, r) \in (0, \Upsilon) \times (\frac{1}{2 \alpha} e^{2 \alpha \tau_0}, \frac{1}{2 \alpha} e^{2 \alpha \tau_1}) \right\}$.

\begin{lem} \label{innerInteriorEstZ}
	For any $\Upsilon > 0$ and $\tau_0 \gg 1$ sufficiently large depending on $\Upsilon$, the following holds:
	If $\tau_1 > \tau_0$ and $Z,U$ satisfy 
		$$Z_\tau = \mathcal{F}_\gamma [Z, U] - \frac{\gamma}{2} Z_\gamma$$
		$$B^2 \le Z \le 1, \quad 0 \le U_\gamma \le \frac{1}{\gamma} \qquad \text{for all } (\gamma, \tau) \in \Omega( \Upsilon, \tau_0, \tau_1)$$
	and the initial data satisfies
		$$|\gamma^2 Z_{\gamma \gamma} | + |\gamma Z_{\gamma} | + | Z| \le C \qquad \text{for all } 0 \le \gamma \le \Upsilon e^{- \alpha \tau_0}, \tau = \tau_0$$
	then there exists a constant $C'$ (depending only on $p,q,k$, and $C$) such that
		$$| \gamma^2 Z_{\gamma \gamma} | + | \gamma Z_\gamma | \le C' \qquad  \text{for all } (\gamma, \tau) \in \Omega\left( \frac{1}{2} \Upsilon, \tau_0, \tau_1 \right)$$
\end{lem}
\begin{proof}
	In terms of the variables
		$$\xi = \gamma e^{\alpha \tau} \qquad r \doteqdot \frac{1}{2 \alpha} e^{2 \alpha \tau} $$ 
	it follows that	
		$$\clZ_r = \mathcal{F}_{\xi} [ \clZ, \clU] - \frac{ \alpha + \frac{1}{2} }{2 \alpha r} \xi \clZ_\xi$$
	We will obtain bounds on $\clZ$ through a rescaling argument.
	Indeed, fix $0 \le \xi_0 \le \frac{1}{\sqrt{2}} \Upsilon$ and $ r_0 > r(\tau_0)$, and define
		$$M(x,s) \doteqdot \clZ \big( \xi_0 ( 1+ x), r_0 + s \xi_0^2 \big)$$ 
	It follows that
	\begin{gather*}
		M_s(x,s) = M M_{xx}(x,s) + M_x \left( \frac{q-1-M}{1+x} - \frac{1}{2} M_x - \frac{ \alpha + \frac{1}{2}}{2 \alpha \left( \frac{r_0}{\xi_0^2} + s \right)} (1+x) \right) \\
		+ \left[ \frac{2(q-1)}{(1+x)^2} ( 1 - M) - 2p \clU_\xi^2 \xi_0^2 M \right]	M
	\end{gather*}
	If $\tau_0$ is taken sufficiently large so that $ \frac{ r(\tau_0)}{\Upsilon^2} > 1$, then $\frac{r_0}{\xi_0^2} > 2$.
	From this estimate and the assumed scale-invariant pointwise bounds on $Z$ and $U_\gamma$, it follows that, up to a semilinearity $-\frac{1}{2} M_x^2$, $M$ satisfies a linear parabolic equation with bounded coefficients on $ |x| \le \frac{1}{2}, \max(-1, s_*) \le s \le 0$, where $s_*(r_0, \xi_0, r(\tau_0))$ is defined by $r_0 + s_* \xi_0^2 = r(\tau_0)$.
	Moreover, the bounds on the coefficients depend only on $p,q$, and $\alpha$.
	Hence, interior estimates for such equations (see e.g. Chapter V.3 of ~\cite{LSU88}) imply that there exists a constant $C''$ depending only on $p,q,\alpha$, and $C$ such that
		$$| \xi_0 \clZ_\xi( \xi_0, r_0 ) | = | M_x(0,0)| \le C''$$
	Since $\xi_0 \in [0 , \frac{1}{\sqrt{2}} \Upsilon ]$ was arbitrary, it follows that
	\begin{equation*} \begin{aligned}
		| \xi \clZ_\xi | \le & C'' \qquad& \text{for all } 0 \le \xi \le \frac{1}{\sqrt{2}} \Upsilon,  \tau_0 \le \tau \le \tau_1\\
		\text{or equivalently } \quad | \gamma Z_\gamma | \le& C'' \qquad& \text{for all } 0 \le \gamma \le \frac{1}{\sqrt{2}} \Upsilon e^{- \alpha \tau}, \tau_0 \le \tau \le \tau_1
	\end{aligned} \end{equation*}
	
	Now, equipped with this bound on the derivative, $M$ can be regarded as solving a \textit{linear} parabolic equation with bounded coefficients.
	As before, interior estimates then yield the corresponding bound on $Z_{\gamma \gamma}$ in the region $\Omega(\frac{1}{2} \Upsilon, \tau_0, \tau_1)$.
\end{proof}

\begin{remark}
	In the above proof, it is important that the semilinearity is of the form $-\frac{1}{2} M_x^2$ for interior estimates to apply.
	Indeed, a result of Angenent-Fila ~\cite{AF96} shows that, in general, interior gradient estimates fail to hold for equations with semilinearities containing larger powers of the gradient.
\end{remark}

\begin{lem} \label{innerInteriorEstU}
	For any $\Upsilon >0$ and $\tau_0 \gg 1$ sufficiently large depending on $\Upsilon$, the following holds:
	if $\tau_1 > \tau_0$ and $Z, U$ satisfy 
		$$U_\tau + \frac{1}{2} ( \gamma U_\gamma - 1) = \mathcal{E}_{\gamma} [ Z, U]$$
		$$\log(\frac{A}{B} \gamma) \le U, \quad 0 \le U_\gamma \le \frac{1}{\gamma} \qquad  B^2 \le Z \le 1 \qquad | \gamma Z_\gamma| \le C $$
		$$\text{for all } (\gamma, \tau) \in \Omega(\Upsilon, \tau_0, \tau_1)$$
	and the initial data satisfies
		$$| \gamma^2 U_{\gamma \gamma} | \le C \qquad \text{for all } 0 \le \gamma \le \Upsilon e^{- \alpha \tau_0}, \tau = \tau_0$$
	then there exists a constant $C'$ (depending only on $p,q,k$, and $C$) such that 
		$$ | \gamma^2 U_{\gamma \gamma} | \le C' \qquad \text{for all } (\gamma, \tau) \in \Omega \left( \frac{1}{2} \Upsilon, \tau_0, \tau_1 \right)$$
\end{lem}
\begin{proof}
	Recall that
		$$\clU_r + \frac{ \alpha + \frac{1}{2}}{2 \alpha r} ( \xi \clU_\xi  - 1 ) = \mathcal{E}_\xi [ \clZ, \clU ]$$
	Differentiating with respect to $\xi$, one obtains the evolution equation for $\clU_\xi$
		$$\partial_r \clU_\xi +  \frac{\alpha + \frac{1}{2}}{2 \alpha r} ( \clU_\xi + \xi \clU_{\xi \xi} ) = \mathcal{L}_\xi [\clZ, \clU] \clU_\xi$$
	The remainder of the proof follows by similar logic as in the proof of lemma \ref{innerInteriorEstZ} by considering 
		$$M(x,s) \doteqdot \xi_0 \clU_\xi \big( \xi_0(1+x), r_0 + \xi_0^2 s \big)$$
	for fixed $0 \le \xi_0 \le \frac{1}{2} \Upsilon$ and $ r_0 > r(\tau_0)$.
\end{proof}

\subsubsection{Barriers for the Inner-Parabolic Interface}

Unfortunately, the barriers contained in the previous sections are insufficient to control certain contributions from the error terms, and it will be necessary to include additional barriers that control the solutions in the inner-parabolic interface.
This subsection carries out the construction of those barriers.
We begin with an interior estimate that will allow for improved barriers in the inner-parabolic interface.

\begin{lem} \label{innerInteriorEstU+}
	Let $0 < a < b \le D \Upsilon_0$, $C_0 > 0$, and $\kappa \in \mathbb{R}$.
	Assume that
		$$B^2 \le \clZ \le 1 \quad \text{and} \quad \log(\frac{A}{B} \xi) \le \clU \qquad \text{ for $\xi \in [a,b]$ and $\tau \ge \tau_0$}$$
	and $\tau_0$ is sufficiently large so that
		$$\frac{1}{2 \alpha} e^{2 \alpha \tau_0} \ge \frac{8}{9} D^2 \Upsilon_0^2$$
	If
		$$|\tl{\clU}| \le C \xi^\kappa \qquad \text{for all } \xi \in [a,b], \quad \tau \ge \tau_0$$
	and there are initial condition bounds
		$$| \xi^2 \tl{\clU}_{\xi \xi} | + | \xi \tl{\clU}_\xi | + |\tl{\clU}| \le C \xi^\kappa \qquad \text{for all } \xi \in [a,b], \quad \tau = \tau_0$$
	then there exists a constant $C'=C'(p,q,k)$ such that 
		$$|\xi^2 \tl{\clU}_{\xi \xi }| + |\xi \tl{\clU}_\xi | \le C' 2^{|\kappa|} C  \xi^{\kappa} \qquad \text{for all } \xi \in \left[2a, \frac{2}{3}b \right] \quad \tau \ge \tau_0$$
\end{lem}
\begin{proof}
	We set
		$$r = \frac{1}{2 \alpha} e^{2 \alpha \tau}$$
	and write the evolution equation for $\tl{\clU}$ in the $\xi, r$ coordinates as
		$$\tl{\clU}_r = \clZ  \tl{\clU}_{\xi \xi} + \frac{\clZ + q-1}{\xi}  \tl{\clU}_\xi + \frac{q-1}{\xi^2} \left( 1 - e^{-2\tl{\clU}} \right) - \frac{\alpha + \frac{1}{2} }{2 \alpha r} \xi \tl{\clU}_\xi$$
		
	As usual, the interior estimate will be proved with a rescaling argument.
	Indeed, fix $\xi_0 \in [2a, \frac{2}{3} b]$ and $r_0 > r(\tau_0)$.
	Define
		$$W(x,s) \doteqdot \tl{\clU}( \xi_0 (1+x), r_0 + s \xi_0^2 )$$
	It follows that $W$ satisifies
		$$W_s(x,s) = \clZ(\xi, r) W_{xx}(x,s) + \frac{q-1+\clZ}{(1+x)} W_x + \frac{q-1}{(1+x)^2}(1 - e^{-2 W}) - \frac{ \alpha + \frac{1}{2} }{2 \alpha ( \frac{r_0}{\xi_0^2} + s )} (1+x) W_x$$
	We observe that
		$$\log \left( \frac{A}{B} \xi \right) \le \clU \implies 0 \le \tl{\clU}  \implies 0 \le 1 - e^{-2W} \le 1$$
		$$\text{In fact, concavity implies } 0 \le 1 - e^{-2W} \le 2W$$
	By assumption, $\tau_0$ is sufficiently large so that
		$$\frac{r_0}{\xi_0^2} \ge \frac{r(\tau_0)}{ \frac{4}{9} b^2} \ge \frac{r(\tau_0)}{ \frac{4}{9} D^2 \Upsilon_0^2} \ge 2$$
	Let $s_*< 0$ be defined by $r_0+s_* \xi_0^2 = r(\tau_0)$.
	For $(x,s) \in [-1/2, 1/2] \times [s_*,0]$, $W(x,s)$ satisfies an inhomogeneous linear parabolic equation with bounded coefficients and bounded inhomogeneous term $\frac{q-1}{(1+x)^2}(1 - e^{-2 W})$.
	Note also that the bounds on the coefficients depend only on $p,q,k$.
	Interior estimates for inhomogeneous linear parabolic equations with bounded coefficients (see e.g. IV.10 of ~\cite{LSU88}), then imply that there exists a $C'' = C''(p,q,k)$ such that
	\begin{equation*} \begin{aligned}
		& |\xi_0^2 \tl{\clU}_{\xi \xi}(\xi_0, r_0)| + | \xi_0 \tl{\clU}_\xi(\xi_0, r_0)|  \\
		=& |W_{xx}(0,0)|+ | W_x(0,0)|\\
		\le& C'' \left( \sup_{|x|\le 1/2, s\in[s_*,0]} \left| \frac{q-1}{(1+x)^2}(1 - e^{-2 W}) \right| + \sup_{|x|\le1/2, s\in[s_*,0]} |W(x,s)| \right)\\
		&+ C'' \sup_{|x|\le 1/2, s=s_*} \left( |W_{xx}| + |W_x| + |W| \right)\\
		\le& C''(8(q-1) + 1 ) \sup_{|x|\le 1/2, s\in[s_*,0]} |W| + C'' \sup_{|x|\le 1/2, s=s_*} \left( |W_{xx}| + |W_x| + |W| \right)\\
		=& C''(8(q-1) + 1 ) \sup_{\xi \in [\xi_0/2, (3/2) \xi_0], r(\tau) \ge r(\tau_0)} | \tl{\clU}(\xi, r) |\\
		&+ C'' \left( \sup_{ \xi \in [\xi_0/2, (3/2) \xi_0], r(\tau) = r(\tau_0)} |\tl{\clU}| +  | \xi_0 \tl{\clU}_\xi  | +  | \xi_0^2 \tl{\clU}_{\xi \xi} | \right)\\
		\le& C''(8(q-1) + 1 ) \sup_{\xi \in [\xi_0/2, (3/2) \xi_0], r(\tau) \ge r(\tau_0)} | \tl{\clU}(\xi, r) |\\
		&+ 4C'' \left( \sup_{ \xi \in [\xi_0/2, (3/2) \xi_0], r(\tau) = r(\tau_0)} |\tl{\clU}| +  | \xi \tl{\clU}_\xi  | +  | \xi^2 \tl{\clU}_{\xi \xi} | \right)\\
		\le& C'' (8(q-1)+1+4) \sup_{\xi \in [\xi_0/2, (3/2) \xi_0]} C \xi^\kappa \\
		\le & C''(8q-3)  2^{|\kappa|} C \xi_0^\kappa\\
	\end{aligned} \end{equation*}
	Setting $C' = C''(8q-3)$ completes the proof of the proposition.
\end{proof}

\begin{lem} \label{innerInteriorEstZ+}
	Let $0 < a < b \le D \Upsilon_0$, $C_0 > 0$ and $\kappa \in \mathbb{R}$.
	Assume that
		$$B^2 \le \cl{Z} \le 1 \quad \text{ and } \quad  | \clU_\xi | \le \frac{1}{\xi} \quad \text{ for } \xi \in [a,b] \text{ and } \tau \ge \tau_0$$
	and $\tau_0 \gg 1$ is sufficiently large so that
		$$\frac{1}{2 \alpha} e^{ 2 \alpha \tau_0} \ge \frac{8}{9} D^2 \Upsilon_0^2$$
	If 
	\begin{equation*} \begin{aligned}
		| \tl{ \cl{U} }_\xi | &\le C \xi^{\kappa -1}  && \text{ for all } \xi \in [a,b], \tau \ge \tau_0 \\
		| \tl{ \cl{Z} } | &\le  C \xi^{\kappa} 	&& \text{ for all } \xi \in [a,b], \tau \ge \tau_0 \\
	\end{aligned} \end{equation*}
	and there are initial condition bounds
		$$| \xi^2 \tl{ \clZ}_{\xi \xi} | + | \xi \tl{\clZ}_{\xi} | + | \tl{ \clZ} | \le C \xi^\kappa 	\qquad \text{ for all } \xi \in [a,b], \quad \tau = \tau_0$$
	then there exists a constant $C' = C'(p,q,k)$ such that
		$$| \xi^2 \tl{ \clZ}_{\xi \xi} | + | \xi \tl{\clZ}_{\xi} | \le C' C 4^{|\kappa|} \xi^\kappa 	\qquad \text{ for all } \xi \in \left[ 4 a, \frac{4}{9} b \right], \quad \tau \ge \tau_0$$
\end{lem}
\begin{proof}
	The proof follows similarly as in the above proof of lemma \ref{innerInteriorEstU+}.
	Begin by setting
		$$r = \frac{1}{2 \alpha} e^{ 2 \alpha \tau}$$
	and writing the evolution equation for $\tl{ \clZ}$ in $\xi, r$ coordinates
		$$ \tl{ \clZ}_r = \cl{F}^l_\xi [ \tl{ \clZ} ] + 2 \cl{Q}_\xi [ \tl{ \clZ}, B^2 ] + \cl{F}^q [ \tl{ \clZ} ] - 2p ( B^2 + \tl{ \clZ} )^2 \cl{U}_\xi^2 + 2p B^4 \frac{1}{\xi^2} - \frac{ \alpha + \frac{1}{2} }{2 \alpha r} \xi \tl{ \clZ} $$
	or equivalently
	\begin{equation*} \begin{aligned}
		\tl{ \cl{Z} }_r
		& = B^2 \tl{\clZ}_{\xi \xi} + \left( \frac{q-1-B^2}{\xi} \right) \tl{\clZ}_\xi - \frac{2(q-1)}{\xi^2}  \tl{\clZ} 
		-4pB^2 \frac{1}{\xi^2} \tl{\clZ} - 4p B^4 \frac{1}{\xi}\tl{\clU}_\xi \\
		& + \mathcal{F}^q_\xi [ \tl{\clZ} ] -2p (\tl{\clZ})^2 \clU_\xi^2 
		  -2p B^4 \tl{\clU}_\xi^2 - 4p B^2 \left( \clU_\xi^2 - \frac{1}{\xi^2} \right) \tl{\clZ}
		 - \frac{ \alpha + \frac{1}{2} }{2 \alpha r}  \xi \tl{\clZ}_\xi \\
		 &= \clZ \tl{ \clZ}_{\xi \xi} + \tl{ \clZ}_\xi \left( \frac{ q-1- \clZ}{\xi } \right) - \frac{1}{2} \tl{ \clZ}_\xi^2
		+ \frac{2(q-1)}{\xi^2} \tl{ \clZ} ( 1 - B^2 - \clZ ) \\
		&- 2p \left( \clZ \clU_\xi + B^2 \frac{1}{\xi} \right) \left( B^2 \tl{ \clU}_\xi + \cl{U}_\xi \tl{\clZ} \right)
		- \frac{ \alpha + \frac{1}{2} }{2 \alpha r}  \xi \tl{\clZ}_\xi \\
	\end{aligned} \end{equation*}

	Fix $\xi_0 \in \left[ 2 a , \frac{2}{3} b \right]$ and $r_0 > r(\tau_0)$.
	Define
		$$W(x,s) \doteqdot \tl{ \clZ} ( \xi_0 (1+x), r_0 + s \xi_0^2 )$$
	It follows that $W$ satisfies
	\begin{equation*} \begin{aligned}
		W_s 
		&= \clZ W_{xx} + W_x \frac{q-1- \clZ}{1+x} - \frac{1}{2} W_x^2 + \frac{2(q-1)}{(1+x)^2} ( 1 - B^2 - \clZ) W\\
		&- 2p \left( \xi_0 \clZ \clU_\xi + \frac{B^2}{1+x} \right) \left( B^2 \xi_0 \tl{ \clU}_\xi + \clU_\xi \xi_0 \tl{ \clZ} \right) \\
		& - \frac{ \alpha + \frac{1}{2} }{2 \alpha (\frac{r_0}{\xi_0^2} + s )}  ( 1+x)W_x \\
	\end{aligned} \end{equation*}
	
	For $x \in \left[ - \frac{1}{2}, \frac{1}{2} \right]$, the assumed bounds apply and so we have
	\begin{equation*} \begin{aligned}
		\left| \xi_0 \clU_\xi \right| & \le \frac{1}{1+x} \\
		\left| B^2 \xi_0 \tl{ \clU}_\xi \right| & \le B^2 \xi_0 C \xi^{\kappa - 1} = B^2 ( 1+x)^{\kappa - 1} C \xi_0^\kappa
	\end{aligned} \end{equation*}
	Hence, $W$ satisfies an inhomogeneous semilinear equation with bounded coefficients and bounded inhomogeneous term.
	Interior estimates then imply that
	\begin{equation*} \begin{aligned}
		| \xi_0 \tl{ \clZ}_\xi( \xi_0, r_0) | = | W_x (0,0) | \lesssim_{p,q,k} C 2^{| \kappa | } \xi_0^{\kappa}
	\end{aligned} \end{equation*}
	Equipped with this estimate, we can then regard $W$ as solving an inhomogeneous \textit{linear} equation with bounded coefficients and bounded inhomogeneous term on the region 
	$ \xi \in \left[ 4 a, \frac{4}{9} b \right]$.
	By the dependence of the bounds on $p,q,k$, interior estimates then imply that
	\begin{equation*} \begin{aligned}
		| \xi_0^2 \tl{ \clZ }_{\xi \xi } (\xi_0, r_0) | + | \xi_0 \tl{ \clZ}_\xi( \xi_0, r_0) | 
		= | W_{xx} (0,0) | + | W_x(0,0)| 
		\le C'_{p,q,k} C 4^{| \kappa |}  \xi_0^{\kappa}
	\end{aligned} \end{equation*}

\end{proof}

\begin{lem} \label{innerParabExtendZBounds}
	Assume that there exists $0 < \eta < -C$ such that
		$$(C- \eta) \xi^{-2k-2B^2 \lambda -1} \le \tl{\clU}_\xi \le (C+ \eta) \xi^{-2k -2B^2 \lambda -1} \qquad \text{for all } (\xi , \tau) \in [\Upsilon_U, \Upsilon_Z] \times [\tau_0, \tau_1]$$
	and $D_{\pm} > 0$ are positive constants such that
		$$D_- < 4p B^2 ( C + \eta) \left[ (2k+2B^2 \lambda)^2 -(n-3)(2k + 2B^2 \lambda) - 2(n-1)  \right]^{-1} $$
		$$< 4p B^2 ( C - \eta) \left[ (2k+2B^2 \lambda)^2 -(n-3)(2k + 2B^2 \lambda) - 2(n-1)  \right]^{-1} < D_+$$
	For any $\Upsilon_U \gg 1$ sufficiently large (depending on $p,q,k,C, D_{\pm}$), $\Upsilon_Z > \Upsilon_U$, and $\tau_0 \gg 1$ sufficiently large (depending on $p,q,k,C, D, \Upsilon_U, \Upsilon_Z$),
	if 
		$$D_{-} \xi^{-2k-2B^2 \lambda} \le \tl{\clZ} \le D_+ \xi^{-2k-2B^2 \lambda} 		\qquad \text{for all } (\xi, \tau) \in \partial_P \left( [\Upsilon_U, \Upsilon_Z] \times [\tau_0, \tau_1] \right) $$
	then
		$$D_{-} \xi^{-2k-2B^2 \lambda} \le \tl{\clZ} \le D_+ \xi^{-2k-2B^2 \lambda} 		\qquad \text{for all } (\xi, \tau) \in  [\Upsilon_U, \Upsilon_Z] \times [\tau_0, \tau_1]$$	
\end{lem}
\begin{proof}
	By the comparison principle, it suffices to check that
		$$\tl{\clZ}^{\pm} \doteqdot D_{\pm} \xi^{-2k -2B^2 \lambda}$$
	are sub/super-solutions on $[\Upsilon_U, \Upsilon_Z] \times [\tau_0, \tau_1]$.
	We estimate
	\begin{equation*} \begin{aligned}
		& \mathcal{F}^l_\xi [ \tl{\clZ}^- ] + 2 \mathcal{Q}_\xi [ \tl{\clZ}^-, B^2 ] + \mathcal{F}^q_\xi [ \tl{\clZ}^- ] - 2p (B^2 + \tl{\clZ}^-)^2 \clU_\xi^2 + 2p B^4 \frac{1}{\xi^2} \\
		&- e^{-2 \alpha \tau} \left( \alpha + \frac{1}{2} \right) \xi \tl{\clZ}^-_\xi \\
		= & B^2 \tl{\clZ}^-_{\xi \xi} + \left( \frac{q-1-B^2}{\xi} \right) \tl{\clZ}^-_\xi - \frac{2(q-1)}{\xi^2}  \tl{\clZ}^- 
		-4pB^2 \frac{1}{\xi^2} \tl{\clZ}^- - 4p B^4 \frac{1}{\xi}\tl{\clU}_\xi \\
		& + \mathcal{F}^q_\xi [ \tl{\clZ}^- ] -2p (\tl{\clZ}^-)^2 \clU_\xi^2 
		  -2p B^4 \tl{\clU}_\xi^2 - 4p B^2 \left( \clU_\xi^2 - \frac{1}{\xi^2} \right) \tl{\clZ}^- \\
		 &- e^{-2 \alpha \tau} \left( \alpha + \frac{1}{2} \right) \xi \tl{\clZ}^-_\xi \\
		 \ge & B^2  \left[ (2k+2B^2 \lambda)^2 -(n-3)(2k + 2B^2 \lambda) - 2(n-1)  \right] D_- \xi^{-2k -2B^2 \lambda - 2} \\
		 &-4pB^4 (C+ \eta) \xi^{-2k -2B^2 \lambda -2}
		 + \mathcal{F}^q_\xi [ \tl{\clZ}^- ] -2p (\tl{\clZ}^-)^2 \clU_\xi^2 \\
		 &-2p B^4 \tl{\clU}_\xi^2 - 4p B^2 \left( \clU_\xi^2 - \frac{1}{\xi^2} \right) \tl{\clZ}^-
		 - e^{-2 \alpha \tau} \left( \alpha + \frac{1}{2} \right) \xi \tl{\clZ}^-_\xi \\
		 \ge & B^2  \left[ (2k+2B^2 \lambda)^2 -(n-3)(2k + 2B^2 \lambda) - 2(n-1) - 4p \right] D_- \xi^{-2k -2B^2 \lambda - 2} \\
		 &-4pB^4 (C+ \eta) \xi^{-2k -2B^2 \lambda -2}
		  + M \Upsilon_U^{-2k-2B^2 \lambda} ( (|C| + |\eta|)^2 + D_-^2 ) \xi^{-2k -2B^2 \lambda - 2} \\
		 &+ \Upsilon_Z^2 e^{-2 \alpha \tau_0} \left( \alpha + \frac{1}{2} \right) (2k + 2B^2 \lambda ) D_- \xi^{-2k - 2B^2 \lambda - 2}
\end{aligned} \end{equation*}
	where $M = M(p,q,k)$ is a constant that depends only on $p,q,k$.
	
	By the assumption on $D_-$, it follows that
		$$B^2  \left[ (2k+2B^2 \lambda)^2 -(n-3)(2k + 2B^2 \lambda) - 2(n-1) - 4p \right] D_-  -4pB^4 (C+ \eta)  > 0$$
	It then follows that $\Upsilon_U$ can be chosen sufficiently large (depending on $p,q,k,C, D$) and $\tau_0$ can be chosen sufficiently large (depending on $p,q,k,C, D, \Upsilon_U, \Upsilon_Z$) so that the above differential equation is positive for all $(\xi, \tau) \in [\Upsilon_U, \Upsilon_Z] \times [\tau_0, \tau_1]$.
	That is, $\tl{\clZ}^-$ is a subsolution on $[\Upsilon_U, \Upsilon_Z] \times [\tau_0, \tau_1]$.
	
	The argument that $\tl{\clZ}^+$ is a supersolution is analogous.
\end{proof}

These interior estimates allow for an improved upper barrier construction in the inner-parabolic interface.
We will proceed to work in terms of the inner region coordinates

	$$\xi = \gamma e^{\alpha_k \tau} 	\qquad r = \frac{1}{2 \alpha} e^{2 \alpha \tau}$$
	$$\mathcal{Z}(\xi, r) = Z(\gamma, \tau) \qquad	 \mathcal{U}(\xi, r) = U(\gamma, \tau) + \alpha \tau$$
	$$\tl{\cl{Z}} = \cl{Z} - B^2 \qquad 		\tl{\cl{U}} = \cl{U} - \log \left( \frac{A}{B} \xi \right)$$
If $Z,U$ satisfy the rescaled Ricci flow equations (\ref{psRF}) then $\mathcal{Z}, \mathcal{U}$ satisfy
\begin{equation} \label{isRF} \begin{aligned}
	 \mathcal{Z}_r + \frac{ \alpha + \frac{1}{2} }{2 \alpha r} \xi \mathcal{Z}_\xi  =& \mathcal{F}_\xi [ \mathcal{Z}, \mathcal{U}_\xi ] \\
	 \mathcal{U}_r + \frac{ \alpha + \frac{1}{2} }{2 \alpha r} ( \xi \mathcal{U}_\xi - 1 )  =& \mathcal{E}_\xi [ \mathcal{Z}, \mathcal{U} ]
\end{aligned} \end{equation}
Let $\Omega$ denote 
	$$\Omega = \left\{ (\xi, \tau) \in (0, \Upsilon) \times (\tau_0, \tau_1) \right\} =  \left\{ (\xi, r) \in (0, \Upsilon) \times \left(\frac{1}{2 \alpha} e^{2 \alpha \tau_0}, \frac{1}{2 \alpha} e^{2 \alpha \tau_1} \right) \right\}$$
and $\partial_P \Omega$ its parabolic boundary.
Note that these variables and sets are identical to those at the beginning of subsubsection \ref{innerIntEsts}.

We shall also assume throughout this subsection that for all $(\xi, \tau) \in \Omega$
	$$B^2 \le \clZ \le 1 \qquad 0 \le \clU_\xi \le \frac{1}{\xi} \qquad \log \left( \frac{A}{B} \xi \right) \le \clU$$
or equivalently
	$$0 \le \tl{\clZ} \le 1- B^2 \qquad - \frac{1}{\xi} \le \tl{\clU}_\xi \le 0 \qquad 0 \le \tl{\clU}$$

\begin{lem} \label{innerFarBarrierUWeak}
	Assume $q \ge 10$.
	For any $C_0, \Upsilon > 0$ and 
		$$a \in \left( \frac{q-1}{2} - \frac{1}{2} \sqrt{(q-9)(q-1)} , \frac{q-1}{2} + \frac{1}{2} \sqrt{(q-9)(q-1)} \right)$$
	and all $\tau_0 \gg 1$ sufficiently large (depending $q, \Upsilon, \alpha_k,$ and $a$), the following holds:
	If
		$$\tl{\clU} \le C_0 \xi^{-a}		\qquad \text{for all } (\xi, \tau) \in \partial_P \Omega$$
	then
		$$\tl{\clU} \le C_0 \xi^{-a}		\qquad \text{for all } (\xi, \tau) \in \Omega$$
\end{lem}
\begin{proof}
	It suffices to check that $\tl{\clU}^+ \doteqdot C_0 \xi^{-a}$ is a supersolution on $\Omega$, i.e.
		$$\clZ \left( \tl{\clU}^+_{\xi \xi} + \frac{ \tl{\clU}^+_\xi}{\xi} \right) + \frac{q-1}{\xi} \tl{\clU}^+_\xi + \frac{q-1}{\xi} \left( 1 - e^{-2 \tl{\clU}^+} \right)	- e^{-2\alpha \tau} \left(\alpha + \frac{1}{2} \right) \xi \tl{\clU}^+_\xi \le 0$$
	for all $(\xi, \tau) \in \Omega$.
	Note in particular that $a > 0$ and in fact even $a > 2k +2B^2 \lambda_k = \frac{n-1}{2} - \frac{1}{2} \sqrt{(n-9)(n-1)}$ since
		$$x \mapsto \frac{x-1}{2} - \frac{1}{2} \sqrt{(x-9)(x-1)} \text{ is decreasing for } x \ge 9$$

	Using concavity of $x \mapsto 1-e^{-2x}$ we estimate
	\begin{equation*} \begin{aligned}
		&\clZ \left( \tl{\clU}^+_{\xi \xi} + \frac{ \tl{\clU}^+_\xi}{\xi} \right) + \frac{q-1}{\xi} \tl{\clU}^+_\xi + \frac{q-1}{\xi} \left( 1 - e^{-2 \tl{\clU}^+} \right)	- e^{-2\alpha \tau} \left(\alpha + \frac{1}{2} \right) \xi \tl{\clU}^+_\xi \\
		\le & \clZ \left( \tl{\clU}^+_{\xi \xi} + \frac{ \tl{\clU}^+_\xi}{\xi} \right) + \frac{q-1}{\xi} \tl{\clU}^+_\xi + \frac{q-1}{\xi} \left( 2 \tl{\clU}^+\right)	- e^{-2\alpha \tau} \left(\alpha + \frac{1}{2} \right) \xi \tl{\clU}^+_\xi \\
		= & \left( \clZ a^2 -(q-1)a + 2(q-1) \right) C_0 \xi^{-a-2}- e^{-2\alpha \tau} \left(\alpha + \frac{1}{2} \right) \xi \tl{\clU}^+_\xi \\
		\le & \left(  a^2 -(q-1)a + 2(q-1) \right) C_0 \xi^{-a-2}
		+a \xi^2 e^{-2\alpha \tau} \left(\alpha + \frac{1}{2} \right) C_0 \xi^{-a-2} \\
		\le & \left(  a^2 -(q-1)a + 2(q-1) \right) C_0 \xi^{-a-2}
		+a \Upsilon^2 e^{-2\alpha \tau_0} \left(\alpha + \frac{1}{2} \right) C_0 \xi^{-a-2} \\
	\end{aligned} \end{equation*}
	Observe that $a \in \left( \frac{q-1}{2} - \frac{1}{2} \sqrt{(q-9)(q-1)} , \frac{q-1}{2} + \frac{1}{2} \sqrt{(q-9)(q-1)} \right) $
	implies that $a^2 -(q-1)a+2(q-1) < 0$.
	Thus, if $\tau_0 \gg 1$ is chosen sufficiently large (depending only on $q, a, \Upsilon, \alpha$) so that
		$$a \Upsilon^2 e^{-2\alpha \tau_0} \left(\alpha + \frac{1}{2} \right)  < \left| a^2 -(q-1)a + 2(q-1) \right|$$
	it follows that $\tl{\clU}^+$ is a supersolution on  $\Omega$.
\end{proof}

\begin{lem} \label{innerFarBarrierZ}
	Assume $q \ge 10$.
	If $C_0, D_0, a, b, \epsilon > 0$ are positive constants such that 
		$$0 < b < a \in \left( \frac{q-1}{2} - \frac{1}{2} \sqrt{(q-9)(q-1)} , \frac{q-1}{2} + \frac{1}{2} \sqrt{(q-9)(q-1)} \right)$$
		$$b < 2 \sqrt{q} + 2 \quad (< n-1)$$
		$$0 < \epsilon < \frac{1}{a} (2k + 2B^2 \lambda_k)$$
	 then for any $\Upsilon \gg 1$ sufficiently large (depending on $p,q,C_0, D_0,M,a,b,\epsilon$) there exists $\tau_0 \gg 1$ sufficiently large (depending on $p,q,C_0, D_0,M,a,b,\epsilon, \Upsilon$) such that the following holds:

	If 
		$$\tl{\clU} \le C_0 \Upsilon^{a-2k-2B^2 \lambda} \xi^{-a} \qquad \text{for all } (\xi, \tau) \in \left[0 , \frac{3}{2}\Upsilon \right] \times[\tau_0, \tau_1]$$
	and
	$$\tl{\clZ}(\xi, \tau) \le D_0 \Upsilon^{\frac{b}{a}(a-2k-2B^2 \lambda) + \epsilon} \xi^{-b} \qquad \text{for all } (\xi, \tau) \in \partial_{P}\Omega$$
	then
		$$\tl{\clZ}(\xi, \tau) \le D_0 \Upsilon^{\frac{b}{a}(a-2k-2B^2 \lambda) + \epsilon} \xi^{-b} \qquad \text{for all } (\xi, \tau) \in \Omega$$
\end{lem}
\begin{proof}
	Denote
		$$C_0' = C_0 \Upsilon^{a - 2k -2B^2 \lambda} 	\qquad D_0' = D_0 \Upsilon^{\frac{b}{a}(a-2k-2B^2 \lambda) + \epsilon}$$
	Recall that $0 \le \tl{\clZ} \le 1-B^2$.
	Therefore, it suffices to confirm that $\tl{\clZ}^+ \doteqdot D_0' \xi^{-b}$ is a supersolution on the region $[\Xi, \Upsilon] \times [\tau_0, \infty)$ where $\Xi$ is defined by
		$$D_0' \Xi^{-b} = 1-B^2$$
	Note in particular that $\Upsilon^{\frac{b}{a}(a-2k-2B^2 \lambda) + \epsilon} \sim D_0' \sim \Xi^b$ as $\Upsilon \nearrow \infty$.
	
	If $M = M(p,q,k, a)$ denotes the constant from the interior estimate proposition \ref{innerInteriorEstU+}, then we have that
		$$| \tl{\clU}_\xi | \le M C_0' \xi^{-a -2}$$
	We now proceed to estimate the equation for $\tl{\clZ}^+$.
	\begin{equation*} \begin{aligned}
		& \mathcal{F}^l_\xi [ \tl{\clZ}^+ ] + 2 \mathcal{Q}_\xi [ \tl{\clZ}^+, B^2 ] + \mathcal{F}^q_\xi [ \tl{\clZ}^+ ] - 2p (B^2 + \tl{\clZ}^+)^2 \clU_\xi^2 + 2p B^4 \frac{1}{\xi^2} \\
		&- e^{-2 \alpha \tau} \left( \alpha + \frac{1}{2} \right) \xi \tl{\clZ}^+_\xi \\
		= & B^2 \tl{\clZ}^+_{\xi \xi} + \left( \frac{q-1-B^2}{\xi} \right) \tl{\clZ}^+_\xi - \frac{2(q-1)}{\xi^2}  \tl{\clZ}^+  \\
		& + \mathcal{F}^q_\xi [ \tl{\clZ}^+ ] -2p (\tl{\clZ}^+)^2 \clU_\xi^2 
		  - 2p B^4 \left( \clU_\xi^2 - \frac{1}{\xi^2} \right) -4p B^2  \clU_\xi^2  \tl{\clZ}^+
		 - e^{-2 \alpha \tau} \left( \alpha + \frac{1}{2} \right) \xi \tl{\clZ}^+_\xi \\
		 \le &  B^2 \tl{\clZ}^+_{\xi \xi} + \left( \frac{q-1-B^2}{\xi} \right) \tl{\clZ}^+_\xi - \frac{2(q-1)}{\xi^2}  \tl{\clZ}^+  \\
		& + \mathcal{F}^q_\xi [ \tl{\clZ}^+ ] 
		+ 4pB^4 \xi^{-1} | \tl{\clU}_\xi |
		 - e^{-2 \alpha \tau} \left( \alpha + \frac{1}{2} \right) \xi \tl{\clZ}^+_\xi \\
		 \le & B^2 \left[ b^2 -(n-3)b-2(n-1) \right] D_0' \xi^{-b-2} \\
		 &+ \left[ \frac{1}{2} b^2 + 2b -2(q-1) \right] D_0'^2 \xi^{-2b-2}
		 + 4 p B^4 M C_0' \xi^{-a-2}\\
		 &+ b \Upsilon^{2} e^{-2\alpha \tau_0} \left( \alpha + \frac{1}{2} \right) D_0' \xi^{-b -2} \\
	\end{aligned} \end{equation*}
	Observe that $0 < b < 2 \sqrt{q} + 2$ implies that 
		$$ \left[ b^2 -(n-3)b-2(n-1) \right],  \left[ \frac{1}{2} b^2 + 2b -2(q-1) \right]  < 0$$
	
	Additionally, using $0 < b < a$ and $\xi \in [\Xi, \Upsilon]$
	\begin{equation*} \begin{aligned} 
		4p B^4 M C_0' \xi^{-a-2} = & \left( 4p B^4 M \frac{C_0'}{D_0'} \xi^{b-a} \right) D_0' \xi^{-b-2} \\
		\le &  \left( 4p B^4 M \frac{C_0'}{D_0'} \Xi^{b-a} \right) D_0' \xi^{-b-2}\\
		\le &  \left( \frac{4p B^4 M}{1-B^2} C_0' \Xi^{-a} \right) D_0' \xi^{-b-2} \\
		= & \left( \frac{4pB^4 M}{1-B^2} (1-B^2)^{a/b} C_0 \Upsilon^{a - 2k -2B^2 \lambda} D_0'^{-a/b} \right)D_0' \xi^{-b-2} \\
		=& \left( \frac{4pB^4 M}{1-B^2} (1-B^2)^{a/b} C_0 D_0 \Upsilon^{-\frac{a}{b} \epsilon} \right) D_0' \xi^{-b-2}
	\end{aligned} \end{equation*}
	Thus, if $\Upsilon$ is sufficiently large depending on $p,q,C_0, D_0,M,a,b,\epsilon$ then
		$$\left| 4 p B^4 M C_0' \xi^{-a-2} \right| \le \frac{1}{3} \left| B^2 \left[ b^2 -(n-3)b-2(n-1) \right] D_0' \xi^{-b-2} \right| \qquad \forall (\xi, \tau) \in \Omega$$
	It then follows that $\tl{\clZ}^+$ is a supersolution on $\Omega$ for $\tau_0$ sufficiently large depending on $p,q,C_0, D_0,M,a,b,\epsilon, \Upsilon$.
\end{proof}

\begin{remark}
	The assumption on $\epsilon > 0$ above is only to ensure that $\Xi \le \Upsilon$ for $\Upsilon \gg 1$.
	In other words, to ensure that the new upper bound isn't simply a consequence of the $\tl{\clZ} \le 1-B^2$ bound.
	In fact, $\Upsilon^\epsilon$ can be replaced with any $f(\Upsilon)$ such that
		$$1 \ll f(\Upsilon) \ll \Upsilon^{\frac{1}{a} (2k + 2B^2 \lambda)} \qquad \text{as } \Upsilon \nearrow \infty$$	
\end{remark}

\begin{lem} \label{innerFarBarrierU}
	Assume $q \ge 10$.
	If $D_0, C_1, \ul{C}, a, b,c,\kappa, \epsilon$ are positive constants such that 
		$$0 < b < a \in \left( \frac{q-1}{2} - \frac{1}{2} \sqrt{(q-9)(q-1)} , \frac{q-1}{2} + \frac{1}{2} \sqrt{(q-9)(q-1)} \right)$$
		$$0 < \epsilon < \frac{1}{a} (2k + 2B^2 \lambda_k)$$
		$$0 < \kappa < \sqrt{(n-9)(n-1)}$$
		$$ c \in ( \kappa, \kappa + 2k + 2B^2 \lambda)$$
		$$1 + \frac{\epsilon}{b}< \frac{1}{a} (2k + 2B^2 \lambda) + \frac{c}{2k + 2B^2 \lambda + \kappa}$$
	 then for any $\Upsilon \gg 1$ sufficiently large there exists $\tau_0 \gg 1$ sufficiently large such that the following holds:
	 
	 If
	 	$$\tl{\clZ}(\xi, \tau) \le D_0 \Upsilon^{\frac{b}{a}(a-2k-2B^2 \lambda) + \epsilon} \xi^{-b} \qquad \text{for all } (\xi, \tau) \in \Omega$$
		$$\tl{\clU}(\Upsilon, \tau) \le \ul{C} \Upsilon^{-2k-2B^2 \lambda} \qquad \text{for all } \tau \in[ \tau_0 , \tau_1]$$
	and
		$$\tl{\clU}(\xi, \tau) \le  C_1 \Upsilon^{c} \xi^{-2k -2B^2 \lambda - \kappa} \qquad \text{for all } (\xi, \tau) \in \partial_{P} \Omega$$
	then
		$$\tl{\clU}(\xi, \tau) \le C_1 \Upsilon^{c} \xi^{-2k -2B^2 \lambda - \kappa} \qquad \text{for all } (\xi, \tau) \in  \Omega$$
\end{lem}
\begin{proof}
Denote
		$$C_1' = C_1 \Upsilon^{c} \qquad D_0' = D_0 \Upsilon^{\frac{b}{a}(a-2k-2B^2 \lambda) + \epsilon} $$
	By the bounds
		$$- \frac{1}{\xi} \le \tl{\clU}_\xi \le 0 \qquad \tl{\clU}(\Upsilon , \tau) \le \ul{C} \Upsilon^{-2k-2B^2 \lambda}$$
	it follows that
		$$\tl{\clU} \le - \log( \xi) + C_\Upsilon \qquad C_\Upsilon \doteqdot \ul{C} \Upsilon^{-2k-2B^2 \lambda} + \log(\Upsilon)$$
	for all $(\xi, \tau) \in \Omega$.
		
	It suffices to check that
		$$\tl{\clU}^+ \doteqdot  C_1'\xi^{-2k -2B^2 \lambda - \kappa}$$
	is a supersolution for $\xi \in [\Xi, \Upsilon]$ where $\Xi > 0$ is defined to be the smallest solution of the equation
		$$C_1'\Xi^{-2k -2B^2 \lambda - \kappa} = - \log(\Xi) + C_\Upsilon$$
	Using concavity of $x \mapsto 1-e^{-2x}$ we estimate
	\begin{equation*} \begin{aligned}
		&\clZ \left( \tl{\clU}^+_{\xi \xi} + \frac{ \tl{\clU}^+_\xi}{\xi} \right) + \frac{q-1}{\xi} \tl{\clU}^+_\xi + \frac{q-1}{\xi} \left( 1 - e^{-2 \tl{\clU}^+} \right)	- e^{-2\alpha \tau} \left(\alpha + \frac{1}{2} \right) \xi \tl{\clU}^+_\xi \\
		\le & B^2 \left( \tl{\clU}^+_{\xi \xi} + \frac{ \tl{\clU}^+_\xi}{\xi} \right) + \frac{q-1}{\xi} \tl{\clU}^+_\xi + \frac{q-1}{\xi} \left( 2 \tl{\clU}^+\right) + \tl{\clZ} \left( \tl{\clU}^+_{\xi \xi} + \frac{ \tl{\clU}^+_\xi}{\xi} \right) 	- e^{-2\alpha \tau} \left(\alpha + \frac{1}{2} \right) \xi \tl{\clU}^+_\xi \\
		\le & B^2 \left[ (2k + 2B^2 \lambda + \kappa)^2-(n-1)(2k + 2B^2 \lambda + \kappa) + 2(n-1) \right] C_1' \xi^{-2k -2B^2 \lambda - \kappa -2}  \\
		& + D_0' \xi^{-b} ( -2k -2B^2 \lambda -\kappa)^2 C_1' \xi^{-2k-2B^2 \lambda - \kappa - 2}
		- e^{-2\alpha \tau} \left(\alpha + \frac{1}{2} \right) \xi \tl{\clU}^+_\xi \\
	\end{aligned} \end{equation*}
	Observe that for $\kappa \in \left( 0, \sqrt{(n-9)(n-1)} \right)$ the first coefficient is negative.
	We now estimate the coefficient of the second term
	\begin{equation*} \begin{aligned}
		D_0' \xi^{-b} & \le D_0 \Upsilon^{\frac{b}{a}(a-2k-2B^2 \lambda) + \epsilon} \Xi^{-b}	\\
		= & D_0 \left( \Upsilon^{1 - \frac{1}{a}(2k + 2B^2 \lambda) + \epsilon'} \Xi^{-1} \right)^b 	&&	(\text{where }  \epsilon/b \doteqdot\epsilon')\\ 
	\end{aligned} \end{equation*}
	To estimate this term, we claim that for any $\eta < \frac{c}{2k+2B^2 \lambda + \kappa} < 1$
		$$\Upsilon^{\eta} \lesssim \Xi \lesssim \Upsilon^{\frac{c}{2k+2B^2 \lambda + \kappa}}	\qquad \text{as } \Upsilon \nearrow + \infty$$
	Let $f(\xi)$ denote $$f(\xi) = \log(\xi) + C_1' \xi^{-2k - 2B^2 \lambda - \kappa} = \log(\xi) + C_1 \Upsilon^{c} \xi^{-2k - 2B^2 \lambda - \kappa} $$
	$f$ achieves its minimum at the unique critical point $\xi_*$ satisfying
		$$\xi_*^{2k+2B^2 \lambda + \kappa} = (2k + 2B^2 \lambda + \kappa) C_1' = (2k + 2B^2 \lambda + \kappa) C_1 \Upsilon^c$$
	Note that
		$$f(\xi_*) \sim \frac{c}{2k + 2B^2 \lambda + \kappa} \log \Upsilon \ll \log(\Upsilon) + \ul{C} \Upsilon^{-2k - 2B^2 \lambda}		\qquad \text{for } \Upsilon \gg 1$$
	It follows that $[\Xi, \Upsilon] \ne \emptyset$ and $\Xi \le \xi_*$ for $\Upsilon \gg 1$.
	Thus, we have the upper bound
		$$\Xi \lesssim \Upsilon^{\frac{c}{2k+2B^2 \lambda + \kappa}}		\qquad \text{for } \Upsilon \gg 1$$
	For the lower bound, it suffices to show that for any $\eta < \frac{c}{\kappa + 2k +2B^2 \lambda} < 1$
		$$f(\Upsilon^\eta) \ge \log(\Upsilon) + \ul{C} \Upsilon^{-2k -2B^2 \lambda} \qquad \text{if } \Upsilon \gg 1$$
	Indeed,
	\begin{equation*} \begin{aligned}
		f(\Upsilon^\eta) = &C_1 \Upsilon^{c - \eta( 2k + 2B^2 \lambda + \kappa) } + \eta \log( \Upsilon)\\
		\sim & \Upsilon^{c - \eta( 2k + 2B^2 \lambda + \kappa)} 		&& (\Upsilon \gg 1) \\
		\gtrsim & \log(\Upsilon) + \ul{C} \Upsilon^{-2k -2B^2 \lambda} 	&& (\Upsilon \gg 1)\\
	\end{aligned} \end{equation*}	
	This completes the proof of the claim that
		$$\Upsilon^\eta \lesssim \Xi \lesssim \Upsilon^{\frac{c}{2k + 2B^2 \lambda + \kappa} }$$
		
	With this estimate, it follows that for any $\eta < \frac{ c}{2k + 2B^2 \lambda + \kappa}$
	\begin{equation*} \begin{aligned}
		\Upsilon^{1 - \frac{1}{a}(2k + 2B^2 \lambda) + \epsilon'} \Xi^{-1} \lesssim & \Upsilon^{1 - \frac{1}{a}(2k + 2B^2 \lambda) + \epsilon' - \eta}
	\end{aligned} \end{equation*}
	
	If $1 + \frac{\epsilon}{b} < \frac{1}{a}(2k +2B^2 \lambda) + \frac{c}{2k + 2B^2 \lambda + \kappa}$, then $ \eta \in \left( 0, \frac{ c}{2k + 2B^2 \lambda + \kappa} \right) $ can be chosen such that
		$$\Upsilon^{1 - \frac{1}{a}(2k + 2B^2 \lambda) + \epsilon' - \eta} \ll 1 \qquad \text{as } \Upsilon \nearrow \infty$$
	
	Therefore, the term $D_0' (2k + 2B^2 \lambda + \kappa)^2 \Xi^{-b}$ can be made arbitrarily small by taking $\Upsilon$ sufficiently large.
	It follows that for any $\Upsilon \gg 1$ sufficiently large depending on $p,q,k, D_0, C_1, \ul{C}, a, b,c,\kappa, \epsilon$ there exists $\tau_0 \gg 1$ sufficiently large depending on $p,q,k, D_0, C_1, \ul{C}, a, b,c,\kappa, \epsilon, \Upsilon$ so that $C_1' \xi^{-2k -2B^2 \lambda - \kappa}$ is a supersolution on $\Omega$.
\end{proof}

\begin{remark} \label{constantsToUse}
	To see that there exists a choice of parameters satisfying the assumptions of the previous two lemmas, consider
		$$a =\frac{q-1}{2} - \frac{1}{2} \sqrt{(q-9)(q-1)} \qquad   c = \kappa = \frac{1}{2} (2k + 2B^2 \lambda) \qquad b = 2k + 2B^2 \lambda$$
	With these parameters
		$$\frac{1}{a} (2k + 2B^2 \lambda ) + \frac{c}{2k + 2B^2 \lambda + \kappa} = \frac{2k + 2B^2 \lambda }{a } + \frac{1}{3} = C_{p,q} > 1$$
	By continuity, one can then make $a$ slightly larger, $\kappa$ and $c$ slightly smaller, and choose $\epsilon$ sufficiently small so that the assumptions of the previous two lemmas hold.
\end{remark}

\subsection{Interior Estimates for the Parabolic Region}

\begin{lem} \label{parabInteriorEstU}
	(Interior estimate for $\tl{U}$ in the Parabolic Region)
	For any $\Gamma , C_0> 0$ and $0 < \eta, \eta' < 1$, there exists $\Upsilon, \tau_0 \gg 1$ sufficiently large (depending on $p,q,k, \Gamma, C_0, \eta, \eta'$) and $C_\Gamma \propto e^{-\alpha \Gamma^2} \ll 1 $ such that
	if $\tau_1 > \tau_0$ and 
	\begin{equation*} \begin{aligned}
		| \tl{Z} | \le& C_0 e^{B^2 \lambda \tau} ( \gamma^{-2k -2B^2 \lambda} + \gamma^{-2B^2\lambda}) \\
		 & \qquad \text{ for all } \gamma \in \left( C_\Gamma \Upsilon e^{- \alpha \tau} , \frac{3}{2} \Gamma \right) , \tau \in [\tau_0, \tau_1] \\
		| \tl{U} - e^{B^2 \lambda \tau} \tl{U}_{\lambda_k} | \le& \eta e^{B^2 \lambda \tau} ( \gamma^{-2k - 2B^2 \lambda} + \gamma^{-2B^2 \lambda}) \\
		&\qquad \text{ for all } \gamma \in \left( C_\Gamma \Upsilon e^{- \alpha \tau} , \frac{3}{2} \Gamma \right) , \tau \in [\tau_0, \tau_1] \\
	\end{aligned} \end{equation*}
	and
	\begin{equation*} \begin{aligned}
		& | \tl{U} - e^{B^2 \lambda \tau} \tl{U}_{\lambda_k} | + | \gamma( \tl{U}_\gamma - e^{B^2 \lambda \tau} \tl{U}_{\lambda_k}' )| + | \gamma^2 (\tl{U}_{\gamma \gamma} - e^{B^2 \lambda \tau} \tl{U}_{\lambda_k}'' ) | \\
		\le & \eta' e^{B^2 \lambda \tau_0} ( \gamma^{-2k -2B^2 \lambda} + \gamma^{-2B^2 \lambda}) \\
		& \qquad  \text{for all } \gamma \in \left( C_\Gamma \Upsilon e^{- \alpha \tau} , \frac{3}{2} \Gamma \right) , \tau =\tau_0 \\
	\end{aligned} \end{equation*}
	then
	\begin{equation*} \begin{aligned}
		| \tl{U}_\gamma - e^{B^2 \lambda \tau} \tl{U}_{\lambda_k}' | \lesssim_{p,q,k,\Gamma}& 
		(\eta + \eta') e^{B^2 \lambda \tau} ( \gamma^{-2k-2B^2 \lambda-1} + \gamma^{-2B^2 \lambda}) \\
		&\qquad \text{ for all } \gamma \in \left( \Upsilon e^{- \alpha \tau} , \Gamma \right) , \tau \in [\tau_0, \tau_1] \\ 
		| \tl{U}_{\gamma \gamma} - e^{B^2 \lambda \tau} \tl{U}_{\lambda_k}'' | \lesssim_{p,q,k, \Gamma}  & 
		(\eta + \eta') e^{B^2 \lambda \tau} ( \gamma^{-2k-2B^2 \lambda-2} + \gamma^{-2B^2 \lambda}) \\
		& \qquad \text{ for all } \gamma \in \left( \Upsilon e^{- \alpha \tau} , \Gamma \right) , \tau \in [\tau_0, \tau_1] \\ 
	\end{aligned} \end{equation*}
\end{lem}
\begin{proof}
	Begin by defining
		$$W(\gamma, \tau) \doteqdot \tl{U}(\gamma, \tau) - e^{B^2 \lambda \tau} \tl{U}_{\lambda_k}(\gamma)$$
	Then $W$ satisfies	
	\begin{equation*} \begin{aligned} 
		\partial_\tau|_\gamma W =& B^2 \cl{D}_U W + \tl{Z}\left( W_{\gamma \gamma} + \frac{1}{\gamma} W_\gamma \right) \\
		&+  \tl{Z} e^{B^2 \lambda_k \tau} \left( \tl{U}_{\lambda_k} '' + \frac{1}{\gamma} \tl{U}_{\lambda_k}' \right) + \frac{n-1}{\gamma^2} \left( 1 - e^{-2\tl{U}} - 2 \tl{U} \right) 
	\end{aligned} \end{equation*}
	
	Fix $\tau_* > \tau_0$ and $\gamma_* \in \left[ \Upsilon e^{- \alpha_k \tau_*} , \Gamma \right] $
 	and define
		$$M(x, s) \doteqdot W( \gamma_* ( 1+ x) , \tau_* + s \gamma_*^2 )$$
	It follows that $M(x,s)$ satisfies
	\begin{equation*} \begin{aligned}
		M_s =& 
		(B^2 + \tl{Z} ) M_{xx} + \left( \frac{ n B^2 + \tl{Z} }{(1+x)} - \frac{ \gamma_*^2 (1+x) }{2} \right) M_x + \frac{2(n-1) B^2}{(1+x)^2} M \\
		&+ \gamma_*^2 \left \{ \tl{Z} e^{B^2 \lambda_k \tau} \left( \tl{U}_{\lambda_k} '' + \frac{1}{\gamma} \tl{U}_{\lambda_k}' \right) + \frac{n-1}{\gamma^2} \left( 1 - e^{-2\tl{U}} - 2 \tl{U} \right)  \right\}
	\end{aligned} \end{equation*}
	
	Observe that for $x \in \left( - \frac{1}{2} , \frac{1}{2} \right)$ and $s \in [-1, 0]$, we have
		$$\gamma_* ( 1 + x ) \le \frac{3}{2} \gamma_* \le \frac{3}{2} \Gamma$$
	and
		$$\gamma_* (1+x) \ge \frac{1}{2} \gamma_* \ge \frac{1}{2} \Upsilon e^{- \alpha \tau_*} = \frac{1}{2} e^{\alpha s \gamma_*^2} \Upsilon e^{ - \alpha (\tau_* + s \gamma_*^2) } \ge \frac{1}{2} e^{- \alpha \Gamma^2} \Upsilon e^{ - \alpha (\tau_* + s \gamma_*^2) } $$
	Hence, with $C_\Gamma = \frac{1}{2} e^{- \alpha \Gamma^2}$,
	we are in a domain where the assumed bounds apply.
	It then follows that
		$$| \tl{Z} | \le C_0 C_\Gamma^{-2k -2B^2 \lambda_k} \Upsilon^{-2k -2B^2 \lambda} + C_0 e^{B^2 \lambda \tau_0 } \Gamma^{-2B^2 \lambda}$$
	In particular, for any $\Gamma > 0$, $\tl{Z}$ can be made arbitrarily small by taking $\Upsilon, \tau_0 \gg 1$ sufficiently large depending on $p,q,k,\Gamma, C_0$.
	
	With this uniform bound on $\tl{Z}$, the PDE for $M$ can be regarded as an inhomogeneous linear parabolic equation.
	For $x \in  \left( -\frac{1}{2}, \frac{1}{2} \right)$, the coefficients are bounded by $C(p,q,k , \Gamma)$ if $\Upsilon, \tau_0 \gg 1$ are sufficiently large (depending on $p,q,k, \Gamma, C_0$).
	Moreover, if $x \in \left(-\frac{1}{2}, \frac{1}{2} \right)$, then the inhomogeneous term satisfies the bound
	\begin{equation*} \begin{aligned}
		&\left | \gamma_*^2 \left \{ \tl{Z} e^{B^2 \lambda_k \tau} \left( \tl{U}_{\lambda_k} '' + \frac{1}{\gamma} \tl{U}_{\lambda_k}' \right) + \frac{n-1}{\gamma^2} \left( 1 - e^{-2\tl{U}} - 2 \tl{U} \right)  \right\} \right| \\
		 \lesssim_{p,q,k} & (1+ C_0) e^{2 B^2 \lambda \tau} \left( \gamma^{-2k -2B^2 \lambda_k} + \gamma^{-2B^2 \lambda} \right)^2 \\
		 \lesssim &(1+ C_0) \left( C_{\Gamma}^{-2k-2B^2 \lambda} \Upsilon^{-2k -2B^2 \lambda} + e^{B^2 \lambda \tau_0} \Gamma^{-2B^2 \lambda} \right) e^{B^2 \lambda \tau} ( \gamma^{-2k -2B^2 \lambda} + \gamma^{-2B^2 \lambda} ) \\
		\lesssim_{p,q,k, \Gamma} &(1+ C_0)  \left( C_{\Gamma}^{-2k-2B^2 \lambda} \Upsilon^{-2k -2B^2 \lambda} + e^{B^2 \lambda \tau_0} \Gamma^{-2B^2 \lambda} \right) e^{B^2 \lambda \tau_*} ( \gamma_*^{-2k -2B^2 \lambda} + \gamma_*^{-2B^2 \lambda} ) \\
	\end{aligned} \end{equation*}
	where $\gamma = \gamma_*(1+x)$ and $\tau = \tau_* + s \gamma_*^2$ in the above inequalities.
	Interior estimates for inhomogeneous linear parabolic equations with bounded coefficients (see ~\cite{LSU88}) now apply.
	There are two cases depending on the size of $\tau_*$ relative to $\gamma_*$.
	
	\noindent Case 1: When $\tau_* - \gamma_*^2 > \tau_0$, there is no dependence on the initial data and it follows that
	\begin{equation*} \begin{aligned}
		&| \gamma_* W_\gamma ( \gamma_*, \tau_*) | + | \gamma_*^2 W_{\gamma \gamma} (\gamma_*, \tau_*) |\\
		=& | M_x(0,0) | + | M_{xx} (0,0) |\\
		\lesssim_{p,q,k, \Gamma} & \sup_{(x,s) \in \left( -\frac{1}{2}, \frac{1}{2} \right) \times (-1, 0)} | M(x,s) | \\
		&+(1+ C_0)  \left( C_{\Gamma}^{-2k-2B^2 \lambda} \Upsilon^{-2k -2B^2 \lambda} + e^{B^2 \lambda \tau_0} \Gamma^{-2B^2 \lambda} \right) e^{B^2 \lambda \tau_*} ( \gamma_*^{-2k -2B^2 \lambda} + \gamma_*^{-2B^2 \lambda} ) \\
		\lesssim_{p,q,k,\Gamma} & \eta e^{B^2 \lambda \tau_*} ( \gamma_*^{-2k-2B^2 \lambda} + \gamma_*^{-2B^2 \lambda} ) \\
		&+(1+ C_0)  \left( C_{\Gamma}^{-2k-2B^2 \lambda} \Upsilon^{-2k -2B^2 \lambda} + e^{B^2 \lambda \tau_0} \Gamma^{-2B^2 \lambda} \right) e^{B^2 \lambda \tau_*} ( \gamma_*^{-2k -2B^2 \lambda} + \gamma_*^{-2B^2 \lambda} ) \\
	\end{aligned} \end{equation*}
	The coefficients on the last term can be bounded by $\eta$ if  $\Upsilon, \tau_0$ are chosen sufficiently large depending on $p,q,k, \Gamma, C_0, \eta$.
	
	\noindent Case 2: When $\tau_* - \gamma_*^2 \le \tau_0$, define $s_* \in [-1, 0]$ such that $\tau_* - s_* \gamma_*^2 = \tau_0$.
	It follows that
	\begin{equation*} \begin{aligned}
		&| \gamma_* W_\gamma ( \gamma_*, \tau_*) | + | \gamma_*^2 W_{\gamma \gamma} (\gamma_*, \tau_*) |\\
		=& | M_x(0,0) | + | M_{xx} (0,0) |\\
		\lesssim_{p,q,k,\Gamma}& 
		\sup_{(x,s) \in \left( -\frac{1}{2}, \frac{1}{2} \right) \times (s_*, 0)} | M(x,s) | \\
		&+(1+ C_0)  \left( C_{\Gamma}^{-2k-2B^2 \lambda} \Upsilon^{-2k -2B^2 \lambda} + e^{B^2 \lambda \tau_0} \Gamma^{-2B^2 \lambda} \right) e^{B^2 \lambda \tau_*} ( \gamma_*^{-2k -2B^2 \lambda} + \gamma_*^{-2B^2 \lambda} ) \\
		& + \sup_{(x,s) \in \left( -\frac{1}{2}, \frac{1}{2} \right) \times (s_*, 0)} \left( | M(x,s) | + |M_x(x,s)| + |M_{xx}(x,s) | \right) \\
		\lesssim_{p,q,k,\Gamma} & (\eta + \eta') e^{B^2 \lambda \tau_*} ( \gamma_*^{-2k-2B^2 \lambda} + \gamma_*^{-2B^2 \lambda} ) \\
		&+(1+ C_0)  \left( C_{\Gamma}^{-2k-2B^2 \lambda} \Upsilon^{-2k -2B^2 \lambda} + e^{B^2 \lambda \tau_0} \Gamma^{-2B^2 \lambda} \right) e^{B^2 \lambda \tau_*} ( \gamma_*^{-2k -2B^2 \lambda} + \gamma_*^{-2B^2 \lambda} ) \\
	\end{aligned} \end{equation*}
	As before, the coefficients on the last term can be bounded by $\eta + \eta'$ if  $\Upsilon, \tau_0$ are chosen sufficiently large depending on $p,q,k, \Gamma, C_0, \eta, \eta'$.
	
	Since $\gamma_*, \tau_*$ were arbitrary, the statement of the lemma follows.
\end{proof}



\begin{lem} \label{parabInteriorEstZ}
	(Interior estimate for $\tl{Z}$ in the Parabolic Region)
	For any $\Gamma > 0$ and $0 < \eta_{0,1,2} < 1$, there exists $\Upsilon, \tau_0 \gg 1$ sufficiently large (depending on $p,q,k,\Gamma, \eta_{0,1,2}$) and $C_\Gamma \propto e^{-\alpha \Gamma^2} \ll 1$
	such that if $\tau_1 > \tau_0$ and 
	\begin{equation*} \begin{aligned}
		| \tl{Z} - e^{B^2 \lambda \tau} \tl{Z}_{\lambda} | \le& \eta_0 e^{B^2 \lambda \tau} ( \gamma^{-2k-2B^2 \lambda} + \gamma^{-2B^2 \lambda} ) \\
		& \qquad  \text{ for all } \gamma \in \left( C_\Gamma \Upsilon e^{- \alpha \tau} , \frac{3}{2} \Gamma \right) , \tau \in [\tau_0, \tau_1] \\
		| \tl{U}_\gamma - e^{B^2 \lambda \tau} \tl{U}_{\lambda}' | \le& \eta_1 e^{B^2 \lambda \tau} \left( \gamma^{-2k - 2B^2 \lambda -1} + \gamma^{-2B^2 \lambda} \right) \\
		& \qquad \text{ for all } \gamma \in \left( C_\Gamma \Upsilon e^{- \alpha \tau} , \frac{3}{2} \Gamma \right) , \tau \in [\tau_0, \tau_1] \\
	\end{aligned} \end{equation*}
	and
	\begin{equation*} \begin{aligned}
		& | \tl{Z} - e^{B^2 \lambda \tau} \tl{Z}_{\lambda} | + | \gamma( \tl{Z}_\gamma - e^{B^2 \lambda \tau} \tl{Z}_{\lambda}') | + | \gamma^2 (\tl{Z}_{\gamma \gamma} - e^{B^2 \lambda \tau} \tl{Z}_{\lambda}'' )  | \\
		\le&  \eta_2 e^{B^2 \lambda \tau_0} ( \gamma^{-2k-2B^2 \lambda} + \gamma^{-2B^2 \lambda} ) \\
		& \qquad \text{ for all } \gamma \in \left( C_\Gamma \Upsilon e^{- \alpha \tau} , \frac{3}{2} \Gamma \right) , \tau = \tau_0
	\end{aligned} \end{equation*}
	then 
	\begin{equation*} \begin{aligned}
		& | \gamma( \tl{Z}_\gamma - e^{B^2 \lambda \tau} \tl{Z}_{\lambda}') | + | \gamma^2 (\tl{Z}_{\gamma \gamma} - e^{B^2 \lambda \tau} \tl{Z}_{\lambda}'' )  |\\
		\lesssim_{p,q,k, \Gamma}  & ( \eta_0 + \eta_1 + \eta_2 ) e^{B^2 \lambda \tau} ( \gamma^{-2k -2B^2 \lambda} + \gamma^{-2B^2 \lambda} ) \\
		& \qquad \text{ for all } \gamma \in \left( \Upsilon e^{- \alpha \tau} , \Gamma \right), \tau \in [\tau_0, \tau_1] 
	\end{aligned} \end{equation*}
\end{lem}
\begin{proof}
	Begin by defining
		$$W(\gamma, \tau) \doteqdot \tl{Z} - e^{B^2 \lambda \tau} \tl{Z}_{\lambda_k}$$
	It follows that
		$$\partial_\tau W = B^2 \cl{D}_Z W + B^2 \cl{N} ( \tl{U} - e^{B^2 \lambda \tau} \tl{U}_{\lambda_k} ) + Err_Z [ \tl{Z}, \tl{U} ]$$
	or equivalently
	\begin{equation*} \begin{aligned}
		\partial_\tau W =& B^2 \cl{D}_Z  W 
		+ \tl{Z} W_{\gamma \gamma} - \frac{1}{2} W_\gamma^2 -e^{B^2 \lambda \tau} \tl{Z}_{\lambda}' W_\gamma- \frac{1}{\gamma} \tl{Z} W_\gamma \\
		& + B^2 \cl{N} ( \tl{U} - e^{B^2 \lambda \tau} \tl{U}_{\lambda_k} ) + \tl{Z} e^{B^2 \lambda \tau} \tl{Z}_\lambda'' - \frac{1}{2} e^{2 B^2 \lambda \tau} (\tl{Z}_\lambda')^2 - \frac{1}{\gamma} e^{B^2 \lambda \tau} \tl{Z}_\lambda' \tl{Z} \\
		& -2p \left\{ B^4 \tl{U}_\gamma^2 + \frac{4B^2 }{\gamma} \tl{Z} \tl{U}_\gamma + 2B^2 \tl{Z} \tl{U}_\gamma^2 + \frac{1}{\gamma^2} \tl{Z}^2 + \frac{2}{\gamma} \tl{Z}^2 \tl{U}_\gamma + \tl{Z}^2 \tl{U}_\gamma^2 \right\}
	\end{aligned} \end{equation*}
	
	Fix $\tau_* > \tau_0$ and $\gamma_* \in \left[ \Upsilon e^{- \alpha_k \tau_*} , \Gamma \right] $
 	and define
		$$M(x, s) \doteqdot W( \gamma_* ( 1+ x) , \tau_* + s \gamma_*^2 )$$
	It follows that $M(x,s)$ satisfies
	\begin{equation*} \begin{aligned}
		M_s =& B^2 M_{xx} + \left[ \frac{B^2 (n-2)}{1+x} - \frac{\gamma_*^2 (1+x)}{2} \right] M_x - \frac{2(n-1)B^2 }{(1+x)^2} M \\
		& + \tl{Z} M_{xx} - \frac{1}{2} M_x^2 - \gamma_* e^{B^2 \lambda \tau} \tl{Z}_\lambda' M_x - \frac{1}{1+x} \tl{Z} M_x \\
		&+ \gamma_*^2 \left\{ B^2 \cl{N} ( \tl{U} - e^{B^2 \lambda \tau} \tl{U}_{\lambda_k} ) + \tl{Z} e^{B^2 \lambda \tau} \tl{Z}_\lambda'' - \frac{1}{2} e^{2 B^2 \lambda \tau} (\tl{Z}_\lambda')^2 - \frac{1}{\gamma} e^{B^2 \lambda \tau} \tl{Z}_\lambda' \tl{Z}  \right\} \\
		& -2p \gamma_*^2 \left\{ B^4 \tl{U}_\gamma^2 + \frac{4B^2 }{\gamma} \tl{Z} \tl{U}_\gamma + 2B^2 \tl{Z} \tl{U}_\gamma^2 + \frac{1}{\gamma^2} \tl{Z}^2 + \frac{2}{\gamma} \tl{Z}^2 \tl{U}_\gamma + \tl{Z}^2 \tl{U}_\gamma^2 \right\}
	\end{aligned} \end{equation*}
	where $\gamma = \gamma_* (1+x)$ and $\tau = \tau_* + s \gamma_*^2$.
	The remainder of the proof follows as in the proof of the previous lemma \ref{parabInteriorEstU}.
	
%
%
\end{proof}
\subsection{Barriers for the Parabolic-Outer Interface}
Next, we proceed to estimate $\tl{U}, \tl{Z}$ pointwise in the interface between the parabolic and outer region.

Throughout this subsection we let $\Omega$ denote the spacetime region
	$$\Omega = \Omega(\Gamma, M, \beta, \tau_0, \tau_1) = \left\{ (\gamma, \tau) \in (0, \infty) \times (\tau_0, \tau_1) : \Gamma < \gamma < M e^{\beta \tau} \right\}$$
and its parabolic boundary
	$$\partial_P \Omega = \left( [ \Gamma, M e^{\beta \tau_0} ] \times \{ \tau_0 \} \right)\cup \{ (\gamma, \tau) \in (0, \infty) \times (\tau_0, \tau_1) :\gamma = \Gamma  \text{ or }  \gamma = M e^{\beta \tau} \}$$

\begin{lem} \label{parabOuterBarriersU}
	(Upper/Lower Barriers for $\tl{U}$ in the Parabolic-Outer Interface)
	For any constant $C$, there exists $D$ (depending only on $C$), $\Gamma \gg 1$ sufficiently large (depending only on $C,D$), and $\rho \ll 1$ sufficiently small (depending only on $C,D$) such that:\\
	if $\tau_0 < \tau_1$, $\beta = 1/2$, $0 < \epsilon \le Z(\gamma, \tau) \le 1$ for all $(\gamma, \tau) \in \Omega(\Gamma, \rho, 1/2, \tau_0, \tau_1)$, and
		$$\tl{U}(\gamma, \tau) \ge (\le) C \gamma^{-2 B^2 \lambda} e^{B^2 \lambda  \tau} + D \gamma^{-2B^2 \lambda-2} e^{B^2 \lambda \tau} \qquad \text{for all } (\gamma, \tau) \in \Omega_P(\Gamma, \rho, 1/2, \tau_0, \tau_1)$$
	then
		$$\tl{U}(\gamma, \tau) \ge (\le) C \gamma^{-2 B^2 \lambda} e^{B^2 \lambda  \tau} + D \gamma^{-2B^2 \lambda-2} e^{B^2 \lambda \tau} \qquad \text{for all } (\gamma, \tau) \in \Omega(\Gamma, \rho, 1/2, \tau_0, \tau_1)$$
\end{lem}
\begin{proof}
	We prove only the lower barrier as the upper barrier proof follows analogously.
	For ease of notation, let
		$$h = B^2 \lambda \qquad k =-2h =  -2B^2 \lambda \ge 0$$
	for the course of this proof.
	
	It suffices to check that
		$$\tl{U}^- = C \gamma^k e^{h  \tau} + D \gamma^{k-2} e^{h \tau}$$
	is a subsolution for the evolution equation satisfied by $\tl{U}$ in $\Omega$.
	Using the assumed bounds on $Z$, we estimate
	\begin{equation*} \begin{aligned}
		& Z (\tl{U}^-_{\gamma \gamma} + \frac{1}{\gamma} \tl{U}^-_\gamma ) + \left( \frac{q-1}{\gamma} - \frac{\gamma}{2} \right) \tl{U}^-_\gamma + \frac{q-1}{\gamma^2} \left( 1 - e^{-2 \tl{U}^-} \right)\\
		=& Z \left( Ck(k-1) + Ck + D(k-2) \gamma^{-2} + D(k-2)(k-3) \gamma^{-2} \right) \gamma^{k-2} e^{h \tau} \\
		& + (q-1) \big( Ck + D(k-2) \gamma^{-2} \big) \gamma^{k-2} e^{h \tau} \\
		&- \frac{1}{2} \big( Ck + D(k-2) \gamma^{-2} \big) \gamma^k e^{h \tau} \\
		&+ 2 (q-1) \big( C + D \gamma^{-2} \big) \gamma^{k-2} e^{h \tau} + 2(q-1) \gamma^{-2} \left( 1 - e^{-2 \tl{U}^- } - 2 \tl{U}^- \right) \\
		=& - \frac{k}{2} e^{h \tau} \big( C \gamma^k + D \gamma^{-2} \big) + e^{h \tau} \gamma^{k-2} \big( Z C k^2 + (q-1) Ck + D + 2 (q-1)C \big) \\
		&+ e^{h \tau} \gamma^{k-4} \big( ZD (k-2) + ZD(k-2)(k-3) + (q-1)D (k-2) + 2 (q-1) D \big) \\
		&+ 2(q-1) \gamma^{-2} \left( 1 - e^{-2 \tl{U}^- } - 2 \tl{U}^- \right) \\
		\ge& - \frac{k}{2} e^{h \tau} \big( C \gamma^k + D \gamma^{-2} \big) + e^{h \tau} \gamma^{k-2} \big( - |C| k^2 - (q-1) |C|k + D - 2 (q-1)|C| \big) \\
		&- e^{h \tau} \gamma^{k-2} \Gamma^{-2} \left| ZD (k-2) + ZD(k-2)(k-3) + (q-1)D (k-2) + 2 (q-1) D \right| \\
		&- 2(q-1) \gamma^{-2} | 1 - e^{-2 \tl{U}^- } - 2 \tl{U}^- |\\
	\end{aligned} \end{equation*}
	Choose $D$ such that
		$$ - |C| k^2 - (q-1) |C|k + D - 2 (q-1)|C| > 0$$
	Observe that
		$$| \tl{U}^- | = \gamma^{k} e^{h \tau} | C + D \gamma^{-2} | \le \gamma^k e^{h \tau} \left( | C | + \frac{|D|}{\Gamma^2} \right) \le \rho^{k} \left( | C | + \frac{|D|}{\Gamma^2} \right)$$
	since $k = -2h \ge 0.$
	Using the fact that $1 - e^{-2x}$ is locally Lipschitz, it follows that for $\Gamma \gg 1$ sufficiently large and $\rho \ll 1$ sufficiently small (both depending on $C$ and $D$) we have
	\begin{equation*} \begin{aligned}
		 &\Gamma^{-2} \left| ZD (k-2) + ZD(k-2)(k-3) + (q-1)D (k-2) + 2 (q-1) D \right| \\
		<& \frac{1}{2} \big( - |C| k^2 - (q-1) |C|k + D - 2 (q-1)|C| \big) \\
	\end{aligned} \end{equation*}
	\begin{equation*} \begin{aligned}
		\text{and } 
		2(q-1) \gamma^{-2} \left| 1 - e^{-2 \tl{U} } - 2 \tl{U} \right|
		&\le 2(q-1)  \gamma^{-2} C' \rho^{k} \left( | C | + \frac{ |D|}{\Gamma^2} \right) | \tl{U}^- | \\
		&\le 2(q-1)  C' \rho^{k} \left( | C | + \frac{ |D|}{\Gamma^2} \right)^2  \gamma^{k-2} e^{h \tau} \\
		& < \frac{1}{2}  \big( - |C| k^2 - (q-1) |C|k + D - 2 (q-1)|C| \big) \\
	\end{aligned} \end{equation*}
	
	Hence,
		$$Z (\tl{U}^-_{\gamma \gamma} + \frac{1}{\gamma} \tl{U}^-_\gamma ) + \left( \frac{q-1}{\gamma} - \frac{\gamma}{2} \right) \tl{U}^-_\gamma + \frac{q-1}{\gamma^2} \left( 1 - e^{-2 \tl{U}^-} \right) $$
	$$> -\frac{k}{2} e^{h \tau} \big( C \gamma^k + D \gamma^{-2} \big) = \partial_\tau \tl{U}^-$$
	for all $(\gamma, \tau) \in \Omega$.
	This completes the proof of the lower barrier.
	
	We note that the upper barrier proof is simplified by the observation that
		$$1 - e^{-2 \tl{U}} - 2 \tl{U} \le 0$$
	Consequently, $\tl{U}^+$ is actually an upper barrier on $\Omega$ for any $\rho > 0$.
\end{proof}

\begin{remark} \label{remarkGamma_UDep}
	If, in the previous lemma \ref{parabOuterBarriersU}, we have $C = C'_{p,q,k} + \eta$ where $| \eta | < 1$,
	then the proof of the lemma shows that we may in fact take 
	$D = D(p,q,k)$,
	$\Gamma \gg 1$ sufficiently large depending only on $p,q,k$ rather than $C, D$,
	and $\rho \ll 1$ sufficiently small depending only on $p,q,k$ rather than $C,D$.
	However, in order to ensure that $\tl{U}^- \le \tl{U}^+$ for $\gamma \ge \Gamma$, it is necessary to have $\Gamma \gg 1$ be sufficiently large depending on $\eta$ as well.
\end{remark}

Before proceeding, we need a scaling argument to show that the derivatives of $\tl{U}$ satisfy corresponding pointwise bounds.

\begin{lem} \label{parabOuterInteriorEstU}
	Assume that there exist $\epsilon,C,  \Gamma, \rho, \tau_0$ with $\Gamma \ge \frac{1}{2 - \sqrt{2}}$ and $\rho e^{\tau_0/2} \ge 1$ such that
		$$0 < \epsilon \le Z \le 1 \text{ and } | \tl{U} | \le C \gamma^{-2B^2 \lambda} e^{B^2 \lambda \tau} \qquad \text{ for all } (\gamma, \tau) \in \Omega( \Gamma, \rho, \beta =1/2, \tau_0, \tau_1)$$
		$$\text{and } | \tl{U}_\gamma | +  | \tl{U}_{\gamma \gamma} | \le C \gamma^{-2B^2 \lambda} e^{B^2 \lambda \tau_0} \qquad \text{ on }  \Gamma \le \gamma \le  \rho e^{\tau_0 / 2} , \tau = \tau_0$$
	Then
		$$| \tl{U}_\gamma | + | \tl{U}_{\gamma \gamma} | \lesssim_{p,q,\epsilon, C \rho} \gamma^{-2B^2 \lambda} e^{B^2 \lambda \tau} \qquad \text{for all } (\gamma, \tau) \in \Omega( 2 \Gamma, \rho/2, 1/2, \tau_0, \tau_1)$$
\end{lem}
\begin{proof}
	Let $\tau_*, \gamma_*$ be such that $\gamma_* \in ( 2  \Gamma, \frac{1}{2} \rho e^{\tau_*/2})$ and $\tau_* > \tau_0$.
	Rescale by defining
		$$W(x, r) \doteqdot \tl{U} \left( \frac{ \gamma_* + x}{\sqrt{ 1- r}}, \tau_* - \log( 1 - r) \right)$$
	Since $\tl{U}(\gamma, \tau)$ satisfies
		$$\tl{U}_\tau = Z (\tl{U}_{\gamma \gamma} + \frac{1}{\gamma} \tl{U}_\gamma ) + \left( \frac{q-1}{\gamma} - \frac{\gamma}{2} \right) \tl{U}_\gamma + \frac{q-1}{\gamma^2} \left( 1 - e^{-2 \tl{U}} \right)$$
	it follows that $W(x,r)$ satisfies the semilinear equation
		$$W_r = Z W_{xx} + W_x \frac{q-1+Z}{\gamma_* + x}  + \frac{2(q-1)}{( \gamma_* + x)^2} W + \frac{q-1}{(\gamma_* + x)^2} \left( 1 - e^{-2W} - 2W \right)$$
	For $x \in (-1, 1)$ and $r \in (-1,0)$,
		$$\frac{\gamma_* + x}{\sqrt{ 1- r} } \ge \frac{ 2 \Gamma - 1}{\sqrt{ 2}} \ge \Gamma \qquad \text{ if } \Gamma \ge \frac{1}{2 - \sqrt{2} }$$
	and
		$$\frac{ \gamma_* + x}{ \sqrt{ 1- r} } \le \frac{ \frac{1}{2} \rho e^{\tau_*/2} + x}{\sqrt{ 1- r} } = \frac{ \frac{1}{2} \rho \sqrt{1-r} e^{\tau/2} + x}{\sqrt{ 1- r}} = \frac{ 1}{2} \rho e^{\tau/2} + \frac{x}{\sqrt{1-r}} \le \frac{ 1}{2} \rho e^{\tau/2} +1$$
		$$\le \rho e^{\tau/2} \qquad \text{ if } \rho e^{\tau_0/2} \ge 1$$
	In other words, we are in a domain where the assumed bounds on $\tl{U}, Z$ apply.
	Now,
		$$Z \ge \epsilon \text{ implies } W \text{ satisfies a strictly parabolic equation on } (x,r) \in (-1,1) \times (-1,0)$$
	and
		$$| W | = | \tl{U} | \le C \gamma^{- 2B^2 \lambda} e^{B^2 \lambda \tau} \le C \rho \qquad \text{ is uniformly bounded.}$$
	While the coefficients depend on $\gamma_*$, they are bounded independently of $\gamma_*$.
	For example,
		$$\left| \frac{q-1+Z}{\gamma_* + x} \right| \le \frac{q-1}{| \gamma_* + x|} + \frac{| Z|}{| \gamma_* + x|} \le \frac{q-1}{\Gamma_0 - 1} + \frac{1}{\Gamma_0 - 1} \le \frac{q}{\Gamma_0 - 1}$$
	Using the uniform bound on $|W|$ to control the semilinear term, we can now apply interior estimates for linear parabolic equations (see e.g. ~\cite{LSU88}).
	There are 2 cases:\\
		1) If $\tau_*$ is large enough so that $\frac{1}{2} e^{\tau_*} > e^{\tau_0}$, then we don't obtain any initial data terms in the interior estimates and simply deduce that
	\begin{equation*} \begin{aligned}
		| W_{xx}(0,0) | + |W_x(0,0)| &\lesssim_{\epsilon, p,q, C\rho} \sup_{(x,r) \in (-1,1) \times (-1,0)} |W| \\
		\implies | \tl{U}_{\gamma \gamma}(\gamma_*, \tau_*)| + | \tl{U}_\gamma(\gamma_*, \tau_*)| 
		&\lesssim \sup_{(x,r) \in (-1,1) \times (-1,0) } \left| \tl{U} \left( \frac{ \gamma_* + x}{\sqrt{ 1- r}}, \tau_* - \log( 1 - r) \right) \right| \\
		&\le \sup_{x,r} C \left( \frac{ \gamma_* + x}{\sqrt{ 1- r}} \right)^{-2B^2_\infty \lambda} e^{B^2_\infty \lambda \big( \tau_* - \log(1-r) \big)} \\
		& \lesssim (\gamma_*)^{-2B^2 \lambda} e^{B^2 \lambda \tau_*} \\
	\end{aligned} \end{equation*}
		2) If $\tau_*$ is such that $\frac{1}{2} e^{\tau_*} \le e^{\tau_0}$, then we also have to include initial data terms in the interior estimate
		\begin{gather*}
			| W_{xx}(0,0) | + |W_x(0,0)| \\
			\lesssim_{\epsilon, p, q, C \rho} \left( \sup_{(x,r) \in (-1,1) \times (-r_0,0)} |W| \right) + \left( \sup_{x \in (-1,1)} | W_x(x, -r_0) | +  | W_{xx}(x, -r_0) | \right) 
		\end{gather*}
		Here $r_0 \in (-1,0)$ is defined such that
			$$\frac{1}{1-r_0} e^{\tau_*} = e^{\tau_0}$$
		The estimate follows analogously using for example
			$$| W_{xx}(x, r_0) | = \left| \frac{1}{1-r_0} \tl{U}_{\gamma \gamma} \left( \frac{\gamma_* + x}{\sqrt{1-r_0}}, \tau_0 \right) \right| \le \frac{1}{2} \left| \tl{U}_{\gamma \gamma} \left( \frac{\gamma_* + x}{\sqrt{1-r_0}}, \tau_0 \right) \right| $$
		and the assumption on the initial data.
\end{proof}

\begin{lem} \label{parabOuterUpperBarrierZ}
	(Upper Barrier for $\tl{Z}$ in the Parabolic-Outer Interface)
	Assume there exists $C_U, \Gamma_U, \rho_U, \tau_U, \tau_1$ such that 
		$$| \tl{U}_\gamma | \le C_U \gamma^{-2B^2\lambda} e^{B^2 \lambda \tau} \qquad \text{for all } (\gamma, \tau) \in \Omega( \Gamma_U, \rho_U, \beta= 1/2, \tau_U, \tau_1)$$
	For $\beta = 1/2$ and any constant $C$, there exists $D$ (depending on $p,q,C_U$), 
	$\Gamma \gg 1$ sufficiently large (depending on $p,q,k, C, D, C_U, \Gamma_U$),
	and $\rho \ll 1 $ sufficiently small (depending on $p,q,k, C, D, C_U, \Gamma_U, \rho_U$) such that if $\tau_U \le \tau_0 < \tau_1$, $Z \ge \epsilon > 0$ for all $(\gamma, \tau) \in \Omega( \Gamma, \rho, \beta = 1/2, \tau_0, \tau_1)$, and 
		$$\tl{Z} \le C \gamma^{-2B^2 \lambda} e^{B^2 \lambda \tau} + D \gamma^{-2B^2 \lambda -1} e^{B^2 \lambda \tau}		\qquad 		\text{for all } (\gamma, \tau) \in \Omega_P( \Gamma, \rho, \beta=1/2, \tau_0, \tau_1)$$
	then
		$$\tl{Z} \le C \gamma^{-2B^2 \lambda} e^{B^2 \lambda \tau} + D \gamma^{-2B^2 \lambda -1} e^{B^2 \lambda \tau}		\qquad 		\text{for all } (\gamma, \tau) \in \Omega( \Gamma, \rho, \beta=1/2, \tau_0, \tau_1)$$
\end{lem}

\begin{proof}
	Recall that $\tl{Z}$ satisfies the equation
		$$\tl{Z}_\tau = 2 \mathcal{Q}^q[Z_0, \tl{Z} ] + \mathcal{F}^l[\tl{Z}] - \frac{\gamma}{2} \tl{Z}_\gamma + \mathcal{F}^q [ \tl{Z} ] - 2p \left\{ Z^2 U_\gamma^2 - B^4 \gamma^{-2} \right\}$$
		$$= Z ( \tl{Z}_{\gamma \gamma} - \frac{1}{\gamma} \tl{Z}_\gamma ) + \tl{Z}_\gamma \left( \frac{q-1}{\gamma} - \frac{\gamma}{2} \right) - \frac{1}{2} \tl{Z}_\gamma^2 + \frac{2(q-1)}{\gamma^2} \tl{Z} (1- Z) -2p\left\{ Z^2 U_\gamma^2 - B^4 \gamma^{-2} \right\}$$
	The term in braces can be estimated as follows
	\begin{equation*} \begin{aligned}
		Z^2 U_\gamma^2 - Z_0^2 U_0'^2 =& Z^2 U_\gamma^2 - \frac{B^4}{\gamma^2} \\
		=& 2B^4 \frac{\tl{U}_\gamma}{\gamma} + B^4 \tl{U}_\gamma^2 + 2 B^2 \tl{Z} \left( \frac{1}{\gamma^2} + 2 \frac{ \tl{U}_\gamma}{\gamma} + \tl{U}_\gamma^2 \right) + \tl{Z}^2 \left( \frac{1}{\gamma^2} + 2 \frac{ \tl{U}_\gamma}{\gamma} + \tl{U}_\gamma^2 \right)\\
		=& 2B^4 \frac{\tl{U}_\gamma}{\gamma} + 2B^2 \tl{Z} \left( \frac{1}{\gamma^2} + 2 \frac{ \tl{U}_\gamma}{\gamma} \right) + \tl{Z}^2 \left( \frac{1}{\gamma^2} + 2 \frac{ \tl{U}_\gamma}{\gamma} \right) + (B^2 + \tl{Z})^2 \tl{U}_\gamma^2 \\
		\ge& 2B^4 \frac{\tl{U}_\gamma}{\gamma} + 2B^2 \tl{Z} \left( \frac{1}{\gamma^2} + 2 \frac{ \tl{U}_\gamma}{\gamma} \right) + \tl{Z}^2 \left( \frac{1}{\gamma^2} + 2 \frac{ \tl{U}_\gamma}{\gamma} \right)
	\end{aligned} \end{equation*}
	From this estimate and $-\frac{1}{2} \tl{Z}_\gamma^2 \le 0$, it follows that
	\begin{equation*} \begin{aligned} 
		\tl{Z}_\tau \le &Z ( \tl{Z}_{\gamma \gamma} - \frac{1}{\gamma} \tl{Z}_\gamma ) + \tl{Z}_\gamma \left( \frac{q-1}{\gamma} - \frac{\gamma}{2} \right) + \frac{2(q-1)}{\gamma^2} \tl{Z} (1- Z) \\
		&-2p \left\{  2B^4 \frac{\tl{U}_\gamma}{\gamma} + 2B^2 \tl{Z} \left( \frac{1}{\gamma^2} + 2 \frac{ \tl{U}_\gamma}{\gamma} \right) + \tl{Z}^2 \left( \frac{1}{\gamma^2} + 2 \frac{ \tl{U}_\gamma}{\gamma} \right) \right\}
	\end{aligned} \end{equation*}
	It now suffices to check that
		$$\tl{Z}^+ \doteqdot C \gamma^{-2B^2 \lambda} e^{B^2 \lambda \tau} + D \gamma^{-2B^2 \lambda -1} e^{B^2 \lambda \tau}	$$
	is a supersolution to this differential equation on $\Omega$.
	For simplicity of notation, we let
		$$ h = B^2 \lambda		\qquad 		k =-2h = -2B^2 \lambda $$
	for the remainder of the proof.
	We estimate
	\begin{equation*} \begin{aligned}
		&Z ( \tl{Z}^+_{\gamma \gamma} - \frac{1}{\gamma} \tl{Z}^+_\gamma ) + \tl{Z}^+_\gamma \left( \frac{q-1}{\gamma} - \frac{\gamma}{2} \right) + \frac{2(q-1)}{\gamma^2} \tl{Z}^+ (1- Z) \\
		&-2p \left\{  2B^4 \frac{\tl{U}_\gamma}{\gamma} + 2B^2 \tl{Z}^+ \left( \frac{1}{\gamma^2} + 2 \frac{ \tl{U}_\gamma}{\gamma} \right) + (\tl{Z}^+)^2 \left( \frac{1}{\gamma^2} + 2 \frac{ \tl{U}_\gamma}{\gamma} \right) \right\} \\
		=& Z \big( Ck (k-1) - Ck + D(k-1)(k-2) \gamma^{-1} - D(k-1) \gamma^{-1} \big) \gamma^{k-2} e^{h \tau} \\
		&+(q-1) \big( Ck + D(k-1) \gamma^{-1} \big) \gamma^{k-2} e^{h \tau} \\
		& - \frac{1}{2} \big( Ck + D(k-1) \gamma^{-1} \big) \gamma^{k} e^{h \tau} \\
		& + 2(q-1)(1 - Z) (C + D \gamma^{-1}) \gamma^{k-2} e^{h \tau} \\
		&-2p \left\{ 2B^4 \frac{\tl{U}_\gamma}{\gamma} + 2B^2 \tl{Z}^+ \left( \frac{1}{\gamma^2} + 2 \frac{ \tl{U}_\gamma}{\gamma} \right) + (\tl{Z}^+)^2 \left( \frac{1}{\gamma^2} + 2 \frac{ \tl{U}_\gamma}{\gamma} \right)  \right\}\\
	\end{aligned} \end{equation*}
	and $|\tl{U}_\gamma| \le C_U \gamma^k e^{h \tau}$ implies that this is in turn less than or equal to
	\begin{equation*} \begin{aligned}
		\le& Z \big( Ck (k-1) - Ck + D(k-1)(k-2) \gamma^{-1} - D(k-1) \gamma^{-1} \big) \gamma^{k-2} e^{h \tau} \\
		&+(q-1) \big( Ck + D(k-1) \gamma^{-1} \big) \gamma^{k-2} e^{h \tau} \\
		& - \frac{1}{2} \big( Ck + D(k-1) \gamma^{-1} \big) \gamma^{k} e^{h \tau} \\
		& + 2(q-1)(1 - Z) (C + D \gamma^{-1}) \gamma^{k-2} e^{h \tau} \\
		& 4p B^4 C_U \gamma^{k-1} e^{h \tau} - 2p (C + D \gamma^{-1} ) (2B^2 + \tl{Z} ) \left( \frac{1}{\gamma} + 2 \tl{U}_\gamma \right) \gamma^{k-1} e^{h \tau} \\
		=& -\frac{1}{2}k (C  + D \gamma^{-1} ) \gamma^k e^{h \tau} + \gamma^{k-1} e^{h \tau} \left( \frac{D}{2} + 4p B^4 C_U - 2p C(2B^2 + \tl{Z}^+) \left( \frac{1}{\gamma} + 2 \tl{U}_\gamma \right) \right) \\
		&+ e^{h \tau} o(\gamma^{k-1}) + Z e^{h \tau} o(\gamma^{k-1})\\
	\end{aligned} \end{equation*}
	If $\Gamma_U \le \Gamma$ and $\rho \le \rho_U$ then on $\Gamma \le \gamma \le \rho e^{\tau /2}$
		$$| \tl{U}_\gamma| \le C_U \gamma^k e^{h \tau} \implies | \tl{U}_\gamma| \le C_U \rho^{k} \qquad (\text{since } k = -2h > 0)$$
	and
		$$| \tl{Z}^+ | = \gamma^{k} e^{h \tau} | C + D \gamma^{-1} | \le \rho^{k} \left( | C| + \frac{ |D|}{\Gamma} \right) \qquad (\text{since } k = -2h > 0)$$
	Hence, we can estimate the coefficient of $\gamma^{k-1}$ by
		$$\frac{D}{2} + 4p B^4 C_U - 2p C(2B^2 + \tl{Z}^+) \left( \frac{1}{\gamma} + 2 \tl{U}_\gamma \right) $$
		$$\le \frac{D}{2} + 4p B^4 C_U + 2p |C| \left(2B^2 + \rho^k |C| + \rho^k \frac{|D|}{\Gamma} \right) \left( \frac{1}{\Gamma} + 2 C_U \rho^k \right)$$
	Choosing $D<0$ such that 
		$$\frac{D}{2} < -4pB^4 C_U$$
	and then taking $\Gamma \gg 1$ sufficiently large and $\rho \ll 1$ sufficiently small (both depending only on $p,q,k, C, D,C_U, \Gamma_U, \rho_U$), it follows that for all $(\gamma, \tau) \in \Omega ( \Gamma, \rho, \beta=1/2, \tau_0, \tau_1)$
	\begin{equation*} \begin{aligned}
		&Z ( \tl{Z}^+_{\gamma \gamma} - \frac{1}{\gamma} \tl{Z}^+_\gamma ) + \tl{Z}^+_\gamma \left( \frac{q-1}{\gamma} - \frac{\gamma}{2} \right) + \frac{2(q-1)}{\gamma^2} \tl{Z}^+ (1- Z) \\
		&-2p \left\{  2B^4 \frac{\tl{U}_\gamma}{\gamma} + 2B^2 \tl{Z}^+ \left( \frac{1}{\gamma^2} + 2 \frac{ \tl{U}_\gamma}{\gamma} \right) + (\tl{Z}^+)^2 \left( \frac{1}{\gamma^2} + 2 \frac{ \tl{U}_\gamma}{\gamma} \right) \right\} \\	
		\le&  -\frac{1}{2}k (C  + D \gamma^{-1} ) \gamma^k e^{h \tau} \\
		= &\partial_\tau \tl{Z}^+
	\end{aligned} \end{equation*}
\end{proof}

\noindent Due to the sign of some of the coupling terms, there is a somewhat weaker lower barrier for $\tl{Z}$.

\begin{lem} \label{parabOuterLowerBarrierZ}
	(Lower Barrier for $\tl{Z}$ in the Parabolic-Outer Interface)
	Assume there exists $C_U, \Gamma_U, \rho_U, \tau_U, \tau_1$ such that 
		$$| \tl{U}_\gamma | \le C_U \gamma^{-2B^2\lambda} e^{B^2 \lambda \tau} \qquad \text{for all } (\gamma, \tau) \in \Omega( \Gamma_U, \rho_U, \beta = 1/2, \tau_U, \tau_1)$$
		
	For any constant $C$, 
	there exists $D$ (depending on $p,q,C_U, C$),
	$\Gamma \gg 1$ sufficiently large (depending on $p,q,k, C, D, C_U, \Gamma_U$),
	and $\rho \ll 1$ sufficiently small (depending on $p,q,k, C, D, C_U, \Gamma_U, \rho_U$)
	such that if 
		$$\beta = \frac{1}{2} \left( \frac{1}{1 + \frac{1}{-2B^2 \lambda}} \right) \in \left( 0,  \frac{1}{2} \right)$$
	$\tau_U \le \tau_0 < \tau_1$, $0 < \epsilon \le Z \le 1$ for all $(\gamma, \tau) \in \Omega(\Gamma, \rho, \beta, \tau_0, \tau_1)$, and
		$$\tl{Z} \ge C \gamma^{-2B^2 \lambda} e^{B^2 \lambda \tau} + D \gamma^{-2B^2 \lambda -1} e^{B^2 \lambda \tau}		\qquad \text{for all } (\gamma, \tau) \in \Omega_P(\Gamma, \rho, \beta, \tau_0, \tau_1)$$
	then	
		$$\tl{Z} \ge C \gamma^{-2B^2 \lambda} e^{B^2 \lambda \tau} + D \gamma^{-2B^2 \lambda -1} e^{B^2 \lambda \tau}		\qquad \text{for all } (\gamma, \tau) \in \Omega(\Gamma, \rho, \beta, \tau_0, \tau_1)$$
\end{lem}

\begin{proof}
	Recall that $\tl{Z}$ satisfies the equation
		$$\tl{Z}_\tau = 2 \mathcal{Q}^q[Z_0, \tl{Z} ] + \mathcal{F}^l[\tl{Z}] - \frac{\gamma}{2} \tl{Z}_\gamma + \mathcal{F}^q [ \tl{Z} ] - 2p \left\{ Z^2 U_\gamma^2 - B^4 \gamma^{-2} \right\}$$
		$$= Z ( \tl{Z}_{\gamma \gamma} - \frac{1}{\gamma} \tl{Z}_\gamma ) + \tl{Z}_\gamma \left( \frac{q-1}{\gamma} - \frac{\gamma}{2} \right) - \frac{1}{2} \tl{Z}_\gamma^2 + \frac{2(q-1)}{\gamma^2} \tl{Z} (1- Z) -2p\left\{ Z^2 U_\gamma^2 - B^4 \gamma^{-2} \right\}$$
	The term in braces can be estimated as follows
	\begin{equation*} \begin{aligned}
		Z^2 U_\gamma^2 - \frac{B^4}{\gamma^2} 
		=& 2B^4 \frac{\tl{U}_\gamma}{\gamma} + B^4 \tl{U}_\gamma^2 + 2 B^2 \tl{Z} \left( \frac{1}{\gamma^2} + 2 \frac{ \tl{U}_\gamma}{\gamma} + \tl{U}_\gamma^2 \right) \\
		&+ \tl{Z}^2 \left( \frac{1}{\gamma^2} + 2 \frac{ \tl{U}_\gamma}{\gamma} + \tl{U}_\gamma^2 \right)\\
		\le& 2 B^4 C_U \gamma^{k-1} e^{h \tau} +  B^4 C_U^2 \gamma^{2k} e^{2h \tau} + 2 B^2 \tl{Z} \left( \frac{1}{\gamma^2} + 2 \frac{ \tl{U}_\gamma}{\gamma} + \tl{U}_\gamma^2 \right) \\
		&+ \tl{Z}^2 \left( \frac{1}{\gamma^2} + 2 \frac{ \tl{U}_\gamma}{\gamma} + \tl{U}_\gamma^2 \right)\\
	\end{aligned} \end{equation*}
	From this estimate, it follows that
	\begin{equation*} \begin{aligned}
		\tl{Z}_\tau \ge & Z ( \tl{Z}_{\gamma \gamma} - \frac{1}{\gamma} \tl{Z}_\gamma ) + \tl{Z}_\gamma \left( \frac{q-1}{\gamma} - \frac{\gamma}{2} \right) - \frac{1}{2} \tl{Z}_\gamma^2 + \frac{2(q-1)}{\gamma^2} \tl{Z} (1- Z) \\
		&-2p\left\{  2 B^4 C_U \gamma^{k-1} e^{h \tau} +  B^4 C_U^2 \gamma^{2k} e^{2h \tau} + 2 B^2 \tl{Z} \left( \frac{1}{\gamma^2} + 2 \frac{ \tl{U}_\gamma}{\gamma} + \tl{U}_\gamma^2 \right) \right. \\
		&\left. + \tl{Z}^2 \left( \frac{1}{\gamma^2} + 2 \frac{ \tl{U}_\gamma}{\gamma} + \tl{U}_\gamma^2 \right) \right\}
	\end{aligned} \end{equation*}
	
	It suffices to check that
		$$\tl{Z}^- = C \gamma^{-2B^2 \lambda} e^{B^2 \lambda \tau} + D \gamma^{-2B^2 \lambda -1} e^{B^2 \lambda \tau}$$
	is a subsolution for this differential equation in $\Omega$.
	Denote $k = -2B^2\lambda$ and $h = B^2 \lambda$.
	We estimate
	\begin{equation*} \begin{aligned}
		&Z ( \tl{Z}^-_{\gamma \gamma} - \frac{1}{\gamma} \tl{Z}^-_\gamma ) + \tl{Z}^-_\gamma \left( \frac{q-1}{\gamma} - \frac{\gamma}{2} \right) - \frac{1}{2} (\tl{Z}^-_\gamma)^2 + \frac{2(q-1)}{\gamma^2} \tl{Z}^- (1- Z) \\
		&-2p\left\{  2 B^4 C_U \gamma^{k-1} e^{h \tau} +  B^4 C_U^2 \gamma^{2k} e^{2h \tau} + 2 B^2 \tl{Z}^- \left( \frac{1}{\gamma^2} + 2 \frac{ \tl{U}_\gamma}{\gamma} + \tl{U}_\gamma^2 \right) \right. \\
		& \left.+ (\tl{Z}^-)^2 \left( \frac{1}{\gamma^2} + 2 \frac{ \tl{U}_\gamma}{\gamma} + \tl{U}_\gamma^2 \right) \right\} \\
		=& Z \big( Ck (k-1) + D(k-1)(k-2) \gamma^{-1} -Ck - D(k-1) \gamma^{-1} \big) \gamma^{k-2} e^{h \tau} \\
		&+(q-1) \big( Ck + D(k-1) \gamma^{-1} \big) \gamma^{k-2} e^{h \tau} \\
		& - \frac{1}{2} \big( Ck + D(k-1)\gamma^{-1} \big) \gamma^k e^{h \tau} \\
		& - \frac{1}{2} ( \tl{Z}_\gamma^- ) \big( Ck + D(k-1) \gamma^{-1}\big) \gamma^{k-1} e^{h \tau} \\
		& +2(q-1) (1-Z) \big( C + D \gamma^{-1} \big) \gamma^{k-2} e^{h \tau} \\
		&-4pB^4 C_U \gamma^{k-1} e^{h \tau} - 2p B^4 C_U | \gamma \tl{U}_\gamma | \gamma^{k-2} e^{h \tau} \\
		&-4pB^4 (C + D \gamma^{-1} ) \big( 1 + 2 \gamma \tl{U}_\gamma + \tl{U}_\gamma^2 \big) \gamma^{k-2} e^{h \tau} \\
		&-2p \tl{Z}^- ( C + D \gamma^{-1} ) \big( 1 + 2 \gamma \tl{U}_\gamma + \tl{U}_\gamma^2 \big) \gamma^{k-2} e^{h \tau}\\
		=& - \frac{k}{2} \left( C + D \gamma^{-1} \right) \gamma^k e^{h \tau}
		+ \left( \frac{D}{2} - 4 p B^4 C_U \right) \gamma^{k-1} e^{h \tau}
		+ (*)
	\end{aligned} \end{equation*}
	Let $D$ be chosen such that
		$$\frac{D}{2} > - 4p B^4 C_U$$
	If $\Gamma_U \le \Gamma$ and $\rho \le \rho_U$ then the choice of $\beta$ implies that
		$$| \gamma \tl{U}_\gamma| \le C_U \rho^{k+1} \qquad \text{ and } \qquad | \tl{U}_\gamma |^2 \le C_U^2 \rho^{2k}$$
	It then follows 
	as in the proof of the upper barrier that the term $(*)$ satisfies
		$$|(*)|  < \left( \frac{D}{2} - 4 p B^4 C_U \right) \gamma^{k-1} e^{h \tau} 	\qquad \text{for all } (\gamma, \tau) \in \Omega( \Gamma, \rho, \beta, \tau_0,\tau_1)$$
	 for suitable choices of $\Gamma,$ and $\rho$.
\end{proof}

\begin{remark} \label{remarkGamma_ZDep}
	As analogously noted in remark \ref{remarkGamma_UDep}, the $\Gamma$ dependence can be simplified in the last two lemmas.
	Indeed, if $C = C'_{p,q,k} + \eta$ where $| \eta | < 1$, then we may take
	$\Gamma \gg 1$ sufficiently large depending only on $p,q,k, C_U, \Gamma_U$
	and $\rho \ll 1$ sufficiently small depending on $p,q,k, C_U, \rho_U$.
	Moreover, if $C_U = C''_{p,q,k} + \eta'$, then the dependence on $C_U$ can also be dropped.
	However, in order to ensure that $\tl{Z}_- \le \tl{Z}^+$ for $\gamma \ge \Gamma$ it is necessary for $\Gamma \gg 1$ to be sufficiently large depending also on $\eta$.
\end{remark}

\begin{lem} \label{parabOuterInteriorEstZ}
	Assume that there exist $\epsilon, C_U, C_Z, \Gamma, \rho, \beta, \tau_0$ with $\Gamma \ge \frac{\sqrt{2}}{2 - \sqrt{3}}$ and $\left( 1 - \frac{1}{\sqrt{2}} \right) \rho e^{\beta \tau_0 } \ge 1$ and $0 < \beta < \frac{1}{2}$ 
	and $\tau_0 \gg 1$ sufficiently large (depending on $p,q,k, \rho, \beta, C_U, C_Z$) such that
		$$0 < \epsilon \le Z \le 1 \text{ and } | \tl{U}_\gamma | \le C_U \gamma^{-2B^2 \lambda} e^{B^2 \lambda \tau} \qquad \text{ for all } ( \gamma, \tau) \in \Omega( \Gamma, \rho, \beta  , \tau_0, \tau_1)$$
		$$\text{ and } | \tl{Z}_\gamma | + | \tl{Z}_{\gamma \gamma} | \le C_Z \gamma^{-2B^2 \lambda} e^{B^2 \lambda \tau_0} \qquad \text{for all } (\gamma, \tau_0) \in \Omega( \Gamma, \rho, \beta, \tau_0, \tau_0)$$
	Then
		$$ | \tl{Z}_\gamma | + | \tl{Z}_{\gamma \gamma} | \lesssim_{p,q,k,\epsilon} (C_U + C_Z) \gamma^{-2B^2 \lambda} e^{B^2 \lambda \tau} 		\qquad \text{ for all } (\gamma, \tau) \in \Omega( 2 \Gamma, \rho / 2, \beta, \tau_0, \tau_1)$$
\end{lem}
\begin{proof}
	Let $\tau_*, \gamma_*$ be such that $\gamma_* \in \left( \sqrt{2} \Gamma, \frac{1}{\sqrt{2}} \rho e^{\beta \tau_*} \right)$ and $\tau_* > \tau_0$.
	Rescale by defining
		$$W(x,r) \doteqdot \tl{Z} \left ( \frac{ \gamma_* + x}{\sqrt{1-r} }, \tau_* - \log( 1-r) \right)$$
	Since $\tl{Z}(\gamma, \tau)$ satisfies
	\begin{equation*} \begin{aligned}
		\tl{Z}_\tau + \frac{\gamma}{2} \tl{Z}_\gamma =&
		 Z \tl{Z}_{\gamma \gamma} - \frac{Z}{\gamma} \tl{Z}_\gamma + \frac{q-1}{\gamma} \tl{Z}_\gamma 
		 - \frac{1}{2} \tl{Z}_\gamma + \frac{2(q-1)}{\gamma^2} (1-Z) \tl{Z}
		 -2p \left\{ Z^2 U_\gamma^2 - B^4 \frac{1}{\gamma^2} \right\}
	\end{aligned} \end{equation*}
	it follows that $W(x,r)$ satisfies
	\begin{equation*} \begin{aligned}
		W_r =& 
		Z W_{xx} - \frac{Z}{\gamma_* + x} W_x + \frac{q-1}{\gamma_* + x} W_x
		- \frac{1}{2} W_x^2 + \frac{2(q-1)}{(\gamma_* +x)^2} (1-Z)W
		-2p \frac{( 2 B^2 U_\gamma^2 + \tl{Z} U_\gamma^2 )}{1-r} W \\
		& + \frac{1}{1-r}\left( B^4 U_\gamma^2 - B^4 \frac{1}{\gamma^2} \right)
	\end{aligned} \end{equation*}
	where $\gamma = \frac{ \gamma_* + x}{\sqrt{1-r} }$.
	Note that for $r \in \left(-\frac{1}{2}, 0 \right)$
	\begin{equation*} \begin{aligned}
		\frac{\gamma_* + x}{\sqrt{1-r}} \ge \frac{ \sqrt{2} \Gamma - 1}{ \sqrt{ \frac{3}{2}}} \ge \Gamma \qquad 
		&& \text{ if } \Gamma \ge \frac{ \sqrt{2} }{2 - \sqrt{3} } > 1 \\
		\frac{\gamma_* + x}{ \sqrt{1-r} } \le \frac{1}{\sqrt{2}} \rho e^{\beta \tau_*} + 1 \le  \rho e^{\beta \tau_*}
		&& \text{ if } 1 \le \left( 1 - \frac{1}{\sqrt{2}} \right) \rho e^{\beta \tau_0} \\
	\end{aligned} \end{equation*}
	
	Hence, we are in a domain where the assumed bounds apply.
	Therefore,
	\begin{equation*} \begin{aligned}
		\left| \frac{1}{1-r} \left( B^4 U_\gamma^2 - B^4 \frac{1}{\gamma^2} \right) \right| 
		\lesssim_{p,q,k}  | \tl{U}_\gamma | 
		\le C_U e^{B^2 \lambda \tau} \gamma^{-2B^2 \lambda} 
		\lesssim_{p,q,k}  C_U e^{B^2 \lambda \tau_*} \gamma_*^{-2B^2 \lambda} 
	\end{aligned} \end{equation*}
	and
	\begin{equation*} \begin{aligned}
		\left| -2p \frac{( 2 B^2 U_\gamma^2 + \tl{Z} U_\gamma^2 )}{1-r}\right|
		& \lesssim_{p,q} | \tl{U}_{\gamma} |^2 \\
		& \le C_U^2 \gamma^{-2B^2 \lambda} e^{B^2 \lambda \tau} \gamma^{-2B^2 \lambda} e^{B^2 \lambda \tau} \\
		&\lesssim_{p,q,k} C_U^2 \left( \rho e^{\beta \tau} \right)^{-2 B^2 \lambda } e^{B^2 \lambda \tau} \gamma_*^{-2B^2 \lambda} e^{B^2 \lambda \tau_*} \\
		&\lesssim_{p,q,k} C_U^2 \left( \rho e^{\beta \tau_0} \right)^{-2 B^2 \lambda } e^{B^2 \lambda \tau_0} \gamma_*^{-2B^2 \lambda} e^{B^2 \lambda \tau_*} \\
	\end{aligned} \end{equation*}
	and the coefficient of $\gamma_*^{-2B^2 \lambda} e^{B^2 \lambda \tau_*}$ can be made arbitrarily small by taking $\tau_0 \gg 1$ sufficiently large depending on $p,q,k, \rho, \beta, C_U$.
	By the assumed bounds and similar logic as in the proof of lemma \ref{parabOuterInteriorEstU}, this equation is an inhomogeneous semilinear equation for $W$ with bounded coefficients where the only nonlinearity for $W$ can be regarded as the $-\frac{1}{2} W_x^2$.
	Interior estimates for such equations (see ~\cite{LSU88}) then imply the gradient bound
	\begin{equation*} \begin{aligned}
		 | \tl{Z}_\gamma(\gamma_*, \tau_*) | 
		= | W_x(0,0) | 
		\lesssim_{p,q,k, \epsilon} (C_U + C_Z)  \gamma_*^{-2B^2 \lambda} e^{B^2 \lambda \tau_*}
	\end{aligned} \end{equation*}
	as in the proof of lemma \ref{parabOuterInteriorEstU}.
	Equipped with such a gradient bound, we may then regard $W$ as satisfying an inhomogeneous \textit{linear} equation with bounded coefficients on the region $\Omega( \sqrt{2} \Gamma, \rho/ \sqrt{2}, \beta, \tau_0, \tau_1)$.
	Interior estimates for such equations then imply the stated second derivative bound.
\end{proof}

\begin{lem} \label{parabOuterInteriorEstU+}
	For $\Gamma \ge \frac{1}{2 - \sqrt{2}}$ 
	and $C_Z, M > 0$ and $ 0 < \eta, \eta_0 < 1$ 
	and $0 < \beta < \frac{1}{2}$, 
	there exists $\tau_0 \gg 1$ sufficiently large (depending on $p,q,k, \beta, \eta, \eta_0, C_Z, M$) such that the following holds:
	If $\tau_1 > \tau_0$ and 
	$\tl{Z}, \tl{U}$ satisfy the estimates
	\begin{equation*} \begin{aligned}
		| \tl{Z} | \le& C_Z e^{B^2 \lambda \tau} \gamma^{-2B^2 \lambda} 
		&& \text{for all } (\gamma, \tau) \in \Omega( \Gamma, M, \beta , \tau_0, \tau_1) \\
		| \tl{U} - e^{B^2 \lambda \tau} \tl{U}_{\lambda} | \le& \eta e^{B^2 \lambda \tau}  \gamma^{-2B^2 \lambda} 
		&& \text{for all }  (\gamma, \tau) \in \Omega( \Gamma, M, \beta , \tau_0, \tau_1) \\
	\end{aligned} \end{equation*}
	and the initial data satisfies
	\begin{equation*} \begin{aligned}
		| \tl{U} - e^{B^2 \lambda \tau} \tl{U}_{\lambda} | + | \tl{U}_\gamma - e^{B^2 \lambda \tau} \tl{U}_{\lambda}' | + | \tl{U}_{\gamma \gamma} - e^{B^2 \lambda \tau} \tl{U}_{\lambda}'' |
		\le \eta_0 e^{B^2 \lambda \tau}  \gamma^{-2B^2 \lambda} \\
		 \text{for all }  (\gamma, \tau) \in \Omega( \Gamma, M, \beta , \tau_0, \tau_0)
	\end{aligned} \end{equation*}
	then
	\begin{equation*} \begin{aligned}
		| \tl{U}_\gamma - e^{B^2 \lambda \tau} \tl{U}_{\lambda}' | + | \tl{U}_{\gamma \gamma} - e^{B^2 \lambda \tau} \tl{U}_{\lambda}'' |
		\lesssim_{p,q,k} (\eta + \eta_0) e^{B^2 \lambda \tau}  \gamma^{-2B^2 \lambda} \\
		 \text{ for all }  (\gamma, \tau) \in \Omega \left( 2\Gamma, \frac{1}{2} M, \beta , \tau_0, \tau_1 \right)
	\end{aligned} \end{equation*}
\end{lem}
\begin{proof}
	Begin by defining
		$$ W(\gamma, \tau) \doteqdot \tl{U}(\gamma, \tau) - e^{B^2 \lambda \tau} \tl{U}_{\lambda_k}(\gamma)$$
	Then $W$ satsfies
	\begin{equation*} \begin{aligned} 
		\partial_\tau|_\gamma W =& B^2 \cl{D}_U W + \tl{Z}\left( W_{\gamma \gamma} + \frac{1}{\gamma} W_\gamma \right) \\
		&+  \tl{Z} e^{B^2 \lambda_k \tau} \left( \tl{U}_{\lambda_k} '' + \frac{1}{\gamma} \tl{U}_{\lambda_k}' \right) + \frac{n-1}{\gamma^2} \left( 1 - e^{-2\tl{U}} - 2 \tl{U} \right) 
	\end{aligned} \end{equation*}
	
	Fix $\tau_* > \tau_0$ and $\gamma_* \in \left[ 2 \Gamma, \frac{1}{2} M e^{\beta \tau_*} \right] $
 	and define
		$$V(x, r) \doteqdot W \left( \frac{\gamma_* + x}{\sqrt{1-r}} , \tau_* - \log(1-r) \right)$$
	It follows that $V(x,r)$ satisfies
	\begin{equation*} \begin{aligned}
		V_r =& B^2 V_{xx} + \frac{B^2 n}{\gamma_* + x} V_x + \frac{ 2(n-1)}{(\gamma_* +x)^2} V 
		+ \tl{Z} \left( V_{xx} + \frac{ 1}{\gamma_* + x} V_x \right) \\
		&+ \frac{1}{1-r} \left \{ \tl{Z} e^{B^2 \lambda_k \tau} \left( \tl{U}_{\lambda_k} '' + \frac{1}{\gamma} \tl{U}_{\lambda_k}' \right) + \frac{n-1}{\gamma^2} \left( 1 - e^{-2\tl{U}} - 2 \tl{U} \right)  \right \}
	\end{aligned} \end{equation*}
	
	By similar logic as in the proof of lemma \ref{parabOuterInteriorEstU},
		$$\frac{\gamma_*  +x}{\sqrt{1-r} } \in \left[ \Gamma, M e^{\beta \tau} \right]$$
	for all $(x, r) \in (-1, 1) \times (-1, 0)$ if $\Gamma \ge \frac{1}{2 - \sqrt{2} }$ and $M e^{\beta \tau_0} \ge 1$.
	Observe
		$$| \tl{Z} | \le C_Z e^{B^2 \lambda \tau} \gamma^{-2B^2 \lambda} \le C_Z \left( e^{-\tau/2} \rho e^{\beta \tau} \right)^{-2B^2 \lambda} \ll 1 \qquad \text{if } \tau_0 \gg 1$$
	and so $V$ satisfies an inhomogeneous \textit{linear} parabolic equation.
	Moreover, the inhomogeneous term satisfies the following bound for $r \in (-1, 0)$
	\begin{equation*} \begin{aligned}
		& \quad \left| \frac{1}{1-r} \left \{ \tl{Z} e^{B^2 \lambda_k \tau} \left( \tl{U}_{\lambda_k} '' + \frac{1}{\gamma} \tl{U}_{\lambda_k}' \right) + \frac{n-1}{\gamma^2} \left( 1 - e^{-2\tl{U}} - 2 \tl{U} \right)  \right \} \right| \\
		& \le C_Z C_{p,q,k} e^{2B^2 \lambda \tau} \gamma^{-4B^2 \lambda -2} + C_{p,q,k} e^{	2 B^2 \lambda \tau} \gamma^{-4B^2 \lambda -2} \\
		& \lesssim_{p,q,k} (C_Z + 1)  \left( e^{-\tau_0/2} M e^{\beta \tau_0} \right)^{-2B^2 \lambda} e^{B^2 \lambda \tau_*} \gamma_*^{-2B^2 \lambda} 
	\end{aligned} \end{equation*} 
	We may now apply interior estimates for such equations.
	There are two cases depending on the size of $\tau_*$
	
	\noindent Case 1: If $\tau_* - \log(2) > \tau_0$ then 
	\begin{equation*} \begin{aligned} 
		& \quad | W_\gamma (\gamma_*, \tau_*) | + | W_{\gamma \gamma} (\gamma_*, \tau_*) |\\
		&= | V_x(0,0) | + | V_{xx}(0,0)| \\
		& \lesssim_{p,q,k}  \sup_{\gamma, \tau} \eta e^{B^2 \lambda \tau} \gamma^{-2B^2 \lambda} 
		+(C_Z +1) (e^{-\tau_0/2} M e^{\beta \tau_0} )^{-2B^2 \lambda} e^{B^2 \lambda \tau_*} \gamma_*^{-2B^2 \lambda} \\
		& \lesssim_{p,q,k} \eta e^{B^2 \lambda \tau_*} \gamma_*^{-2B^2 \lambda} \\
	\end{aligned} \end{equation*}
	If $\tau_0 \gg 1$ is sufficiently large depending on $p,q,k, \beta, \eta, C_Z, M$.
	
	\noindent Case 2: If $\tau_* - \log(2) \le \tau_0$ then we must include contributions of the initial data.
	An analogous argument to above yields
	\begin{equation*} \begin{aligned} 
		&\quad  | W_\gamma (\gamma_*, \tau_*) | + | W_{\gamma \gamma} (\gamma_*, \tau_*) |\\
		&= | V_x(0,0) | + | V_{xx}(0,0)| \\
		&\lesssim_{p,q,k}  ( \eta + \eta_0) e^{B^2 \lambda \tau_*} \gamma_*^{-2B^2 \lambda} 
		+ (C_Z +1)(e^{-\tau_0/2} M e^{\beta \tau_0} )^{-2B^2 \lambda} e^{B^2 \lambda \tau_*} \gamma_*^{-2B^2 \lambda} \\
		& \lesssim  ( \eta + \eta_0) e^{B^2 \lambda \tau_*} \gamma_*^{-2B^2 \lambda} 
	\end{aligned} \end{equation*}
	If $\tau_0 \gg 1$ is sufficiently large depending on $p,q,k, \beta, \eta, \eta_0, C_Z, M$.
\end{proof}

\begin{lem} \label{parabOuterInteriorEstZ+}
	For $\Gamma \ge \frac{ \sqrt{2} }{2 - \sqrt{3}}$
	and $ M > 0$ 
	and $ 0 < \eta_0, \eta_1, \eta_2 < 1$ 
	and $0 < \beta < \frac{1}{2}$, 
	there exists $\tau_0 \gg 1$ sufficiently large (depending on $p,q,k, M, \beta, \eta_0, \eta_1, \eta_2$) such that the following holds:
	If $\tau_1 > \tau_0$ and 
	\begin{equation*} \begin{aligned}
		| \tl{U}_\gamma - e^{B^2 \lambda \tau} \tl{U}'_{\lambda_k} (\gamma) | 
		\le& \eta_1 e^{B^2 \lambda \tau} \gamma^{-2B^2 \lambda} 
		&& \text{ for } (\gamma, \tau) \in \Omega( \Gamma, M , \beta , \tau_0, \tau_1) \\
		| \tl{Z} - e^{B^2 \lambda \tau} \tl{Z}_{\lambda_k}(\gamma) |
		\le& \eta_2 e^{B^2 \lambda \tau} \gamma^{-2B^2 \lambda}
		&& \text{ for } (\gamma, \tau) \in \Omega( \Gamma, M , \beta , \tau_0, \tau_1) \\
	\end{aligned} \end{equation*}
	and
	\begin{equation*} \begin{aligned}
		| \tl{Z} - e^{B^2 \lambda \tau} \tl{Z}_{\lambda_k}(\gamma) |
		+ | \tl{Z}_\gamma - e^{B^2 \lambda \tau} \tl{Z}_{\lambda_k}'(\gamma) |
		+ | \tl{Z}_{\gamma \gamma} - e^{B^2 \lambda \tau} \tl{Z}_{\lambda_k}''(\gamma) |
		\le& \eta_0 e^{B^2 \lambda \tau} \gamma^{-2B^2 \lambda} \\
		 \text{ for }  (\gamma, \tau_0) \in \Omega( \Gamma, M , \beta , \tau_0, \tau_0) \\
	\end{aligned} \end{equation*}
	then
	\begin{equation*} \begin{aligned}
		| \tl{Z}_\gamma - e^{B^2 \lambda \tau} \tl{Z}_{\lambda_k}'(\gamma) |
		+ | \tl{Z}_{\gamma \gamma} - e^{B^2 \lambda \tau} \tl{Z}_{\lambda_k}''(\gamma) |
		\lesssim_{p,q,k} & ( \eta_0 + \eta_1 + \eta_2) e^{B^2 \lambda \tau} \gamma^{-2B^2 \lambda} \\
		 \text{ for }  (\gamma, \tau) \in \Omega \left( 2\Gamma, \frac{1}{2}M , \beta , \tau_0, \tau_1 \right) \\
	\end{aligned} \end{equation*}
\end{lem}
\begin{proof}
	The proof is nearly identical to the one above and so we omit it here.
\end{proof}

\subsection{Outer Region Pointwise Estimates}
\subsubsection{Curvature Bounds}
	
\begin{lem} \label{outerCurvatureBounds}
	Assume that there exists $\Gamma > 0$ and $\tau_0$ such that
		$$| \tl{U} | + | \tl{U}_\gamma | + | \tl{U}_{\gamma \gamma} | \le C \gamma^{-2B^2 \lambda_k} e^{B^2 \lambda \tau}$$
		$$\text{ and } | \tl{Z} | + | \tl{Z}_{\gamma} | \le C \gamma^{-2B^2 \lambda_k } e^{B^2 \lambda_k \tau}$$
		$$\text{for all } ( \gamma, \tau) \in \{ \tau = \tau_0, \gamma \ge \Gamma \} \cup \{ \tau > \tau_0, \gamma = \Gamma \}$$
	If $\Gamma \gg 1$ is sufficiently large (depending only on $p,q$) and $\tau_0 \gg 1$ is sufficiently large (depending on $p,q,k, C, \Gamma$), then
		$$| Rm(\psi, t) | \lesssim_{p,q} \frac{1}{T-t} = e^{\tau} \qquad \text{for all } \tau \ge \tau_0, \psi \ge \Gamma e^{- \tau/2}$$
	In particular,
		$$| \tl{Z}_{\gamma} | \lesssim_{p,q} \gamma \qquad \text{for all } \tau \ge \tau_0, \gamma \ge \Gamma$$
\end{lem}
\begin{proof}
	Let $\Omega \subset S^{p+1} \times S^q \times \mathbb{R}$ denote
		$$\Omega_T \doteqdot \{ (\psi, t) | \tau(t) \ge \tau_0, \psi \ge \Gamma e^{-\tau(t)/2} \}$$
		$$\text{and let } S_T \text{ be the parabolic boundary of } \Omega_T $$
	The assumed bounds on $\tl{U}, \tl{Z}$ imply that for all $  (\psi, t) \in S_T$
		$$|u| \le \log(\psi) + C \psi^{-2B^2 \lambda} \qquad | u_\psi | \le \frac{1}{\psi} + C e^{\tau/2} \psi^{-2B^2 \lambda} \qquad | u_{\psi \psi} | \le  \frac{1}{\psi^2} + C e^\tau \psi^{-2B^2 \lambda}$$
		$$| z | \le 1 +  C \psi^{-2 B^2 \lambda} \qquad \text{and} \qquad | z_{\psi} | \le C e^{\tau/2} \psi^{-2B^2 \lambda}  $$
	It follows that 
		$$| Rm(\psi, t) | \le c e^{\tau(t)} \qquad \forall (\psi, t) \in S_T$$
	where $c$ can be made arbitrarily small by taking $\Gamma \gg 1$ sufficiently large and $\tau_0 \gg 1$ sufficiently large depending on $p,q,k, C, \Gamma$.
	
	Since under Ricci flow
		$$| (\partial_t - \Delta ) Rm | \le C_n | Rm |^2$$
	for some constant $C_n$ depending only on dimension, it follows from the parabolic maximum principle on $\Omega_T$ that if $c$ is taken sufficiently small depending on $C_n$ then 
		$$| Rm(\psi, t) | \lesssim_n \frac{1}{T-t} \qquad \forall (\psi, t) \in \Omega_T$$
		
	The final statement of the lemma then follows immediately from observing that $\frac{z_\psi}{2 \psi}$ is a sectional curvature.
\end{proof}

\subsubsection{Barriers}
Throughout this subsection, let $\psi_* = \psi_*(t)$ denote the largest value of $\psi$ on $(S^{q+1} \times S^p, g(t) )$.
By reflection symmetry and monotonicity of $\psi$, there is a unique $s_*(t) = s(x=0,t)$ such that $\psi(s_*(t), t) = \psi_*(t)$.
Moreover, the metric is reflection symmetric across the hypersurface defined by $\{ s= s_*(t) \}$.

Throughout this subsection we let $\Omega$ denote the spacetime region
	$$\Omega = \Omega(M, \beta, t_0, t_1) = \left\{ (\psi, t) \in (0, \infty) \times (t_0, t_1) : M e^{\beta \tau(t)} < \psi < \psi_*(t) \right\}$$
and its parabolic boundary
\begin{equation*} \begin{aligned}
	\partial_P \Omega = & \left( [ M e^{\beta \tau(t_0)} , \psi_*(t_0) ] \times \{ t_0 \} \right) \\
	& \cup \{ (\psi, t) \in (0, \infty) \times (t_0, t_1) : \psi = M e^{\beta \tau(t)} \text{ or } \psi = \psi_*(t) \}
\end{aligned} \end{equation*}

Recall that our initial simplifying assumptions guarantee that $0 < z \le 1$ in $\Omega$ and smoothness of the metric implies that $\lim_{\psi \nearrow \psi_*(t)} z(\psi, t) = 0$.

\begin{lem} \label{outerConeBarrierU}
	Let $c \ge \frac{A}{B}$.
	If $u \ge \log( c \psi)$ for all $(\psi, t) \in \partial_{P} \Omega$ then $u \ge \log( c \psi)$ for all $(\psi, t) \in \Omega$.
\end{lem}
\begin{proof}
	The proof is identical to that of lemma \ref{innerConeBarrierU}.
\end{proof}
	
\begin{lem} \label{outerConeBarrierGradU}
	If $u \ge \log \left( \frac{A}{B} \psi \right)$ for all $(\psi, t) \in \Omega$ and $u_\psi \le \frac{C}{\psi}$ on $\partial_{P} \Omega$ for some positive constant $C > 0$, then $u_\psi \le \frac{C}{\psi}$ for all $(\psi, t) \in \Omega$.
\end{lem}	
\begin{proof}
	The proof is identical to that of lemma \ref{innerConeBarrierGradU}.
\end{proof}

\begin{lem} \label{outerSineConeBarrierZ}
	Assume $u_\psi \le \frac{C}{\psi}$ in $\Omega$ for some constant $C > 0$ and 
	let
		$$z^- (\psi, t) = z^-(\psi, t; C, K_0) \doteqdot \max\left(  0, \frac{q-1}{q-1 + p C^2} - \frac{1}{2} K(t; K_0) \psi^2  \right)$$
		$$\text{where } K(t; K_0) = \frac{K_0}{1 - K_0 (C^2 p + q)t}	\quad \text{ and } \quad K_0 > 0$$
	If $z(\psi, t) \ge z^-(\psi, t )$ for all $(\psi, t) \in \partial_{P} \Omega$
	then $z(\psi, t) \ge z^-(\psi, t)$ for all $(\psi, t) \in \Omega$.
	
	In particular,
		$$\psi_*(t) \ge \sqrt{ 2 \frac{q-1}{q-1+pC^2} \frac{1- K_0 (p C^2 + q) t}{K_0} }$$
\end{lem}
\begin{proof}
	Since the max of two subsolutions is a subsolution, it suffices to check that
		$$\hat{z}^-(\psi, t) \doteqdot \frac{q-1}{q-1 + p C^2} - \frac{1}{2} K(t; K_0) \psi^2$$
	is a subsolution of the PDE satisfied by $z(\psi, t)$ on $\Omega$, i.e.
		$$\hat{z}^-_t \le \mathcal{F}^l_\psi [ \hat{z}^- ] + \mathcal{F}^q_\psi [ \hat{z}^- ] -2 p (\hat{z}^-)^2 u_\psi^2$$
	By the estimate $u_\psi \le \frac{C}{\psi}$, it is straightforward to check that 
	\begin{equation*} \begin{aligned}
		& \mathcal{F}^l_\psi [ \hat{z}^- ] + \mathcal{F}^q_\psi [ \hat{z}^- ] -2 p (\hat{z}^-)^2 u_\psi^2 \\
		\ge & \mathcal{F}^l_\psi [ \hat{z}^- ] + \mathcal{F}^q_\psi [ \hat{z}^- ] -2 p (\hat{z}^-)^2 \frac{C^2}{\psi^2} \\
		= & \partial_t \hat{z}^-(\psi, t)
	\end{aligned} \end{equation*}
	Hence, $z^-$ is a subsolution and the comparison principle completes the proof of the first statement of the proposition.
	
	The second statement follows from the observations that $z(\psi_*(t), t) = 0$ and that 
		$$z(\psi, t) \ge z^-(\psi, t) > 0 		\qquad \text{for all } \psi < \sqrt{ 2 \frac{q-1}{q-1+pC^2} \frac{1- K_0 (p C^2 + q) t}{K_0} }$$
\end{proof}
\begin{remark}
	When $C=1$, $\hat{z}^-(\psi, t)$ is one of the profile functions for the sine cone solution to Ricci flow.
\end{remark}

%% file: NoInnBlowup.tex

So far, the bounds we have established do not preclude the possibility that some of the Ricci flow solutions under consideration might form a singularity at $t = T_* < T$ (equivalently, $\tau = \tau(T_*) < \infty$) in the inner region.
In this section, we use a blow-up argument to rule out such a possibility.

Throughout this section, we will assume that $U, Z$ satisfy the bounds
\begin{gather*}
	\log \left( \frac{A}{B} \gamma \right) \le U 	\qquad 
	0 \le U_\gamma \le \frac{1}{\gamma} \qquad 
	| U_{\gamma \gamma} | \lesssim \gamma^{-2} \\ 
	B^2 \le Z \le 1 \qquad 
	| Z_\gamma | \lesssim \gamma^{-1} \qquad
	0 \le \gamma U_{\gamma \gamma} + U_\gamma 
\end{gather*}
	on $0 \le \gamma \le \Upsilon_U e^{- \alpha \tau}, \tau \in[ \tau_0, \tau(T^*) ]$ and
		$$U(\Upsilon_U e^{- \alpha \tau} , \tau) > \log \left( \frac{A}{B} \Upsilon_U e^{-\alpha \tau} \right) + C \Upsilon_U^{-2k-2B^2 \lambda_k}$$
	for some positive constant $C = C(p,q,k)$.

In terms of the warping functions $\phi(s,t), \psi(s,t)$ of the metric, these bounds can be written equivalently as
\begin{gather*}
	\frac{A}{B} \psi \le \phi \qquad
	0 \le \frac{ \phi_s}{\phi} \le \frac{\psi_s}{\psi} \qquad 
	\left| \frac{1}{\psi_s} \partial_s \left( \frac{ \phi_s}{\psi_s \phi} \right) \right| \lesssim \frac{1}{\psi^2} \\
	B \le \psi_s \le 1 \qquad 
	| \psi_{ss} | \lesssim \frac{1}{\psi} \qquad 
	0 \le \partial_s \left( \frac{ \phi_s \psi}{\psi_s \phi} \right)
\end{gather*}
on the spacetime region
	$$\left\{ (s,t) \in [0, \infty) \times \big( t(\tau_0),  T_* \big) : 0 \le \psi(s,t) \le \Upsilon_U (T - t)^{-\alpha - \frac{1}{2} } \right\}$$
and
	$$\log \phi > \log \left( \frac{A}{B} \psi \right) + C \Upsilon_U^{-2k -2B^2 \lambda_k}$$
for all points such that $\psi = \Upsilon_U e^{- \alpha \tau} e^{- \tau/2}$.
Note that these estimates are invariant under rescaling the metric, that is under the rescaling $(\lambda \phi, \lambda \psi)\left( s/ \lambda, t \right)$.
We record some additional consequences of these bounds.
	\begin{prop}
		Under the assumptions above, the following hold on the spacetime region
			$$\left\{ (s,t) \in [0, \infty) \times \big( t(\tau_0),  T_* \big) : 0 \le \psi(s,t) \le \Upsilon_U (T - t)^{-\alpha - \frac{1}{2} } \right\},$$
		\begin{enumerate}
			\item $B s \le \psi \le s$
				\begin{proof}
					This estimate follows from $B \le \psi_s \le 1$ and $\psi(0,t) = 0$.
				\end{proof}
			\item for some $\delta = \delta(p,q,k,\Upsilon_U)>0$,
				$$\left(\frac{A}{B} + \delta  \right) \psi \le \phi$$
					\begin{proof}
						The bound $\frac{\phi_s}{\phi} \le \frac{\psi_s}{\psi}$ may be rewritten as 
							$$\partial_\psi ( \log \phi ) \le \frac{1}{\psi}$$
						The estimate then follows from this bound combined with the fact that
							$$\log \phi > \log \left( \frac{A}{B} \psi \right) + C \Upsilon_U^{-2k -2B^2 \lambda_k}$$
						when $\psi = \Upsilon_U (T - t)^{-\alpha - \frac{1}{2} } $.
					\end{proof}
			\item $| Rm | \lesssim_{p,q} \psi^{-2}$
			\item $| \partial_t (\psi^2) | \lesssim_{p,q} 1$
				\begin{proof}
					Use $| \partial_t \psi | \lesssim_{p,q} \psi | Rm| \lesssim \psi^{-1}$.
				\end{proof}
		\end{enumerate}
	\end{prop}
	
Recall that our metrics are on $S^{q+1} \times S^{p}$ and for $p \in S^{q+1} \times S^{p}$ we have denoted $s(p, t) = dist_{g(t)} (p, S^p_*)$ where $S^p_* \subset S^{q+1} \times S^p$ is a connected component of the submanifold where $\psi$ vanishes.
Suppose a singularity forms at $T_* < T$ or $\tau_* \doteqdot \tau( T_* ) < \infty$.
Define
	$$\mathcal{S}_* \doteqdot \Upsilon_0 e^{-\alpha_k \tau_* - \tau_*/2}$$
	$$\mathcal{B}(t) \doteqdot \{ p \in S^{q+1} \times S^p | 0 \le s(p,t) \le \mathcal{S}_* \} = B_{g(t)}( S^p_*, \mathcal{S}_*) $$
	In other words, $\mathcal{B}(t)$ is a metric neighborhood of $S^p_*$ consisting of points whose distance from this submanifold is less than $\mathcal{S}_*$.
	
	\textit{If a singularity forms in $\mathcal{B}(t)$ as $t \nearrow T_*$}, then we can pick a sequence $(p_k, t_k) \in S^{q+1} \times S^{p} \times (T_* -1, T_*)$ such that
		$$t_0 < t_1 < ... < t_k < ... \nearrow T_*$$
		$$\sup_{p \in \ol{\mathcal{B}}(t), t \le t_k } | Rm(p,t)|_{g(t)} = | Rm(p_k, t_k) |_{g(t_k)} \nearrow +\infty \text{ as } k \nearrow \infty$$
		$$proj_{S^p} p_k \in S^p \text{ is independent of $k$}$$
	Let $s_k \doteqdot s(p_k, t_k)$.
	Observe that
		$$| Rm | \lesssim \psi^{-2} \text{ and } \psi_s \ge B \implies s_k \to 0 \text{ as } k \nearrow \infty$$
	Define
		$$\epsilon_k \doteqdot | Rm(p_k, t_k) |_{g(t_k)}^{-1/2} \qquad \text{ and } \qquad \psi_k \doteqdot \psi(s_k, t_k)$$
	Observe that
		$$| Rm| \le \frac{C_0}{\psi^2} \implies \psi_k \le \sqrt{C_0} \epsilon_k$$
		$$\text{ and so } B s_k \le \psi_k  \implies s_k \le \frac{ \sqrt{C_0} }{B} \epsilon_k$$
	
	Define rescaled metrics
		$$g_k (t) \doteqdot \epsilon_k^{-2} g( t_k + \epsilon_k^2 t )$$
	and $\sigma (s,t_k) \doteqdot \epsilon_k^{-1} s$ so that
		$$g_k(0) = d \sigma^2 + \Phi_k(\sigma)^2 g_{S^p} + \Psi_k (\sigma)^2 g_{S^q}$$
		$$\Phi_k(\sigma) \doteqdot \epsilon_k^{-1} \phi(s,t_k)$$
		$$\Psi_k(\sigma) \doteqdot \epsilon_k^{-1} \psi(s, t_k)$$
	By construction, the metrics $g_k(t)$ are defined for at least for $t \in (- \epsilon_k^{-2} t_k, 0]$ and have $|Rm(g_k(t))| \le 1$ for all $t \in (- \epsilon_k^{-2} t_k, 0]$.
	Note that $B_{g_k(t)}(   S^p_* , \mathcal{S}_*/ \epsilon_k) = B_{g(t)} (   S^p_* , \mathcal{S}_* )$ and $\sigma(s_k, t_k) \le \frac{ \sqrt{ C_0}}{B}$.
	There are two cases which we now distinguish between:
	\begin{enumerate}
		\item $\limsup_{k \nearrow \infty} \Phi_k(0) = \limsup_{k \nearrow \infty} \inf_{\sigma} \Phi_k(\sigma) = \infty$
		\item or $\limsup_{k \nearrow \infty} \Phi_k (0) < \infty$
	\end{enumerate}
		
\subsection{Case 1: the $S^p$ Factor Limits to $\mathbb{R}^p$}
In the first case, by passing to a subsequence we may assume without loss of generality that $\lim_{k \nearrow \infty} \Phi_k (0) = \infty$.
After perhaps passing to a further subsequence, we then obtain a limiting complete non-flat ancient solution $g(t)$ of the Ricci flow on $\mathbb{R}^{q+1} \times \mathbb{R}^p$ defined for $t \in (-\infty, 0]$ with uniformly bounded curvature $| Rm | \le 1$ and which splits as the product of an ancient rotationally symmetric metric on $\mathbb{R}^{q+1}$ and the stationary Euclidean metric on $\mathbb{R}^p$.
Moreover, by scaling invariance of the bounds $B^2 \le \psi_s \le 1$, the warping function of the rotationally symmetric ancient solution also satisfies such a bound.
The following proposition, implicit in the work of Oliynyk \& Woolgar ~\cite{OW07}, rules out the possibility of such a blow-up limit. 
\begin{prop} \label{noSuchRotlySymm}
	If
		$$g(t) = ds^2 + \psi(s,t)^2 g_{S^q} = \frac{1}{z(\psi, t)} d\psi^2 + \psi^2 g_{S^q} \qquad t \in (-\infty, T)$$
	is a rotationally symmetric, complete, ancient Ricci flow solution on $\mathbb{R}^{q+1}$ such that
		$$0 < B^2 \le z(\psi, t) \le 1, \text{ and }  |Rm| \le 1$$
	 for all $(\psi, t) \in [0, \infty) \times (-\infty, T)$,
	 then $g(t)$ is flat for all $t$.
\end{prop}

\begin{proof}
Fix $1<m<2$ and $t_0 < T$, and consider the quantity
	$$v_m(\psi,t ) \doteqdot (1+t-t_0) \frac{1 - z(\psi, t)}{ \psi^m + \psi^2} \frac{ 1 }{z(\psi,t)}$$
Note that $v_m(0, t) = \lim_{\psi \nearrow \infty} v_m(\psi, t) = 0$ for all $t\in (-\infty, T)$.
By computing the evolution equation for $v_m$ and applying a maximum principle, Oliynyk \& Woolgar ~\cite{OW07} show that, for rotationally symmetric Ricci flows as above, $v_m$ satisfies the bound
	$$v_m( \psi,t) \le \max \left\{ \frac{1}{6 \inf_{(\psi, t) \in [0, \infty) \times [t_0, T)} z(\psi,t)^2}, \max_{\psi \in [0, \infty)} v_m(\psi, t_0) \right\} 	$$
for all $(\psi, t) \in [0, \infty) \times [t_0, T)$.
Now, using that $g(t)$ is an ancient solution with $|Rm| \le 1$ and $z = \psi_s^2 \ge B^2$, we can estimate
	$$\max_{\psi} v_m( \psi, t_0) = \max_{\psi} \frac{1 - \psi_s^2}{ \psi^m + \psi^2} \frac{ 1 }{\psi_s^2} \le \max_\psi \frac{1 - \psi_s^2}{ \psi^2} \frac{ 1 }{\psi_s^2} \le \frac{1}{B^2}$$
Thus, $v_m( \psi, t)$ is bounded by a constant independent of $m$ and $t_0$.
Taking $m \nearrow 2$ and $t_0 \searrow -\infty$, it follows that $\psi_s \equiv 1$ for any $t \in (-\infty, T)$.
\end{proof}

\subsection{Case 2: the $S^p$ Factor Remains}
Now, consider the case that $$\limsup_{k \nearrow \infty} \Phi_k (0) < \infty.$$
After perhaps passing to a subsequence, we obtain a limiting complete non-flat ancient solution $g(t)$ of the Ricci flow on $\mathbb{R}^{q+1} \times S^p$ defined for $t \in (-\infty, 0]$ with uniformly bounded curvature $| Rm | \le 1$ and which is a doubly-warped product.
Moreover, by scaling invariance of the assumed bounds, the warping functions $\phi, \psi$ of the limiting metric satisfy the bounds 
		\begin{equation*} 
			B \le \psi_s \le 1,  \qquad  
			\frac{A}{B} + \delta \le \frac{ \phi}{\psi}, \qquad
			\frac{ \phi_s \psi}{\psi_s \phi} \le 1,  \quad \text{and} \quad 
			0 \le \partial_s \left( \frac{ \phi_s \psi}{\psi_s \phi} \right) 
		\end{equation*}
The next proposition rules out the possibility of such a blow-up limit.
\begin{prop}
	For any $\delta > 0$, there does not exist a complete, doubly-warped, ancient Ricci flow solution 
		$$\left( \mathbb{R}^{q+1} \times S^p, ds^2 + \phi(s,t)^2 g_{S^p} + \psi(s,t)^2 g_{S^q} \right)$$
	satisfying
		\begin{equation*} 
			B \le \psi_s \le 1,  \qquad  
			\frac{A}{B} + \delta \le \frac{ \phi}{\psi}, \qquad
			\frac{ \phi_s \psi}{\psi_s \phi} \le 1,  \quad \text{and} \quad 
			0 \le \partial_s \left( \frac{ \phi_s \psi}{\psi_s \phi} \right) 
		\end{equation*}
\end{prop}

Before providing the proof, we first give a heuristic argument for the above result.
Namely, if such a solution were to exist, the bounds $B \le \psi_s \le 1$ and $\frac{A}{B} + \delta \le \frac{ \phi}{\psi}$ would suggest that the metric is asymptotically conical, that is
	$$\phi \sim \tl{A} s \qquad \psi \sim \tl{B} s \qquad 	\text{as } s \nearrow +\infty$$
with cone angles $\tl{A} > A$ and $\tl{B} \ge B$.
For such metrics, the scalar curvature $R$ is asymptotic to
	$$R \sim \frac{1}{s^2} \left( p(p-1) \frac{ 1 - \tl{A}^2}{\tl{A}^2} + q(q-1) \frac{1- \tl{B}^2}{ \tl{B}^2} -2pq  \right) < 0$$
However, a result of Bing-Long Chen ~\cite{BLC09} shows that ancient Ricci flow solutions have nonnegative scalar curvature.

While this heuristic argument is not rigorous because it made several assumptions on the asymptotics of the warping functions and their derivatives, a similar argument using the assumed properties of $\frac{ \phi_s \psi}{\psi_s \phi}$ will show that the scalar curvature is negative at some $s$ sufficiently large.

\begin{proof}
In what follows, all the computations are done at a fixed time $t \in (-\infty, T)$ and so the dependence on $t$ will often be elided.

Smoothness implies $\frac{ \phi_s \psi}{\psi_s \phi} = 0$ at $s=0$.
Additionally, $\frac{ \phi_s \psi}{\psi_s \phi}$ is increasing in $s$, and so $0 \le \frac{ \phi_s \psi}{\psi_s \phi} \le 1$.
In particular, $\lim_{s \to \infty} \frac{ \psi \phi_s}{\phi \psi_s}$ exists and is in $[0,1]$.
Moreover,
	$$ \frac{ \psi \phi_s}{\phi \psi_s} \le 1 \implies \phi_s \psi - \psi_s \phi \le 0 \implies \frac{\phi}{\psi} \text{ is decreasing in $s$}$$
In particular, $\lim_{s \nearrow + \infty} \frac{\phi}{\psi}$ exists and is in $\left[ \frac{A}{B} + \delta, \infty \right)$.
Multiplying this limit with that of $\frac{ \psi \phi_s}{\phi \psi_s}$, it follows that $\lim_{s \nearrow \infty} \frac{\phi_s}{\psi_s}$ also exists.
L'Hopital's rule now applies and shows that $\lim_{s \nearrow \infty} \frac{\phi}{\psi} = \lim \frac{\phi_s}{\psi_s}$.
We additionally conclude that $\lim_{s \nearrow \infty}  \frac{ \psi \phi_s}{\phi \psi_s} = 1$.
In summary,
	$$\lim_{s \nearrow \infty} \frac{ \phi }{\psi} = \lim_{s \nearrow \infty} \frac{ \phi_s}{\psi_s}  \in \left[ \frac{A}{B} + \delta, \infty \right) \qquad \text{and} \qquad \lim_{s \nearrow \infty} \frac{ \phi_s \psi}{\psi_s \phi} = 1$$

Next, we claim that $\liminf_{s \nearrow \infty} \phi_s \ge A + \delta B$.
Indeed, 
	$$\lim_{s \nearrow \infty} \frac{ \phi_s}{\psi_s}  \ge \frac{A}{B} + \delta$$
Thus, for any $\epsilon > 0$, there exists $s_0$ such that for $s \ge s_0$
	$$\phi_s \ge \left( \frac{A}{B} + \delta - \epsilon \right) \psi_s$$
	Taking $\liminf_{s \nearrow \infty}$, it follows that
		$$\liminf_{s \nearrow \infty} \phi_s \ge \left( \frac{A}{B} + \delta - \epsilon \right) \liminf_{s \nearrow \infty} \psi_s$$
	Since $\epsilon > 0$ is arbitrary, we deduce the claim
		$$\liminf_{s \nearrow \infty} \phi_s \ge \left( \frac{A}{B} + \delta \right) \liminf_{s \nearrow \infty} \psi_s \ge A + \delta B$$

Next, consider the asymptotic behavior of a simplified version of the scalar curvature obtained by omitting the terms containing $\phi_{ss}, \psi_{ss}$ and its rescaling by $\frac{\psi^2}{\psi_s^2}$, that is
\begin{equation*} \begin{aligned} 
	&\limsup_{s \nearrow \infty} \frac{ \psi^2}{\psi_s^2} \left\{ p(p-1) \frac{ 1 - \phi_s^2}{\phi^2} -2pq \frac{ \phi_s \psi_s}{\phi \psi} + q(q-1) \frac{ 1- \psi_s^2}{\psi^2} \right\} \\
	=& \limsup_{s \nearrow \infty} p(p-1) \frac{\psi^2 \phi_s^2}{\psi_s^2 \phi^2} \frac{1}{\phi_s^2} - p(p-1) \frac{ \phi_s^2 \psi^2}{\phi^2 \psi_s^2} -2pq \frac{ \phi_s \psi}{\phi \psi_s} + q(q-1) \frac{ 1}{\psi_s^2} - q(q-1) \\
	\le& p(p-1) \frac{1}{\liminf \phi_s^2} -p(p-1) -2pq + q(q-1) \frac{1}{\liminf \psi_s^2} - q(q-1) \\
	\le&  \frac{p(p-1)}{(A+ \delta B)^2} -p(p-1) -2pq + q(q-1) \frac{1}{B^2} - q(q-1) \\
	<& p(p-1) \frac{1}{ A^2} - \frac{p(p-1)}{A^2}\left( 1 - \left( 1 + \frac{\delta B}{A} \right)^{-2} \right) - p(p-1) -2pq + q(q-1) \frac{1}{B^2}\\ 
	& \qquad - q(q-1) \\
	=& p(n-1) - p(n-1) \left( 1 - \left( 1 + \frac{\delta B}{A} \right)^{-2} \right) -p(p-1) -2pq + q(n-1) - q(q-1)\\
	=& -p(n-1) \left( 1 - \left( 1 + \frac{\delta B}{A} \right)^{-2} \right)< 0\\
\end{aligned} \end{equation*}

We proceed to investigate some of the second derivative terms that appear in the scalar curvature.
Observe that
\begin{equation*} \begin{aligned}
	& &0 \le& \partial_s \left( \frac{ \phi_s \psi}{ \phi \psi_s} \right)\\
	&\implies&  ( \phi_{ss} \psi + \phi_s \psi_s) \phi \psi_s \ge& (\phi \psi_{ss} + \phi_s \psi_s) \phi_s \psi \\
	&\implies& \frac{ \phi_{ss} \psi_s }{\phi \psi} + \frac{ \phi_s \psi_s^2}{\phi \psi^2} \ge& \frac{ \psi_{ss} \phi_s}{\psi \phi} + \frac{\phi_s^2 \psi_s}{\phi^2 \psi}\\
	&\implies& -\frac{ \phi_{ss} \psi_s }{\phi \psi} \le &-\frac{ \psi_{ss} \phi_s}{\psi \phi} - \frac{\phi_s^2 \psi_s}{\phi^2 \psi} + \frac{ \phi_s \psi_s^2}{\phi \psi^2}\\
	&\implies& - \frac{\phi_{ss}}{\phi} \le &- \frac{ \psi_{ss} }{\psi} \frac{ \phi_s \psi}{\phi \psi_s} - \frac{ \phi_s^2}{ \phi^2} + \frac{ \phi_s \psi_s}{ \phi \psi} \\
\end{aligned} \end{equation*}
It now follows that
\begin{equation*}\begin{aligned}
	& \frac{\psi^2}{\psi_s^2} \left\{ -2p \frac{\phi_{ss}}{ \phi } - 2q \frac{ \psi_{ss} }{\psi} \right\}  \\
	\le& \frac{\psi^2}{\psi_s^2} \left\{ -2p \frac{ \psi_{ss} }{\psi} \frac{ \phi_s \psi}{\phi \psi_s} -2p  \frac{ \phi_s^2}{ \phi^2} + 2p \frac{ \phi_s \psi_s}{ \phi \psi}  -2q  \frac{ \psi_{ss} }{\psi} \right\} \\
	= & -2p \frac{ \psi_{ss} \psi}{\psi_s^2} \frac{ \phi_s \psi }{\phi \psi_s} - 2p \frac{ \phi_s^2 \psi^2}{ \phi^2 \psi_s^2} + 2p \frac{ \phi_s \psi }{\phi \psi_s} - 2q \frac{ \psi_{ss} \psi}{\psi_s^2} \\
	= &- \frac{ \psi_{ss} \psi}{\psi_s^2} \left( 2p \frac{ \phi_s \psi}{\phi \psi_s} + 2q \right) - 2p \frac{ \phi_s^2 \psi^2}{ \phi^2 \psi_s^2} + 2p \frac{ \phi_s \psi }{\phi \psi_s} \\
	\le &\frac{ | \psi_{ss} | \psi}{\psi_s^2} \left( 2p \frac{ \phi_s \psi}{\phi \psi_s} + 2q \right) - 2p \frac{ \phi_s^2 \psi^2}{ \phi^2 \psi_s^2} + 2p \frac{ \phi_s \psi }{\phi \psi_s} \\
	\le & \frac{ | \psi_{ss} | \psi}{\psi_s^2} \left( 2p + 2q \right) - 2p \frac{ \phi_s^2 \psi^2}{ \phi^2 \psi_s^2} + 2p \frac{ \phi_s \psi }{\phi \psi_s} \\
\end{aligned} \end{equation*}
	
Next, we claim that for any $\epsilon > 0$ and $s_0 > 0$, there exists $ s > s_0$ such that
		$$| \psi_{ss} (s) | \le \frac{ \epsilon }{s}$$
	Suppose otherwise for the sake of contradiction.
	Then there exists $\epsilon > 0$ and $s_0 > 0$ such that
		$$| \psi_{ss} (s) | > \frac{ \epsilon }{s} \qquad  \text{for all } s > s_0$$
	By continuity of $\psi_{ss}$, either
	\begin{enumerate}
		\item $\psi_{ss} (s)  > \frac{ \epsilon }{s}$  $ \text{for all } s > s_0$, or
		\item $\psi_{ss} (s)  < -\frac{ \epsilon }{s}$  $ \text{for all } s > s_0$
	\end{enumerate} 
	In the first case, integrating from $s_0$ implies
		$$\psi_s(s) \ge \psi_s(s_0) + \epsilon \log( s/s_0) \qquad  \text{for all } s > s_0$$
		$$\implies \psi_s > 1 \qquad \text{for some } s \gg 1$$
	This violates the inequality $\psi_s \le 1$.
	Similarly, in the second case, integrating implies
		$$\psi_s(s) \le \psi_s(s_0) - \epsilon \log( s/ s_0) \qquad \text{for all } s > s_0$$
		$$\implies \psi_s < B \qquad \text{for some } s \gg 1$$
	This violates the inequality $B \le \psi_s$.
	In either case, a contradiction arises and thus the claim is proven.
	
Finally, we can show that the scalar curvature $R$ is negative for some $s \gg 1$ sufficiently large.
\begin{equation*} \begin{aligned}
	&\frac{\psi^2}{\psi_s^2} R \\
	= & \frac{\psi^2}{\psi_s^2} \left\{ -2p \frac{\phi_{ss}}{ \phi } - 2q \frac{ \psi_{ss} }{\psi} \right\} 
	+ \frac{ \psi^2}{\psi_s^2} \left\{ p(p-1) \frac{ 1 - \phi_s^2}{\phi^2} -2pq \frac{ \phi_s \psi_s}{\phi \psi} + q(q-1) \frac{ 1- \psi_s^2}{\psi^2} \right\} \\
	\le& \frac{ | \psi_{ss} | \psi}{\psi_s^2} \left( 2p + 2q \right) - 2p \frac{ \phi_s^2 \psi^2}{ \phi^2 \psi_s^2} + 2p \frac{ \phi_s \psi }{\phi \psi_s} \\
	& \qquad + \frac{ \psi^2}{\psi_s^2} \left\{ p(p-1) \frac{ 1 - \phi_s^2}{\phi^2} -2pq \frac{ \phi_s \psi_s}{\phi \psi} + q(q-1) \frac{ 1- \psi_s^2}{\psi^2} \right\} \\
	\le & \frac{ | \psi_{ss} | s}{B^2} \left( 2p + 2q \right) + 2p \frac{ \phi_s \psi }{\phi \psi_s}  \left( 1 - \frac{\phi_s \psi}{\phi \psi_s} \right) \\
	& \qquad + \frac{ \psi^2}{\psi_s^2} \left\{ p(p-1) \frac{ 1 - \phi_s^2}{\phi^2} -2pq \frac{ \phi_s \psi_s}{\phi \psi} + q(q-1) \frac{ 1- \psi_s^2}{\psi^2} \right\} \\
\end{aligned} \end{equation*}
The above claims show that for any $\epsilon > 0$ there exists $s_0 > 0$ such that
\begin{equation*}\begin{aligned}
	& 2p \frac{ \phi_s \psi }{\phi \psi_s}  \left( 1 - \frac{\phi_s \psi}{\phi \psi_s} \right) \le \epsilon 
	&&  \text{for all } s > s_0, \\
	 & \frac{ \psi^2}{\psi_s^2} \left\{ p(p-1) \frac{ 1 - \phi_s^2}{\phi^2} -2pq \frac{ \phi_s \psi_s}{\phi \psi} + q(q-1) \frac{ 1- \psi_s^2}{\psi^2} \right\} \\
	& \qquad  \le -p(n-1) \left( 1 - \left( 1 + \frac{\delta B}{A} \right)^{-2} \right) + \epsilon 
	&&  \text{for all } s > s_0 , \text{ and }\\
	& \frac{ | \psi_{ss} | s}{B^2} \left( 2p + 2q \right) \le \epsilon &&  \text{for some } s > s_0\\
\end{aligned} \end{equation*} 
By taking
	$$\epsilon < \frac{p(n-1)}{3} \left( 1 - \left( 1 + \frac{\delta B}{A} \right)^{-2} \right)$$
it follows that the scalar curvature is negative at some large $s$.
This contradicts the result of Bing-Long Chen ~\cite{BLC09} that ancient Ricci flow solutions have nonnegative scalar curvature.
\end{proof}

%% file: CoeffEst.tex
The goal of this section is to prove the following theorem
\begin{thm} \label{smallCoeffs}

	Let $(\tl{Z}_{\ul{p}, \ul{q}}, \tl{U}_{\ul{p}, \ul{q}} )$ be as in definition \ref{assignmentOfInitData}.
	There exists $\delta = \delta(p,q,k) >0$  such that 
		$$(\ul{p}, \ul{q}) \in \mathcal{W}_{\tau_0, \tau_1} \text{ and } {P}_{\tau_0, \tau_1}(\ul{p}, \ul{q} ) = 0 \text{ for some } \tau_1 > \tau_0 $$
	implies
	\begin{equation*} \begin{aligned}
		| \ul{p} | &\lesssim_{p,q,k, \ul{D}, \Upsilon_U, \Upsilon_Z, \ul{\Upsilon} } e^{B^2 \lambda_k \tau_0} e^{- \delta \tau_0} \\
		 \text{and } \quad  
	 | \ul{q} | &\le C_{p,q,k, \ul{D}, \Upsilon_U, \Upsilon_Z, \ul{\Upsilon}} e^{B^2 \lambda_k \tau_0} e^{- \delta \tau_0} + C_{p,q,k}\eta_1^U e^{B^2 \lambda_k \tau_0} 
	\end{aligned} \end{equation*}
	if
	$\ul{\Upsilon}, \Upsilon_U, \Upsilon_Z$, and 
	$\tau_0 \gg 1$ are sufficiently large. 
\end{thm}
The proof is based on weighted $H^{-1}$ estimates of the error terms.

\subsection{Parabolic Region Contributions}
We will use the next lemma frequently to obtain such estimates.

\begin{lem} \label{inH^-1}
	Let $a \in \mathbb{N}$ and $b \in \mathbb{R}$ with $a > 1$ and $b > 0$.
	Let $g(\gamma) \in L^1_{loc}(\mathbb{R}_+, d\gamma)$ with 
	$$|g(\gamma)| \lesssim \gamma^{\kappa}  \text{ a.e.}$$
	\begin{enumerate}
	\item	If $\kappa > \frac{-3-a}{2}$ then $g \in H^{-1}_{a,b}$ with
		$$\| g \|_{H^{-1}_{a,b}} \le C \left( a,b, \kappa, \sup_{\gamma} \frac{ | g |}{\gamma^{\kappa}} \right)$$
	\item If $\kappa > \frac{-1-a}{2}$ then $g \in L^2_{a,b}$ with
		$$\| g \|_{L^2_{a,b}} \le C \left(a,b, \kappa, \sup_{\gamma} \frac{ | g |}{\gamma^{\kappa}} \right)$$
	\end{enumerate}
\end{lem}
\begin{proof}
	Assume $\kappa \ne -1$. The case $\kappa = -1$ can be handled separately.\\
	Let $v \in C_c^\infty( \mathbb{R}^{a+1} )$ and let $u(\gamma)$ denote its spherical averages
		$$u(\gamma) \doteqdot  \int_{\mathbb{S}^a} v(\gamma, \theta) d \theta $$
	Then
	\begin{equation*} \begin{aligned}
		\left| \int_0^\infty g u d\mu_{a,b} \right| \lesssim& \int_0^\infty \gamma^\kappa | u | d\mu_{a,b} \\
		=& - \int_0^\infty \left( \frac{1}{\kappa+1} \gamma^{\kappa+1} \right) \partial_\gamma \left( |u| \gamma^a e^{-b\gamma^2/2} \right) d\gamma \\
		\lesssim& \int_0^\infty \gamma^{\kappa+1} \left( | u_\gamma| + \frac{a}{\gamma} | u| + b \gamma |u| \right) d\mu_{a,b} \\ 
		\le& \left( \int_0^\infty \gamma^{2\kappa+2} \gamma^a e^{-b\gamma^2/2} d\gamma \right)^{1/2} \left( \int_0^\infty \left( | u_\gamma| + \frac{a}{\gamma} | u| + b \gamma |u| \right)^2 d\mu_{a,b} \right)^{1/2}  \\
		\le& \left( \int_0^\infty \gamma^{2\kappa+2} \gamma^a e^{-b\gamma^2/2} d\gamma \right)^{1/2} \left( \| u_\gamma \| + \| \frac{a}{\gamma} u \| + \| b \gamma u \| \right)\\
	\end{aligned} \end{equation*}
	By lemmas \ref{MultCts} and \ref{DivCts},
		$$\left( \| u_\gamma \|_{L^2_{a,b}} + \| \frac{a}{\gamma} u \|_{L^2_{a,b}} + \| b \gamma u \|_{L^2_{a,b}} \right) \lesssim \| u \|_{H^1_{a,b}} \lesssim \| v \|_{H^1_{a,b}}$$
	Moreover,
		$$ \int_0^\infty \gamma^{2\kappa+2} \gamma^a e^{-b\gamma^2/2} d\gamma < \infty$$
	if $2\kappa + 2 +a > -1$, or equivalently
		$$\text{ if } \kappa > \frac{-3-a}{2}$$
	The result then follows from density of smooth compactly supported functions.\\
	The proof of (2) is similar.
\end{proof}

\begin{remark}
	A similar estimate holds when $|g| \lesssim \gamma^{\kappa_1} + \gamma^{\kappa_2}$ by noting that
		$$\| g \| \lesssim \| \gamma^{\kappa_1} \| + \| \gamma^{\kappa_2} \|$$
\end{remark}

\begin{prop} \label{parabolicH^-1EstU}
There exists $\epsilon > 0$ such that
if 
	$$(\tl{Z}, \tl{U}) \in \cl{P}[ \ul{D}, l, \kappa, \ul{\eta}, \tau_0, \tau_1, \Upsilon_{U,Z}, M, \alpha, \beta, \lambda_k ]$$
then for all $\tau \in [\tau_0, \tau_1]$
	$$\chi_{[\Upsilon_U e^{-\alpha \tau}, M e^{\beta \tau} ]} Err_U[\tl{Z}, \tl{U}](\tau) \in H^{-1}_{n, \frac{1}{2B^2}}$$
	$$\text{with } \| \chi_{[\Upsilon_U e^{-\alpha \tau}, M e^{\beta \tau} ]} Err_U[\tl{Z}, \tl{U}](\tau) \|_{H^{-1}_{n,\frac{1}{2B^2}}} \lesssim_{p,q,k, \ul{D}, \Upsilon_U} e^{B^2 \lambda_k \tau - \epsilon \tau}$$
\end{prop}
\begin{proof}
	We obtain pointwise estimates for $\chi_{[\Upsilon_U e^{-\alpha \tau}, M e^{\beta_k \tau} ]} Err_U[\tl{Z}, \tl{U}](\tau)$ and then apply lemma \ref{inH^-1}.
%

	If $\tl{Z}, \tl{U} \in \cl{P}$, then proposition \ref{ptwiseEstsErr} implies
		$$| \chi_{[\Upsilon_U e^{-\alpha \tau}, M e^{\beta \tau} ]} Err_U[\tl{Z}, \tl{U}] | \lesssim_{p,q,k,\ul{D}} \chi_{[\Upsilon_U e^{-\alpha \tau}, M e^{\beta \tau} ]}  e^{2B^2 \lambda \tau} \left( \Upsilon_U^l \gamma^{ -4k-4B^2 \lambda - 2} + \gamma^{-4B^2 \lambda } \right)$$
			
	
	Observe that for all $ n \ge 10$
		$$ 0 < 2 - \frac{ 1 - n + \sqrt{ (n-9)(n-1)} }{1 - n + \sqrt{(n-9)(n-1)} + 2} < (n+3) \left[ n+1 - \sqrt{ (n-9)(n-1)} \right]^{-1}$$
	Select $\delta > 0$ such that
		$$ 0 < 2 - \frac{ 1 - n + \sqrt{(n-9)(n-1)} }{1 - n + \sqrt{ (n-9)(n-1)} + 2} < 2 \delta <  (n+3) \left[ n+1 - \sqrt{ (n-9)(n-1)} \right]^{-1}$$
	It follows that
	\begin{equation*} \begin{aligned}
		&\chi_{[\Upsilon_U e^{-\alpha \tau}, M e^{\beta \tau} ]}\gamma^{-4k-4B^2\lambda - 2} e^{2 B_\infty^2 \lambda \tau} \\
		=& \chi_{[\Upsilon_U e^{-\alpha \tau}, M e^{\beta \tau} ]}
		 \gamma^{\delta (-4k-4B^2 \lambda - 2)} \gamma^{(1-\delta)(-4k-4 B^2 \lambda - 2)} e^{2 B_\infty^2 \lambda \tau}\\
		\le& \gamma^{\delta (-4k-4B^2 \lambda - 2) } \left[ \Upsilon_U e^{-\alpha \tau} \right]^{(1-\delta)(-4k-4B^2 \lambda - 2)} e^{2 B_\infty^2 \lambda \tau}\\
		=&\Upsilon_U^{(1-\delta)(-4k-4B^2 \lambda - 2)} \gamma^{\delta (-4k-4B^2 \lambda - 2) } e^{-\alpha \tau (1-\delta)(-4k-4B^2 \lambda - 2)} e^{2 B_\infty^2 \lambda \tau}\\
	\end{aligned} \end{equation*}
	By the above choice of $\delta$,
		$$\delta (-4k-4B^2 \lambda - 2) > \frac{ - 3 - n}{2} \qquad \& \qquad -\alpha  (1-\delta)(-4k-4B^2 \lambda - 2) +2 B_\infty^2 \lambda < B^2 \lambda$$
		

	The statement of the proposition now follows from the remark following lemma \ref{inH^-1} if $\epsilon > 0$ is taken to satisfy
	$$ -\alpha  (1-\delta)(-4k-4B^2 \lambda - 2) +2 B_\infty^2 \lambda < B^2 \lambda - \epsilon < B^2 \lambda$$
	
 \end{proof}

 \begin{prop} \label{parabolicH^-1EstZ}
There exists $\epsilon > 0$ such that
if 
	$$(\tl{Z}, \tl{U}) \in \cl{P} [ \ul{D},l, \kappa, \ul{\eta}, \tau_0, \tau_1, \Upsilon_{U,Z}, M, \alpha, \beta, \lambda_k ]$$
then for all $\tau \in [\tau_0, \tau_1]$
	$$\chi_{[\Upsilon_Z e^{-\alpha \tau}, M e^{\beta \tau} ]} Err_Z[\tl{Z}, \tl{U}](\tau) \in H^{-1}_{n-2,\frac{1}{2B^2}}$$
	$$\text{with } \| \chi_{[\Upsilon_Z e^{-\alpha \tau}, M e^{\beta \tau} ]} Err_Z[\tl{Z}, \tl{U}](\tau) \|_{H^{-1}_{n-2,\frac{1}{2B^2}}} \lesssim_{p,q,k, \ul{D}, \Upsilon_Z} e^{B^2 \lambda_k \tau - \epsilon \tau}$$
\end{prop}
\begin{proof}
	The proof is similar to the proof of proposition \ref{parabolicH^-1EstZ}.
	Namely, we first obtain pointwise estimates and then apply lemma \ref{inH^-1}.
	
	By the estimates in proposition \ref{ptwiseEstsErr},
		$$\chi{[\Upsilon_Z e^{-\alpha \tau}, M e^{\beta \tau} ]} Err_Z[ \tl{Z}, \tl{U}](\tau) \lesssim_{p,q,k, \ul{D}} e^{2B^2 \lambda \tau} \left( \gamma^{-4k - 4B^2 \lambda - 2} + \gamma^{-4B^2 \lambda-2} \right)$$

	It can be verified that for all $n \ge 10$
		$$ 2 - \frac{ 1 - n + \sqrt{ (n-9)(n-1)} }{1 - n + \sqrt{ (n-9)(n-1)} + 2} < (n+1) \left[ n+1 - \sqrt{(n-9)(n-1)} \right]^{-1}$$
	Select $\delta > 0$ such that
		$$  2 - \frac{ 1 - n + \sqrt{ (n-9)(n-1)} }{1 - n + \sqrt{ (n-9)(n-1)} + 2} < 2 \delta < (n+1) \left[ n+1 - \sqrt{ (n-9)(n-1)} \right]^{-1}$$
	
	It follows that
	\begin{equation*} \begin{aligned}
		&\chi_{[\Upsilon_Z e^{-\alpha \tau}, M e^{\beta \tau} ]} e^{2B^2 \lambda_k \tau} \gamma^{-4k-4B^2 \lambda_k -2} \\
		\le&
		\left[ \Upsilon_Z e^{- \alpha \tau} \right]^{(1-\delta)(-4k-4B^2 \lambda_k -2)} e^{2B^2 \lambda_k \tau} \gamma^{\delta ( - 4k - 4B^2 \lambda_k -2)} \\
		\le & \Upsilon_Z^{(1-\delta)(-4k-4B^2 \lambda_k -2)} e^{ - \alpha \tau (1-\delta)(-4k-4B^2 \lambda_k -2) + 2B^2 \lambda_k \tau} \gamma^{\delta ( - 4k - 4B^2 \lambda_k -2)} \\
	\end{aligned} \end{equation*}
	By the choice of $\delta$, 
		$$\delta ( - 4k - 4B^2 \lambda_k -2) > \frac{ 3-(n-2)}{2} \qquad \& \qquad - \alpha  (1-\delta)(-4k-4B^2 \lambda_k -2) + 2B^2 \lambda_k < B^2 \lambda_k$$
	
	The statement of the proposition now follows from the remark following lemma \ref{inH^-1} if $ \epsilon > 0$ is taken to satisfy
		$$- \alpha  (1-\delta)(-4k-4B^2 \lambda_k -2) + 2B^2 \lambda_k < B^2 \lambda_k - \epsilon< B^2 \lambda_k$$
\end{proof}

\begin{prop} \label{parabolicH^-1EstN}
	If 
	$$(\tl{Z}, \tl{U}) \in \cl{P}[ \ul{D}, l , \kappa, \ul{\eta}, \tau_0, \tau_1, \Upsilon_{U,Z}, M, \alpha, \beta, \lambda_k]$$
then for all $\tau \in [\tau_0, \tau_1]$
		$$\chi_{[\Upsilon_U e^{-\alpha \tau}, M e^{\beta \tau} ]}  \mathcal{N} ( \tl{U} - e^{B^2 \lambda_k \tau} \tl{U}_{\lambda_k} ) \in H^{-1}_{n-2,\frac{1}{2B^2}}$$
		$$\text{with } \left\| \chi_{[\Upsilon_U e^{-\alpha \tau}, M e^{\beta \tau} ]}  \mathcal{N} ( \tl{U} - e^{B^2 \lambda_k \tau} \tl{U}_{\lambda_k} ) \right\|_{H^{-1}_{n-2,\frac{1}{2B^2}} } \lesssim_{p,q,k} \eta_1^U e^{B^2 \lambda_k \tau} $$
\end{prop}
\begin{proof}
	There is the pointwise estimate
		$$\chi_{[\Upsilon_U e^{-\alpha \tau}, M e^{\beta \tau} ]}  \mathcal{N} ( \tl{U} - e^{B^2 \lambda_k \tau} \tl{U}_{\lambda_k} )  \lesssim_{p,q} \eta_1^U e^{B^2 \lambda_k \tau} \left( \gamma^{-2k - 2B^2 \lambda_k -2} + \gamma^{-2B^2 \lambda_k -1} \right)$$
	by proposition \ref{ptwiseEstsErr}.
	Observe that for $n \ge 10$
		$$-2k - 2B^2 \lambda_k -2 > \frac{-3-(n-2)}{2}$$
	Hence, lemma \ref{inH^-1} applies and the proposition follows immediately.
\end{proof}

\subsection{Inner Region Contributions}
\begin{lem} \label{inH^-1Inner}
	Let $$\chi(\gamma) = \chi_{[0, \Upsilon e^{-\alpha \tau}]}(\gamma)$$ denote the characteristic function for $\{ \gamma \in \mathbb{R} :  0 \le \gamma \le  \Upsilon e^{-\alpha \tau} \}$.
	Let $a \in \mathbb{N}$, $b \in \mathbb{R}$, and $0 \le \epsilon < \frac{1}{2}$ with $a>1$ and $b > 0$.
	Then
		$$\chi(\gamma) \gamma^{-2 - \epsilon(a-1)} \in H^{-1}_{a,b} \qquad \text{with} \qquad \| \chi \gamma^{-2} \|_{H^{-1}_{a,b}} \lesssim_{a,b} \Upsilon^{ (1-\epsilon) \frac{ (a-1)}{2} } e^{-(1-\epsilon)\frac{a-1}{2} \alpha \tau}$$
\end{lem}
\begin{proof}
	Let $v \in C_c^\infty( \mathbb{R}^{a+1} )$ and let $u(\gamma)$ denote its spherical averages
		$$u(\gamma) \doteqdot  \int_{\mathbb{S}^a} v(\gamma, \theta) d \theta $$
	Then
	\begin{equation*} \begin{aligned}
		&\left| \int_{\mathbb{R}^{a+1}} \chi \gamma^{-2 - \epsilon(a-1)} v e^{-b \gamma^2/2} d(\gamma, \theta) \right| \\
		&\le  \int_{0}^\infty \chi(\gamma) \gamma^{-2 - \epsilon (a-1)} |u| \gamma^a e^{-b \gamma^2/2} d \gamma \\
		 &\le \left( \int \chi^2 \gamma^{a-2 - \epsilon(a-1)} e^{-b\gamma^2/2} d \gamma \right)^{1/2} \left( \int \gamma^{-2} u^2 \gamma^a e^{-b \gamma^2/2} d\gamma \right)^{1/2}\\
		&\le  \left( \int_0^{\Upsilon e^{-\alpha \tau}} \gamma^{a-2 - \epsilon(a-1)} d \gamma \right)^{1/2} \left \| \frac{1}{\gamma} u \right \|_{L^2_{a,b}}\\
		&\le  \left( \frac{1}{(a-1)(1-\epsilon)} \Upsilon^{(1-\epsilon)(a-1)} e^{-\alpha (1- \epsilon) (a-1) \tau} \right)^{1/2} \left \| \frac{1}{\gamma} u  \right \|_{L^2_{a,b}}\\
		&\lesssim_{a,b} \left(  \Upsilon^{(1-\epsilon)(a-1)} e^{-\alpha( 1- \epsilon) (a-1) \tau} \right)^{1/2} \| u \|_{H^1_{a,b}} && \text{(by lemma \ref{DivCts})}\\
		&\lesssim   \Upsilon^{(1- \epsilon) (a-1)/2} e^{- (1- \epsilon) \frac{a-1}{2} \alpha \tau} \| v \|_{H^1_{a,b} } \\
	\end{aligned} \end{equation*}
	
\end{proof}

\begin{prop} \label{innerH^-1EstErr}
	For any $\Upsilon > 0$ and $\tau_0$ such that $\Upsilon e^{- \alpha \tau_0} \le 1$, if
		$$\left| \frac{ \tl{U} }{\log(\gamma)} \right| + | \gamma \tl{U}_\gamma | + | \gamma^2 \tl{U}_{\gamma \gamma} | + | \tl{Z} | + | \gamma \tl{Z}_{\gamma} | + | \gamma^2 \tl{Z}_{\gamma \gamma} | \lesssim 1$$
		$$\text{for all }  \tau \in [\tau_0, \tau_1], \gamma \in (0 , \Upsilon e^{- \alpha \tau}]$$
	then
	\begin{equation*} \begin{aligned}
		 \chi_{[0, \Upsilon e^{- \alpha \tau}]} Err_U  (\tau) &\in H^{-1}_{n, \frac{1}{2B^2} } \\
		 \chi_{[0, \Upsilon e^{- \alpha \tau}]} Err_Z  (\tau) &\in H^{-1}_{n-2, \frac{1}{2B^2} } \\
		 \chi_{[0, \Upsilon e^{- \alpha \tau}]} \mathcal{N} (\tl{U} - e^{B^2 \lambda_k \tau} \tl{U}_{\lambda_k} )  (\tau) &\in H^{-1}_{n-2, \frac{1}{2B^2} } \\
	\end{aligned} \end{equation*}
	with estimates
	\begin{equation*} \begin{aligned}
		\| \chi_{[0, \Upsilon e^{- \alpha \tau}]} Err_U  (\tau) \|_{H^{-1}_{n,\frac{1}{2B^2}}} 
		&\lesssim_{p,q,k, \epsilon}  \Upsilon^{(1-\epsilon) \frac{n-1}{2}} e^{-(1-\epsilon) \frac{n-1}{2} \alpha_k \tau} 
		\\
		\| \chi_{[0, \Upsilon e^{- \alpha \tau}]} Err_Z  (\tau)\|_{H^{-1}_{n-2,\frac{1}{2B^2}}} 
		&\lesssim_{p,q,k} \Upsilon^{\frac{n-3}{2}} e^{- \frac{n-3}{2} \alpha_k \tau}
		\\
		 \|  \chi_{[0, \Upsilon e^{- \alpha \tau}]} \mathcal{N} (\tl{U} - e^{B^2 \lambda_k \tau} \tl{U}_{\lambda_k} )  (\tau)  \|_{H^{-1}_{n-2,\frac{1}{2B^2}}} 
		 &\lesssim_{p,q,k} \Upsilon^{\frac{n-3}{2}} e^{- \frac{n-3}{2} \alpha_k \tau} + \sqrt{\Upsilon} e^{\left( B^2 \lambda_k  - \frac{\alpha_k}{2} \right)\tau} 
		 \\
	\end{aligned} \end{equation*} 
	for any $0 < \epsilon < 1/2$.
\end{prop}
\begin{proof}
	The pointwise bounds 
		$$\left| \frac{ \tl{U} }{\log(\gamma)} \right| + | \gamma \tl{U}_\gamma | + | \gamma^2 \tl{U}_{\gamma \gamma} | + | \tl{Z} | + | \gamma \tl{Z}_{\gamma} | + | \gamma^2 \tl{Z}_{\gamma \gamma} | \lesssim 1$$
	imply that, for all $\tau \in [\tau_0, \tau_1], \gamma \in (0 , \Upsilon e^{- \alpha \tau}]$,
	\begin{equation*} \begin{aligned}
		| \chi_{[0, \Upsilon e^{- \alpha \tau}]} Err_U | 
		&\lesssim_{p,q,k} \frac{1}{\gamma^2} ( 1 + | \log (\gamma) | ) 
		\lesssim_{\epsilon} \frac{1}{\gamma^{2 + \epsilon n}} \\ 
		\text{and } \quad |  \chi_{[0, \Upsilon e^{- \alpha \tau}]} Err_Z  | 
		&\lesssim_{p,q,k} \frac{1}{\gamma^2} \\
	\end{aligned} \end{equation*}
	for any $0 < \epsilon < 1/2$.
	Hence, lemma \ref{inH^-1Inner} with $a = n$ and $a = n-2$ imply  
		$$ \chi_{[0, \Upsilon e^{- \alpha \tau}]} Err_U  (\tau) \in H^{-1}_{n,\frac{1}{2B^2}} \quad \text{and} \quad \chi_{[0, \Upsilon e^{- \alpha \tau}]} Err_Z  (\tau) \in H^{-1}_{n-2,\frac{1}{2B^2}}$$
	with corresponding estimates on the $H^{-1}$ norms.
	
	To estimate $\chi_{[0, \Upsilon e^{- \alpha \tau}]} \mathcal{N}(\tl{U} - e^{B^2 \lambda_k \tau} \tl{U}_{\lambda_k} )$, first note that the assumed pointwise bounds and the form of the eigenfunction $\tl{U}_{\lambda_k}(\gamma)$ imply 
		$$\chi_{[0, \Upsilon e^{- \alpha \tau}]} \mathcal{N}(\tl{U} - e^{B^2 \lambda_k \tau} \tl{U}_{\lambda_k} ) \lesssim \chi_{[0, \Upsilon e^{- \alpha \tau}]} \left( \frac{1}{\gamma^2} + e^{B^2 \lambda_k \tau} \gamma^{-2k - 2B^2 \lambda_k -2} \right)$$
	for all $ \tau \in [\tau_0, \tau_1]$.
	The factor of $\gamma^{-2}$ can be estimated as above using lemma \ref{inH^-1Inner}.
	The $H^{-1}_{n-2,b}$ norm of the remaining term is estimated directly.
	Namely, let $v \in C_c^\infty( \mathbb{R}^{a+1} )$ and let $u(\gamma)$ denote its spherical averages
		$$u(\gamma) \doteqdot  \int_{\mathbb{S}^a} v(\gamma, \theta) d \theta $$
	Then
	\begin{equation*} \begin{aligned}
		& \quad \left| \int_{\mathbb{R}^{n-1}} \chi_{[0, \Upsilon e^{- \alpha \tau}]} e^{B^2 \lambda_k \tau} \gamma^{-2k - 2B^2 \lambda_k -2} v e^{-\gamma^2/(4B^2)} d(\gamma, \theta) \right| \\
		&\le e^{B^2 \lambda_k \tau} \int_0^{\Upsilon e^{-\alpha \tau} } \gamma^{-2k - 2B^2 \lambda_k -2} |u| \gamma^{n-2} e^{- \gamma^2 /(4B^2)} d\gamma \\
		&\le  e^{B^2 \lambda_k \tau} \left( \int_0^{\Upsilon e^{- \alpha \tau} } \gamma^{2(-2k - 2B^2 \lambda_k -1)} \gamma^{n-2} e^{- \gamma^2/(4B^2)} d\gamma \right)^{1/2} \| \gamma^{-1} u \|_{L^2_{n-2, \frac{1}{2B^2} }} \\
		&\lesssim_{p,q}  e^{B^2 \lambda_k \tau} \left( \int_0^{\Upsilon e^{- \alpha \tau} } \gamma^{2(-2k - 2B^2 \lambda_k -1)} \gamma^{n-2} e^{- \gamma^2/(4B^2)} d\gamma \right)^{1/2} \|  u \|_{H^1_{n-2, \frac{1}{2B^2} }} && (\text{lemma \ref{DivCts}})\\
	\end{aligned} \end{equation*}
	Because
		$$2(-2k -2B^2 \lambda_k -1) + n-2 = \sqrt{(n-9)(n-1)} -3 \ge 0 \qquad (\text{for all } n \ge 10)$$
	it follows that
		$$\|  \chi_{[0, \Upsilon e^{- \alpha \tau}]} e^{B^2 \lambda_k \tau} \gamma^{-2k - 2B^2 \lambda_k -2} \|_{H^{-1}_{n-2,\frac{1}{2B^2}} } \lesssim_{p,q,k} e^{B^2 \lambda_k \tau}  e^{-\frac{1}{2} \alpha \tau} \sqrt{\Upsilon}$$
\end{proof}

\begin{remark}
	In the setting of the previous proposition, 
	there exists $0 < \epsilon' = \epsilon'(p,q,k)$
	such that
	\begin{equation*} \begin{aligned}
		& \quad  \| \chi_{[0, \Upsilon e^{- \alpha \tau}]} Err_U  (\tau) \|_{H^{-1}_{n, \frac{1}{2B^2} }} 
		+ \| \chi_{[0, \Upsilon e^{- \alpha \tau}]} Err_Z  (\tau)\|_{H^{-1}_{n-2, \frac{1}{2B^2} }} \\
		& \qquad + \|  \chi_{[0, \Upsilon e^{- \alpha \tau}]} N(\tl{U} - e^{B^2 \lambda_k \tau} \tl{U}_{\lambda_k} )  (\tau)  \|_{H^{-1}_{n-2, \frac{1}{2B^2} }} \\
		& \lesssim_{p,q,k, \Upsilon} e^{B^2 \lambda_k \tau - \epsilon ' \tau} \\
		& \ll e^{B^2 \lambda_k \tau} && (\text{as } \tau \nearrow \infty)
	\end{aligned} \end{equation*}
	Indeed, this asymptotic estimate follows from the fact that
		$$n \ge 10 \implies \sqrt{(n-9)(n-1)} > 2 $$
		$$\implies \frac{n-1}{2} \alpha_k > \frac{n-3}{2} \alpha_k = \frac{n-1-2}{n-1 - \sqrt{ (n-9)(n-1)}} (-  B^2 \lambda_k) > -  B^2 \lambda_k$$	
\end{remark}

\subsection{Outer Region Contributions}
\begin{lem} \label{inH^-1Outer}
	Let $a \in \mathbb{N}$ and $b \in \mathbb{R}$ with $a > 1$ and $b > 0$.
	Fix $0 < \beta < 1/2$.
	For any $\epsilon \in (0,1)$, there exists $\tau_0 \gg 1$ sufficiently large such that if
	$g(\gamma, \tau) \in L^1_{loc}$ is supported on 
		$$\{ \tau \ge \tau_0 \quad \& \quad \gamma \ge M e^{\beta \tau} \}$$
	and satisfies the pointwise estimate
		$$| g| \le C ( 1 + e^{\kappa_1 \tau} ) ( 1 + \gamma^{\kappa_2} ), \qquad (\kappa_1, \kappa_2 > 0)$$
	then $g \in L^2_{a,b} \subset H^{-1}_{a,b}$ and 
		$$\| g \|_{H^{-1}_{a,b}} \le \| g \|_{L^2_{a,b}} \lesssim_{a,b} e^{- (1-\epsilon) \frac{b}{4} M^2 e^{2 \beta \tau} } \ll e^{B^2 \lambda_k \tau}$$
\end{lem}
\begin{proof}
	Let $v \in C_c^\infty( \mathbb{R}^{a+1} )$ and let $u(\gamma)$ denote its spherical averages
		$$u(\gamma) \doteqdot  \int_{\mathbb{S}^a} v(\gamma, \theta) d \theta $$
	Then
	\begin{equation*} \begin{aligned}
		\left| \int_{\mathbb{R}^{a+1}} g v \gamma^a e^{- b \gamma^2 /2 } d( \gamma, \theta) \right| \le & C(1 + e^{\kappa_1 \tau} ) \int_{M e^{\beta \tau}}^\infty ( 1 + \gamma^{\kappa_2} ) |u| \gamma^a e^{-b\gamma^2/2} d \gamma\\
		= & C(1 +e^{\kappa_1 \tau} ) \left[  \int_{M e^{\beta \tau}}^\infty |u| \gamma^a e^{-b\gamma^2/2} d \gamma + \int_{M e^{\beta \tau}}^\infty \gamma^{\kappa_2} |u| \gamma^a e^{-b\gamma^2/2} d \gamma \right] \\
		\doteqdot & C( 1 + e^{\kappa_1 \tau}  ) ( I + II ) \\
	\end{aligned} \end{equation*}
	$I$ and $II$ can each be estimated with H{\"o}lder's inequality and lemma \ref{asympsOfExpInt}.
	For example,
	\begin{equation*} \begin{aligned}
		II &\le \left( \int_{M e^{\beta \tau}}^\infty \gamma^{2 \kappa_2} \gamma^a e^{- b \gamma^2 /2} d\gamma \right)^{1/2} \| u \|_{L^2_{a,b}} \\
		& \lesssim_b (M e^{\beta \tau})^{ \kappa_2 + \frac{a-1}{2}} e^{ -\frac{b}{4} (M e^{\beta \tau})^2 } \| u \|_{L^2_{a,b}} 
		&& (\text{lemma \ref{asympsOfExpInt}})\\
		& \lesssim_{a,b}  (M e^{\beta \tau})^{ \kappa_2 + \frac{a-1}{2}} e^{ -\frac{b}{4} (M e^{\beta \tau})^2 } \| v \|_{L^2_{a,b}}  
	\end{aligned} \end{equation*}
	A similar argument shows that
		$$I \lesssim_{a,b} (M e^{\beta \tau})^{\frac{a-1}{2}} e^{-\frac{b}{4} (M e^{\beta \tau})^2 } \| v \|_{L^2_{a,b}}$$
	The conclusion follows from observing that the doubly-exponential term dominates at large $\tau$.
\end{proof}

\begin{prop} \label{outerH^-1EstErr}
	For any $M > 0$ and $0 < \beta< 1/2$ and $\epsilon \in (0,1)$, if 
	$\tau_0 \gg 1$ is sufficiently large depending on $p,q,k, M, \beta, \epsilon$ and
	$$(\tl{Z}, \tl{U}) \in \cl{P} [ \ul{D}, l, \kappa, \ul{\eta} , \tau_0, \tau_1, \Upsilon_{U,Z}, M, \alpha, \beta, \lambda_k ]$$
	then
	\begin{equation*} \begin{aligned}
		 \chi_{[M e^{\beta \tau}, \infty )} \grave{\chi} Err_U  (\tau) &\in H^{-1}_{n,\frac{1}{2B^2}} \\
		 \chi_{[M e^{\beta \tau}, \infty )} \grave{\chi} Err_Z  (\tau) &\in H^{-1}_{n-2,\frac{1}{2B^2}} \\
		 \chi_{[M e^{\beta \tau},\infty)} \left( \grave{\chi} \mathcal{N} \tl{U} - \mathcal{N} e^{B^2 \lambda_k \tau} \tl{U}_{\lambda_k} \right)  (\tau) &\in H^{-1}_{n-2,\frac{1}{2B^2}} \\
		\chi_{[M e^{\beta \tau},\infty)} \left( [ \grave{\chi}, B^2 \mathcal{D}_U ] \tl{U} + \grave{\chi}_\tau \tl{U} \right) &\in H^{-1}_{n,\frac{1}{2B^2}} \\
		\chi_{[M e^{\beta \tau},\infty)} \left( [\grave{\chi}, B^2 \mathcal{D}_Z ] \tl{Z} + \grave{\chi}_\tau \tl{Z}  \right) &\in H^{-1}_{n-2,\frac{1}{2B^2}} \\
	\end{aligned} \end{equation*}
	with estimates
	\begin{equation*} \begin{aligned}
		\|  \chi_{[M e^{\beta \tau}, \infty )} \grave{\chi} Err_U  (\tau) \|_{H^{-1}_{n,\frac{1}{2B^2}}} 
		+ \| \chi_{[M e^{\beta \tau},\infty)} \left( [ \grave{\chi}, B^2 \mathcal{D}_U ] \tl{U} + \grave{\chi}_\tau \tl{U} \right) \|_{H^{-1}_{n,\frac{1}{2B^2}} }  \\
		\lesssim_{p,q,k, \ul{D}} e^{B^2 \lambda_k \tau - \epsilon \tau}
	\end{aligned} \end{equation*}
	\begin{equation*} \begin{aligned}
		\| \chi_{[M e^{\beta \tau}, \infty )} \grave{\chi} Err_Z  (\tau) \|_{H^{-1}_{n-2,\frac{1}{2B^2}} } + \|  \chi_{[M e^{\beta \tau},\infty)} \left( \grave{\chi} \mathcal{N} \tl{U} - \mathcal{N} e^{B^2 \lambda_k \tau} \tl{U}_{\lambda_k} \right)  (\tau) \|_{H^{-1}_{n-2,\frac{1}{2B^2}}} \\
		+ \| \chi_{[M e^{\beta \tau},\infty)} \left( [\grave{\chi}, B^2 \mathcal{D}_Z ] \tl{Z} + \grave{\chi}_\tau \tl{Z}  \right)  \|_{H^{-1}_{n-2,\frac{1}{2B^2}} } \lesssim_{p,q,k, \ul{D}} e^{B^2 \lambda_k \tau - \epsilon \tau}
	\end{aligned} \end{equation*}
\end{prop}

\begin{proof}
	We begin with pointwise estimates.
	Proposition \ref{ptwiseEstsErr} implies that
	\begin{equation*} \begin{aligned}
		\chi_{[M e^{\beta \tau}, \infty )} \grave{\chi} Err_U \lesssim_{p,q,k, \ul{D}}& e^{2B^2 \lambda_k \tau} \gamma^{-4B^2 \lambda_k}\\
		\chi_{[M e^{\beta \tau}, \infty )} \grave{\chi} Err_Z \lesssim_{p,q,k, \ul{D}}& e^{2B^2 \lambda_k \tau} \gamma^{-4B^2 \lambda_k}\\
		\chi_{[M e^{\beta \tau},\infty)} \left( \grave{\chi} \mathcal{N} \tl{U} - \mathcal{N} e^{B^2 \lambda_k \tau} \tl{U}_{\lambda_k} \right) 
		\lesssim_{p,q,k, \ul{D}}& e^{B^2 \lambda_k \tau} \gamma^{-2B^2 \lambda_k + 1} \\
		\chi_{[M e^{\beta \tau},\infty)} \left( [ \grave{\chi}, B^2 \mathcal{D}_U ] \tl{U} + \grave{\chi}_\tau \tl{U} \right) 
		\lesssim_{p,q,k, \ul{D}}&  e^{B^2 \lambda_k \tau} \gamma^{-2B^2 \lambda_k + 1} \\ 
		\chi_{[M e^{\beta \tau},\infty)} \left( [\grave{\chi}, B^2 \mathcal{D}_Z ] \tl{Z} + \grave{\chi}_\tau \tl{Z}  \right)
		\lesssim_{p,q,k, \ul{D}}&  e^{B^2 \lambda_k \tau} \gamma^{-2B^2 \lambda_k + 1} \\
	\end{aligned} \end{equation*}
	The remainder of the proof now follows from lemma \ref{inH^-1Outer} and estimating the resulting doubly-exponential bound by $e^{B^2 \lambda_k \tau - \epsilon \tau}$ with $\epsilon \in (0,1)$.
\end{proof}

We summarize the results of the last three subsections for the readers' convenience and for reference in later sections.

\begin{prop} \label{summH^-1EstErr}
	There exists $\epsilon = \epsilon(p,q,k) > 0$ such that if 
		$$\left( \tl{Z}, \tl{U} \right) \in \mathcal{P} [ \ul{D}, l, \kappa, \ul{\eta},  \tau_0, \tau_1, \Upsilon_U, \Upsilon_Z, M , \alpha , \beta, \lambda_k ]$$
	and $\tau_0 \gg 1$ is sufficiently large depending on $p,q,k, \Upsilon_U, \Upsilon_Z, M, \beta$
	then 
	\begin{equation*} \begin{aligned}
		\| \grave{\chi} Err_U(\tau) |_{H^{-1}_{n, \frac{1}{2B^2} } } 
		&\le C_{p,q,k, \ul{D}, \Upsilon_U} e^{B^2 \lambda \tau - \epsilon \tau} \\
		\| \grave{\chi} Err_Z(\tau) |_{H^{-1}_{n-2, \frac{1}{2B^2} } } 
		&\le C_{p,q,k, \ul{D}, \Upsilon_Z} e^{B^2 \lambda \tau - \epsilon \tau} \\
		\| \grave{\chi} \mathcal{N} \tl{U} - \mathcal{N} e^{B^2 \lambda \tau} \tl{U}_{\lambda_k} \|_{H^{-1}_{n-2,b}}
		& \le C_{p,q,k,\ul{D},  \Upsilon_U  } e^{B^2 \lambda \tau - \epsilon \tau} 
		+ C_{p,q,k, \ul{D} } \eta_1^U e^{B^2 \lambda \tau} \\
		\| B^2 [ \grave{\chi}, \mathcal{D}_U ] \tl{U} + \grave{\chi}_\tau \tl{U} \|_{H^{-1}_{n, \frac{1}{2B^2} } }
		&\le C_{p,q,k, \ul{D}} e^{B^2 \lambda \tau - \epsilon \tau} \\
		\| B^2 [ \grave{\chi}, \mathcal{D}_Z ] \tl{Z} + \grave{\chi}_\tau \tl{U} \|_{H^{-1}_{n-2, \frac{1}{2B^2} } }
		&\le C_{p,q,k, \ul{D}} e^{B^2 \lambda \tau - \epsilon \tau} \\
	\end{aligned} \end{equation*}
	for all $\tau \in [ \tau_0, \tau_1 ]$.
\end{prop}

\subsection{Initial Data Contributions} 
\begin{lem} \label{L^2EstInitU}
	For initial data satisfying
		$$0 \le \tl{U}_{\ul{p}, \ul{q}} ( \tau_0, \gamma) \le \left[ C- \log \left( \frac{A}{B} \gamma \right) \right]  \qquad \text{for all } 0 \le \gamma \le \ul{\Upsilon} e^{- \alpha_k \tau_0}$$
	and
		$$0 \le \tl{U}_{\ul{p}, \ul{q}} ( \tau_0, \gamma) \le D + \frac{\tau_0}{2} - \log \left( \frac{A}{B} \gamma \right) \qquad \text{for all } \gamma \ge M e^{\ol{\beta} \tau_0} \quad \big( \ol{\beta} \in (0, 1/2) \big)$$ 
	there exists $\epsilon = \epsilon(p,q,k) >0$ such that
		$$\left \| \check{\grave{U}}_{\ul{p}, \ul{q}} ( \tau_0, \cdot) - \tl{U}_{\lambda_k}(\cdot) \right \|_{L^2_{n,\frac{1}{2B^2}}} \lesssim_{p,q,k,\ul{\Upsilon}, C, D} e^{- \epsilon \tau_0} \ll 1$$
	for $\tau_0 \gg 1$ sufficiently large depending on $p,q,k, D, M , \ol{\beta}$. 
	
	In particular, for all $ j \ne k$
		$$\left| \left \langle \check{\grave{U}}_{\ul{p}, \ul{q}}(\tau_0, \cdot) , \tl{U}_{\lambda_j} \right \rangle_{L^2_{n,\frac{1}{2B^2}}} \right| \lesssim_{p,q,k, \ul{\Upsilon}, C, D} e^{- \epsilon \tau_0} \ll 1$$
\end{lem}

\begin{proof}
	Because 
		$$\check{\grave{U}}_{\ul{p}, \ul{q}}(\tau_0, \gamma) = \tl{U}_{\lambda_k}(\gamma) \qquad \text{for all } \ul{\Upsilon} e^{-\alpha_k \tau_0} \le \gamma \le M e^{\ol{\beta} \tau_0}$$
	it follows that
	\begin{equation*} \begin{aligned}
		\left \| \check{\grave{U}}_{\ul{p}, \ul{q}} ( \tau_0, \cdot) - \tl{U}_{\lambda_k}(\cdot) \right\|_{L^2_{n,\frac{1}{2B^2}}}^2 
		\le & \int_0^{\ul{\Upsilon} e^{- \alpha \tau_0} } \left| \check{\grave{U}}_{\ul{p}, \ul{q}} ( \tau_0, \cdot) - \tl{U}_{\lambda_k}(\cdot) \right|^2 \gamma^{n} e^{- \gamma^2/(4B^2)} d \gamma \\
		&+ \int_{M e^{\ol{\beta} \tau_0}}^\infty \left| \check{\grave{U}}_{\ul{p}, \ul{q}} ( \tau_0, \cdot) - \tl{U}_{\lambda_k}(\cdot) \right|^2 \gamma^{n} e^{- \gamma^2/(4B^2)} d \gamma \\
	\end{aligned} \end{equation*}
	
	For $\gamma \le \ul{\Upsilon} e^{- \alpha \tau_0}$,
	\begin{equation*} \begin{aligned}
		\left| \check{\grave{U}}_{\ul{p}, \ul{q}} - \tl{U}_{\lambda_k} \right| 
		&=\left|  e^{- B^2 \lambda_k \tau_0} \left[ \grave{U}_{\ul{p}, \ul{q}}( \tau_0, \gamma) - \sum_{j = 1}^{k-1} p_j \tl{U}_{\lambda_j} \right] - \tl{U}_{\lambda_k} \right| \\
		&\le \left| C - \log \left( \frac{A}{B} \gamma \right) \right| e^{-B^2 \lambda_k \tau_0} \\
		& \quad + \sum_{j=1}^{k-1} | p_j| e^{- B_\infty^2 \lambda_k \tau_0} C_j \gamma^{- 2k - 2 B^2 \lambda_k}  + C_k \gamma^{-2k - 2B^2 \lambda_k}\\
		&\lesssim_{p,q,k} \left| C - \log \left( \frac{A}{B} \gamma \right) \right| e^{-B^2 \lambda_k \tau_0} 
		+ \gamma^{-2k- 2B^2 \lambda_k} 
		&& (\ul{p} \in B_{\epsilon_0 e^{B^2 \lambda_k \tau_0}}) \\
		&\le  \gamma^{-2k- 2B^2 \lambda_k} + (C + C_\delta) \gamma^{-\delta} e^{-B^2 \lambda_k \tau_0} 
	\end{aligned} \end{equation*}
	for any $\delta > 0$.
	It then follows that
	\begin{equation*} \begin{aligned}
		 & \quad  \int_0^{\ul{\Upsilon} e^{- \alpha \tau_0} } \left| \check{\grave{U}}_{\ul{p}, \ul{q}} - \tl{U}_{\lambda_k} \right|^2 \gamma^{n} e^{-  \gamma^2/(4B^2)} d \gamma \\
		& \lesssim_{p,q,k}   \int_0^{\ul{\Upsilon} e^{- \alpha \tau_0} } \gamma^{-4k - 4B^2 \lambda_k } \gamma^n d \gamma  
		 + ( C + C_\delta)^2 e^{- 2 B^2 \lambda_k \tau_0} \int_0^{\ul{\Upsilon} e^{- \alpha \tau_0}} \gamma^{n - 2 \delta} d \gamma\\
		  &=  \int_0^{\ul{\Upsilon} e^{- \alpha \tau_0} } \gamma^{1 + \sqrt{ (n-9)(n-1)} } d \gamma  
		  + ( C + C_\delta)^2 e^{- 2 B^2 \lambda_k \tau_0} \int_0^{\ul{\Upsilon} e^{- \alpha \tau_0}} \gamma^{n - 2 \delta} d \gamma\\
		 &\lesssim_{p,q,k, \ul{\Upsilon}} e^{- \alpha_k \left( 2+ \sqrt{ (n-9)(n-1)} \right) \tau_0 } 
		 + ( C + C_\delta)^2 e^{\tau_0 ( - 2 B^2 \lambda_k - \alpha_k (n+1-2\delta))} \\
	\end{aligned} \end{equation*}
	It is possible to take $0 < \delta \ll 1$ sufficiently small so that
		$$- 2 B^2 \lambda_k - \alpha_k (n+1-2 \delta) <0$$
	Such a choice of $\delta$ yields the desired estimate.
	
	For $\gamma \ge M e^{\ol{\beta} \tau_0}$, a similar pointwise estimate reveals
	\begin{equation*} \begin{aligned}
		\left| \check{\grave{U}}_{\ul{p}, \ul{q}} - \tl{U}_{\lambda_k} \right| 
		&=\left|  e^{- B^2 \lambda_k \tau_0} \left[ \grave{U}_{\ul{p}, \ul{q}}( \tau_0, \gamma) - \sum_{j = 1}^{k-1} p_j \tl{U}_{\lambda_j} \right] - \tl{U}_{\lambda_k} \right| \\
		& \lesssim_{p,q,k} e^{-B^2 \lambda_k \tau_0}  \left| D + \frac{\tau_0}{2} \right| 
		+ e^{- B^2 \lambda_k \tau_0} \left| \log \left( \frac{A}{B} \gamma \right) \right|
		 + \gamma^{-2 B^2 \lambda_k } \\
	\end{aligned} \end{equation*} 
	It then follows that
	\begin{equation*} \begin{aligned}
		& \int_{M e^{\ol{\beta} \tau_0}}^\infty | \check{\grave{U}}_{\ul{p}, \ul{q}} - \tl{U}_{\lambda_k} |^2 \gamma^n e^{- \gamma^2 /(4B^2) } d \gamma \\
		\lesssim_{p,q,k}& \left| D+ \frac{\tau_0}{2}  \right|^2 e^{-2 B^2 \lambda_k \tau_0} \int_{M e^{\ol{\beta} \tau_0}}^\infty \gamma^n e^{-  \gamma^2 /(4B^2)} d \gamma \\
		&+ e^{-2B^2 \lambda_k \tau_0} \int_{M e^{\ol{\beta} \tau_0}}^\infty \left| \log \left( \frac{A}{B} \gamma \right) \right|^2 \gamma^n e^{- \gamma^2/(4B^2)} d \gamma \\
		&+ \int_{M e^{\ol{\beta} \tau_0}}^\infty \gamma^{-  \lambda_k /(4B^2)} \gamma^n e^{- \gamma^2/(4B^2)} d \gamma \\
	\end{aligned} \end{equation*}
	By applying lemma \ref{asympsOfExpInt}, each term will have a factor with doubly-exponential decay in $\tau_0$ that dominates.
	The first statement of the lemma follows.
	The second statement of the lemma then follows from orthogonality of the eigenfunctions and Cauchy-Schwartz.
\end{proof}

\begin{lem} \label{L^2EstInitZ}
	For initial data satisfying
		$$-B^2 \le \tl{Z}_{\ul{p}, \ul{q} } (\gamma, \tau_0) \le 1-B^2 \qquad \text{for all } 0 \le \gamma \le \ul{\Upsilon} e^{- \alpha_k \tau_0} \text{ and } \gamma \ge M e^{\ol{\beta} \tau_0}$$
	there exists $\epsilon = \epsilon(p,q,k) > 0$ so that
		$$\left \| \check{\grave{Z}}_{\ul{p}, \ul{q}}(\cdot, \tau_0) - \tl{Z}_{\lambda_k} \right\|_{L^2_{n-2,\frac{1}{2B^2}}} \lesssim_{p,q,k, \ul{\Upsilon}} e^{-\epsilon \tau_0} \ll 1$$
	if $\tau_0 \gg 1$ is sufficiently large depending on $p,q,k, M ,\ol{\beta}$.
	
	In particular, for all $ j$
		$$\left| \left \langle \check{\grave{Z}}_{\ul{p}, \ul{q}}(\cdot, \tau_0) - \tl{Z}_{\lambda_k}, \tl{Z}_{h_j} \right\rangle \right| \lesssim_{p,q,k, \ul{\Upsilon}} e^{- \epsilon \tau_0} \ll 1$$
\end{lem}
\begin{proof}
	Because
		$$\check{\grave{Z}}_{\ul{p}, \ul{q}}(\gamma, \tau_0) = \tl{Z}_{\lambda_k}(\gamma) \qquad 		\text{for all } \ul{\Upsilon} e^{-\alpha_k \tau_0} < \gamma < M e^{\beta \tau_0}$$
	it follows that
	\begin{equation*} \begin{aligned}
		\left \| \check{\grave{Z}}_{\ul{p}, \ul{q}}(\cdot, \tau_0) - \tl{Z}_{\lambda_k} \right\|_{L^2_{n-2, \frac{1}{2B^2} }}^2 
		\le& \int_0^{\ul{\Upsilon} e^{-\alpha_k \tau_0} } \left| \check{\grave{Z}}_{\ul{p}, \ul{q}} - \tl{Z}_{\lambda_k} \right|^2 \gamma^{n-2} e^{- \gamma^2/(4B^2)} d \gamma\\
		& + \int_{M e^{\ol{\beta} \tau_0}}^\infty \left| \check{\grave{Z}}_{\ul{p}, \ul{q}} - \tl{Z}_{\lambda_k} \right|^2 \gamma^{n-2} e^{- \gamma^2/(4B^2)} d \gamma
	\end{aligned} \end{equation*}
		
	In the inner region,
	\begin{equation*} \begin{aligned}
		& \quad \left| \check{\grave{Z}}_{\ul{p}, \ul{q}}(\gamma, \tau_0) - \tl{Z}_{\lambda_k} \right|\\
		&= \left| e^{-B^2 \lambda_k \tau_0} \left[ \grave{Z}_{\ul{p}, \ul{q}}(\gamma, \tau_0) - \sum_{j=0}^K q_j \tl{Z}_{h_j}(\gamma) \right] - \tl{Z}_{\lambda_k} \right| \\
		&\le e^{-B^2 \lambda_k \tau_0} \max( B^2, 1-B^2) \\
		& \quad + \sum_{j=0}^K |q_j| D_j \gamma^{-2k - 2B^2 \lambda_k } e^{-B^2 \lambda_k \tau_0} + D_k \gamma^{-2k - 2B^2 \lambda_k } \\
		&\lesssim_{p,q,k} e^{-B^2 \lambda_k \tau_0} + \gamma^{-2k - 2B^2 \lambda_k}
		&& ( \ul{q} \in B_{\epsilon_0 e^{B^2 \lambda \tau_0} }) \\
	\end{aligned} \end{equation*}
	
	It then follows that
	\begin{equation*} \begin{aligned}
		& \quad  \int_0^{\ul{\Upsilon} e^{- \alpha_k \tau_0}} \left| \check{\grave{Z}}_{\ul{p}, \ul{q}} - \tl{Z}_{\lambda_k} \right|^2 \gamma^{n-2} e^{- \gamma^2/(4B^2)} d \gamma \\
		&\lesssim_{p,q,k} e^{-2B^2 \lambda_k \tau_0} \int_0^{\ul{\Upsilon} e^{- \alpha_k \tau_0}} \gamma^{n-2}  d \gamma 
		+  \int_0^{\ul{\Upsilon} e^{- \alpha_k \tau_0}} \gamma^{-4k - 4B^2 \lambda_k} \gamma^{n-2}  d \gamma \\
		&= e^{-2B^2 \lambda_k \tau_0} \int_0^{\ul{\Upsilon} e^{- \alpha_k \tau_0}} \gamma^{n-2}  d \gamma 
		+  \int_0^{\ul{\Upsilon} e^{- \alpha_k \tau_0}} \gamma^{-1 + \sqrt{(n-9)(n-1)} } d \gamma \\
		&\lesssim_{p,q,k, \ul{\Upsilon}} e^{-2B^2 \lambda_k \tau_0 - (n-1)\alpha_k \tau_0} + e^{- \sqrt{ (n-9)(n-1)} \alpha_k \tau_0} \\ 
		&\ll  1
	\end{aligned} \end{equation*}
	
	In the outer region, a similar pointwise estimate reveals
		$$\left| \check{\grave{Z}}_{\ul{p}, \ul{q}}(\gamma, \tau_0) - \tl{Z}_{\lambda_k} \right| \lesssim_{p,q,k} e^{-B^2 \lambda_k \tau_0} + \gamma^{-2B^2 \lambda_k -2}$$
	We apply this pointwise estimate to obtain integral estimates
	\begin{equation*} \begin{aligned}
		& \quad \int_{M e^{\ol{\beta} \tau_0}}^\infty \left| \check{\grave{Z}}_{\ul{p}, \ul{q}}(\gamma, \tau_0) - \tl{Z}_{\lambda_k} \right|^2 \gamma^{n-2} e^{- \gamma^2 /(4B^2)} d \gamma \\
		&\lesssim_{p,q,k} e^{-2 B^2 \lambda_k \tau_0} \int_{M e^{\ol{\beta} \tau_0}}^\infty  \gamma^{n-2} e^{- \frac{\gamma^2}{4B^2}} d \gamma
		+ \int_{M e^{\ol{\beta} \tau_0}}^\infty \gamma^{-4B^2 \lambda_k -4} \gamma^{n-2} e^{- \frac{\gamma^2}{4B^2}} d \gamma
	\end{aligned} \end{equation*}
	By applying lemma \ref{asympsOfExpInt}, each term will have a factor wtih doubly-exponential decay in $\tau_0$ that dominates. The statement of the lemma follows.
\end{proof}	

\subsection{Proof of Theorem \ref{smallCoeffs}}
\begin{lem} \label{smallCoeffsU}
	Let $\tau_1 > \tau_0$.
	If ${P}_{\tau_0, \tau_1} ( \ul{p}, \ul{q}) = 0$
	and there exists $\epsilon >0$ such that
	\begin{equation*} \begin{aligned}
		\left| \langle \check{U}_{\ul{p}} ( \tau_0), \tl{U}_{\lambda_j} \rangle_{L^2_{n,\frac{1}{2B^2}}} \right| 
		&\lesssim e^{- \epsilon \tau_0} 
		&& \text{for all } j < k \\
		\text{and } \| Err_U [ \tl{Z}, \tl{U} ] ( \tau, \cdot) \|_{H^{-1}_{n,\frac{1}{2B^2}}} 
		&\lesssim e^{B^2 \lambda_k \tau} e^{- \epsilon \tau} 
		&& \text{for all } \tau \in [ \tau_0, \tau_1]
	\end{aligned} \end{equation*}
	then
		$$| \ul{p} | \lesssim_{p,q,k} e^{B^2 \lambda_k \tau_0} e^{- \epsilon \tau_0} \ll e^{B^2 \lambda_k \tau_0} \ll 1$$
\end{lem}
\begin{proof}
	Recall that
		$$\partial_\tau \grave{U} = B^2 \mathcal{D}_U \grave{U} + \grave{\chi} Err_U [ \tl{Z}, \tl{U} ] + B^2 [ \grave{\chi}, \mathcal{D}_U ] \tl{U} + \grave{\chi}_\tau \tl{U}$$
	so the variation of constants formula implies that
		\begin{equation*} \begin{aligned} 
			\grave{U}_{\ul{p}, \ul{q}}( \tau_1) 
			=& e^{(\tau_1 - \tau_0) B^2 \mathcal{D}_U} \grave{U}_{\ul{p}, \ul{q}}(\tau_0) + \int_{\tau_0}^{\tau_1} e^{(\tau_1 - \tau) B^2 \mathcal{D}_U} \grave{\chi} Err_U [ \tl{Z}, \tl{U} ] (\tau) d \tau \\
			& + \int_{\tau_0}^{\tau_1} e^{B^2 \mathcal{D}_U (\tau_1 - \tau)} 
			\left(		B^2 [ \grave{\chi}, \mathcal{D}_U ] \tl{U} + \grave{\chi}_\tau \tl{U} \right)(\tau)	d\tau
		\end{aligned} \end{equation*}
	Letting $j < k$ and taking the $L^2_{n,\frac{1}{2B^2}}$ inner product of both sides with $\tl{U}_{\lambda_j}$, it follows that
	\begin{equation*} \begin{aligned}
		0 &= \langle \grave{U}_{\ul{p}, \ul{q}}(\tau_1), \tl{U}_{\lambda_j} \rangle \\ 
		&= e^{(\tau_1 - \tau_0) B^2 \lambda_j} p_j
		+ e^{( \tau_1 - \tau_0) B^2 \lambda_j + B^2 \lambda_k \tau_0} \langle \check{\grave{U}}_{\ul{p}, \ul{q}}(\tau_0), \tl{U}_{\lambda_j} \rangle  \\
		& \quad +\int_{\tau_0}^{\tau_1} e^{(\tau_1 - \tau) B^2 \lambda_j } \langle \grave{\chi} Err_U [ \tl{Z}, \tl{U} ], \tl{U}_{\lambda_j} \rangle d \tau \\
		& \quad +\int_{\tau_0}^{\tau_1} e^{(\tau_1 - \tau) B^2 \lambda_j } \langle B^2 [ \grave{\chi}, \mathcal{D}_U ] \tl{U} + \grave{\chi}_\tau \tl{U} , \tl{U}_{\lambda_j} \rangle d\tau\\
		\implies | p_j| &\le e^{B^2 \lambda_k \tau_0} \left|  \langle \check{\grave{U}}_{\ul{p}, \ul{q}}(\tau_0), \tl{U}_{\lambda_j} \rangle \right| \\
		&\quad + \int_{\tau_0}^{\tau_1} e^{(\tau_0 - \tau) B^2 \lambda_j } \left| \langle  \grave{\chi} Err_U [ \tl{Z}, \tl{U} ], \tl{U}_{\lambda_j} \rangle \right| d \tau \\
		&\quad + \int_{\tau_0}^{\tau_1} e^{(\tau_0 - \tau) B^2 \lambda_j } \left|  \langle B^2 [ \grave{\chi}, \mathcal{D}_U ] \tl{U} + \grave{\chi}_\tau \tl{U} , \tl{U}_{\lambda_j} \rangle \right| d \tau\\
		&\le e^{B^2 \lambda_k \tau_0} \left|  \langle \check{\grave{U}}_{\ul{p}, \ul{q}}(\tau_0), \tl{U}_{\lambda_j} \rangle \right| \\
		& \quad + \left\|  \tl{U}_{\lambda_j} \right\|_{H^1_{n,\frac{1}{2B^2}}} \int_{\tau_0}^{\tau_1} e^{(\tau_0 - \tau) B^2 \lambda_j } 	\left\|   \grave{\chi} Err_U [ \tl{Z}, \tl{U} ]( \tau) \right\|_{H^{-1}_{n,\frac{1}{2B^2}}}  d \tau \\
		& \quad + \left\|  \tl{U}_{\lambda_j} \right\|_{H^1_{n,\frac{1}{2B^2}}} \int_{\tau_0}^{\tau_1} e^{(\tau_0 - \tau) B^2 \lambda_j }   \left\| B^2 [ \grave{\chi}, \mathcal{D}_U ] \tl{U}(\tau) + \grave{\chi}_\tau \tl{U}(\tau) \right\|_{H^{-1}_{n,\frac{1}{2B^2}}} d \tau \\
		&\lesssim_{p,q,k}  e^{B^2 \lambda_k \tau_0} e^{- \epsilon \tau_0} 
		+ \| \tl{U}_{\lambda_j} \|_{H^1_{n,\frac{1}{2B^2}}} \int_{\tau_0}^{\tau_1} e^{(\tau_0 - \tau) B^2 \lambda_j } e^{B^2 \lambda_k \tau} e^{- \epsilon \tau}  d \tau  \\ 
		&\lesssim_{p,q,k}  e^{B^2 \lambda_k \tau_0} e^{- \epsilon \tau_0}\\
	\end{aligned} \end{equation*}
\end{proof}

\begin{lem} \label{smallCoeffsZ}
	Let $\tau_1 > \tau_0$.
	If ${P}_{\tau_0, \tau_1}( \ul{p}, \ul{q} ) = 0$ and there exists $\epsilon > 0$ such that
	\begin{equation*} \begin{aligned}
		\left| \left \langle \check{\grave{Z}}_{\ul{p}, \ul{q}} ( \tau_0) - \tl{Z}_{\lambda_k}, \tl{Z}_{h_j} \right \rangle_{L^2_{n-2,\frac{1}{2B^2}}} \right| & \lesssim e^{-\epsilon \tau_0} 
		&&  \text{for all } j \le K  \\
		\| Err_Z [ \tl{Z}, \tl{U} ](\tau) \|_{H^{-1}_{n-2,b}} & \lesssim e^{B^2 \lambda_k \tau} e^{- \epsilon \tau} 		
		&& \text{for all } \tau \in [ \tau_0, \tau_1] \\
		\| [\grave{\chi}, B^2 \mathcal{D}_Z ] \tl{Z} + \grave{\chi}_\tau \tl{Z}   \|_{H^{-1}_{n-2,\frac{1}{2B^2}}}
		& \lesssim e^{B^2 \lambda_k \tau} e^{- \epsilon \tau} 		
		&& \text{for all } \tau \in [ \tau_0, \tau_1] \\
		\| \grave{\chi} \cl{N} \tl{U}_{\ul{p}, \ul{q}} ( \tau) - \cl{N} e^{B^2 \lambda_k \tau} \tl{U}_{\lambda_k} \|_{H^{-1}_{n-2,\frac{1}{2B^2}}} & \lesssim e^{B^2 \lambda_k \tau} ( e^{- \epsilon \tau} + \eta_1^U ) 		
		&&  \text{for all } \tau \in [ \tau_0, \tau_1] \\
	\end{aligned} \end{equation*}
	then
		$$| \ul{q} | \lesssim_{p,q,k} e^{B^2 \lambda_k \tau_0} ( e^{- \epsilon \tau_0} + \eta_1^U)$$
\end{lem}
\begin{proof}
	The variation of constants formula implies that 
	\begin{equation*} \begin{aligned}
		\grave{Z}(\tau_1) - e^{B^2 \lambda_k \tau_1} \tl{Z}_{\lambda_k} 
		=& e^{B^2 \cl{D}_Z (\tau_1 - \tau_0)} \left( \grave{Z}(\tau_0) - \tl{Z}_{\lambda_k} e^{B^2 \lambda_k \tau_0 }\right) \\
		&+ \int_{\tau_0}^{\tau_1} e^{B^2 \cl{D}_Z ( \tau_1 - {\tau} )} \left( \grave{\chi} \cl{N} \tl{U}_{\ul{p}, \ul{q}} ({\tau}) - B^2 \cl{N} \tl{U}_{\lambda_k} e^{B^2 \lambda_k {\tau} } \right) d {\tau} \\
		&+ \int_{\tau_0}^{\tau_1} e^{B^2 \cl{D}_Z ( \tau_1 - {\tau} )} \left( \grave{\chi} Err_Z [ \tl{Z}, \tl{U} ]({\tau})   \right) d {\tau} \\
		&+ \int_{\tau_0}^{\tau_1} e^{B^2 \cl{D}_Z ( \tau_1 - {\tau} )} \left( [\grave{\chi}, B^2 \mathcal{D}_Z ] \tl{Z} + \grave{\chi}_\tau \tl{Z}  \right) d {\tau} \\
	\end{aligned} \end{equation*}
	By taking the $L^2_{n-2,b}$ inner product with $\tl{Z}_{h_j}$ for $j \le K$ and using $\cl{P}_{\tau_0, \tau_1} ( \ul{p}, \ul{q} ) = 0$, it follows that
	\begin{equation*} \begin{aligned}
		0  &= \left \langle 	\grave{Z}(\tau_1) - e^{B^2 \lambda_k \tau_1} \tl{Z}_{\lambda_k}, \tl{Z}_{h_j} \right \rangle_{L^2_{n-2,\frac{1}{2B^2}}} \\
		 &= e^{B^2 h_j (\tau_1 - \tau_0)} \left \langle \left( \grave{Z}(\tau_0) - \tl{Z}_{\lambda_k} e^{B^2 \lambda_k \tau_0 }\right) , \tl{Z}_{h_j} \right \rangle \\
		 & \quad + \int_{\tau_0}^{\tau_1} e^{B^2 h_j (\tau_1 - {\tau} ) }
		 \left \langle \left( \grave{\chi} \cl{N} \tl{U}_{\ul{p}, \ul{q}} ({\tau}) - B^2 \cl{N} \tl{U}_{\lambda_k} e^{B^2 \lambda_k {\tau} } \right), \tl{Z}_{h_j} \right \rangle d {\tau}\\
		  & \quad + \int_{\tau_0}^{\tau_1} e^{B^2 h_j (\tau_1 - {\tau} ) }
		 \left \langle \left( \grave{\chi} Err_Z [ \tl{Z}, \tl{U} ]({\tau})   \right) , \tl{Z}_{h_j} \right \rangle d {\tau}\\
		  & \quad + \int_{\tau_0}^{\tau_1} e^{B^2 h_j (\tau_1 - {\tau} ) }
		 \left \langle \left( [\grave{\chi}, B^2 \mathcal{D}_Z ] \tl{Z} + \grave{\chi}_\tau \tl{Z}  \right), \tl{Z}_{h_j} \right \rangle d {\tau}\\
		 \implies |q_j| 
		 &\le e^{B^2 \lambda_k \tau_0} \left| \left \langle \check{\grave{Z}}(\tau_0) - \tl{Z}_{\lambda_k}, \tl{Z}_{h_j} \right \rangle \right| \\
		 & \quad  + \| \tl{Z}_{h_j} \|_{H^1_{n-2,\frac{1}{2B^2}}} 
		 \int_{\tau_0}^{\tau_1} e^{B^2 h_j ( \tau_0 - {\tau} ) }
		 \| \grave{\chi} \cl{N} \tl{U}_{\ul{p}, \ul{q}} ({\tau}) - B^2 \cl{N} \tl{U}_{\lambda_k} e^{B^2 \lambda_k {\tau} }  \|_{H^{-1}_{n-2,\frac{1}{2B^2}}} d {\tau} \\
		 & \quad + \| \tl{Z}_{h_j} \|_{H^1_{n-2,\frac{1}{2B^2}}} 
		 \int_{\tau_0}^{\tau_1} e^{B^2 h_j ( \tau_0 - {\tau} ) }
		 \| \grave{\chi} Err_Z [ \tl{Z}, \tl{U} ]({\tau})   \|_{H^{-1}_{n-2,\frac{1}{2B^2}}} d {\tau} \\
		 & \quad + \| \tl{Z}_{h_j} \|_{H^1_{n-2,\frac{1}{2B^2}}}
		  \int_{\tau_0}^{\tau_1} e^{B^2 h_j ( \tau_0 - {\tau} ) }
		 \| [\grave{\chi}, B^2 \mathcal{D}_Z ] \tl{Z} + \grave{\chi}_\tau \tl{Z} \|_{H^{-1}_{n-2,\frac{1}{2B^2}}} d {\tau} \\
		 &\lesssim_{p,q,k}  e^{B^2 \lambda_k \tau_0 } e^{- \epsilon \tau_0} 
		 + \int_{\tau_0}^{\tau_1} e^{B^2 h_j ( \tau_0 - {\tau} ) } e^{B^2 \lambda_k {\tau} } ( e^{- \epsilon {\tau} } + \eta_1^U ) d {\tau} \\
		& \lesssim_{p,q,k}  e^{B^2 \lambda_k \tau_0} ( e^{- \epsilon \tau_0} + \eta_1^U ) 
	\end{aligned} \end{equation*}	
\end{proof}

\begin{proof} (of theorem \ref{smallCoeffs})
	By the construction of $\tl{Z}_{\ul{p}, \ul{q}}, \tl{U}_{\ul{p}, \ul{q}}$, the pointwise estimates of theorem \ref{ptwiseEstsZU} and \ref{ptwiseEstsErr} apply if $\ul{\Upsilon}, \Upsilon_U, \Upsilon_Z, \tau_0 \gg 1$ are sufficiently large.
	These pointwise bounds and initial data bounds imply the $H^{-1}$ estimates from lemmas \ref{parabolicH^-1EstU} to \ref{L^2EstInitZ}.
	Hence, lemmas \ref{smallCoeffsU} and \ref{smallCoeffsZ} apply and conclude the proof of the theorem.
\end{proof}

%% file: APrioriShortTimeEsts2.tex

Now, we will proceed to establish pointwise estimates on the ``parabolic region" for the evolution of the terms that are not included in the linearization.
We begin by estimating on ``short-time" regimes $\tau_0 < \tau \le \tau_1 \le \tau_0 + 1$.
For the remainder of the paper, we will use $``\mathcal{P}"$ as shorthand for $\mathcal{P}[ \ul{D}, l, \kappa, \ul{\eta}, \tau_0, \tau_1, \Upsilon_U, \Upsilon_Z, M, \alpha, \beta, \lambda_k]$.

\subsection{Estimates for Initial Data Contributions}
	\begin{lem}
		Let $( \tl{Z}_{\ul{p}, \ul{q}}, \tl{U}_{\ul{p}, \ul{q}})$ be as in \ref{assignmentOfInitData}.
		Assume that $\ul{\Upsilon}, \Upsilon_U, \Upsilon_Z, \tau_0 \gg 1$ are sufficiently large so that the conclusions of theorems \ref{ptwiseEstsZU} and \ref{smallCoeffs} hold.
		For any $\nu \in (0,1)$, there exists 
		$\Upsilon_U \gg 1$ sufficiently large depending on $p,q,k, l, \ul{\Upsilon}, \ul{D}$ and 
		$\tau_0 \gg 1$ sufficiently large  
		such that
			$$\left| e^{B^2 \cl{D}_U (\tau - \tau_0)} \left( \grave{U}_{\ul{p}, \ul{q}}(\tau_0) - e^{B^2 \lambda_k \tau_0} \tl{U}_{\lambda_k} \right)(\gamma) \right| \le \nu e^{B^2 \lambda_k \tau} \left( \gamma^{-2k - 2B^2 \lambda_k} + \gamma^{2B^2 \lambda_k} \right)$$
			$$\text{for all } \quad  \gamma \in [ \Upsilon_U e^{-\alpha_k \tau}, M e^{\beta \tau} ], \quad \tau_0 \le \tau  \le \tau_0+1$$
	\end{lem}
	\begin{proof}
		We rewrite
			$$\grave{U}_{\ul{p}, \ul{q}}(\tau_0) - e^{B^2 \lambda_k \tau_0} \tl{U}_{\lambda_k} = \sum_{j = 0}^{k-1} p_j \tl{U}_{\lambda_j} + e^{B^2 \lambda_k \tau_0} \left( \check{\grave{U}}_{\ul{p}, \ul{q}} (\tau_0) - \tl{U}_{\lambda_k} \right) $$
		and estimate the action of the semigroup on each term.
		\begin{equation*} \begin{aligned}
			& \quad \left| e^{B^2 \cl{D}_U (\tau - \tau_0)} \sum p_j \tl{U}_{\lambda_j}  \right| \\
			&\le \sum | p_j| e^{B^2 \lambda_j (\tau - \tau_0)} | \tl{U}_{\lambda_j} | \\
			& \lesssim_{p,q,k} \left( \gamma^{-2 k -2 B^2  \lambda_k} + \gamma^{ - 2B^2 \lambda_k } \right) \sum | p_j | 
			&& (\ref{eigfuncUAsymps})\\
			&\lesssim_{p,q,k, \ul{D}, \Upsilon_U, \Upsilon_Z, \ul{\Upsilon}}
			 e^{B^2 \lambda_k \tau} e^{- \epsilon \tau_0}  \left( \gamma^{-2 k -2 B^2 \lambda_k } + \gamma^{ - 2B^2 \lambda_k } \right) 
			 && (\text{theorem } \ref{smallCoeffs})\\
			&\le  \frac{1}{2}\nu e^{B^2 \lambda_k \tau} \left( \gamma^{-2 k - 2B^2 \lambda_k } + \gamma^{ - 2B^2 \lambda_k } \right)
		\end{aligned} \end{equation*}
		if $\tau_0 \gg 1$ is sufficiently large. 
		
		For the other term, we'll use lemma \ref{maximalFuncBoundU}.
		First note that $\tl{Z}, \tl{U} \in \cl{P}$ imply
			\begin{equation*}
				\left | \check{\grave{U}}_{\ul{p}, \ul{q}}(\tau_0) - \tl{U}_{\lambda_k} \right|(\gamma) \lesssim_{p,q,k, \ul{D}} \left\{
					\begin{array}{ll}
						\Upsilon_U^l \gamma^{-2k - 2B^2 \lambda_k}, \quad & \text{if } \gamma < \ul{\Upsilon} e^{- \alpha_k \tau_0} \\
						0, & \text{if } \ul{\Upsilon} e^{- \alpha_k \tau_0} \le \gamma \le M e^{\ol{\beta} \tau_0} \\
						\gamma^{-2 B^2 \lambda_k},  & \text{if } \gamma > M e^{\ol{\beta} \tau_0} \\
					\end{array} \right.
			\end{equation*}
		Now, let
			$$\chi_{inner}(\gamma) \doteqdot \chi_{(0, \ul{\Upsilon} e^{- \alpha_k \tau_0} ]}(\gamma)$$
		For $\cl{M}_U$ defined in equation \ref{maximalFuncU}, lemma \ref{maximalFuncBoundU} implies
			\begin{equation*} \begin{aligned}
				& \quad \cl{M}_U ( \chi_{inner} \gamma^{-2k - 2B^2 \lambda_k} ) (\gamma) \\
				&= \frac{ \int_0^\gamma \chi_{inner}(\ol{\gamma}) \ol{\gamma}^{1 + \sqrt{(n-9)(n-1)}} e^{- \ol{\gamma}^2/(4B^2)} d \ol{\gamma} }{\int_0^\gamma \ol{\gamma}^{1 + \sqrt{(n-9)(n-1)}} e^{- \ol{\gamma}^2/(4B^2)} d \ol{\gamma} }\\
				 &\le  \frac{ \int_0^{ \ul{\Upsilon} e^{- \alpha_k \tau_0} }  \ol{\gamma}^{1 + \sqrt{(n-9)(n-1)}} e^{- \ol{\gamma}^2/(4B^2)} d \ol{\gamma} }{\int_0^\gamma \ol{\gamma}^{1 + \sqrt{(n-9)(n-1)}} e^{- \ol{\gamma}^2/(4B^2)} d \ol{\gamma} }\\
				& \lesssim_{p,q}  \frac{ ( \ul{\Upsilon} e^{- \alpha_k \tau_0} )^{2 + \sqrt{(n-9)(n-1) } } }{\int_0^\gamma \ol{\gamma}^{1 + \sqrt{(n-9)(n-1)}} e^{- \ol{\gamma}^2/(4B^2)} d \ol{\gamma} }\\
				& \lesssim_{p,q}  \left( \frac{ \ul{\Upsilon} e^{- \alpha_k \tau_0}} { \Upsilon e^{- \alpha_k \tau} } \right)^{2 + \sqrt{ (n-9)(n-1) } } 
				&& \left(\gamma > \Upsilon_U e^{- \alpha_k \tau} \right)\\
				&\lesssim_{p,q,k}  \left( \frac{ \ul{\Upsilon} }{ \Upsilon_U} \right)^{2 + \sqrt{ (n-9)(n-1) } } 
			\end{aligned} \end{equation*}
		It follows that
			\begin{equation*} \begin{aligned}				
				& \quad \left| e^{B^2 D_U (\tau - \tau_0)} e^{B^2 \lambda_k \tau_0} \left( \check{\grave{U}}_{\ul{p}, \ul{q}}(\tau_0) - \tl{U}_{\lambda_k} \right) \chi_{inner} \right| \\
				&\lesssim_{p,q,k, \ul{D}}  e^{B^2 \lambda_k \tau_0} \Upsilon_U^l \left( \gamma e^{- \frac{1}{2} (\tau - \tau_0)} \right)^{-2k- 2B^2 \lambda_k  } \cl{M}_U (\chi_{inner} \gamma^{-2k- 2B^2 \lambda_k} ) \\
				&\lesssim_{p,q,k} e^{B^2 \lambda_k \tau_0} \gamma^{-2k - 2B^2 \lambda_k} \left( \frac{ \ul{\Upsilon} }{ \Upsilon_U} \right)^{2 + \sqrt{ (n-9)(n-1) } } \Upsilon_U^l\\
			\end{aligned} \end{equation*}
		using $0 \le \tau - \tau_0 \le 1$.
		Because
			$$l < \frac{1}{2} ( 2k + 2B^2 \lambda_k) = \frac{n-1}{4} - \frac{1}{4} \sqrt{(n-9)(n-1)} \le 2 + \sqrt{ (n-9)(n-1)},$$
		taking $\Upsilon_U \gg 1$ sufficiently large depending on $p,q,k, \ul{D}, l , \ul{\Upsilon}$ yields the desired estimate.
		
		Finally, let
			$$\chi_{outer}(\gamma) \doteqdot \chi_{[M e^{\ol{\beta} \tau_0}, \infty)}(\gamma)$$
		and observe
			$$\gamma^{2k + 2B^2 \lambda_k} \gamma^{-2B^2 \lambda_k} \chi_{outer}= \chi_{outer} \gamma^{2k}$$
		is nondecreasing.
		Hence, lemma \ref{maximalFuncBoundU} implies
			\begin{equation*} \begin{aligned}
				\cl{M}_U( \chi_{outer} \gamma^{-2 B^2 \lambda_k})(\gamma) 
				&= \frac{ \int_\gamma^\infty \chi_{outer} \ol{\gamma}^{2k} \ol{\gamma}^{1 + \sqrt{ (n-9)(n-1) } } e^{- \ol{\gamma}^2/(4B^2)} d \ol{\gamma} }{ \int_\gamma^\infty \ol{\gamma}^{1 + \sqrt{ (n-9)(n-1)} } e^{ -  \ol{\gamma}^2/(4B^2)} d \ol{\gamma} } \\
				&\le  \frac{ \int_{ M e^{\ol{\beta} \tau_0}}^\infty \ol{\gamma}^{2k} \ol{\gamma}^{1 + \sqrt{ (n-9)(n-1)} } e^{-  \ol{\gamma}^2/(4B^2)} d \ol{\gamma} }{ \int_\gamma^\infty \ol{\gamma}^{1 + \sqrt{ (n-9)(n-1) } } e^{ -  \ol{\gamma}^2/(4B^2)} d \ol{\gamma} } \\
				& \lesssim_{p,q,k}  \frac{ ( M e^{\ol{\beta} \tau_0} )^{2k + \sqrt{ (n-9)(n-1) } } e^{- \frac{1}{4B^2} (  M e^{\ol{\beta} \tau_0})^2} }{ \int_\gamma^\infty \ol{\gamma}^{1 + \sqrt{ (n-9)(n-1) } } e^{ -  \ol{\gamma}^2/(4B^2)} d \ol{\gamma} }  && (\ref{asympsOfExpInt}) \\
			\end{aligned} \end{equation*}
			It follows that for $\gamma \le M e^{\beta \tau}$
			\begin{equation*} \begin{aligned}
				& \quad |\gamma^{-2k - 2B^2 \lambda_k} \cl{M}_U (\chi_{outer} \gamma^{-2B^2 \lambda_k} )| \\
				&\lesssim_{p,q,k} \gamma^{-2k - 2B^2 \lambda_k} \frac{ ( M e^{\ol{\beta} \tau_0} )^{2k + \sqrt{ (n-9)(n-1) } } e^{- \frac{1}{4B^2} (  M e^{\ol{\beta} \tau_0})^2} }{ \int_\gamma^\infty \ol{\gamma}^{1 + \sqrt{ (n-9)(n-1) } } e^{ -  \ol{\gamma}^2/(4B^2)} d \ol{\gamma} } \\
				& \lesssim_{p,q,k}
				( M e^{\ol{\beta} \tau_0} )^{2k + \sqrt{ (n-9)(n-1) } } e^{- \frac{1}{4B^2} (  M e^{\ol{\beta} \tau_0})^2} \\
				& \quad \cdot \left( \gamma^{-2k -2B^2 \lambda} + \gamma^{-2B^2 \lambda} 
				\left( M e^{\beta \tau} \right)^{-2k - \sqrt{(n-9)(n-1)} } e^{ \frac{1}{4B^2} ( M e^{ \beta \tau} )^2 } \right)			
				&& (\ref{asympsOfExpInt})\\
				&\lesssim_{p,q,k}
				( M e^{\ol{\beta} \tau_0} )^{2k + \sqrt{ (n-9)(n-1) } } e^{- \frac{1}{4B^2} (  M e^{\ol{\beta} \tau_0})^2} \gamma^{-2k -2B^2 \lambda_k} \\
				& \quad +
				 \left( e^{(\ol{\beta} - \beta) \tau_0} \right)^{2k + \sqrt{(n-9)(n-1)}} 
				 e^{ - \frac{M^2}{4B^2} \left( e^{ 2 \ol{\beta} \tau_0} - C e^{2 \beta \tau_0} \right) }
				 \gamma^{-2B^2 \lambda}
			\end{aligned} \end{equation*}
		Therefore, we get the desired bounds if $\tau_0 \gg 1$ is sufficiently large.
	\end{proof}
	
	\begin{lem}
		Let $( \tl{Z}_{\ul{p}, \ul{q}}, \tl{U}_{\ul{p}, \ul{q}})$ be as in \ref{assignmentOfInitData}.
		Assume that $\ul{\Upsilon}, \Upsilon_U, \Upsilon_Z, \tau_0 \gg 1$ are sufficiently large so that the conclusions of theorems \ref{ptwiseEstsZU} and \ref{smallCoeffs} hold.
		For any $\Gamma > 0$ and $\nu \in (0,1)$,
		there exists 
		$\Upsilon_Z \gg 1$ sufficiently large depending on $p,q,k, l, \ul{D}, \ul{\Upsilon}, \Upsilon_U$, and
		$\tau_0 \gg 1$ sufficiently large
		such that
		\begin{gather*}
			\left| e^{B^2 \cl{D}_Z (\tau_1 - \tau_0)} \left( \grave{Z}_{\ul{p}, \ul{q}}(\tau_0)(\gamma) - e^{B^2 \lambda_k \tau_0} \tl{Z}_{\lambda_k} \right) \right| \le \nu e^{B^2 \lambda_k \tau} \left( \gamma^{-2k - 2B^2 \lambda_k} + \gamma^{-2B^2 \lambda_k} \right) \\
			\le ( C_{p,q,k} \eta_1^U + \nu) e^{B^2 \lambda_k \tau} \left( \gamma^{-2k - 2B^2 \lambda_k} + \gamma^{-2B^2 \lambda_k} \right)
		\end{gather*}
		for all $\gamma \in [ \Upsilon_Z e^{-\alpha_k \tau}, \Gamma ], \quad \tau_0 \le \tau  \le \tau_0+1$.
	\end{lem}
	\begin{proof}
		We rewrite
			$$\grave{Z}_{\ul{p}, \ul{q}}(\tau_0) - e^{B^2 \lambda_k \tau_0} \tl{Z}_{\lambda_k} = 	\sum_{j=0}^K q_j \tl{Z}_{h_j} + e^{B^2 \lambda_k \tau_0} \left( \check{ \grave{Z} }_{\ul{p}, \ul{q}}(\tau_0) - \tl{Z}_{\lambda_k} \right)$$
		and estimate the action of the semigroup on each term.
		
		One part of the estimate reads
		\begin{equation*} \begin{aligned}
			\left| e^{B^2 D_Z (\tau_1 - \tau_0)} \sum_j q_j \tl{Z}_{h_j} \right| 
			&= \sum_j |q_j| e^{B^2 h_j (\tau - \tau_0) } \left| \tl{Z}_{h_j} \right| \\
			&\lesssim_{p,q,k} \sum_j | q_j | e^{B^2 h_j (\tau - \tau_0)} ( \gamma^2 + \gamma^{-2B^2 h_j} ) 
			&& ( \ref{eigfuncZAsymps} )\\
			&\lesssim  e^{B^2 \lambda_k \tau_0}  ( \gamma^2 + \gamma^{-2B^2 h_j} ) \\
			& \quad \cdot \left( C_{p,q,k, \ul{D}, \Upsilon_U, \Upsilon_Z, \ul{\Upsilon}} e^{- \epsilon \tau_0} 
			+ C_{p,q,k} \eta_1^U \right)
			&& (\text{theorem \ref{smallCoeffs}} )\\
			&\lesssim_{p,q,k} \left( C_{p,q,k, \ul{D}, \Upsilon_U, \Upsilon_Z, \ul{\Upsilon}} e^{- \epsilon \tau_0} 
			+ C_{p,q,k} \eta_1^U \right) \\
			& \quad \cdot e^{B^2 \lambda_k \tau }  ( \gamma^2 + \gamma^{ -2B^2 \lambda_k}) &&(j \le K)\\
		\end{aligned} \end{equation*}

		From $\tl{Z}, \tl{U} \in \cl{P}$, we have the pointwise estimates
		\begin{equation*}
			|\check{ \grave{Z}}(\tau_0) - \tl{Z}_{\lambda_k}| \lesssim_{p,q,k, \ul{D}} \left\{
				\begin{array}{ll}
					\Upsilon_U^l \gamma^{-2k - 2B^2 \lambda_k}, \qquad & \gamma < \ul{\Upsilon} e^{ - \alpha_k \tau_0} \\
					0 , &  \ul{\Upsilon} e^{ - \alpha_k \tau_0} \le \gamma \le  M e^{ \ol{\beta} \tau_0}\\
					\gamma^{-2B^2 \lambda_k }, & \gamma > M e^{ \ol{\beta} \tau_0} \\
				\end{array} \right.
		\end{equation*}

		We use a maximal function estimate to bound these contributions.
		First, note that
		\begin{gather*}
			\chi_{inner} \gamma^{-2k - 2B^2 \lambda_k} \gamma^{-2} \text{ is decreasing,} 
			\text{ where } \chi_{inner} (\gamma) \doteqdot \chi_{( 0 , \ul{\Upsilon} e^{- \alpha \tau_0} ]} (\gamma)
		\end{gather*}
		Hence, for $\Upsilon_Z e^{ - \alpha_k \tau} \le \gamma \le M e^{ \beta \tau}$ 
		and $\tau_0 \le \tau \le \tau_0 + 1$
		\begin{equation*} \begin{aligned}
			& \quad \left| e^{B^2 \cl{D}_Z (\tau - \tau_0)} e^{B^2 \lambda_k \tau_0} \Upsilon_U^l \chi_{inner} \gamma^{-2k- 2B^2 \lambda_k} \right| \\
			& \lesssim_{p,q}  e^{B^2 \lambda_k \tau_0} \Upsilon_U^l \gamma^2 e^{ - \frac{1}{2} (\tau - \tau_0)} 
			 \frac{ \int_0^\gamma \chi_{inner} \ol{\gamma}^{-2k- 2B^2 \lambda_k } \ol{\gamma}^n e^{ -  \ol{\gamma}^2/(4B^2)} d\ol{\gamma} }{ \int_0^\gamma \ol{\gamma}^{n+2} e^{ -  \ol{\gamma}^2/(4B^2)} d\ol{\gamma}}
			 && (\ref{maximalFuncBoundZ})\\
			&\lesssim  \frac{ \int_0^\gamma \chi_{inner} \ol{\gamma}^{-2k- 2B^2 \lambda_k } \ol{\gamma}^n e^{ -  \ol{\gamma}^2/(4B^2)} d\ol{\gamma} }{  \int_0^\gamma \ol{\gamma}^{-2k- 2B^2 \lambda_k } \ol{\gamma}^n e^{ -  \ol{\gamma}^2/(4B^2)} d\ol{\gamma} } 
			e^{B^2 \lambda_k \tau_0} \Upsilon_U^l	\frac{ \gamma^2 \int_0^\gamma \ol{\gamma}^{ -2k-2B^2 \lambda_k +n} e^{-  \ol{\gamma}^2/(4B^2)} d \ol{\gamma} }{ \int_0^\gamma \ol{\gamma}^{n+2} e^{ -  \ol{\gamma}^2/(4B^2)} d \ol{\gamma} }\\
			&\le  \frac{ \int_0^{\ul{\Upsilon} e^{ - \alpha_k \tau_0}} \gamma^{ n - 2k - 2B^2 \lambda_k} e^{ -  \gamma^2 /(4B^2)} d\gamma}{ \int_0^{ \Upsilon_Z e^{ - \alpha_k \tau }}  \gamma^{ n - 2k - 2B^2 \lambda_k} e^{ - \gamma^2 /(4B^2)} d\gamma} 
			e^{B^2 \lambda_k \tau_0}	\Upsilon_U^l
			\frac{ \gamma^2 \int_0^\gamma \ol{\gamma}^{ -2k-2B^2 \lambda_k +n} e^{-  \ol{\gamma}^2/(4B^2)} d \ol{\gamma} }{ \int_0^\gamma \ol{\gamma}^{n+2} e^{ -  \ol{\gamma}^2/(4B^2)} d \ol{\gamma} }\\
			&\lesssim_{p,q,k}   \frac{ \int_0^{\ul{\Upsilon} e^{ - \alpha_k \tau_0}} \gamma^{ n - 2k -2B^2 \lambda_k} e^{ -  \gamma^2 /(4B^2)} d\gamma}{ \int_0^{ \Upsilon_0 e^{ - \alpha_k \tau }}  \gamma^{ n - 2k - 2B^2 \lambda_k} e^{ -  \gamma^2 /(4B^2)} d\gamma} e^{B^2 \lambda_k \tau_0}	 \Upsilon_U^l 
			\left( \frac{ \gamma^2 \gamma^{ -2k - 2B^2 \lambda_k  +n +1} }{\gamma^{n+3}} + \gamma^2 \right)\\
			& \lesssim_{p,q,k} \left[ \frac{ \ul{\Upsilon} e^{ - \alpha_k \tau_0} }{ \Upsilon_Z e^{ - \alpha_k \tau} } \right]^{n-2k- 2B^2 \lambda_k+1} e^{B^2 \lambda_k \tau_0} \Upsilon_U^l
			\left( \frac{ \gamma^2 \gamma^{ -2k - 2B^2 \lambda_k  +n +1} }{\gamma^{n+3}} + \gamma^2 \right)\\
			& \lesssim_{p,q,k}  \left[ \frac{ \ul{\Upsilon}  }{ \Upsilon_Z  } \right]^{n-2k- 2B^2 \lambda_k+1} e^{B^2 \lambda_k \tau}	 \Upsilon_U^l
			 \left(  \gamma^{ -2k - 2B^2 \lambda_k  } + \gamma^2 \right) \\
		\end{aligned} \end{equation*}
		Since $n-2k -2B^2  \lambda_k > \frac{ 1+n}{2} + \frac{1}{2} \sqrt{ (n-9)(n-1)} > 0$, it follows that if 
		$\Upsilon_Z \gg 1$ is sufficiently large depending on $p,q,k, \ul{D}, \ul{\Upsilon}, \Upsilon_U$ then we get the desired contribution from the inner region.
	
		Finally, let
			$$\chi_{outer}(\gamma) \doteqdot \chi_{[M e^{\ol{\beta} \tau_0}, \infty)}(\gamma)$$
		and observe
			$\gamma^{-2} \gamma^{-2B^2 \lambda_k} \chi_{outer}$
		is either nondecreasing or can be bounded above by $(M e^{\ol{\beta} \tau_0})^{ -2B^2 -2} \chi_{outer}$. 
		In the first case, with $\cl{M}_Z$ as defined in equation \ref{maximalFuncZ}, lemma \ref{maximalFuncBoundZ} implies
		\begin{equation*} \begin{aligned}
			& \quad \left | e^{B^2 \cl{D}_Z ( \tau - \tau_0)} \chi_{outer} \gamma^{-2B^2 \lambda_k} \right|\\
			&\lesssim_{p,q}  \gamma^2 e^{-(\tau - \tau_0)} \cl{M}_Z \chi_{outer} \gamma^{-2B^2 \lambda_k} \\
			&= \gamma^2 e^{-(\tau - \tau_0)}
			 \frac{ \int_\gamma^\infty \chi_{outer} \ol{\gamma}^{-2B^2 \lambda} \ol{\gamma}^{-2} \ol{\gamma}^{n+2} e^{-  \ol{\gamma}^2/(4B^2)} d \ol{\gamma} } {\int_\gamma^\infty \ol{\gamma}^{n+2} e^{- \ol{\gamma}^2/(4B^2)} d \ol{\gamma} } \\
			 &\lesssim_{p,q,k}  \gamma^2 
			 \frac{ ( M e^{\ol{\beta} \tau_0} )^{-2B^2 \lambda + n -1 } e^{- (M e^{\ol{\beta} \tau_0} )^2 / (4B^2)} }{\int_\gamma^\infty \ol{\gamma}^{n+2} e^{- \ol{\gamma}^2 /(4B^2) } d \ol{\gamma} }\\
			& \lesssim_{p,q,k} ( M e^{\ol{\beta} \tau_0} )^{-2B^2 \lambda + n -1 } e^{- (M e^{\ol{\beta} \tau_0} )^2 / (4B^2)} \\
			 & \quad \cdot  \left( \gamma^2 
			 + \gamma^{-2B^2 \lambda} \frac{1}{ ( M e^{\beta \tau} )^{-2B^2 \lambda + n -1} e^{- ( M e^{\beta \tau})^2/(4B^2) } } \right) 
			 && (\ref{asympsOfExpInt}) \\
			& \lesssim_{p,q,k}
			 ( M e^{\ol{\beta} \tau_0} )^{-2B^2 \lambda + n -1 } e^{- (M e^{\ol{\beta} \tau_0} )^2 / (4B^2)} \gamma^{-2k -2B^2 \lambda_k}\\
			 & \quad + \left( e^{( \ol{\beta} - \beta ) \tau_0} \right)^{-2B^2 \lambda + n -1} 
			 e^{- \frac{M^2}{4B^2} \left[ e^{2 \ol{\beta} \tau_0} -C e^{ 2 \beta \tau_0} \right] }
			 \gamma^{-2 B^2 \lambda}\\
			 %
		\end{aligned} \end{equation*}
		Since $\ol{\beta} > \beta$, taking $\tau_0 \gg 1$ sufficiently large yields the estimate.
		
		In the second case where $-2B^2 \lambda -2 < 0$, we again use lemma \ref{maximalFuncBoundZ}
		but we also need to restrict to $\gamma < \Gamma$
		\begin{equation*} \begin{aligned}
			& \quad \left| e^{B^2 \cl{D}_Z(\tau - \tau_0)} \chi_{outer} \gamma^{-2B^2 \lambda} \right| \\
			&\lesssim_{p,q}  \gamma^2 e^{-(\tau - \tau_0)} \cl{M}_Z \chi_{outer} \gamma^{-2B^2 \lambda} \\
			&=  \gamma^2 e^{-(\tau - \tau_0)} \sup_{I \ni \gamma} 
			\frac{ \int_I \chi_{outer} \ol{\gamma}^{-2B^2 \lambda - 2} \ol{\gamma}^{n+1} e^{-  \ol{\gamma}^2 /(4B^2)} d \ol{\gamma} } {\int_I \ol{\gamma}^{n+1} e^{- \ol{\gamma}^2 /(4B^2)} d \ol{\gamma} } \\
			&=   \gamma^2 e^{-(\tau - \tau_0)} \sup_{b \ge M e^{\ol{\beta} \tau_0}} 
			\frac{ \int_{M e^{\ol{\beta} \tau_0}}^b \ol{\gamma}^{-2B^2 \lambda - 2} \ol{\gamma}^{n+2} e^{-  \ol{\gamma}^2 / (4B^2)} d \ol{\gamma} } {\int_\gamma^b \ol{\gamma}^{n+2} e^{-  \ol{\gamma}^2 /(4B^2) } d \ol{\gamma}} \\
			&\le \gamma^2 
			\frac{ \int_{M e^{ \ol{\beta} \tau_0}}^\infty \ol{\gamma}^{n+2} e^{-  \ol{\gamma}^2 /(4B^2) } d \ol{\gamma} }{ \int_\Gamma^{M e^{ \ol{\beta} \tau_0}} \ol{\gamma}^{n+2} e^{-  \ol{\gamma}^2 /(4B^2) } d \ol{\gamma} }
			&& (-2B^2 \lambda - 2 < 0 \text{ and } \gamma \le \Gamma) \\
			&\le \gamma^2 ( M e^{\ol{\beta} \tau_0} )^{n+2} e^{ -  (M e^{\ol{\beta} \tau_0})^2 /(4B^2) }  \frac{ 1}{ \int_\Gamma^{M e^{ \ol{\beta} \tau_0}} \ol{\gamma}^{n+2} e^{-  \ol{\gamma}^2 /(4B^2) } d \ol{\gamma} }
		\end{aligned} \end{equation*}
		For any $\Gamma > 0$, we can take $\tau_0 \gg 1$ sufficiently large depending on $p,q,k, \Gamma, M, \ol{\beta}$ such that 
			$$\int_\Gamma^{M e^{ \ol{\beta} \tau_0}} \ol{\gamma}^{n+2} e^{-  \ol{\gamma}^2 /(4B^2) } d \ol{\gamma}  \ge \frac{1}{2} \int_{\Gamma}^\infty \ol{\gamma}^{n+2} e^{-  \ol{\gamma}^2 /(4B^2) } d \ol{\gamma} = C(\Gamma) > 0$$
		It then follows that
		\begin{equation*} \begin{aligned}
			 & \quad \gamma^2 ( M e^{\ol{\beta} \tau_0} )^{n+2} e^{ -  (M e^{\ol{\beta} \tau_0})^2 /(4B^2) }  \frac{ 1}{ \int_\Gamma^{M e^{ \ol{\beta} \tau_0}} \ol{\gamma}^{n+2} e^{-  \ol{\gamma}^2 /(4B^2) } d \ol{\gamma} } \\
			& \lesssim_{\Gamma}  \gamma^2 ( M e^{\ol{\beta} \tau_0} )^{n+2} e^{ -  (M e^{\ol{\beta} \tau_0})^2 /(4B^2)}   \\
			& \lesssim_{p,q,k, \Gamma} \gamma^{-2B^2 \lambda_k} ( M e^{\ol{\beta} \tau_0} )^{n+2} e^{ - (M e^{\ol{\beta} \tau_0})^2 /(4B^2) }   
			&&  (-2B^2 \lambda - 2 < 0 \text{ and } \gamma \le \Gamma)\\ 
		\end{aligned} \end{equation*}
		By taking $\tau_0 \gg 1$ sufficiently large, the estimate follows.
	\end{proof}
	

\subsection{Estimates for $Err$ Contributions}
\subsubsection{Estimates for $Err_U$ Contributions}
\begin{lem} \label{polyGrowthU}
	If $C_0 > 0$ and $m <2k + 2B^2 \lambda + \sqrt{(n-9)(n-1)}$, then 
	\begin{gather*}
		\left| \int_{\tau_0}^{\tau} e^{B^2 \cl{D}_U (\tau - \ol{\tau})} \chi_{[0, C_0 \gamma]} (\ol{\gamma}) \ol{\gamma}^{-m-2} d\ol{\tau} \right| (\gamma) \lesssim_{p,q,C_0, m} \gamma^{-m} \\
		 \text{for all }  \tau \in [\tau_0, \tau_0+1] \text{ and } \gamma \in [\Upsilon_U e^{- \alpha \tau}, \infty)
	\end{gather*}
\end{lem}
\begin{proof}
	\begin{equation*} \begin{aligned}
		& \quad \left| \int_{\tau_0}^{\tau_1} e^{B^2 \cl{D}_U (\tau - \ol{\tau})} \chi_{[0, C_0 \gamma]} (\ol{\gamma}) \ol{\gamma}^{-m-2} d \ol{\tau} \right| \\
		& \lesssim_{p,q}  \gamma^{-2k- 2B^2 \lambda} \int_{\tau_0}^{\tau} (\tau - \ol{\tau})^{-1 - \frac{1}{2} \sqrt{ (n-9)(n-1) }} 
		  \int_0^{C_0 \gamma} H_U \ol{\gamma}^{-m-2+ \frac{n+1}{2} + \frac{1}{2} \sqrt{ (n-9)(n-1) } } d \ol{\gamma} d \ol{\tau} \\
		& \qquad (\text{by proposition } \ref{kernelU}) \\
		& \le \gamma^{-2k - 2B^2 \lambda} \int_{\tau_0}^{\tau} (\tau - \ol{\tau})^{-1 - \frac{1}{2} \sqrt{ (n-9)(n-1) } }
		\int_0^{C_0 \gamma} exp \left[ - \frac{1}{4B^2} \left( \frac{ e^{- (\tau - \ol{\tau})/2} \gamma - \ol{\gamma} }{ \sqrt{ \tau - \ol{\tau}} } \right)^2 \right]\\
		& \quad \cdot\left( 1 + \frac{1}{2 B^2 C_0 \sqrt{e} } \frac{ \ol{\gamma}^2}{\tau - \ol{\tau}} \right)^{-\frac{1}{2} - \frac{1}{2} \sqrt{ (n-9)(n-1) } } 
		\ol{\gamma}^{-m-2 + \frac{n+1}{2} + \frac{1}{2} \sqrt{ (n-9)(n-1) } } d \ol{\gamma} d \ol{\tau}\\
		&= \gamma^{-2k - 2B^2 \lambda} \int_{\tau_0}^{\tau} (\tau - \ol{\tau})^{-1 - \frac{1}{2} \sqrt{ (n-9)(n-1) }  - \frac{m}{2} -1 + \frac{n+1}{4} + \frac{1}{4} \sqrt{ (n-9)(n-1) } + \frac{1}{2}} \\
		& \quad  \cdot \int_{0}^{C_0 \gamma / \sqrt{\tau - \ol{\tau}} } exp \left[ - \frac{1}{4B^2} \left( \frac{ e^{- (\tau - \ol{\tau})/2} \gamma }{ \sqrt{ \tau - \ol{\tau}} }  - u\right)^2 \right]\\
		& \quad \cdot\left( 1 + \frac{1}{2B^2 C_0 \sqrt{e} } u^2 \right)^{-\frac{1}{2} - \frac{1}{2} \sqrt{ (n-9)(n-1) } } 
		 u^{-m-2 + \frac{n+1}{2} + \frac{1}{2} \sqrt{ (n-9)(n-1) } } d u d \ol{\tau}\\
		& \qquad \left( \text{where } u = \ol{\gamma}/ \sqrt{ \tau - \ol{\tau}} \right) \\
		&= \gamma^{-m} \int_{\gamma^2 / (\tau - \tau_0)}^\infty \xi^{ \frac{m}{2} - \frac{n}{4} + \frac{1}{4} \sqrt{ (n-9)(n-1) } - \frac{3}{4} }
		\int_0^{C_0 \sqrt{\xi} } exp \left[ - \frac{1}{4B^2} \left( e^{- \gamma^2/(2 \xi)} \sqrt{\xi}  - u\right)^2 \right]\\
		&\quad  \cdot\left( 1 + \frac{1}{2 B^2 C_0 \sqrt{e} } u^2 \right)^{-\frac{1}{2} - \frac{1}{2} \sqrt{ (n-9)(n-1) } } 
		u^{-m-2 + \frac{n+1}{2} + \frac{1}{2} \sqrt{ (n-9)(n-1) } } d u d \xi \\
		& \qquad \big( \text{where }  \xi = \gamma^2/ ( \tau - \ol{\tau}) \big)\\
		&\le  \gamma^{-m} \int_{\gamma^2}^\infty \xi^{ \frac{m}{2} - \frac{n}{4} + \frac{1}{4} \sqrt{ (n-9)(n-1) } - \frac{3}{4} } 
		\int_0^{C_0 \sqrt{\xi} } exp \left[ - \frac{1}{4B^2} \left( e^{- \gamma^2/(2 \xi)} \sqrt{\xi}  - u\right)^2 \right]\\
		&\quad  \cdot\left( 1 + \frac{1}{2 B^2 C_0 \sqrt{e} } u^2 \right)^{-\frac{1}{2} - \frac{1}{2} \sqrt{ (n-9)(n-1) } } 
		u^{-m-2 + \frac{n+1}{2} + \frac{1}{2} \sqrt{ (n-9)(n-1) } } d u d \xi
	\end{aligned} \end{equation*}

\noindent At this point, it suffices to prove that the integral is finite, uniformly in $\gamma$.
To this end, define $\xi$-dependent domains
	\begin{gather*}
	D_1(\xi) \doteqdot \left\{ u \in \left[ 0, C_0 \sqrt{\xi} \right] : \left| e^{- \gamma^2/(2 \xi) } \sqrt{\xi} - u \right| \ge \frac{1}{2}  e^{- \gamma^2/(2 \xi) } \sqrt{\xi}  \right\} \\
	D_2(\xi) \doteqdot \left\{ u \in \left[ 0, C_0 \sqrt{\xi} \right] : \left| e^{- \gamma^2/(2 \xi) } \sqrt{\xi} - u \right| \le \frac{1}{2}  e^{- \gamma^2/(2 \xi) } \sqrt{\xi}  \right\}
	\end{gather*}
and split the integral into two integrals according to $D_1$ and $D_2$.
It follows that
	\begin{equation*} \begin{aligned}
		&\quad \int_{\gamma^2}^\infty \xi^{ \frac{m}{2} - \frac{n}{4} + \frac{1}{4} \sqrt{ (n-9)(n-1) } - \frac{3}{4} } 
		\int_{D_1(\xi)} exp \left[ - \frac{1}{4B^2} \left(  e^{- \gamma^2/(2 \xi)} \sqrt{\xi}  - u\right)^2 \right]\\
		& \quad \cdot\left( 1 + \frac{1}{2B^2 C_0 \sqrt{e} } u^2 \right)^{-\frac{1}{2} - \frac{1}{2} \sqrt{ (n-9)(n-1) } } u^{-m-2 + \frac{n+1}{2} + \frac{1}{2} \sqrt{ (n-9)(n-1) } } d u d \xi\\
		&\le \int_{\gamma^2}^\infty \xi^{ \frac{m}{2} - \frac{n}{4} + \frac{1}{4} \sqrt{ (n-9)(n-1) } - \frac{3}{4} } \\
		& \quad \cdot \int_{D_1(\xi)} exp \left[ - \frac{1}{4B^2} \frac{1}{4} e^{- \gamma^2/\xi} \xi \right] u^{-m-2 + \frac{n+1}{2} + \frac{1}{2} \sqrt{ (n-9)(n-1) } } d u d \xi\\
		&\le \int_{\gamma^2}^\infty \xi^{ \frac{m}{2} - \frac{n}{4} + \frac{1}{4} \sqrt{ (n-9)(n-1) } - \frac{3}{4} } exp \left[ - \frac{1}{16B^2 e} \xi \right]\\
		& \quad \cdot \int_0^{C_0 \sqrt{\xi} }  u^{-m-2 + \frac{n+1}{2} + \frac{1}{2} \sqrt{ (n-9)(n-1) } } d u d \xi\\
		&\lesssim_{m,n} C_0^{-m + \frac{n-1}{2} + \frac{1}{2} \sqrt{(n-9)(n-1)} } \int_{\gamma^2}^\infty \xi^{\frac{1}{2} \sqrt{ (n-9)(n-1) } - 1} e^{- \frac{1}{16 B^2 e} \xi} d \xi \\
		& \qquad (\text{by the upper bound on $m$})\\
		&\le  C_0^{-m + \frac{n-1}{2} + \frac{1}{2} \sqrt{(n-9)(n-1)} } \int_{0}^\infty \xi^{\frac{1}{2} \sqrt{...} - 1} e^{- \frac{1}{16 B^2 e} \xi} d \xi \\
		&\lesssim_{p,q}  C_0^{-m + \frac{n-1}{2} + \frac{1}{2} \sqrt{(n-9)(n-1)} }		
		 = C_0^{-m + 2k + 2B^2 \lambda + \sqrt{(n-9)(n-1)}}	\\
	\end{aligned} \end{equation*}
	
To estimate the $D_2(\xi)$ integral, first observe that
\begin{equation*} \begin{aligned}
	&|u - e^{-\gamma^2/(2\xi)} \sqrt{\xi} | \le \frac{1}{2} e^{- \gamma^2/(2\xi)} \sqrt{\xi}\\
	\iff&  \frac{1}{2} e^{- \gamma^2/(2\xi)} \sqrt{\xi} \le u \le  \frac{3}{2} e^{- \gamma^2/(2\xi)} \sqrt{\xi}\\
	 \implies&  \frac{1}{2 \sqrt{e}}  \sqrt{\xi} \le u \le  \frac{3}{2} \sqrt{\xi}\\
\end{aligned} \end{equation*}
It follows that
	\begin{equation*} \begin{aligned}
		& \quad \int_{\gamma^2}^\infty \xi^{ \frac{m}{2} - \frac{n}{4} + \frac{1}{4} \sqrt{ (n-9)(n-1) } - \frac{3}{4} } 
		\int_{D_2(\xi)} exp \left[ - \frac{1}{4B^2} \left(  e^{- \gamma^2/(2 \xi)} \sqrt{\xi}  - u\right)^2 \right]\\
		& \quad \cdot\left( 1 + \frac{1}{2B^2 C_0 \sqrt{e} } u^2 \right)^{-\frac{1}{2} - \frac{1}{2} \sqrt{ (n-9)(n-1) } } u^{-m-2 + \frac{n+1}{2} + \frac{1}{2} \sqrt{ (n-9)(n-1) } } d u d \xi\\
		&\le \int_{\gamma^2}^\infty \xi^{ \frac{m}{2} - \frac{n}{4} + \frac{1}{4} \sqrt{...} - \frac{3}{4} } \\
		& \quad \cdot \int_{ \frac{1}{2 \sqrt{e} } \sqrt{\xi} }^{ \frac{3}{2} \sqrt{\xi} } \left( 1 + \frac{1}{2B^2 C_0 \sqrt{e} } u^2 \right)^{-\frac{1}{2} - \frac{1}{2} \sqrt{ (n-9)(n-1) } } u^{-m-2 + \frac{n+1}{2} + \frac{1}{2} \sqrt{ (n-9)(n-1) } } d u d \xi\\
		&\lesssim_{m,n} \int_{\gamma^2}^\infty \xi^{ \frac{1}{2} \sqrt{ (n-9)(n-1) } - 1} \left( 1 + \frac{1}{8 B^2 C_0 e \sqrt{e} } \xi \right)^{- \frac{1}{2} - \frac{1}{2} \sqrt{ (n-9)(n-1) } } d \xi\\
		&\le \int_{0}^\infty \xi^{ \frac{1}{2} \sqrt{ (n-9)(n-1) } - 1} \left( 1 + \frac{1}{8 B^2 C_0 e \sqrt{e} } \xi \right)^{- \frac{1}{2} - \frac{1}{2} \sqrt{ (n-9)(n-1) } } d \xi\\
		&\lesssim_{p,q,C_0}  1 \\
	\end{aligned} \end{equation*}
This completes the proof.
\end{proof}

\begin{lem}
	For any $\nu \in (0,1)$ and $0 < \kappa < \sqrt{(n-9)(n-1)}$ and $l < \sqrt{(n-9)(n-1)}$, there exists $\Upsilon_U \gg 1$ sufficiently large (depending on $p,q,k,l, \kappa, \nu$) such that
	\begin{gather*}
		\left| \int_{\tau_0}^{\tau} e^{B^2 \cl{D}_U(\tau - \ol{\tau})} \left[ \chi_{(0, e^{-\alpha_k \ol{\tau}} )}(\ol{\gamma}) \Upsilon_U^l ( \ol{\gamma} e^{\alpha \ol{\tau}} )^{-2k-2B^2 \lambda - \kappa} \ol{\gamma}^{-2} \right] d \ol{\tau}\right|(\gamma) \le \nu e^{B^2 \lambda_k \tau} \gamma^{-2k - 2B^2 \lambda_k} \\
		 \text{for all }  \quad 0 \le \tau - \tau_0 \le 1, \quad \gamma \ge \Upsilon_U e^{-\alpha \tau}
	\end{gather*}
\end{lem}
\begin{proof}
	Begin by observing that
	\begin{equation*} \begin{aligned}
		& \left| \int_{\tau_0}^{\tau} e^{B^2 \cl{D}_U(\tau - \ol{\tau})} \chi_{(0, e^{-\alpha_k \ol{\tau}} )}(\ol{\gamma}) \Upsilon_U^l ( \ol{\gamma} e^{\alpha \ol{\tau}} )^{-2k-2B^2 \lambda - \kappa} \ol{\gamma}^{-2} d \ol{\tau}\right| \\
		\lesssim_{p,q,k} &\Upsilon_U^l e^{B^2 \lambda \tau - \kappa \alpha \tau}  \int_{\tau_0}^{\tau} e^{B^2 \cl{D}_U(\tau - \ol{\tau})} \chi_{(0, \frac{ e^{\alpha}}{\Upsilon_U} \gamma )}  \ol{\gamma}^{-2k-2B^2 \lambda - \kappa} \ol{\gamma}^{-2} d \ol{\tau} \\	
	\end{aligned} \end{equation*}
	Now, apply the estimate in the proof of lemma \ref{polyGrowthU} and note that $D_2(\xi) = \emptyset$ if $\Upsilon_U$ is sufficiently large depending on $p,q,k$.
	It thus follows that 
	\begin{equation*} \begin{aligned}
		& \quad \Upsilon_U^l e^{B^2 \lambda \tau - \kappa \alpha \tau}  \int_{\tau_0}^{\tau} e^{B^2 \cl{D}_U(\tau - \ol{\tau})} \chi_{(0, \frac{ e^{\alpha}}{\Upsilon_U} \gamma )}  \ol{\gamma}^{-2k-2B^2 \lambda - \kappa} \ol{\gamma}^{-2} d \ol{\tau} \\	
		&\lesssim_{p,q,k, \kappa}  \Upsilon_U^l e^{B^2 \lambda \tau - \kappa \alpha \tau}  \Upsilon_U ^{ \kappa- \sqrt{(n-9)(n-1)}} \gamma^{-2k -2B^2 \lambda -\kappa} \\
		&\le  \Upsilon_U^{l - \sqrt{(n-9)(n-1)} } e^{B^2 \lambda_k \tau} \gamma^{-2k -2B^2 \lambda} 		&& (\Upsilon_U  e^{-\alpha \tau} \le \gamma) \\
	\end{aligned} \end{equation*}
	Taking $\Upsilon_U \gg 1$ sufficiently large completes the proof.

\end{proof}

\begin{lem}
	For any $l < k + B^2 \lambda_k$ and $\nu \in (0,1)$, there exists $\Upsilon_U \gg 1$ sufficiently large (depending on $p,q,k, l$ and $\nu$) such that
	\begin{gather*}
		\left| \int_{\tau_0}^{\tau} e^{B^2 \cl{D}_U(\tau - \ol{\tau})} \chi_{[e^{- \alpha_k \ol{\tau}}, \Upsilon_U e^{-\alpha_k \ol{\tau}} ]}(\ol{\gamma}) \Upsilon_U^{2l} \ol{\gamma}^{-4k-4B^2 \lambda_k  - 2} e^{2B^2 \lambda_k \ol{\tau} } d \ol{\tau}\right|(\gamma)\\
		 \le \nu e^{B^2 \lambda_k \tau} \gamma^{-2k - 2B^2 \lambda_k} \\
	\end{gather*}
	for all $ 0 \le \tau - \tau_0 \le 1$ and $\gamma \ge \Upsilon_U e^{-\alpha \tau}$.
\end{lem}
\begin{proof}
	We estimate the integrand and apply lemma \ref{polyGrowthU}.
	Note that 
		$$4k + 4B^2 \lambda = n-1 - \sqrt{(n-9)(n-1)} < \frac{n+1}{2} + \frac{1}{2} \sqrt{(n-9)(n-1)} -1$$
	so that the assumptions of lemma \ref{polyGrowthU} are valid.
	\begin{equation*} \begin{aligned}
		& \quad  \left| \int_{\tau_0}^{\tau} e^{B^2 \cl{D}_U(\tau - \ol{\tau})} \chi_{[e^{-\alpha \ol{\tau}} , \Upsilon_U e^{- \alpha_k \ol{\tau}},  ]}(\ol{\gamma}) \Upsilon_U^{2l} \ol{\gamma}^{-4k-4B^2 \lambda_k  - 2} e^{2B^2 \lambda_k \tau } d \ol{\tau}\right|(\gamma) \\
		&\le  \Upsilon_U^{2l} e^{2B^2 \lambda \tau_0} \left| \int_{\tau_0}^{\tau} e^{B^2 \cl{D}_U(\tau - \ol{\tau})} \chi_{[0, e^\alpha \gamma ]}(\ol{\gamma}) \ol{\gamma}^{-4k-4B^2 \lambda_k  - 2}  d \ol{\tau}\right|(\gamma) \\
		&\lesssim_{p,q,k}  \Upsilon_U^{2l}e^{2B^2 \lambda \tau} \gamma^{-4k -4B^2 \lambda } \\
		&\le  \Upsilon_U^{2l-2k-2B^2 \lambda} e^{B^2 \lambda \tau} \gamma^{-2k -2B^2 \lambda} 
	\end{aligned} \end{equation*}
\end{proof}

\begin{lem} \label{polyGrowthUBig}
	If $2k +2B^2 \lambda < m < 2k + 2B^2 \lambda +  \sqrt{(n-9)(n-1)} $, then
	\begin{gather*}
		\left| \int_{\tau_0}^{\tau} e^{B^2 \cl{D}_U (\tau - \ol{\tau})} \chi_{[4 \gamma, \infty)} (\ol{\gamma}) \ol{\gamma}^{-m-2} d\ol{\tau} \right| (\gamma) \lesssim_{p,q, m} \gamma^{-m} \\
		 \text{for all } \tau \in [\tau_0, \tau_0+1] \text{ and } \gamma \in [\Upsilon_U e^{- \alpha \tau}, \infty)
	\end{gather*}
\end{lem}


\begin{proof}
	\begin{equation*} \begin{aligned}
		& \quad  \left| \int_{\tau_0}^{\tau} e^{B^2 \cl{D}_U (\tau - \ol{\tau})} \chi_{[4 \gamma, \infty)} (\ol{\gamma}) \ol{\gamma}^{-m-2} d\ol{\tau} \right| (\gamma) \\
		&\lesssim_{p,q} \gamma^{-2k - 2B^2 \lambda} \int_{\tau_0}^{\tau} (\tau - \ol{\tau})^{-1- \frac{1}{2} \sqrt{ (n-9)(n-1) } } 
		\int_{4 \gamma}^{\infty} \ol{\gamma}^{-m - 2 } H_U \ol{\gamma}^{\frac{n+1}{2} + \frac{1}{2} \sqrt{  (n-9)(n-1)  } } d \ol{\gamma} d \ol{\tau}\\	
		& \qquad  ( \text{by proposition }\ref{kernelU})\\
		&\le   \gamma^{-2k - 2B^2 \lambda} \int_{\tau_0}^{\tau} (\tau - \ol{\tau})^{-1- \frac{1}{2} \sqrt{  (n-9)(n-1)  } } e^{-\frac{1}{4B^2} \frac{c \gamma^2}{\tau - \ol{\tau}} } \\
		& \quad \cdot  \int_{4 \gamma}^\infty \ol{\gamma}^{\frac{n+1}{2} + \frac{1}{2} \sqrt{  (n-9)(n-1)  } -m-2 } e^{-\frac{1}{4B^2} \frac{ c \ol{\gamma}^2}{\tau - \ol{\tau}} } d \ol{\gamma} d \ol{\tau} \\
		&=  \gamma^{-2k - 2B^2 \lambda}  \int_{\tau_0}^{\tau_1} e^{- \frac{1}{4B^2} \frac{c \gamma^2}{\tau - \ol{\tau}}} (\tau - \ol{\tau})^{-1 - \frac{1}{2} \sqrt{  (n-9)(n-1)  } + \frac{n+1}{4} + \frac{1}{4} \sqrt{ (n-9)(n-1) } - \frac{m}{2} - \frac{1}{2} } \\
		&\quad  \cdot \int_{4 \gamma (\tau - \ol{\tau})^{-1/2} }^{\infty} u^{ \frac{n+1}{2} + \frac{1}{2} \sqrt{  (n-9)(n-1)  } -m-2 } e^{-\frac{1}{4B^2} c u^2} du d \ol{\tau}\\
		& \qquad (\text{where } u = \ol{\gamma} (\tau - \ol{\tau})^{-1/2} ) \\
		&=  \gamma^{-2k - 2 B^2 \lambda} 
		\int_{\tau_0}^{\tau} e^{- \frac{1}{4B^2} \frac{c \gamma^2}{\tau - \ol{\tau}}} (\tau - \ol{\tau})^{-1+ \frac{n-1}{4} - \frac{1}{4} \sqrt{ (n-9)(n-1) }  - \frac{m}{2}} \\
		& \quad \cdot \int_{4 \gamma (\tau - \ol{\tau})^{-1/2} }^{\infty} u^{ \frac{n+1}{2} + \frac{1}{2} \sqrt{  (n-9)(n-1)  } -m-2 } e^{-\frac{1}{4B^2} c u^2} du d \ol{\tau}\\
		&\lesssim  \gamma^{-2k - 2B^2  \lambda} 
		\int_{\gamma^2 (\tau - \tau_0)^{-1}}^\infty e^{- \frac{1}{4B^2} c \xi} \left( \frac{\xi}{\gamma^2} \right)^{1 + \frac{1-n}{4} + \frac{1}{4}\sqrt{ (n-9)(n-1) }  + \frac{m}{2} } \frac{\gamma^2}{\xi^2}\\
		& \quad \cdot \int_{4 \sqrt{ \xi} }^\infty u^{ \frac{n+1}{2} + \frac{1}{2} \sqrt{  (n-9)(n-1) } -m-2 } e^{- \frac{1}{4B^2} c u^2} du  d \xi \\
		& \qquad (\text{where } \xi = \gamma^2 (\tau - \ol{\tau})^{-1} ) \\
		&=  \gamma^{-2k - 2B^2 \lambda} \gamma^{  \frac{n-1}{2}- \frac{1}{2}\sqrt{ (n-9)(n-1) }  - m }
		 \int_{\gamma^2 (\tau - \tau_0)^{-1}}^\infty e^{- \frac{1}{4B^2} c \xi } \xi^{ \frac{1-n}{4} + \frac{1}{4} \sqrt{ (n-9)(n-1) } + \frac{m}{2} -1} \\
		& \qquad  \cdot \int_{4 \sqrt{ \xi} }^\infty u^{\frac{n+1}{2} + \frac{1}{2} \sqrt{  (n-9)(n-1)  } -m-2} e^{- \frac{1}{4B^2} c u^2 } du d\xi \\
		&\lesssim_{p,q,m}  \gamma^{-m} \int_{\gamma^2}^\infty e^{- \frac{1}{4B^2} c \xi } \xi^{ \frac{1-n}{4} + \frac{1}{4} \sqrt{ (n-9)(n-1) } + \frac{m}{2} -1} d \xi \\
		&\lesssim_{p,q,m} \gamma^{-m} 	
	\end{aligned} \end{equation*}
	where the assumed bounds on $m$ are used to estimate the $du$ and $d \xi$ integrals in the last two lines.
\end{proof}

\begin{lem}
	For any $l < 2k + 2B^2 \lambda_k$ and $\nu \in (0,1)$, there exists $\Upsilon_U \gg 1$ sufficiently large (depending on $p,q,k, l $ and $\nu$) such that
	\begin{gather*}
		\left| \int_{\tau_0}^{\tau} e^{B^2 \cl{D}_U(\tau - \ol{\tau})} \chi_{[\Upsilon_U e^{- \alpha_k \ol{\tau}}, \Upsilon_Z e^{-\alpha_k \ol{\tau}} ]}(\ol{\gamma}) \Upsilon_U^l \ol{\gamma}^{-4k-4B^2 \lambda_k  - 2} e^{2B^2 \lambda_k \tau } d \ol{\tau}\right|(\gamma)\\
		 \le \nu e^{B^2 \lambda_k \tau} \gamma^{-2k - 2B^2 \lambda_k} 
	\end{gather*}
	for all $0 \le \tau - \tau_0 \le 1$ and  $\gamma \ge \Upsilon_U e^{-\alpha \tau}$.
\end{lem}
\begin{proof}
	Begin by splitting the integral
	\begin{equation*} \begin{aligned}
		& \left| \int_{\tau_0}^{\tau} e^{B^2 \cl{D}_U(\tau - \ol{\tau})} \chi_{[\Upsilon_U e^{- \alpha_k \ol{\tau}}, \Upsilon_Z e^{-\alpha_k \ol{\tau}} ]}(\ol{\gamma}) \Upsilon_U^l \ol{\gamma}^{-4k-4B^2 \lambda_k  - 2} e^{2B^2 \lambda_k \ol{\tau} } d \ol{\tau}\right|(\gamma) \\
		 \le  & \Upsilon_U^l \left| \int_{\tau_0}^{\tau} e^{B^2 \cl{D}_U(\tau - \ol{\tau})} \chi_{[\Upsilon_U e^{- \alpha_k \ol{\tau}}, 4 \gamma ]}(\ol{\gamma}) \ol{\gamma}^{-4k-4B^2 \lambda_k  - 2} e^{2B^2 \lambda_k \ol{\tau} } d \ol{\tau}\right|(\gamma) \\
		& +\Upsilon_U^l  \left| \int_{\tau_0}^{\tau} e^{B^2 \cl{D}_U(\tau - \ol{\tau})} \chi_{[4 \gamma, \Upsilon_Z e^{-\alpha_k \ol{\tau}} ]}(\ol{\gamma}) \ol{\gamma}^{-4k-4B^2 \lambda_k  - 2} e^{2B^2 \lambda_k \ol{\tau} } d \ol{\tau}\right|(\gamma) \\
		 \le  &\Upsilon_U^l  \left| \int_{\tau_0}^{\tau} e^{B^2 \cl{D}_U(\tau - \ol{\tau})} \chi_{[0, 4 \gamma ]}(\ol{\gamma}) \ol{\gamma}^{-4k-4B^2 \lambda_k  - 2} e^{2B^2 \lambda_k \ol{\tau} } d \ol{\tau}\right|(\gamma) \\
		& +\Upsilon_U^l  \left| \int_{\tau_0}^{\tau} e^{B^2 \cl{D}_U(\tau - \ol{\tau})} \chi_{[4 \gamma, \infty ]}(\ol{\gamma}) \ol{\gamma}^{-4k-4B^2 \lambda_k  - 2} e^{2B^2 \lambda_k \ol{\tau} } d \ol{\tau}\right|(\gamma) \\		
		=& (I) + (II)
	\end{aligned} \end{equation*}
	Lemma \ref{polyGrowthU} implies that
		$$(I) \lesssim_{p,q,k} \Upsilon_U^l \gamma^{-4k-4B^2 \lambda_k} e^{2B^2 \lambda_k \tau} \lesssim_{p,q,k} \Upsilon_U^{l-2k -2B^2 \lambda} e^{B^2 \lambda \tau} \gamma^{-2k -2B^2 \lambda}$$	
	Similarly, lemma \ref{polyGrowthUBig} implies
		$$(II) \lesssim_{p,q,k} \Upsilon_U^{l-2k -2B^2 \lambda} e^{B^2 \lambda \tau} \gamma^{-2k -2B^2 \lambda}$$
\end{proof}

\begin{lem} \label{shortTimeErrUParabNeg}
	For any $\nu \in (0,1)$, there exists $\Upsilon_U \gg 1$ sufficiently large (depending on $p,q,k$ and $\nu$) such that
		$$\left| \int_{\tau_0}^{\tau} e^{B^2 \cl{D}_U(\tau - \ol{\tau})} \chi_{[ \Upsilon_Z e^{- \alpha_k \ol{\tau}}, M e^{\beta \ol{\tau}} ]}(\ol{\gamma}) \ol{\gamma}^{-4k-4B^2 \lambda_k - 2} e^{2B^2 \lambda_k \tau} d \ol{\tau}\right|(\gamma) \le \nu e^{B^2 \lambda_k \tau}  \gamma^{-2k - 2B^2 \lambda_k}  $$
		$$ \text{for all }  \quad 0 \le \tau - \tau_0 \le 1, \quad  \Upsilon_U e^{-\alpha \tau}  \le \gamma \le M e^{\beta \tau}$$
\end{lem}
\begin{proof}
	Begin by splitting the integral
	\begin{equation*} \begin{aligned}
		& \left| \int_{\tau_0}^{\tau} e^{B^2 \cl{D}_U(\tau - \ol{\tau})} \chi_{[ \Upsilon_Z e^{- \alpha_k \ol{\tau}}, M e^{\beta \ol{\tau}} ]}(\ol{\gamma}) \ol{\gamma}^{-4k-4B^2 \lambda_k - 2} e^{2B^2 \lambda_k \tau} d \ol{\tau}\right|(\gamma) \\
		\le & \left| \int_{\tau_0}^{\tau} e^{B^2 \cl{D}_U(\tau - \ol{\tau})} \chi_{[ \Upsilon_Z e^{- \alpha_k \ol{\tau}}, 4 \gamma ]}(\ol{\gamma}) \ol{\gamma}^{-4k-4B^2 \lambda_k - 2} e^{2B^2 \lambda_k \tau} d \ol{\tau}\right|(\gamma) \\
		& + \left| \int_{\tau_0}^{\tau} e^{B^2 \cl{D}_U(\tau - \ol{\tau})} \chi_{[4 \gamma, M e^{\beta \ol{\tau}} ]}(\ol{\gamma}) \ol{\gamma}^{-4k-4B^2 \lambda_k - 2} e^{2B^2 \lambda_k \tau} d \ol{\tau}\right|(\gamma) \\
		\le & \left| \int_{\tau_0}^{\tau} e^{B^2 \cl{D}_U(\tau - \ol{\tau})} \chi_{[ 0, 4 \gamma ]}(\ol{\gamma}) \ol{\gamma}^{-4k-4B^2 \lambda_k - 2} e^{2B^2 \lambda_k \tau} d \ol{\tau}\right|(\gamma) \\
		& + \left| \int_{\tau_0}^{\tau} e^{B^2 \cl{D}_U(\tau - \ol{\tau})} \chi_{[4 \gamma, \infty )}(\ol{\gamma}) \ol{\gamma}^{-4k-4B^2 \lambda_k - 2} e^{2B^2 \lambda_k \tau} d \ol{\tau}\right|(\gamma) \\
		= & (I) + (II)
	\end{aligned} \end{equation*}
	
	As in the previous lemma, lemmas \ref{polyGrowthU} and \ref{polyGrowthUBig} imply
		$$(I), (II) \lesssim_{p,q,k} \Upsilon_U^{-2k - 2B^2 \lambda } e^{B^2 \lambda \tau} \gamma^{-2k -2 B^2 \lambda}$$
\end{proof}

\begin{lem}
	For any $\nu \in (0,1)$, $\beta \in \left(0, \frac{1}{2} \right)$, and $M > 0$, there exists $\tau_0 \gg 1$ sufficiently large (depending on $p,q,k,M, \beta, \nu$) such that
	\begin{gather*}
		\left| \int_{\tau_0}^{\tau} e^{B^2 \cl{D}_U(\tau - \ol{\tau})} \chi_{[ \Upsilon_Z e^{- \alpha_k \ol{\tau}}, M e^{\beta \ol{\tau}} +1]}(\ol{\gamma}) \ol{\gamma}^{-4B^2 \lambda_k } e^{2B^2 \lambda_k \tau} d \ol{\tau}\right|(\gamma) \\
		\le \nu e^{B^2 \lambda_k \tau} \left( \gamma^{-2k - 2B^2 \lambda_k}  + \gamma^{-2B^2 \lambda_k} \right)\\
	\end{gather*}
	for all $0 \le \tau - \tau_0 \le 1$ and $\Upsilon_U e^{-\alpha \tau} \le \gamma \le M e^{\beta \tau}$.
\end{lem}
\begin{proof}
	Applying lemma \ref{maximalFuncBoundU} and using that 
	$$-4B^2  \lambda + 2k + 2B^2 \lambda > 0		\qquad \text{and} \qquad 0 \le \tau - \tau_0 \le 1,$$ we deduce
	\begin{equation*} \begin{aligned}
		& \quad \left| \int_{\tau_0}^{\tau} e^{B^2 \cl{D}_U (\tau  - \ol{\tau})}  e^{2B^2 \lambda_k \ol{\tau}} \ol{\gamma}^{-4 B^2 \lambda} d \ol{\tau} \right| \\
		&\lesssim_{p,q,k}  e^{2B^2 \lambda_k \tau_0} \\
		& \quad \cdot \int_{\tau_0}^{\tau} \gamma^{-2k - 2B^2 \lambda}
		\frac{ \int_{\gamma}^\infty \ol{\gamma}^{2k  - 2B^2 \lambda  } \ol{\gamma}^{1 + \sqrt{(n-9)(n-1)} } e^{- \ol{\gamma}^2/(4B^2)} d \ol{\gamma} }{ \int_{\gamma}^\infty  \ol{\gamma}^{1 + \sqrt{(n-9)(n-1)} } e^{-  \ol{\gamma}^2/(4B^2)} d \ol{\gamma} } d \ol{\tau} \\
		&\lesssim_{p,q,k} e^{2B^2 \lambda_k \tau_0} \int_{\tau_0}^{\tau} \gamma^{-2k - 2B^2 \lambda} ( 1+ \gamma^{2k - 2B^2 \lambda } ) d \ol{\tau} && (\text{lemma }\ref{asympsOfExpInt}) \\
		&\lesssim_{p,q,k}  e^{2B^2 \lambda_k \tau}  (\gamma^{-2k - 2B^2 \lambda} + \gamma^{ -4B^2 \lambda } ) \\
		&\le  e^{B^2 \lambda_k \tau_0} e^{B^2 \lambda_k \tau} \gamma^{-2k -2B^2 \lambda } + M^{-2B^2 \lambda} e^{( 1 - 2 \beta) B^2 \lambda \tau_0 } e^{B^2 \lambda \tau} \gamma^{-2B^2 \lambda}
	\end{aligned} \end{equation*}
	The statement of the lemma then follows immediately.
\end{proof}

We summarize the estimates in this subsection in the following statement:
\begin{lem} \label{summShortTimeEstErrU}
	For any $\nu \in (0,1)$, there exists
		$\Upsilon_U \gg 1$ sufficiently large depending on $p,q,k, l, \kappa , \ul{D}, \nu$, and
		$\tau_0 \gg 1$ sufficiently large depending on $p,q,k, M , \beta, \ul{D}, \nu$
	such that $(\tl{Z}, \tl{U}) \in \mathcal{P}$ implies
	$$\left| \int_{\tau_0}^\tau e^{B^2 \mathcal{D}_U ( \tau - \ol{\tau} ) } \grave{\chi} Err_U(\ol{\tau} ) d \ol{\tau} \right| (\gamma) \le \nu e^{B^2 \lambda_k \tau} \left( \gamma^{-2k -2B^2 \lambda_k} + \gamma^{-2B^2 \lambda_k} \right)$$
	for all $\tau_0 \le \tau \le \tau_1 \le \tau_0 + 1$ and $\Upsilon_U e^{- \alpha_k \tau} \le \gamma \le M e^{ \beta \tau}$.
\end{lem}

\subsubsection{Estimates for $Err_Z$ Contributions}
\begin{lem}
	For any $\nu \in (0,1)$, there exists $\Upsilon_Z \gg 1$ sufficiently large (depending on $p,q,k, \nu$) such that
	\begin{gather*}
		\left| \int_{\tau_0}^{\tau} e^{B^2 \cl{D}_Z(\tau - \ol{\tau})} \chi_{(0, e^{-\alpha_k \ol{\tau}} )}(\ol{\gamma}) \ol{\gamma}^{-2} d \ol{\tau}\right|(\gamma) \le \nu e^{B^2 \lambda_k \tau} \gamma^{-2k - 2B^2 \lambda_k} \\
		 \text{for all }  \quad 0 \le \tau - \tau_0 \le 1, \quad \gamma \ge \Upsilon_Z e^{-\alpha \tau}
	\end{gather*}
\end{lem}
\begin{proof}
	\begin{equation*} \begin{aligned}
		&\quad \left| \int_{\tau_0}^{\tau} e^{B^2 \cl{D}_Z (\tau - \ol{\tau})} \chi_{ [0,  e^{- \alpha_k \ol{\tau} } ]} \ol{\gamma}^{-2} d \ol{\tau} \right| (\gamma) \\
		& \lesssim_{p,q} \int_{\tau_0}^{\tau} \gamma^2 (\tau - \ol{\tau})^{-1 - \frac{n+1}{2} } 
		\int_0^\infty H_Z \chi_{ [0,  e^{- \alpha_k \ol{\tau} } ]} \ol{\gamma}^{n-2} d \ol{\gamma} d \ol{\tau} \\
		& \qquad (\text{by proposition \ref{kernelZ}})\\
		&\le \gamma^2 e^{B^2 \lambda_k \tau_0} \int_{\tau_0}^{\tau} \gamma^2 (\tau - \ol{\tau})^{-1 - \frac{n+1}{2} } 
		\int_0^{e^{- \alpha_k \ol{\tau}} } H_Z  \ol{\gamma}^{n-2-2k - 2B^2 \lambda} d \ol{\gamma} d \ol{\tau} \\
		&\le \gamma^2 e^{B^2 \lambda_k \tau_0} \int_{\tau_0}^{\tau} exp\left[ - \frac{c}{4B^2} \frac{\gamma^2}{\tau - \ol{\tau}} \right] (\tau - \ol{\tau})^{-1 - \frac{n+1}{2} } 
		\int_0^{ e^{- \alpha_k \ol{\tau}}} \ol{\gamma}^{n - 2 - 2k - 2B^2 \lambda} d \ol{\gamma} d \ol{\tau}\\
		&= \gamma^2 e^{B^2 \lambda \tau_0} \int_{\tau_0}^{\tau} exp \left[ - \frac{c}{4B^2 } \frac{\gamma^2}{\tau - \ol{\tau}} \right] (\tau - \ol{\tau})^{-2 -k - B^2 \lambda} \int_0^{ e^{- \alpha_k \ol{\tau}} / \sqrt{\tau - \ol{\tau}} }	u^{n-2-2k-2B^2 \lambda} du d \ol{\tau}\\
		& \qquad \left(\text{where } u = \ol{\gamma} (\tau - \ol{\tau})^{-1/2} \right)\\
		&\le \gamma^2  e^{B^2 \lambda \tau_0} \int_{\gamma^2/(\tau - \tau_0)}^{\infty} exp \left[ - \frac{c}{4B^2} \xi \right] \left( \frac{\gamma^2}{\xi} \right)^{-2 - k -B^2 \lambda} \frac{\gamma^2}{\xi^2}
		\int_0^{ \frac{e^{\alpha}}{\Upsilon_Z}  \sqrt{\xi} } u^{n-2-2k-2B^2 \lambda} du d\xi\\
		& \qquad (\text{where } \xi = \gamma^2 (\tau - \ol{\tau})^{-1} )\\
		&\le \gamma^{-2k-2B^2 \lambda}  e^{B^2 \lambda \tau_0} \int_{\gamma^2/(\tau - \tau_0)}^{\infty} exp \left[ - \frac{c}{4B^2} \xi \right] \xi^{k + B^2 \lambda}
		\int_0^{ \frac{e^\alpha}{\Upsilon_Z} \sqrt{\xi} } u^{n-2-2k-2B^2 \lambda} du d\xi\\
		&\lesssim_{p,q,k}  \Upsilon_Z^{\frac{1-n}{2} - \frac{1}{2} \sqrt{(n-9)(n-1)} }  e^{B^2 \lambda \tau} \gamma^{-2k - 2B^2 \lambda}
		\int_0^\infty \xi^{ \frac{n-1}{2} } exp \left[ - \frac{c}{4B^2} \xi \right] d \xi \\
		&\lesssim_{p,q,k}  \Upsilon_Z^{\frac{1-n}{2} - \frac{1}{2} \sqrt{(n-9)(n-1)} }  e^{B^2 \lambda \tau} \gamma^{-2k - 2B^2 \lambda}
	\end{aligned} \end{equation*}
	Since $\frac{1-n}{2} - \frac{1}{2} \sqrt{(n-9)(n-1)} < 0$, we get the desired estimate by taking $\Upsilon_Z \gg 1$ sufficiently large.
\end{proof}

\begin{lem} \label{polyGrowthZ}
If $C_0 > 0$ and $m < n-1$, then 
	\begin{gather*}
		\left| \int_{\tau_0}^{\tau} e^{B^2 \cl{D}_Z (\tau - \ol{\tau})} \chi_{[0, C_0 \gamma]} \ol{\gamma}^{-m-2} d \ol{\tau} \right| (\gamma)  \lesssim_{p,q,m,C_0}  \gamma^{-m} \\
	\text{for all } \quad 0 \le \tau - \tau_0 \le 1, \quad \gamma \ge \Upsilon_Z e^{- \alpha \tau}
	\end{gather*}
\end{lem}
\begin{proof}
	\begin{equation*} \begin{aligned}
		& \left| \int_{\tau_0}^{\tau} e^{B^2 \cl{D}_Z (\tau - \ol{\tau})} \chi_{[0, C_0 \gamma]} \ol{\gamma}^{-m-2} d \ol{\tau} \right|\\
		&\lesssim_{p,q}  \gamma^2 \int_{\tau_0}^{\tau} (\tau -  \ol{\tau})^{-1 - \frac{n+1}{2} } 
		\int_{0}^{C_0 \gamma } H_Z \ol{\gamma}^{n-m-2} d \ol{\gamma} d \ol{\tau} \\
		&= \gamma^2 \int_{\tau_0}^{\tau} (\tau - \ol{\tau})^{-1 - \frac{n+1}{2} } 
		\int_{0}^{C_0 \gamma } exp \left[ - \frac{1}{4B^2} \frac{ ( e^{-(\tau_1 - \tau)/2} \gamma - \ol{\gamma} )^2 }{1 - e^{-(\tau_1 - \tau)} } \right] \\
		& \quad \cdot \left( 1 + \frac{1}{2B^2} \frac{ e^{-(\tau_1 - \tau)/2} \gamma \ol{\gamma} }{1 - e^{-(\tau_1 - \tau)} } \right)^{-\frac{1}{2} - \frac{n+1}{2} } \ol{\gamma}^{n-m-2} d \ol{\gamma}  d\tau \\
		&\le  \gamma^2 \int_{\tau_0}^{\tau} (\tau - \ol{\tau})^{-1 - \frac{n+1}{2} }
		\int_{0}^{C_0 \gamma} exp \left[ - \frac{1}{4B^2 } \left( \frac{  e^{-(\tau - \ol{\tau})/2} \gamma - \ol{\gamma} }{ \sqrt{\tau - \ol{\tau}} } \right)^2 \right] \\
		& \quad  \cdot  \left( 1 + \frac{1}{2B^2 C_0 \sqrt{e} } \frac{ \ol{\gamma}^2}{\tau - \ol{\tau}} \right)^{-\frac{1}{2} - \frac{n+1}{2} } \ol{\gamma}^{n-m-2} d \ol{\gamma} d \ol{\tau}\\
		&= \gamma^2 \int_{\tau_0}^{\tau} (\tau - \ol{\tau})^{-2 - \frac{m}{2}}
		\int_{0}^{C_0 \gamma / \sqrt{\tau - \ol{\tau}} } exp \left[ - \frac{1}{4B^2} \left( \frac{  e^{-(\tau - \ol{\tau})/2} \gamma }{ \sqrt{\tau - \ol{\tau}} } - u \right)^2 \right] \\
		& \quad \cdot \left( 1 + \frac{1}{2B^2 C_0 \sqrt{e}} u^2 \right)^{-\frac{1}{2} - \frac{n+1}{2} } u^{n-m-2} du d \ol{\tau}\\
		& \qquad \left(\text{where } u = \ol{\gamma} / \sqrt{ \tau - \ol{\tau}} \right)\\
		&= \gamma^{-m} \int_{\gamma^2/(\tau - \tau_0)}^\infty \xi^{ \frac{m}{2} }
		\int_{0 }^{C_0 \sqrt{\xi} } exp \left[ - \frac{1}{4B^2} \left( e^{- \gamma^2/(2 \xi) } \sqrt{ \xi} - u \right)^2 \right]  \\
		& \quad \cdot \left( 1 + \frac{1}{2B^2 C_0 \sqrt{e}} u^2 \right)^{-\frac{1}{2} - \frac{n+1}{2} } u^{n-m-2} du d \xi\\
		& \qquad ( \text{where } \xi = \gamma^2 / (\tau - \ol{\tau}) )\\
		&\le \gamma^{-m} \int_{\gamma^2}^\infty \xi^{ \frac{m}{2} }
		\int_{0}^{C_0 \sqrt{\xi} } exp \left[ - \frac{1}{4B^2} \left( e^{- \gamma^2/(2 \xi) } \sqrt{ \xi} - u \right)^2 \right]  \\
		&\quad  \cdot \left( 1 + \frac{1}{2B^2 C \sqrt{e}} u^2 \right)^{-\frac{1}{2} - \frac{n+1}{2} } u^{n-m-2} du d \xi\\ 
	\end{aligned} \end{equation*}
	
	\noindent At this point, it suffices to prove that the integral is finite.
	To this end define
	\begin{gather*}
		D_1(\xi) \doteqdot \left \{ u \in \left( 0, C_0 \sqrt{\xi} \right) :  \left| u - e^{- \gamma^2/(2\xi)} \sqrt{ \xi} \right| \ge \frac{1}{2} e^{- \gamma^2/(2\xi)} \sqrt{\xi} \right\} \\
		D_2(\xi) \doteqdot \left \{ u \in \left( 0, C_0 \sqrt{\xi} \right) : \left| u - e^{- \gamma^2/(2\xi)} \sqrt{ \xi} \right | \le \frac{1}{2} e^{- \gamma^2/(2\xi)} \sqrt{\xi} \right\}
	\end{gather*}
	and split the integral as 
	\begin{equation*} \begin{aligned}
	&\quad \int_{\gamma^2}^\infty \xi^{ \frac{m}{2} } \int_{0}^{C_0 \sqrt{\xi} } exp \left[ - \frac{1}{4B^2} \left( e^{- \gamma^2/(2 \xi) } \sqrt{ \xi} - u \right)^2 \right] \\
	& \quad \cdot  \left( 1 + \frac{1}{2 B^2 C_0 \sqrt{e}} u^2 \right)^{-\frac{1}{2} - \frac{n+1}{2} } u^{n-m-2} du d \xi \\
	&= \int_{\gamma^2}^\infty \xi^{ \frac{m}{2} }	\int_{D_1(\xi)} exp \left[ - \frac{1}{4B^2} \left( e^{- \gamma^2/(2 \xi) } \sqrt{ \xi} - u \right)^2 \right] \\
	& \qquad \cdot  \left( 1 + \frac{1}{2B^2 C_0 \sqrt{e}} u^2 \right)^{-\frac{1}{2} - \frac{n+1}{2} } u^{n-m-2} du d \xi \\
	& \quad + \int_{\gamma^2}^\infty \xi^{ \frac{m}{2} }	\int_{D_2(\xi) } exp \left[ - \frac{1}{4B^2} \left( e^{- \gamma^2/(2 \xi) } \sqrt{ \xi} - u \right)^2 \right]  \\
	& \qquad \cdot \left( 1 + \frac{1}{2B^2 C_0 \sqrt{e}} u^2 \right)^{-\frac{1}{2} - \frac{n+1}{2} } u^{n-m-2} du d \xi \\
	&= (I) + (II) 
	\end{aligned} \end{equation*}
	Observe that
	\begin{equation*} \begin{aligned}
		(I) &\le \int_{\gamma^2}^\infty \xi^{m/2} \int_{D_1} exp\left[ - \frac{1}{4B^2} \frac{1}{4} e^{- \gamma^2/\xi} \xi \right] u^{n-m-2} du d \xi \\
		&\le \int_{\gamma^2}^\infty \xi^{m/2} e^{- \frac{1}{16 B^2 e} \xi} \int_0^{C_0 \sqrt{\xi}} u^{n-m-2} du d\xi\\
		&\lesssim_{C_0, n, m} \int_{\gamma^2}^\infty \xi^{\frac{n}{2} - \frac{1}{2} } e^{- \frac{1}{16 B^2 e} \xi} d\xi && (\text{since } n-m-2>-1)\\
		&\le \int_0^\infty \xi^{\frac{n}{2} - \frac{1}{2} } e^{- \frac{1}{16 B^2 e} \xi} d\xi  < \infty
	\end{aligned} \end{equation*}
	
	To estimate $(II)$, we begin by observing that
	\begin{equation*} \begin{aligned}
		&|u - e^{-\gamma^2/(2\xi)} \sqrt{\xi} | \le \frac{1}{2} e^{- \gamma^2/(2\xi)} \sqrt{\xi} \\
		 \iff&   \frac{1}{2} e^{- \gamma^2/(2\xi)} \sqrt{\xi} \le u \le  \frac{3}{2} e^{- \gamma^2/(2\xi)} \sqrt{\xi} \\
		 \implies&  \frac{1}{2 \sqrt{e}}  \sqrt{\xi} \le u \le  \frac{3}{2} \sqrt{\xi}
	\end{aligned} \end{equation*}
	since $\gamma^2 \le \xi$.
	It follows that
	\begin{equation*} \begin{aligned}
		(II) &\le \int_{\gamma^2}^\infty \xi^{m/2} \int_{ \frac{1}{2 \sqrt{e}} \sqrt{\xi}}^{\frac{3}{2} \sqrt{\xi}} \left( 1 + \frac{1}{2 B^2 C_0 \sqrt{e}} u^2 \right)^{-\frac{1}{2} - \frac{n+1}{2} } u^{n-m-2} du d\xi \\
		&\lesssim_{n,m} \int_{\gamma^2}^\infty \xi^{\frac{n-1}{2}} \left( 1 + \frac{1}{2 B^2 C_0 \sqrt{e}} \frac{\xi}{4 e} \right)^{- \frac{1}{2} - \frac{n+1}{2} } d\xi\\
		&\le \int_{0}^\infty \xi^{\frac{n-1}{2}} \left( 1 + \frac{1}{2 B^2 C_0 \sqrt{e}} \frac{\xi}{4 e} \right)^{- \frac{1}{2} - \frac{n+1}{2} } d\xi \\
		&\lesssim_{p,q, C_0}  1 
	\end{aligned} \end{equation*}
\end{proof}

\begin{lem}
	For any $l < k + B^2 \lambda_k$ and  $\nu \in (0,1)$, there exists $\Upsilon_Z \gg 1$ sufficiently large (depending on $p,q,k,l$ and $\nu$) such that
	\begin{gather*}
		\left| \int_{\tau_0}^{\tau} e^{B^2 \cl{D}_Z(\tau - \ol{\tau})} \chi_{[e^{- \alpha_k \ol{\tau}}, \Upsilon_Z e^{-\alpha_k \ol{\tau}} ]}(\ol{\gamma}) \Upsilon_U^{2l}\ol{\gamma}^{-4k-4B^2 \lambda_k  - 2} e^{2B^2 \lambda_k \ol{\tau} } d \ol{\tau}\right|(\gamma) \\
		\le \nu e^{B^2 \lambda_k \tau} \gamma^{-2k - 2B^2 \lambda_k} \\
	\end{gather*}
	for all $0 \le \tau - \tau_0 \le 1$ and $\gamma \ge \Upsilon_Z e^{-\alpha \tau}$.
\end{lem}
\begin{proof}
	\begin{equation*} \begin{aligned}
		& \quad \left| \int_{\tau_0}^{\tau} e^{B^2 \cl{D}_Z(\tau - \ol{\tau})} \chi_{[e^{- \alpha_k \ol{\tau}}, \Upsilon_Z e^{-\alpha_k \ol{\tau}} ]}(\ol{\gamma}) \Upsilon_U^{2l} \ol{\gamma}^{-4k-4B^2 \lambda_k  - 2} e^{2B^2 \lambda_k \ol{\tau} } d \ol{\tau}\right|(\gamma) \\
		&\le  \Upsilon_U^{2l} \left| \int_{\tau_0}^{\tau} e^{B^2 \cl{D}_Z(\tau - \ol{\tau})} \chi_{[0, e^\alpha \gamma ]}(\ol{\gamma}) \ol{\gamma}^{-4k-4B^2 \lambda_k  - 2} e^{2B^2 \lambda_k \ol{\tau} } d \ol{\tau}\right|(\gamma) \\	
		&\lesssim_{p,q,k}  \Upsilon_U^{2l} e^{2B^2 \lambda_k \tau} \gamma^{-4k -4B^2 \lambda_k}	&& (\text{lemma \ref{polyGrowthZ}}) \\
		&\le \Upsilon_U^{2l} \Upsilon_Z^{-2k-2B^2 \lambda_k} e^{B^2 \lambda_k \tau} \gamma^{-2k -2B^2 \lambda_k} \\
		&\le  \Upsilon_Z^{2l -2k -2B^2 \lambda_k} e^{B^2 \lambda_k \tau} \gamma^{-2k -2B^2 \lambda_k}		&& ( \Upsilon_U < \Upsilon_Z) \\
	\end{aligned} \end{equation*}
	Note that lemma \ref{polyGrowthZ} applies in this case since
		$$m = 4k + 4B^2 \lambda_k  = n-1 - \sqrt{(n-9)(n-1)} < n-1$$
\end{proof}

\begin{lem} \label{polyGrowthZBig}
	If $-2 < m < n-1$, then
	\begin{gather*}
		\left| \int_{\tau_0}^{\tau} e^{B^2 \cl{D}_Z (\tau - \ol{\tau})} \chi_{[4 \gamma, \infty)} (\ol{\gamma}) \ol{\gamma}^{-m-2} d\ol{\tau} \right| (\gamma) \lesssim_{p,q, m} \gamma^{-m} \\
	\text{for all } \quad 0 \le \tau - \tau_0 \le 1, \quad \gamma \ge \Upsilon_Z e^{- \alpha \tau} 
	\end{gather*}
\end{lem}
\begin{proof}
	The proof is similar to that of lemma \ref{polyGrowthUBig}.
	\begin{equation*} \begin{aligned}
		& \quad \left| \int_{\tau_0}^{\tau} e^{B^2 \cl{D}_Z (\tau - \ol{\tau})} \chi_{[4 \gamma, \infty)} (\ol{\gamma}) \ol{\gamma}^{-m-2} d\ol{\tau} \right| \\
		&\lesssim_{p,q}  \gamma^{2} \int_{\tau_0}^\tau ( \tau - \ol{\tau} )^{-1 - \frac{n+1}{2} } \\
		&\quad  \cdot \int_{4 \gamma}^\infty H_Z \ol{\gamma}^{n-m-2} d \ol{\gamma} d \ol{\tau} 	
		&& (\text{proposition } \ref{kernelZ})\\
		&\le  \gamma^2 \int_{\tau_0}^\tau ( \tau - \ol{\tau} )^{-1 - \frac{n+1}{2} } e^{ -\frac{1}{4 B^2 } \frac{c \gamma^2}{\tau - \ol{\tau} }} \\
		& \quad \cdot \int_{4 \gamma}^\infty e^{-\frac{1}{4 B^2 } \frac{c \ol{\gamma}^2}{\tau - \ol{\tau}} } \ol{\gamma}^{n-m-2} d \ol{\gamma} d \ol{\tau} \\
		& \qquad  \left(\text{where } u = \ol{\gamma} ( \tau - \ol{\tau})^{-1/2} \right) \\
		&=  \gamma^2 \int_{\tau_0}^\tau ( \tau - \ol{\tau} )^{-2 - \frac{m}{2} } e^{ -\frac{1}{4 B^2 } \frac{c \gamma^2}{\tau - \ol{\tau} }} \\
		& \quad \cdot \int_{4 \gamma ( \tau - \ol{\tau})^{-1/2}}^\infty e^{ -\frac{1}{4 B^2 } c u^2 } u^{n-m-2} du d \ol{\tau} \\
		& \qquad (\text{where } \xi = \gamma^2/ (\tau - \ol{\tau}) ) \\
		&\lesssim  \gamma^{-m} \int_{\gamma^2}^\infty \xi^{m/2} e^{ -\frac{1}{4 B^2 }c \xi} \int_{4 \sqrt{\xi} }^\infty e^{- \frac{1}{4B^2} c u^2} u^{n-m-2} du d \xi \\
		&\lesssim_{p,q,m}  \gamma^{-m} 		&& (-2 < m < n-1)\\
	\end{aligned} \end{equation*}
\end{proof}

\begin{lem}
	For any $\nu \in (0,1)$, there exists $\Upsilon_Z \gg 1$ sufficiently large/small (depending on $p,q,k$ and $\nu$) such that
	\begin{gather*}
		\left| \int_{\tau_0}^{\tau} e^{B^2 \cl{D}_Z(\tau - \ol{\tau})} \chi_{[ \Upsilon_Z e^{- \alpha_k \ol{\tau}}, M e^{\beta \ol{\tau}} ]}(\ol{\gamma}) \ol{\gamma}^{-4k-4B^2 \lambda_k - 2} e^{2B^2 \lambda_k \ol{\tau}} d \ol{\tau}\right|(\gamma) \le \nu e^{B^2 \lambda_k \tau}  \gamma^{-2k - 2B^2 \lambda_k}  \\
		 \text{for all }  \quad 0 \le \tau - \tau_0 \le 1, \quad  \Upsilon_Z e^{-\alpha \tau}  \le \gamma \le M e^{\beta \tau}
	\end{gather*}
\end{lem}
\begin{proof}
	The proof is similar to the proof of lemma \ref{shortTimeErrUParabNeg}.
	We begin by splitting the integral
	\begin{equation*} \begin{aligned}
		& \left| \int_{\tau_0}^{\tau} e^{B^2 \cl{D}_Z(\tau - \ol{\tau})} \chi_{[ \Upsilon_Z e^{- \alpha_k \ol{\tau}}, M e^{\beta \ol{\tau}} ]}(\ol{\gamma}) \ol{\gamma}^{-4k-4B^2 \lambda_k - 2} e^{2B^2 \lambda_k \ol{\tau}} d \ol{\tau}\right|(\gamma) \\
		\le & \left| \int_{\tau_0}^{\tau} e^{B^2 \cl{D}_Z(\tau - \ol{\tau})} \chi_{[ 0, 4 \gamma ]}(\ol{\gamma}) \ol{\gamma}^{-4k-4B^2 \lambda_k - 2} e^{2B^2 \lambda_k \ol{\tau}} d \ol{\tau}\right|(\gamma) \\
		& + \left| \int_{\tau_0}^{\tau} e^{B^2 \cl{D}_Z(\tau - \ol{\tau})} \chi_{[ 4 \gamma, \infty)}(\ol{\gamma}) \ol{\gamma}^{-4k-4B^2 \lambda_k - 2} e^{2B^2 \lambda_k \ol{\tau}} d \ol{\tau}\right|(\gamma) \\
		=& (I) + (II)
	\end{aligned} \end{equation*}
	Lemmas \ref{polyGrowthZ} and \ref{polyGrowthZBig} imply
		$$(I), (II) \lesssim_{p,q,k} e^{2B^2 \lambda_k \tau} \gamma^{-4k -4B^2 \lambda_k} \le \Upsilon_Z^{-2k-2B^2 \lambda_k} e^{B^2 \lambda_k \tau} \gamma^{-2k -2B^2 \lambda_k}$$
	Note that here
		$$m = 4k + 4B^2 \lambda_k = n-1 - \sqrt{(n-9)(n-1)} \in (0, n-1)$$
	so the assumptions of lemmas \ref{polyGrowthZ} and \ref{polyGrowthZBig} hold.
\end{proof}

\begin{lem}
	For any $\nu \in (0,1)$, $M > 0$, and $\beta \in  \left(0, \frac{1}{2} \right)$, there exists $\tau_0 \gg 1$ sufficiently large (depending on $p,q,k, \beta, M, \nu$) such that
	\begin{gather*}
		\left| \int_{\tau_0}^{\tau} e^{B^2 \cl{D}_Z(\tau - \ol{\tau})} \chi_{[ \Upsilon_Z e^{- \alpha_k \ol{\tau}}, M e^{\beta \ol{\tau}} +1 ]}(\ol{\gamma}) \ol{\gamma}^{-4B^2 \lambda_k } e^{2B^2 \lambda_k \ol{\tau}} d \ol{\tau}\right|(\gamma) \\
		\le \nu e^{B^2 \lambda_k \tau} \left( \gamma^{-2k - 2B^2 \lambda_k}  + \gamma^{-2B^2 \lambda_k} \right) \\
	\end{gather*}
 for all $0 \le \tau - \tau_0 \le 1$ and  $\Upsilon_Z e^{-\alpha \tau} \le \gamma \le M e^{\beta \tau}$.
\end{lem}
\begin{proof}
	Note that for $n \ge 10$, the choice of $k$ implies $-4B^2 \lambda_k \ge 2$.
	We use the maximal function bound from lemma \ref{maximalFuncBoundZ}.
	\begin{equation*} \begin{aligned}
		& \quad \left| \int_{\tau_0}^{\tau} e^{B^2 \cl{D}_Z(\tau - \ol{\tau})} \ol{\gamma}^{-4B^2 \lambda_k } e^{2B^2 \lambda_k \ol{\tau}} d \ol{\tau}\right|  \\
		&\lesssim_{p,q,k} e^{2B^2 \lambda \tau_0} \int_{\tau_0}^\tau \gamma^2 e^{- (\tau - \ol{\tau})}
		\frac{ \int_\gamma^\infty \ol{\gamma}^{-4B^2 \lambda - 2} \ol{\gamma}^{n+2} e^{-\ol{\gamma}^2 / (4B^2 )} d \ol{\gamma} }{ \int_\gamma^\infty  \ol{\gamma}^{n+2} e^{-\ol{\gamma}^2 / (4B^2 )} d \ol{\gamma} } d \ol{\tau}\\
		&\lesssim_{p,q,k}  e^{2B^2 \lambda \tau_0} \int_{\tau_0}^\tau \gamma^2 e^{- (\tau - \ol{\tau})} ( 1 + \gamma^{-4B^2 \lambda - 2} ) d \ol{\tau} 
		&& (\text{lemma }\ref{asympsOfExpInt})\\
		&\lesssim_{p,q,k}  e^{2B^2 \lambda \tau_0} (\gamma^2 + \gamma^{-4B^2 \lambda} ) \\
		&\lesssim  e^{B^2 \lambda \tau_0} e^{B^2 \lambda \tau} \gamma^{-2k -2B^2 \lambda} + M^{-2B^2 \lambda} e^{(1-2 \beta) B^2 \lambda \tau_0} e^{B^2 \lambda \tau} \gamma^{-2B^2 \lambda}
	\end{aligned} \end{equation*}
\end{proof}

We summarize the estimates in this subsection in the following statement:
\begin{lem} \label{summShortTimeEstErrZ}
	For any $\nu \in (0,1)$, there exists
		$\Upsilon_Z \gg 1$ sufficiently large depending on $p,q,k, l , \ul{D}, \nu$, and
		$\tau_0 \gg 1$ sufficiently large depending on $p,q,k, M , \beta, \ul{D}, \nu$
	such that $(\tl{Z}, \tl{U}) \in \mathcal{P}$ implies
	$$\left| \int_{\tau_0}^\tau e^{B^2 \mathcal{D}_Z ( \tau - \ol{\tau} ) } \grave{\chi} Err_Z(\ol{\tau} ) d \ol{\tau} \right| (\gamma) \le \nu e^{B^2 \lambda_k \tau} \left( \gamma^{-2k -2B^2 \lambda_k} + \gamma^{-2B^2 \lambda_k} \right)$$
	for all $\tau_0 \le \tau \le \tau_1 \le \tau_0 + 1$ and $\Upsilon_U e^{- \alpha_k \tau} \le \gamma \le M e^{ \beta \tau}$.
\end{lem}

\subsection{Estimates for $\mathcal{N}$ Contributions}

\begin{lem}
	For any $l$ and any $\nu \in (0,1)$, there exists
	$$\left( \frac{ \Upsilon_U^{1 + \frac{l}{n-2k-2B^2 \lambda_k -1}}}{\Upsilon_Z} \right) \ll 1$$
	sufficiently small (depending on $p,q,k,l, \nu$) such that
	\begin{gather*}
		\left| \int_{\tau_0}^{\tau} e^{B^2 \cl{D}_Z(\tau - \ol{\tau})} \chi_{[ 0, \Upsilon_U e^{- \alpha_k \ol{\tau}} ]}(\ol{\gamma}) \Upsilon_U^l \ol{\gamma}^{-2k-2B^2 \lambda_k-2 } e^{B^2 \lambda_k \ol{\tau}} d \ol{\tau}\right|(\gamma) \le \nu e^{B^2 \lambda_k \tau}  \gamma^{-2k - 2B^2 \lambda_k}  \\
		\text{for all }  \quad 0 \le \tau - \tau_0 \le 1, \quad \Upsilon_Z e^{-\alpha \tau} \le \gamma \le M e^{\beta \tau}
	\end{gather*}
\end{lem}
\begin{proof}
	By proposition \ref{kernelZ},
	\begin{equation*} \begin{aligned}
		& \left| \int_{\tau_0}^{\tau} e^{B^2 \cl{D}_Z(\tau - \ol{\tau})} \chi_{[ 0, \Upsilon_U e^{- \alpha_k \ol{\tau}} ]}(\ol{\gamma}) \Upsilon_U^l\ol{\gamma}^{-2k-2B^2 \lambda_k-2 } e^{B^2 \lambda_k \ol{\tau}} d \ol{\tau} \right|   \\
		\lesssim_{p,q,k} & \Upsilon_U^l e^{B^2 \lambda_k \tau} \gamma^2 \int_{\tau_0}^\tau (\tau - \ol{\tau})^{-1 - (n+1)/2}
		\int_0^{\Upsilon_U e^{- \alpha \ol{\tau}}} H_Z  \ol{\gamma}^{n-2k-2B^2 \lambda - 2} d \ol{\gamma} d \ol{\tau} 
	\end{aligned} \end{equation*}
	Without loss of generality, say $2 e^{\alpha} \le \frac{ \Upsilon_Z}{\Upsilon_U}$.
	Then there exists a universal constant $c> 0$ such that this quantity can in turn be estimated by
	\begin{equation*} \begin{aligned} 
		&\le \Upsilon_U^l e^{B^2 \lambda_k \tau} \gamma^2 
		\int_{\tau_0}^\tau (\tau - \ol{\tau})^{-1 - (n+1)/2} exp \left[ - \frac{c}{4B^2} \frac{ \gamma^2}{\tau  - \ol{\tau} } \right] \\
		& \quad \cdot \int_0^{\Upsilon_U e^{- \alpha \ol{\tau} } } \ol{\gamma}^{n - 2k -2B^2 \lambda -2} d \ol{\gamma} d \ol{\tau} \\
		&= \Upsilon_U^l e^{B^2 \lambda \tau} \gamma^2
		\int_{\tau_0}^\tau ( \tau - \ol{\tau})^{- 1 - (n+1)/2 + n/2 - k - B^2 \lambda - 1 + 1/2} exp \left[ - \frac{c}{4B^2} \frac{ \gamma^2}{\tau  - \ol{\tau} } \right] \\
		& \quad \cdot  \int_0^{\Upsilon_U e^{- \alpha \ol{\tau} } / \sqrt{ \tau - \ol{\tau} } } u^{n - 2k -2 B^2 \lambda - 2} du d \ol{\tau} \\
		& \qquad \left( \text{where } u = \ol{\gamma} ( \tau - \ol{\tau} )^{-1/2} \right) \\
		&\le  \Upsilon_U^l e^{B^2 \lambda \tau} \gamma^2
		\int_{\tau_0}^\tau ( \tau - \ol{\tau})^{-k - B^2 \lambda - 2 } exp \left[ - \frac{c}{4B^2} \frac{ \gamma^2}{\tau  - \ol{\tau} } \right] \\
		& \quad \cdot \int_0^{\Upsilon_U e^{- \alpha \tau_0 } / \sqrt{ \tau - \ol{\tau} } } u^{n - 2k -2 B^2 \lambda - 2} du d \ol{\tau} \\
		&\lesssim  \Upsilon_U^l e^{B^2 \lambda \tau} \gamma^2
		\int_{\gamma^2/ (\tau - \tau_0)}^\infty \left( \frac{ \gamma^2}{\xi} \right)^{-k - B^2 \lambda - 2} exp \left[ - \frac{ c}{4B^2} \xi \right] \frac{ \gamma^2}{\xi^2} \\
		& \quad \cdot \int_0^{ e^\alpha \frac{ \Upsilon_U}{\Upsilon_Z} \sqrt{\xi} } u^{n-2k -2B^2 \lambda - 2} du d \xi \\
		& \qquad (\text{where } \xi = \gamma^2 / ( \tau - \ol{\tau} ) ) \\
		&\le  \Upsilon_U^l e^{B^2 \lambda \tau}  \gamma^{-2k -2B^2 \lambda }
		\int_{\gamma^2}^\infty \xi^{k + B^2 \lambda} exp \left[ - \frac{ c}{4B^2} \xi \right] 
		\int_0^{ e^{\alpha} \frac{ \Upsilon_U}{\Upsilon_Z} \sqrt{\xi} } u^{n-2k -2B^2 \lambda - 2} du d \xi \\
		&\lesssim_{p,q,k} \Upsilon_U^l e^{B^2 \lambda \tau} \gamma^{-2k -2B^2 \lambda } \left( \frac{ \Upsilon_U}{\Upsilon_Z} \right)^{n - 2k -2B^2 \lambda -1}
		\int_{\gamma^2}^\infty \xi^{\frac{n}{2} -1} exp \left[ - \frac{c}{4B^2} \xi \right] d \xi \\
		&\lesssim_{p,q,k}  \Upsilon_U^l e^{B^2 \lambda \tau} \gamma^{-2k -2B^2 \lambda } \left( \frac{ \Upsilon_U}{\Upsilon_Z} \right)^{n - 2k -2B^2 \lambda -1}\\
		&=  e^{B^2 \lambda \tau} \gamma^{-2k -2B^2 \lambda } \left( \frac{ \Upsilon_U^{1 + \frac{l}{n-2k-2B^2 \lambda -1}}}{\Upsilon_Z} \right)^{n - 2k -2B^2 \lambda -1}
	\end{aligned} \end{equation*}
	Note that
		$$n - 2k -2B^2 \lambda -1 = \frac{ n+1}{2} + \frac{1}{2} \sqrt{ (n-9)(n-1) } - 1 > 0$$
	so that this quantity can be made arbitrarily small by making 
	$$\frac{ \Upsilon_U^{1 + \frac{l}{n-2k-2B^2 \lambda -1}}}{\Upsilon_Z} \ll 1$$
\end{proof}

\begin{lem}
	For any $\nu \in (0,1)$, there exists $\eta_1^U \ll 1$ sufficiently small (depending on $p,q,k, \nu$) such that
	\begin{gather*}
		\left| \int_{\tau_0}^{\tau} e^{B^2 \cl{D}_Z(\tau - \ol{\tau})} \chi_{[ \Upsilon_U e^{- \alpha_k \ol{\tau}} , M e^{\beta \ol{\tau}} ]}(\ol{\gamma})  \eta_1^U  \ol{\gamma}^{-2k-2B^2 \lambda_k-2 }  e^{B^2 \lambda_k \ol{\tau}} d \ol{\tau}\right|(\gamma) \le \nu e^{B^2 \lambda_k \tau}  \gamma^{-2k - 2B^2 \lambda_k}  \\
		\text{for all }  \quad 0 \le \tau - \tau_0 \le 1, \quad \Upsilon_Z e^{-\alpha \tau} \le \gamma \le M e^{\beta \tau}
	\end{gather*}
\end{lem}
\begin{proof}
	\begin{equation*} \begin{aligned}
		& \left| \int_{\tau_0}^{\tau} e^{B^2 \cl{D}_Z(\tau - \ol{\tau})} \chi_{[ \Upsilon_U e^{- \alpha_k \ol{\tau}} , M e^{\beta \ol{\tau}} ]}(\ol{\gamma})  \eta_1^U  \ol{\gamma}^{-2k-2B^2 \lambda_k-2 }  e^{B^2 \lambda_k \ol{\tau}} d \ol{\tau}\right|(\gamma) \\
		\le & \left| \int_{\tau_0}^{\tau} e^{B^2 \cl{D}_Z(\tau - \ol{\tau})} \chi_{[ 0, 4 \gamma ]}(\ol{\gamma})  \eta_1^U  \ol{\gamma}^{-2k-2B^2 \lambda_k-2 }  e^{B^2 \lambda_k \ol{\tau}} d \ol{\tau}\right|(\gamma) \\
		& + \left| \int_{\tau_0}^{\tau} e^{B^2 \cl{D}_Z(\tau - \ol{\tau})} \chi_{[ 4 \gamma, \infty )}(\ol{\gamma})  \eta_1^U  \ol{\gamma}^{-2k-2B^2 \lambda_k-2 }  e^{B^2 \lambda_k \ol{\tau}} d \ol{\tau}\right|(\gamma) \\
		=& (I) + (II)
	\end{aligned} \end{equation*}
	By lemmas \ref{polyGrowthZ} and \ref{polyGrowthZBig},
		$$(I), (II) \lesssim_{p,q,k} \eta_1^U e^{B^2 \lambda_k \tau} \gamma^{-2k -2B^2 \lambda_k}$$
\end{proof}

\begin{lem}
	For any $\nu \in (0,1)$, there exists $\eta_1^U \ll 1$ sufficiently small (depending on $p,q,k, \nu$) such that
	\begin{gather*}
		\left| \int_{\tau_0}^{\tau} e^{B^2 \cl{D}_Z(\tau - \ol{\tau})} \chi_{[ \Upsilon_U e^{- \alpha_k \ol{\tau}} , M e^{\beta \ol{\tau}} ]}(\ol{\gamma})  \eta_1^U  \ol{\gamma}^{-2B^2 \lambda_k-1 }  e^{B^2 \lambda_k \ol{\tau}} d \ol{\tau}\right|(\gamma) \\
		\le \nu e^{B^2 \lambda_k \tau}  \left( \gamma^{-2k - 2B^2 \lambda_k} + \gamma^{-2B^2 \lambda_k} \right)  
	\end{gather*}
	 for all  $ 0 \le \tau - \tau_0 \le 1$ and  $\Upsilon_Z e^{-\alpha \tau} \le \gamma \le M e^{\beta \tau}$.
\end{lem}
\begin{proof}
	If $-2B^2 \lambda - 1 \ge 2$, then we use the maximal function estimate of lemma \ref{maximalFuncBoundZ} to deduce
	\begin{equation*} \begin{aligned}
		&  \quad \left| \int_{\tau_0}^{\tau} e^{B^2 \cl{D}_Z(\tau - \ol{\tau})} \chi_{[ \Upsilon_U e^{- \alpha_k \ol{\tau}} , M e^{\beta \ol{\tau}} ]}(\ol{\gamma})  \eta_1^U  \ol{\gamma}^{-2B^2 \lambda_k-1 }  e^{B^2 \lambda_k \ol{\tau}} d \ol{\tau}\right| \\
		&\le  \left| \int_{\tau_0}^{\tau} e^{B^2 \cl{D}_Z(\tau - \ol{\tau})} \eta_1^U  \ol{\gamma}^{-2B^2 \lambda_k-1 }  e^{B^2 \lambda_k \ol{\tau}} d \ol{\tau}\right| \\
		&\lesssim_{p,q,k}  \eta_1^U e^{B^2 \lambda_k \tau} \int_{\tau_0}^\tau \gamma^2 e^{- (\tau - \ol{\tau})}
		\frac{ \int_\gamma^\infty \ol{\gamma}^{-2B^2 \lambda - 3} \ol{\gamma}^{n+2} e^{- \ol{\gamma}^2/ (4B^2) } d \ol{\gamma} }{ \int_\gamma^\infty \ol{\gamma}^{n+2} e^{- \ol{\gamma}^2/ (4B^2) } d \ol{\gamma} } \\
		&\lesssim_{p,q,k}   \eta_1^U e^{B^2 \lambda_k \tau} \int_{\tau_0}^\tau \gamma^2 ( 1 + \gamma^{-2B^2 \lambda_k -3} ) d \ol{\tau} 
		&& (\ref{asympsOfExpInt})\\
		&\le  \eta_1^U e^{B^2 \lambda_k \tau} ( \gamma^2 + \gamma^{-2B^2 \lambda -1 }) \\
		&\lesssim_{p,q,k}  \eta_1^U e^{B^2 \lambda_k \tau} ( \gamma^{-2k -2B^2 \lambda} + \gamma^{-2B^2 \lambda} )
	\end{aligned} \end{equation*}
	
	If $-2B^2 \lambda - 1 < 2$, then lemma \ref{maximalFuncBoundZ} implies
	\begin{equation*} \begin{aligned}
		& \quad \left| \int_{\tau_0}^{\tau} e^{B^2 \cl{D}_Z(\tau - \ol{\tau})} 
		\chi_{[ \Upsilon_U e^{- \alpha_k \ol{\tau}} , M e^{\beta \ol{\tau}} ]}(\ol{\gamma})  
		\eta_1^U  \ol{\gamma}^{-2B^2 \lambda_k-1 }  e^{B^2 \lambda_k \ol{\tau}} d \ol{\tau}\right| \\
		&\le  \left| \int_{\tau_0}^{\tau} e^{B^2 \cl{D}_Z(\tau - \ol{\tau})} 
		\eta_1^U  \ol{\gamma}^{-2B^2 \lambda_k-1 }  e^{B^2 \lambda_k \ol{\tau}} d \ol{\tau}\right| \\
		&\lesssim_{p,q,k}  \eta_1^U e^{B^2 \lambda_k \tau} \int_{\tau_0}^\tau \gamma^2 e^{- (\tau - \ol{\tau})}
		\frac{ \int_0^\gamma \ol{\gamma}^{-2B^2 \lambda - 3} \ol{\gamma}^{n+2} e^{- \ol{\gamma}^2/ (4B^2) } d \ol{\gamma} }{ \int_0^\gamma \ol{\gamma}^{n+2} e^{- \ol{\gamma}^2/ (4B^2) } d \ol{\gamma} } d \ol{\tau}\\
		& \lesssim_{p,q,k}  \eta_1^U e^{B^2 \lambda_k \tau} \int_{\tau_0}^\tau \gamma^2 e^{- (\tau - \ol{\tau})} ( \gamma^{-2B^2 \lambda - 3} + 1) d \ol{\tau}	\\
		&\le  \eta_1^U e^{B^2 \lambda_k \tau} ( \gamma^{-2B^2 \lambda-1} + \gamma^2 ) \\
		&\lesssim_{p,q,k}  \eta_1^U e^{B^2 \lambda_k \tau} ( \gamma^{-2k -2B^2 \lambda_k} + \gamma^{-2B^2 \lambda_k} ) 
		&& (\text{since } k \ge 1)\\
	\end{aligned} \end{equation*}
\end{proof}

\begin{lem}
	For any $\nu \in (0,1)$, $M, \Gamma > 0$, and $\beta \in \left(0, \frac{1}{2} \right)$, there exists $\tau_0 \gg 1$ sufficiently large (depending on $p, q, k, M, \beta, \Gamma, \nu$) such that
	\begin{gather*}
		\left| \int_{\tau_0}^{\tau} e^{B^2 \cl{D}_Z(\tau - \ol{\tau})} \chi_{[ M e^{\beta \ol{\tau}}, M e^{\beta \ol{\tau}}+1 ]}(\ol{\gamma})   \ol{\gamma}^{-2B^2 \lambda_k-1 }  e^{B^2 \lambda_k \ol{\tau}} d \ol{\tau}\right|(\gamma) \\
		\le \nu e^{B^2 \lambda_k \tau}  \left( \gamma^{ -2k - 2B^2 \lambda_k}   + \gamma^{-2B^2 \lambda_k } \right)
	\end{gather*}
	for all $0 \le \tau - \tau_0 \le 1$ and $\Upsilon_Z e^{-\alpha \tau} \le \gamma \le \Gamma$.
\end{lem}
\begin{proof}
	\begin{equation*} \begin{aligned}
		& \left| \int_{\tau_0}^{\tau} e^{B^2 \cl{D}_Z(\tau - \ol{\tau})} \chi_{[ M e^{\beta \ol{\tau}}, M e^{\beta \ol{\tau}}+1 ]}(\ol{\gamma})   \ol{\gamma}^{-2B^2 \lambda_k-1 }  e^{B^2 \lambda_k \ol{\tau}} d \ol{\tau}\right|(\gamma) \\
		\lesssim_{p,q,k} & e^{B^2 \lambda_k \tau} M^{-2} e^{-2 \beta \tau_0} \int_{\tau_0}^{\tau} \left|  e^{B^2 \cl{D}_Z(\tau - \ol{\tau})} \chi_{[ M e^{\beta \ol{\tau}}, M e^{\beta \ol{\tau}}+1 ]}(\ol{\gamma})   \ol{\gamma}^{-2B^2 \lambda_k+1 }  \right| d \ol{\tau} \\
	\end{aligned} \end{equation*}
	The reader is now referred to the proof of lemma \ref{shortTimeEstCommZ}, where it is shown that this quantity is bounded by
	\begin{equation*} \begin{aligned}
		\lesssim_{p,q,k} & M^{-2} e^{-2 \beta \tau_0} e^{- C e^{2 \beta \tau_0} } e^{B^2 \lambda \tau} \gamma^2 \\
		\lesssim_{p,q,k} & M^{-2} e^{-2 \beta \tau_0} e^{-C e^{2 \beta \tau_0}} ( 1 + \Gamma^{2 + 2B^2 \lambda} ) e^{B^2 \lambda \tau} ( \gamma^{-2k - 2B^2 \lambda} + \gamma^{-2B^2 \lambda} )
	\end{aligned} \end{equation*}
	Taking $\tau_0 \gg 1$ sufficiently large completes the proof.
\end{proof}

We summarize the estimates of this subsection in the following lemma.

\begin{lem} \label{summShortTimeEstN}
	For any $\nu \in (0,1)$ there exists
	$\Upsilon_Z \gg 1$ sufficiently large depending on $p,q,k, l , \ul{D}, \Upsilon_U, \nu$
	and $\tau_0 \gg 1$ sufficiently large depending on $p,q,k, M, \beta, \Gamma, \ul{D}, \nu$ such that
	$(\tl{Z}, \tl{U}) \in \mathcal{P}$ implies
	\begin{gather*}
		\left| \int_{\tau_0}^\tau e^{B^2 \cl{D}_Z ( \tau - \ol{\tau} ) } 
		\left ( \grave{\chi} B^2 \mathcal{N} \tl{U} - B^2 \mathcal{N} e^{B^2 \lambda \ol{\tau}} \tl{U}_{\lambda} \right)
		d \ol{\tau} \right| \\
		\le \nu e^{B^2 \lambda_k \tau} \left( \gamma^{-2k - 2B^2 \lambda_k} + \gamma^{-2B^2 \lambda_k} \right) 
		+ C_{p,q} \eta_1^U e^{B^2 \lambda_k \tau} \left( \gamma^{-2k - 2B^2 \lambda_k} + \gamma^{-2B^2 \lambda_k} \right) 
	\end{gather*}
	for all $\tau_0 \le \tau \le \tau_1 \le \tau_0 + 1$ and $\Upsilon_Z e^{- \alpha \tau} \le \gamma \le \Gamma$.
\end{lem}

\subsection{Estimates for Commutator Contributions}

\begin{lem} \label{shortTimeEstCommU}
	For any $\nu \in (0,1)$ and $\Gamma > 0$, there exists $\tau_0 \gg 1$ sufficiently large (depending on $p, q, k, M, \beta, \Gamma, \nu$) such that
	\begin{gather*}
		\left| \int_{\tau_0}^{\tau} e^{B^2 \cl{D}_U(\tau - \ol{\tau})} \chi_{[ M e^{\beta \ol{\tau}}, M e^{\beta \ol{\tau}} +1 ]}(\ol{\gamma}) \ol{\gamma}^{-2B^2 \lambda_k + 1} e^{B^2 \lambda_k \ol{\tau}} d \ol{\tau}\right|(\gamma) \le \nu e^{B^2 \lambda_k \tau}  \gamma^{-2k - 2B^2 \lambda_k} \\
		\text{for all }  \quad 0 \le \tau - \tau_0 \le 1, \quad \Upsilon_U e^{-\alpha \tau} \le \gamma \le \Gamma
	\end{gather*}
\end{lem}
\begin{proof}
	By proposition \ref{kernelU},
	\begin{equation*} \begin{aligned}
		& \left| \int_{\tau_0}^{\tau} e^{B^2 \cl{D}_U (\tau - \ol{\tau})} \chi_{[M e^{\beta \ol{\tau}}, M e^{\beta \ol{\tau}} + 1]}(\ol{\gamma}) \ol{\gamma}^{-2B^2 \lambda_k +1} e^{B^2 \lambda_k \ol{\tau}} d \ol{\tau} \right| \\
		\lesssim_{p,q,k} & e^{B^2 \lambda_k \tau} \gamma^{-2k -2 B^2 \lambda_k}  \int_{\tau_0}^{\tau} (\tau - \ol{\tau})^{-1 - \frac{1}{2} \sqrt{(n-9)(n-1)}}\\
		& \cdot \int_{M e^{\beta \ol{\tau}}}^\infty  \ol{\gamma}^{-2B^2 \lambda_k +1}  H_U(\gamma, \ol{\gamma}; \tau - \ol{\tau}) \ol{\gamma}^{\frac{n+1}{2} + \frac{1}{2} \sqrt{(n-9)(n-1)}} d \ol{\gamma} d \ol{\tau} \\
	\end{aligned} \end{equation*}	
	If $\tau_0 \gg 1$ is sufficiently large depending on $M, \beta, \Gamma$, then 
	$$\gamma \le \Gamma \text{ and } \ol{\gamma} \ge M e^{\beta \ol{\tau}} \ge M e^{\beta \tau_0} \implies ( e^{- (\tau - \ol{\tau})/2} \gamma - \ol{\gamma} )^2 \gtrsim e^{2 \beta \ol{\tau}}$$
	Hence,
	\begin{equation*} \begin{aligned}
		H_U( \gamma, \ol{\gamma} ; \tau - \ol{\tau} ) \le & exp \left( - \frac{1}{8B^2} \frac{C e^{2 \beta \ol{\tau}} }{(\tau - \ol{\tau})} \right) exp \left( - \frac{1}{8B^2} \frac{\ol{\gamma}^2 }{(\tau - \ol{\tau})} \right) \\
	\end{aligned} \end{equation*}
	From this estimate on $H_U$, it follows that
	\begin{equation*} \begin{aligned}
		& \quad e^{B^2 \lambda_k \tau} \gamma^{-2k -2 B^2 \lambda_k} \int_{\tau_0}^{\tau}  (\tau - \ol{\tau})^{-1 - \frac{1}{2} \sqrt{(n-9)(n-1)}}\\
		& \quad \cdot \int_{M e^{\beta \ol{\tau}}}^\infty  \ol{\gamma}^{-2B^2 \lambda_k +1}  H_U(\gamma, \ol{\gamma}; \tau - \ol{\tau}) \ol{\gamma}^{\frac{n+1}{2} + \frac{1}{2} \sqrt{(n-9)(n-1)}} d \ol{\gamma} d \ol{\tau} \\
		&\le  e^{B^2 \lambda_k \tau} \gamma^{-2k -2 B^2 \lambda_k} \int_{\tau_0}^{\tau}  (\tau - \ol{\tau})^{-1 - \frac{1}{2} \sqrt{(n-9)(n-1)}}	exp \left( - C \frac{e^{2 \beta \ol{\tau}} }{\tau - \ol{\tau}} \right)\\
		& \quad  \cdot \int_{M e^{\beta \ol{\tau}}}^\infty e^{-\frac{1}{8B^2} \ol{\gamma}^2/ (\tau - \ol{\tau})} \ol{\gamma}^{-2B^2 \lambda_k +1 + \frac{n+1}{2} + \frac{1}{2} \sqrt{(n-9)(n-1)}} d \ol{\gamma} d \ol{\tau} \\
		&=  e^{B^2 \lambda_k \tau} \gamma^{-2k -2 B^2 \lambda_k} \int_{\tau_0}^{\tau}  (\tau - \ol{\tau})^{-1 - \frac{1}{2} \sqrt{(n-9)(n-1)}+ \frac{1}{2} 	+ \frac{n+1}{4} + \frac{1}{4} \sqrt{(n-9)(n-1)} - B^2 \lambda_k + \frac{1}{2}} \\
		& \quad \cdot exp \left( - C \frac{e^{2 \beta \ol{\tau}} }{\tau - \ol{\tau}} \right)
		\int_0^\infty e^{-\frac{1}{8B^2} u^2} u^{-2B^2 \lambda_k +1 + \frac{n+1}{2} + \frac{1}{2} \sqrt{(n-9)(n-1)}} d u d \ol{\tau} \\
		& \qquad (\text{where } u = \ol{\gamma}/\sqrt{\tau - \ol{\tau}} )\\
		& \lesssim_{p,q,k} e^{B^2 \lambda_k \tau} \gamma^{-2k -2 B^2 \lambda_k} \\
		& \quad \cdot \int_{\tau_0}^{\tau}  (\tau - \ol{\tau})^{-1 - \frac{1}{2} \sqrt{(n-9)(n-1)}+ \frac{1}{2} + \frac{n+1}{4} + \frac{1}{4} \sqrt{(n-9)(n-1)} - B^2 \lambda_k + \frac{1}{2}}
		exp \left( - C e^{2 \beta \ol{\tau}}  \right) d \ol{\tau} \\
		&=  e^{B^2 \lambda_k \tau} \gamma^{-2k -2 B^2 \lambda_k} \int_{\tau_0}^{\tau}  (\tau - \ol{\tau})^{k + \frac{1}{2}}
		exp \left( - C e^{2 \beta \ol{\tau}}  \right) d \ol{\tau} \\
		&\lesssim_k  e^{B^2 \lambda_k \tau} \gamma^{-2k -2 B^2 \lambda_k} e^{-C e^{2 \beta \tau_0} } \\
	\end{aligned} \end{equation*}
	Taking $\tau_0 \gg 1$, the estimate follows.
\end{proof}

\begin{lem} \label{shortTimeEstCommZ}
	For any $\nu \in (0,1)$ and $\Gamma > 0$, there exists $\tau_0 \gg 1$ sufficiently large (depending on $p,q,k, M, \beta, \Gamma, \nu$) such that
	\begin{gather*}
		\left| \int_{\tau_0}^{\tau} e^{B^2 \cl{D}_Z(\tau - \ol{\tau})} \chi_{[ M e^{\beta \ol{\tau}}, M e^{\beta \ol{\tau}} +1 ]}(\ol{\gamma}) \ol{\gamma}^{-2B^2 \lambda_k + 1} e^{B^2 \lambda_k \ol{\tau}} d \ol{\tau}\right|(\gamma) \\
		\le \nu e^{B^2 \lambda_k \tau} \left( \gamma^{-2k - 2B^2 \lambda_k}  + \gamma^{-2B^2 \lambda_k} \right) 
	\end{gather*}
	for all $ 0 \le \tau - \tau_0 \le 1$ and  $\Upsilon_Z e^{-\alpha \tau} \le \gamma \le \Gamma$.
\end{lem}
\begin{proof}
	The proof is similar to the proof of the previous lemma.
	\begin{equation*} \begin{aligned}
		&\quad  \left| \int_{\tau_0}^{\tau} e^{B^2 \cl{D}_Z(\tau - \ol{\tau})} \chi_{[ M e^{\beta \ol{\tau}}, M e^{\beta \ol{\tau}} +1 ]}(\ol{\gamma}) \ol{\gamma}^{-2B^2 \lambda_k + 1} e^{B^2 \lambda_k \ol{\tau}} d \ol{\tau}\right| \\
		& \lesssim_{p,q,k}  e^{B^2 \lambda \tau} \gamma^2 \int_{\tau_0}^\tau ( \tau - \ol{\tau})^{-1 - \frac{n+1}{2} }
		\int_{M e^{\beta \ol{\tau} }}^\infty H_Z \ol{\gamma}^{n - 2B^2 \lambda + 1} d \ol{\gamma} d \ol{\tau} \\
		&\le  e^{B^2 \lambda \tau} \gamma^2 \int_{\tau_0}^\tau ( \tau - \ol{\tau})^{- 1 - \frac{n+1}{2} } exp \left[ -C \frac{ e^{2 \beta \ol{\tau}}}{\tau - \ol{\tau}} \right]
		\int_{M e^{\beta \ol{\tau}}}^\infty  exp \left[ -C \frac{ \ol{\gamma}^2}{\tau - \ol{\tau}} \right] \ol{\gamma}^{n -2B^2 \lambda + 1} d \ol{\gamma} d \ol{\tau}\\
		&\le  e^{B^2 \lambda \tau} \gamma^2 e^{- C e^{2 \beta \tau_0} }	 \int_{\tau_0}^\tau ( \tau - \ol{\tau} )^{-B^2 \lambda - \frac{1}{2} }
		\int_{M e^{\beta \ol{\tau}}/ \sqrt{ \tau - \ol{\tau}}}^\infty exp [ - C u^2] u^{n-2B^2 \lambda + 1} du d \ol{\tau} \\
		& \qquad  \left(\text{where } u = \ol{\gamma} / \sqrt{ \tau - \ol{\tau}} \right) \\
		& \lesssim_{p,q,k}  e^{B^2 \lambda \tau} \gamma^2 e^{- C e^{2 \beta \tau_0} }
	\end{aligned} \end{equation*}
	If $2 \le -2B^2 \lambda$ then 
		$$e^{B^2 \lambda \tau} \gamma^2 e^{- C e^{2 \beta \tau_0} } \lesssim_{p,q,k}  e^{- C e^{2 \beta \tau_0} }  e^{B^2 \lambda \tau} \left( \gamma^{-2k -2B^2 \lambda_k} + \gamma^{-2B^2 \lambda_k} \right)$$
	If $2 > -2B^2 \lambda$ then
		$$e^{B^2 \lambda \tau} \gamma^2 e^{- C e^{2 \beta \tau_0} } \le e^{- C e^{2 \beta \tau_0} } \Gamma^{2 + 2B^2 \lambda} e^{B^2 \lambda_k \tau} \gamma^{-2B^2 \lambda} $$
	In either case, the estimate follows from taking $\tau_0 \gg 1$ sufficiently large.
\end{proof}

%

%% file: APrioriLongTimeEsts2.tex
For $\Gamma > 0$ arbitrary, let $\Sigma_{\Upsilon, \Gamma}$ denote the spacetime region
		$$\Sigma_{\Upsilon, \Gamma} \doteqdot \{ (\gamma, \tau) : \quad \Upsilon e^{- \alpha_k \tau} \le \gamma \le \Gamma, \quad \tau_0 + 1 \le \tau \le \tau_1 \}$$
	
	In this section, it is assumed that
		$${P}_{\tau_0, \tau_1} ( \underline{p}, \underline{q}) = 0,$$
	that
		$$( \tl{Z}_{\ul{p}, \ul{q}}, \tl{U}_{\ul{p}, \ul{q}} ) \in \mathcal{P}[ \ul{D}, l, \kappa, \ul{\eta}, \Upsilon_{U,Z}, M, \alpha_k, \beta, \lambda_k ],$$
	and that there exists $\delta > 0$ such that
	\begin{gather*}
		| \underline{p} | \le C_{p,q,k, \ul{D}, \Upsilon_U, \Upsilon_Z, \ul{\Upsilon}} e^{B_\infty^2 \lambda_k \tau_0 - \delta \tau_0}		\\
		  | \underline{q} | \le
		  C_{p,q,k, \ul{D}, \Upsilon_U, \Upsilon_Z, \ul{\Upsilon}} e^{B_\infty^2 \lambda_k \tau_0 - \delta \tau_0}
		  + C_{p,q,k} \eta_1^U e^{B^2 \lambda_k \tau_0}\\
	\end{gather*}
	Such assumptions are justified by the results of sections \ref{PointwiseEst} and \ref{CoeffEst}.
	
The goal of this section is to prove the bound from the previous section without the short-time assumption and only in $\Sigma_{\Upsilon, \Gamma}$, with the understanding that the barriers from lemmas \ref{parabOuterBarriersU} - \ref{parabOuterInteriorEstZ+} will be used to control the solution for $\gamma > \Gamma$.
Note that $\tl{U}(\gamma, \tau) = \grave{U}( \gamma, \tau)$ for all $(\gamma, \tau) \in \Sigma_{\Upsilon, \gamma}$ if $\tau_0 \gg 1$ is sufficiently large depending on $M, \beta, \Gamma$.

We begin with a specific way to write the solution.
Adjacent to each term, we indicate the lemma in which that term is estimated.
\begin{lem}
	For ${P}_{\tau_0, \tau_1} (\ul{p}, \ul{q}) = 0$ and $\tau \in [\tau_0, \tau_1]$,
	\begin{equation*} \begin{aligned}
		&\grave{U}_{\ul{p}, \ul{q}}(\tau) -e^{B^2 \lambda_k  \tau} \tl{U}_{\lambda_k} \\ =
		& e^{B^2 \lambda_k \tau} \left\langle \check{\grave{U}}_{\ul{p}, \ul{q}} (\tau_0) - \tl{U}_{\lambda_k} , \tl{U}_{\lambda_k} \right\rangle_{n, \frac{1}{2B^2}} \tl{U}_{\lambda_k} 
		&& (\ref{domModeInitU})\\
		& + \sum_{j = k+1}^\infty e^{B^2 \lambda_j (\tau - \tau_0)} e^{B^2 \lambda_k \tau_0} \left\langle \check{\grave{U}}_{\ul{p}, \ul{q}}(\tau_0) , \tl{U}_{\lambda_j} \right\rangle_{n, \frac{1}{2B^2}} \tl{U}_{\lambda_j} 
		&& (\ref{fastModeInitU}) \\
		& + \sum_{j = k }^\infty \tl{U}_{\lambda_j} \int_{\tau_0}^\tau e^{B^2 \lambda_j (\tau - \ol{\tau})} \left \langle \grave{\chi}  Err_U (\tl{Z}, \tl{U} ) (\ol{\tau}), \tl{U}_{\lambda_j} \right \rangle_{n, \frac{1}{2B^2}} d \ol{\tau} 
		&& (\ref{fastModeErrU})\\
		& + \sum_{j = k }^\infty \tl{U}_{\lambda_j} \int_{\tau_0}^\tau e^{B^2 \lambda_j (\tau - \ol{\tau})} \left \langle B^2 [ \grave{\chi}, \mathcal{D}_U ] \tl{U} + \grave{\chi}_\tau \tl{U}, \tl{U}_{\lambda_j} \right \rangle_{n, \frac{1}{2B^2}} d \ol{\tau}
		&& (\ref{fastModeCommU}) \\
		& - \sum_{j = 0}^k \tl{U}_{\lambda_j} \int_{\tau}^{\tau_1} e^{B^2 \lambda_j (\tau - \ol{\tau})} \left \langle \grave{\chi}  Err_U (\tl{Z}, \tl{U} )(\ol{\tau}), \tl{U}_{\lambda_j} \right \rangle_{n, \frac{1}{2B^2}} d \ol{\tau} 
		&& (\ref{slowModeErrU}) \\
		& - \sum_{j = 0}^k \tl{U}_{\lambda_j} \int_{\tau}^{\tau_1} e^{B^2 \lambda_j (\tau - \ol{\tau})} \left \langle B^2 [ \grave{\chi}, \mathcal{D}_U ] \tl{U} + \grave{\chi}_\tau \tl{U}(\ol{\tau}), \tl{U}_{\lambda_j} \right \rangle_{n, \frac{1}{2B^2}} d \ol{\tau} 
		&& (\ref{slowModeCommU})\\
	\end{aligned} \end{equation*}
\end{lem}
\begin{proof}
	$P_{\tau_0, \tau_1}( \ul{p}, \ul{q} ) = 0$ and the variation of constants formula imply that
	\begin{equation*}\begin{aligned}
		&\left \langle \grave{U}(\tau_0) , \tl{U}_{\lambda_j} \right \rangle_{n, \frac{1}{2B^2}}  \\
		=& - \int_{\tau_0}^{\tau_1} e^{B^2 \lambda_j ( \tau_0 - \ol{\tau} ) } \left \langle \grave{\chi}  Err_U (\tl{Z}, \tl{U} ) + B^2 [ \grave{\chi}, \mathcal{D}_U ] \tl{U} + \grave{\chi}_\tau \tl{U} ( \ol{\tau}) , \tl{U}_{\lambda_j} \right \rangle_{n, \frac{1}{2B^2}} d \ol{\tau}
	\end{aligned} \end{equation*}
	Using this formula in the variation of constants formula yields the result.
\end{proof}
		
Similarly, we have the following formula for $\grave{Z}$
\begin{lem}
	For ${P}_{\tau_0, \tau_1} (\ul{p}, \ul{q}) = 0$ and $\tau \in [\tau_0, \tau_1]$,
	\begin{equation*} \begin{aligned}
		& \grave{Z}_{\ul{p}, \ul{q}} (\tau) - e^{B^2 \lambda_k \tau} \tl{Z}_{\lambda_k} \\
		= & \sum_{j = K+1}^\infty e^{B^2 h_j (\tau - \tau_0)} e^{B^2 \lambda_k \tau_0} \left \langle \check{\grave{Z}}(\tau_0) -  \tl{Z}_{\lambda_k}, \tl{Z}_{h_j} \right \rangle_{n-2, \frac{1}{2B^2}} \tl{Z}_{h_j} 
		&& (\ref{fastModeInitZ}) \\
		&+ \sum_{j = K+1}^\infty \tl{Z}_{h_j} \int_{\tau_0}^\tau e^{B^2 h_j(\tau - \ol{\tau})} \left \langle \grave{\chi} B^2 \cl{N} \tl{U} - B^2 \cl{N} e^{B^2 \lambda_k \ol{\tau} } \tl{U}_{\lambda_k}, \tl{Z}_{h_j} \right \rangle_{n-2, \frac{1}{2B^2}} d \ol{\tau} 
		&& (\ref{fastModeN})\\
		&+ \sum_{j = K+1}^\infty \tl{Z}_{h_j} \int_{\tau_0}^\tau e^{B^2 h_j(\tau - \ol{\tau})} \left \langle \grave{\chi}  Err_Z (\tl{Z}, \tl{U} ) (\ol{\tau}), \tl{Z}_{h_j} \right \rangle_{n-2, \frac{1}{2B^2}} d \ol{\tau} 
		&& (\ref{fastModeErrZ})\\
		&+ \sum_{j = K+1}^\infty \tl{Z}_{h_j} \int_{\tau_0}^\tau e^{B^2 h_j(\tau - \ol{\tau})} \left \langle B^2 [ \grave{\chi}, \mathcal{D}_Z ] \tl{Z} + \grave{\chi}_\tau \tl{Z}(\ol{\tau}), \tl{Z}_{h_j} \right \rangle_{n-2, \frac{1}{2B^2}} d \ol{\tau} 
		&& (\ref{fastModeCommZ})\\
		& - \sum_{j = 0}^K \tl{Z}_{h_j} \int_{\tau}^{\tau_1} e^{B^2 h_j (\tau - \ol{\tau} ) } \left \langle \grave{\chi} B^2 \cl{N} \tl{U} - B^2 \cl{N} e^{B^2 \lambda_k \ol{\tau} } \tl{U}_{\lambda_k} , \tl{Z}_{h_j} \right \rangle_{n-2, \frac{1}{2B^2}} d \ol{\tau} 
		&& (\ref{slowModeN})\\
		& - \sum_{j = 0}^K \tl{Z}_{h_j} \int_{\tau}^{\tau_1} e^{B^2 h_j (\tau - \ol{\tau} ) } \left \langle \grave{\chi}  Err_Z (\tl{Z}, \tl{U} )( \ol{\tau}) , \tl{Z}_{h_j} \right \rangle_{n-2, \frac{1}{2B^2}} d \ol{\tau} 
		&& (\ref{slowModeErrZ})\\
		& - \sum_{j = 0}^K \tl{Z}_{h_j} \int_{\tau}^{\tau_1} e^{B^2 h_j (\tau - \ol{\tau} ) } \left \langle B^2 [ \grave{\chi}, \mathcal{D}_Z ] \tl{Z} + \grave{\chi}_\tau \tl{Z}( \ol{\tau}) , \tl{Z}_{h_j} \right \rangle_{n-2, \frac{1}{2B^2}} d \ol{\tau} 
		&& (\ref{slowModeCommZ})\\
	\end{aligned} \end{equation*} 
\end{lem}
\begin{proof}
	The proof is identical to that of the previous lemma.
\end{proof}

\begin{lem} \label{slowModeErrU}
	For any $\nu \in (0, 1)$, there exists 
	$\tau_0 \gg 1$ depending on $p,q,k, \Upsilon_U, M, \beta, \ul{D}, \nu$ such that
		$$\left| \sum_{j = 0}^{k -1} \tl{U}_{\lambda_j} \int_{\tau}^{\tau_1} e^{B^2 \lambda_j (\tau - \ol{\tau})} \langle \grave{\chi} Err_U(\gamma, \ol{\tau}), \tl{U}_{\lambda_j} \rangle_{n, \frac{1}{2B^2}} d \ol{\tau} \right| \le \nu \left(\gamma^{-2k - 2B^2 \lambda_k } + \gamma^{- 2B^2 \lambda_k } \right) e^{B^2 \lambda_k \tau}$$
\end{lem}
\begin{proof}
	By finiteness of the sum, it suffices to prove the claim for any $j < k$.
	\begin{equation*} \begin{aligned}
		& \quad \left|  \tl{U}_{\lambda_j} (\gamma) \int_{\tau}^{\tau_1} e^{B^2 \lambda_j (\tau - \ol{\tau})} \langle \grave{\chi} Err_U(\gamma, \ol{\tau} ), \tl{U}_{\lambda_j} \rangle_{n, \frac{1}{2B^2}} d \ol{\tau} \right|  \\
		&\le C_j (\gamma^{-2k - 2B^2 \lambda_k} + \gamma^{-2B^2\lambda_j} ) \\
		& \quad \cdot \left| \int_{\tau}^{\tau_1} e^{B^2 \lambda_j (\tau - \ol{\tau})} \langle \grave{\chi} Err, \tl{U}_{\lambda_j} \rangle_{n, \frac{1}{2B^2}} d \ol{\tau} \right| 
		&&  	\text{(proposition \ref{eigfuncUAsymps})} \\
		&\le  C_j (\gamma^{-2k - 2B^2 \lambda_k} + \gamma^{-2B^2 \lambda_k} )\\
		& \quad \cdot  \left| \int_{\tau}^{\tau_1} e^{B^2 \lambda_j (\tau - \ol{\tau})} \langle  \grave{\chi} Err, \tl{U}_{\lambda_j} \rangle_{n, \frac{1}{2B^2}} d \ol{\tau} \right| 
		&& (\lambda_k < \lambda_j) \\
		&\le C_j (\gamma^{-2k - 2B^2 \lambda_k} + \gamma^{- 2B^2 \lambda_k} )  \\
		& \quad \cdot \int_{\tau}^{\tau_1} e^{B^2 \lambda_j (\tau - \ol{\tau})} \cancelto{1}{\| \grave{\chi} \|_{\infty} } \| Err\|_{H^{-1}_{n,b}} \| \tl{U}_{\lambda_j} \|_{H^1_{n,\frac{1}{2B^2 }}} d \ol{\tau} \\
		& \lesssim_{p,q,k, \Upsilon_U, \ul{D} }  (\gamma^{-2k - 2B^2 \lambda_k} + \gamma^{- 2B^2 \lambda_k} )  \\
		& \quad \cdot \int_{\tau}^{\tau_1} e^{B^2 \lambda_j (\tau - \ol{\tau})} e^{B^2 \lambda_k \ol{\tau} } e^{- \delta \ol{\tau} } d \ol{\tau} \| \tl{U}_{\lambda_j} \|_{H^1_{n,\frac{1}{2B^2}}} 
		 &&(\ref{summH^-1EstErr}) \\
		&\lesssim_{p,q,k} (\gamma^{-2k - 2B^2 \lambda_k} + \gamma^{-2B^2 \lambda_k} ) \| \tl{U}_{\lambda_j} \|_{H^1_{n, \frac{1}{2B^2}}}  e^{- \delta \tau_0} e^{B^2 \lambda_k \tau}
	\end{aligned} \end{equation*}
	By taking $\tau_0 \gg 1$, the claim follows.
\end{proof}

\begin{lem} \label{slowModeErrZ}
	For any $\nu \in (0,1)$, there exists 
	$\tau_0 \gg 1$ sufficiently large depending on $p,q,k, \Upsilon_Z, M , \beta, \ul{D}, \nu$ such that 
		$$\left| \sum_{j = 0}^{K} \tl{Z}_{h_j} \int_{\tau}^{\tau_1} e^{B^2 h_j (\tau - \ol{\tau} )} \langle \grave{\chi} Err_Z(\ol{\tau}), \tl{Z}_{h_j} \rangle_{n-2, \frac{1}{2B^2}} d \ol{\tau} \right| \le \nu e^{B^2 \lambda_k \tau} ( \gamma^{-2k - 2B^2 \lambda_k} + \gamma^{ - 2B^2 \lambda_k } )$$
\end{lem}
\begin{proof}
	Observe that proposition \ref{eigfuncZAsymps} implies that
		$$\tl{Z}_{h_j} \le C_j (\gamma^2 + \gamma^{- 2B^2 h_j } )$$
	and $\lambda_k \le h_j$ for all $0 \le j \le K$.
	With these facts in hand the proof follows by similar logic as in the proof of the above lemma \ref{slowModeErrU}.
\end{proof}

\begin{lem} \label{slowModeCommU}
	For any $\nu \in (0, 1)$, 
	there exists $\tau_0 \gg 1$ depending on $p,q,k, M, \beta, \ul{D}, \nu$ such that
	\begin{gather*}
		 \left| \sum_{j = 0}^k \tl{U}_{\lambda_j} \int_{\tau}^{\tau_1} e^{B^2 \lambda_j (\tau - \ol{\tau})} \left \langle B^2 [ \grave{\chi}, \mathcal{D}_U ] \tl{U} + \grave{\chi}_\tau \tl{U}(\ol{\tau}), \tl{U}_{\lambda_j} \right \rangle_{n, \frac{1}{2B^2}} d \ol{\tau} \right| \\
		 \le \nu \left(\gamma^{-2k - 2B^2  \lambda_k} + \gamma^{- 2B^2 \lambda_k } \right) e^{B^2 \lambda_k \tau}
	\end{gather*}
\end{lem}
\begin{proof}
	By finiteness of the sum, it suffices to prove the claim for any $j \le k$.
	\begin{equation*} \begin{aligned}
		& \quad  \left|  \tl{U}_{\lambda_j} \int_{\tau}^{\tau_1} e^{B^2 \lambda_j (\tau - \ol{\tau})} \left \langle B^2 [ \grave{\chi}, \mathcal{D}_U ] \tl{U} + \grave{\chi}_\tau \tl{U}(\ol{\tau}), \tl{U}_{\lambda_j} \right \rangle_{n, \frac{1}{2B^2}} d \ol{\tau} \right|  \\
		&\lesssim_k  ( \gamma^{-2k -2B^2 \lambda_k}  + \gamma^{-2B^2 \lambda_j} ) \\
		& \quad \cdot \left| \int_{\tau}^{\tau_1} e^{B^2 \lambda_j (\tau - \ol{\tau})} \left \langle B^2 [ \grave{\chi}, \mathcal{D}_U ] \tl{U} + \grave{\chi}_\tau \tl{U}(\ol{\tau}), \tl{U}_{\lambda_j} \right \rangle_{n, \frac{1}{2B^2}} d \ol{\tau} \right| 		
		&& (\text{proposition } \ref{eigfuncUAsymps})\\
		&\lesssim_{ k}  ( \gamma^{-2k -2B^2 \lambda_k}  + \gamma^{-2B^2 \lambda_k} ) \\
		& \quad \cdot  \left| \int_{\tau}^{\tau_1} e^{B^2 \lambda_j (\tau - \ol{\tau})} \left \langle B^2 [ \grave{\chi}, \mathcal{D}_U ] \tl{U} + \grave{\chi}_\tau \tl{U}(\ol{\tau}), \tl{U}_{\lambda_j} \right \rangle_{n, \frac{1}{2B^2}} d \ol{\tau} \right| 		
		&& (\lambda_k \le \lambda_j)\\
		& \le  ( \gamma^{-2k -2B^2 \lambda}  + \gamma^{-2B^2 \lambda_k} ) \\
		& \quad \cdot  \int_{\tau}^{\tau_1} e^{B^2 \lambda_j (\tau - \ol{\tau})} \| B^2 [ \grave{\chi}, \mathcal{D}_U ] \tl{U} + \grave{\chi}_\tau \tl{U}(\ol{\tau}) \|_{H^{-1}_{n, \frac{1}{2B^2}}}
		\| \tl{U}_{\lambda_j} \|_{H^1_{n , \frac{1}{2B^2}}} d \ol{\tau} \\
		&\lesssim_{p,q,k, \ul{D}}   ( \gamma^{-2k -2B^2 \lambda}  + \gamma^{-2B^2 \lambda_k} ) \\
		& \quad \cdot \int_{\tau}^{\tau_1} e^{B^2 \lambda_j (\tau - \ol{\tau})} e^{B^2 \lambda_k \ol{\tau} - \delta \ol{\tau} }
		 d \ol{\tau} \| \tl{U}_{\lambda_j} \|_{H^1_{n, \frac{1}{2B^2}}}
		 && (\ref{outerH^-1EstErr})\\
		 &\lesssim_{p,q,k}  ( \gamma^{-2k -2B^2 \lambda}  + \gamma^{-2B^2 \lambda_k} ) e^{B^2 \lambda_k \tau} e^{- \delta \tau_0} \\
	\end{aligned} \end{equation*}
	By taking $\tau_0 \gg 1$, the claim follows.
\end{proof}

\begin{lem} \label{slowModeCommZ}
	For any $\nu \in (0, 1)$, there exists  
	$\tau_0 \gg 1$ depending on $p,q,k, M , \beta, \ul{D}, \nu$ such that
	\begin{gather*}
		 \left| \sum_{j = 0}^K \tl{Z}_{h_j} \int_{\tau}^{\tau_1} e^{B^2 h_j (\tau - \ol{\tau} ) } \left \langle B^2 [ \grave{\chi}, \mathcal{D}_Z ] \tl{Z} + \grave{\chi}_\tau \tl{Z}( \ol{\tau}) , \tl{Z}_{h_j} \right \rangle_{n-2, \frac{1}{2B^2}} d \ol{\tau} \right| \\
		 \le \nu \left(\gamma^{-2k - 2B^2  \lambda_k} + \gamma^{- 2B^2 \lambda_k } \right) e^{B^2 \lambda_k \tau}
	\end{gather*}
\end{lem}
\begin{proof}
	The proof is identical to the proof of \ref{slowModeCommU} above using $\lambda_k \le h_j$ for $0 \le j \le K$ and \ref{outerH^-1EstErr}.
\end{proof}

\begin{lem} \label{slowModeN}
	For some $\epsilon(p,q,k) > 0$,
	\begin{gather*}
	\left| \sum_{j = 0}^K \tl{Z}_{h_j} \int_{\tau}^{\tau_1} e^{B^2 h_j (\tau - \ol{\tau} ) } \left \langle \grave{\chi} B^2 \cl{N} \tl{U} - B^2 \cl{N} e^{B^2 \lambda_k \ol{\tau} } \tl{U}_{\lambda_k} , \tl{Z}_{h_j} \right \rangle_{n-2, \frac{1}{2B^2}} d \ol{\tau} \right| \\
	\le ( C_{p,q,k, \Upsilon_U, \ul{D}} e^{- \epsilon \tau_0} + C_{p,q,k, \ul{D}} \eta_1^U ) e^{B^2 \lambda_k \tau} ( \gamma^{-2k -2B^2 \lambda_k} + \gamma^{-2B^2 \lambda_k} )
	\end{gather*}
\end{lem}
\begin{proof}
	Note that 
	lemma \ref{summH^-1EstErr}
	implies that there exists $\epsilon(p,q,k) > 0$ such that 
		$$\| \grave{\chi} B^2 \cl{N} \tl{U} - B^2 \cl{N} e^{B^2 \lambda_k \ol{\tau} } \tl{U}_{\lambda_k} \|_{H^{-1}_{n-2, \frac{1}{2B^2}}} \le e^{B^2 \lambda_k \ol{\tau}}( C_{p,q,k, \Upsilon_U, \ul{D}} e^{-\epsilon \ol{\tau}} + C_{p,q,k, \ul{D}}\eta_1^U)$$
	The rest of the proof then follows by similar logic as the above proofs.
\end{proof}

\begin{lem} \label{domModeInitU}
	For any $\nu \in (0,1)$, there exists $\tau_0 \gg 1$ sufficiently large depending on $p,q,k, \ul{\Upsilon}, M, \ol{\beta}, \nu$ such that
		$$\left| e^{B^2 \lambda_k \tau} \left\langle \check{\grave{U}}_{\ul{p}, \ul{q}} (\tau_0) - \tl{U}_{\lambda_k} , \tl{U}_{\lambda_k} \right\rangle_{n, \frac{1}{2B^2}} \tl{U}_{\lambda_k} \right| \le \nu e^{B^2 \lambda_k \tau} ( \gamma^{-2k -2B^2 \lambda_k} + \gamma^{-2B^2 \lambda_k} )$$
\end{lem}
\begin{proof}
	\begin{equation*} \begin{aligned}
		& \quad \left| e^{B^2 \lambda_k \tau} \left\langle \check{\grave{U}}_{\ul{p}, \ul{q}} (\tau_0) - \tl{U}_{\lambda_k} , \tl{U}_{\lambda_k} \right\rangle_{n, \frac{1}{2B^2}} \tl{U}_{\lambda_k} \right| \\
		& \lesssim_k  e^{B^2 \lambda_k \tau} ( \gamma^{-2k -2B^2 \lambda_k} + \gamma^{-2B^2 \lambda_k} ) \| \check{\grave{U}}_{\ul{p}, \ul{q}} (\tau_0) - \tl{U}_{\lambda_k}  \|_{n, \frac{1}{2B^2}} \\
		& \lesssim_{p,q,k, \ul{\Upsilon}}  e^{-\epsilon \tau_0} e^{B^2 \lambda_k \tau} ( \gamma^{-2k -2B^2 \lambda_k} + \gamma^{-2B^2 \lambda_k} )
		&& (\ref{L^2EstInitU})\\
	\end{aligned} \end{equation*}
	Taking $\tau_0 \gg 1$ sufficiently large completes the proof.
\end{proof}

\begin{lem} \label{fastModeInitU}
	For any $\nu \in (0,1)$ and $\Gamma > 0$, there exists $\tau_0 \gg 1$ sufficiently large depending on $p,q,k, \ul{\Upsilon}, M, \ol{\beta}, \Gamma, \nu$ such that 
		$$\left| \sum_{j = k+1}^\infty e^{B^2 \lambda_j (\tau - \tau_0)} \tl{U}_{\lambda_j} e^{B^2 \lambda_k \tau_0} \langle \check{\grave{U}}_{\ul{p}, \ul{q}} (\tau_0), \tl{U}_{\lambda_j} \rangle_{n, \frac{1}{2B^2}} \right| \le \nu e^{B^2 \lambda_k \tau} 
		\gamma^{-2k - 2B^2 \lambda_k}$$
		$$\text{for all } (\gamma, \tau) \in \Sigma_\Gamma$$
\end{lem}
\begin{proof}
	For this proof, all the arbitrary constants $C$ and implicit constants that appear will be uniform in $j$.
	First, note that for $j \ge k+1$
	\begin{equation*} \begin{aligned}
		e^{B^2 \lambda_j (\tau - \tau_0)} e^{B^2 \lambda_k \tau_0} 
		&= e^{B^2 (\lambda_j - \lambda_{k+1} ) (\tau - \tau_0)} e^{B^2 (\lambda_k - \lambda_{k+1}) \tau_0} e^{B^2 \lambda_{k+1} \tau} \\
		&\le e^{B^2 (\lambda_j - \lambda_{k+1} )} e^{B^2 (\lambda_k - \lambda_{k+1} ) \tau_0} e^{B^2 \lambda_{k+1} \tau} 
		&& ( \lambda_j \le \lambda_{k+1}) \\
		&= e^{B^2 (\lambda_j - \lambda_{k+1} )} e^{B^2 (\lambda_{k+1} - \lambda_k)(\tau - \tau_0)} e^{B^2 \lambda_k \tau} \\
		&\le e^{B^2 (\lambda_j - \lambda_{k+1} )} e^{B^2 \lambda_k \tau}  && (\lambda_{k+1} - \lambda_k < 0)\\
		& \lesssim_{p,q,k} e^{- B^2 j } e^{B^2 \lambda_k \tau}
	\end{aligned} \end{equation*}

	It follows that
	\begin{equation*} \begin{aligned}
		& \quad  \left| \sum_{j = k+1}^\infty e^{B^2 \lambda_j (\tau - \tau_0)} \tl{U}_{\lambda_j} e^{B^2 \lambda_k \tau_0} \langle \check{\grave{U}}_{\ul{p}, \ul{q}} (\tau_0), \tl{U}_{\lambda_j} \rangle_{n, \frac{1}{2B^2}} \right| \\
		&\lesssim_{p,q,k}  e^{B^2 \lambda_k \tau} \sum_{j = k+1}^\infty e^{-B^2 j } \left| \langle \check{\grave{U}}_{\ul{p}, \ul{q}}(\tau_0), \tl{U}_{\lambda_j} \rangle_{n, \frac{1}{2B^2}} \right|  \left| \tl{U}_{\lambda_j} (\gamma) \right| \\
		& \lesssim_{p,q,k, \ul{\Upsilon}}  e^{- \epsilon \tau_0} e^{B^2 \lambda_k \tau} \sum_{j = k+1}^\infty e^{-B^2 j }  \left| \tl{U}_{\lambda_j} (\gamma) \right| && (\ref{L^2EstInitU})\\
		& \lesssim_{p,q,k} e^{- \epsilon \tau_0} e^{B^2 \lambda_k \tau} \gamma^{-2k - 2B^2 \lambda_k }\\
		& \quad \cdot  \sum_{j=k+1}^\infty e^{-B^2 j}  \left( \frac{j!}{\Gamma(j + \frac{1}{2} \sqrt{(n-9)(n-1)} +1)} \right)^{1/2} \\
		& \quad \cdot \left| L_j^{\left(\frac{1}{2} \sqrt{(n-9)(n-1)} \right)} \left( \frac{ \gamma^2}{4B^2} \right) \right| 
		&& (\text{proposition } \ref{eigConstU})\\
		& \lesssim_{p,q, \Gamma}  e^{- \epsilon \tau_0} e^{B^2 \lambda_k \tau} \gamma^{-2k - 2B^2 \lambda_k } \sum_{j=k+1}^\infty e^{-B^2 j}  j^{C_{p,q,\Gamma}}
		&& (\text{lemma } \ref{laguerreGrowthJ})\\
		& \lesssim_{p,q,\Gamma}  e^{- \epsilon \tau_0} e^{B^2 \lambda_k \tau} \gamma^{-2k - 2B^2 \lambda_k}
	\end{aligned} \end{equation*}
	Taking $\tau_0 \gg 1$ sufficiently large completes the proof.
\end{proof}

\begin{lem} \label{fastModeInitZ}
	For $\nu \in (0,1)$ and $\Gamma > 0$, there exists $\tau_0 \gg 1$ sufficiently large depending on $p,q,k, \ul{\Upsilon}, M, \ol{\beta}, \Gamma, \nu$ such that
		$$\left| \sum_{j = K+1}^\infty e^{B^2 h_j ( \tau - \tau_0)} e^{B^2 \lambda_k \tau_0} \left \langle \check{ \grave{Z}}_{\ul{p}, \ul{q}}(\tau_0) - \tl{Z}_{\lambda_k}, \tl{Z}_{h_j} \right \rangle_{n-2, \frac{1}{2B^2}} \tl{Z}_{h_j} \right| $$
		$$\le \nu e^{B^2 \lambda_k \tau} ( \gamma^{-2k -2B^2 \lambda_k} + \gamma^{-2B^2 \lambda_k} )$$
\end{lem}
\begin{proof}
	By similar logic as in the previous proof,
		$$e^{B^2 h_j(\tau - \tau_0)} e^{B^2 \lambda_k \tau_0} \lesssim_{p,q,k} e^{-B^2 j} e^{B^2 \lambda_k \tau} 		\qquad \text{for all } j \ge K+1$$
	\ref{L^2EstInitZ} then implies
	\begin{equation*} \begin{aligned}
		& \quad \left| \sum_{j = K+1}^\infty e^{B^2 h_j ( \tau - \tau_0)} e^{B^2 \lambda_k \tau_0} \left \langle \check{ \grave{Z}}_{\ul{p}, \ul{q}} (\tau_0) - \tl{Z}_{\lambda_k}, \tl{Z}_{h_j} \right \rangle_{n-2, \frac{1}{2B^2}} \tl{Z}_{h_j} \right| \\
		&\lesssim_{p,q,k, \ul{\Upsilon}}  e^{B^2 \lambda_k \tau} e^{-\epsilon \tau_0} \sum_{j = K+1}^\infty e^{-B^2 j} | \tl{Z}_{h_j} | \\
		& \lesssim_{p,q}  e^{B^2 \lambda_k \tau} e^{-\epsilon \tau_0} \gamma^2\\
		& \quad \cdot  \sum_{j = K+1}^\infty e^{-B^2 j} \left( \frac{j!}{\Gamma( j + \frac{n+1}{2} + 1)} \right)^{1/2}
		\left| L_j^{ \left( \frac{n+1}{2} \right)} \left( \frac{\gamma^2}{4 B^2 } \right) \right| \\
		& \lesssim_{p,q,\Gamma}  e^{B^2 \lambda_k \tau} e^{-\epsilon \tau_0} \gamma^2 \sum_{j = K+1}^\infty e^{-B^2 j} j^{C_{p,q,\Gamma}}
		&& (\text{lemma } \ref{laguerreGrowthJ})\\
		& \lesssim_{p,q,\Gamma}  e^{B^2 \lambda_k \tau} e^{-\epsilon \tau_0} \gamma^2 \\
		& \lesssim_{p,q,k,\Gamma}  e^{B^2 \lambda_k \tau} e^{-\epsilon \tau_0} ( \gamma^{-2k -2B^2 \lambda_k} + \gamma^{-2B^2 \lambda_k} )
		&& (  (\gamma, \tau) \in \Sigma_{\Upsilon, \Gamma} ) \\
	\end{aligned} \end{equation*}
	Taking $\tau_0 \gg 1$ sufficiently large completes the proof.
\end{proof}

\begin{lem} \label{fastModeErrU}
	For any $\nu \in (0,1)$ and $\Gamma > 0$, there exists 
	$\Upsilon_U \gg 1$ sufficiently large depending on $p,q,k, l, \kappa, \ul{D}, \nu$,
	and $\tau_0 \gg 1$ sufficiently large depending on $p,q,k, \Upsilon_U, M, \beta, \ul{D}, \nu$
	such that
		$$\left| \sum_{j = k }^\infty \tl{U}_{\lambda_j} \int_{\tau_0}^\tau e^{B^2 \lambda_j (\tau - \ol{\tau})} \left \langle \grave{\chi}  Err_U (\tl{Z}, \tl{U} ) (\ol{\tau}), \tl{U}_{\lambda_j} \right \rangle_{n, \frac{1}{2B^2}} d \ol{\tau}  \right|$$
		$$\le \nu e^{B^2 \lambda_k \tau} ( \gamma^{-2k -2B^2 \lambda_k} + \gamma^{-2 B^2 \lambda_k} )$$
	for all $(\gamma, \tau) \in \Sigma_{\Upsilon_U, \Gamma}$.
\end{lem}
\begin{proof}
	Begin by splitting the integral as
	\begin{equation*} \begin{aligned}
		&\sum_{j = k }^\infty \tl{U}_{\lambda_j} \int_{\tau_0}^\tau e^{B^2 \lambda_j (\tau - \ol{\tau})} \left \langle \grave{\chi}  Err_U (\tl{Z}, \tl{U} ) (\ol{\tau}), \tl{U}_{\lambda_j} \right \rangle_{n, \frac{1}{2B^2}} d \ol{\tau} \\
		= & \sum_{j = k}^\infty \tl{U}_{\lambda_j} \left( \int_{\tau_0}^{\tau - 1} e^{B^2 \lambda_j (\tau - \ol{\tau})} \left \langle \grave{\chi}  Err_U (\tl{Z}, \tl{U} ) (\ol{\tau}), \tl{U}_{\lambda_j} \right \rangle_{n, \frac{1}{2B^2}} d \ol{\tau} \right. \\
		& \qquad + \left. \int_{\tau - 1}^\tau  e^{B^2 \lambda_j (\tau - \ol{\tau})} \left \langle \grave{\chi}  Err_U (\tl{Z}, \tl{U} ) (\ol{\tau}), \tl{U}_{\lambda_j} \right \rangle_{n, \frac{1}{2B^2}} d \ol{\tau}\right) \\
	\end{aligned} \end{equation*}
	using $\tau \ge \tau_0 + 1$.
	
	The second term can be estimated as follows
	\begin{equation*} \begin{aligned}
		& \quad \left|  \sum_{j = k}^\infty \tl{U}_{\lambda_j} \int_{\tau - 1}^\tau e^{B^2 \lambda_j (\tau - \ol{\tau})} \langle \grave{\chi} Err_U, \tl{U}_{\lambda_j} \rangle_{n, \frac{1}{2B^2}} d \ol{\tau}  \right| \\
		&=\left|  \int_{\tau - 1}^\tau e^{B^2 \cl{D}_U (\tau - \ol{\tau} )} \grave{\chi} Err_U(\gamma, \ol{\tau}) d \ol{\tau} - \sum_{j = 0}^{k-1} \tl{U}_{\lambda_j} \int_{\tau - 1}^\tau e^{B^2 \lambda_j (\tau - \ol{\tau} )} \langle \grave{\chi} Err_U, \tl{U}_{\lambda_j} \rangle_{n, \frac{1}{2B^2}} d \ol{\tau} \right| \\
		&\le  \left|  \int_{\tau - 1}^\tau e^{B^2 \cl{D}_U (\tau - \ol{\tau} )} \grave{\chi} Err_U(\gamma, \ol{\tau}) d \ol{\tau} \right| \\
		& \quad +  \sum_{j = 0}^{k-1} | \tl{U}_{\lambda_j} | \| \tl{U}_{\lambda_j} \|_{H^1_{n, \frac{1}{2B^2}}}  \int_{\tau - 1}^\tau e^{B^2 \lambda_j ( \tau - \ol{\tau}) } \| \grave{\chi} Err_U( \ol{\tau} ) \|_{H^{-1}_{n, \frac{1}{2B^2}}} d \ol{\tau} \\
		&\lesssim_{p,q,k}  \left|  \int_{\tau - 1}^\tau e^{B^2 \cl{D}_U (\tau - \ol{\tau} )} \grave{\chi} Err_U(\gamma, \ol{\tau}) d \ol{\tau} \right| \\
		&\quad  +  \left( \gamma^{-2k -2B^2 \lambda_k} + \gamma^{-2B^2 \lambda_k} \right) \int_{\tau - 1}^\tau  \| \grave{\chi} Err_U( \ol{\tau} ) \|_{H^{-1}_{n, \frac{1}{2B^2}}} d \ol{\tau} \\
		& \le \frac{1}{4} \nu e^{B^2 \lambda \tau} ( \gamma^{-2k -2B^2 \lambda_k} + \gamma^{-2B^2\lambda_k}) 
		&& (\ref{summShortTimeEstErrU} )\\
		& \quad + C_{p,q,k, \ul{D}, \Upsilon_U}  \left( \gamma^{-2k -2B^2 \lambda_k} + \gamma^{-2B^2 \lambda_k} \right) 
		\int_{\tau - 1}^\tau e^{B^2 \lambda \ol{\tau} - \delta \ol{\tau} } d \ol{\tau} 
		&& (\ref{summH^-1EstErr})\\
		& \le \frac{1}{2}  \nu e^{B^2 \lambda \tau} ( \gamma^{-2k -2B^2 \lambda_k} + \gamma^{-2B^2\lambda_k}) 
	\end{aligned} \end{equation*}
	if $\Upsilon_U \gg 1$ is sufficiently large depending on $p,q,k, l, \kappa, \ul{D}, \nu$,
	and $\tau_0 \gg 1$ is sufficiently large depending on $p,q,k, \Upsilon_U, M, \beta, \ul{D}, \nu$

	
	For the other integral, begin by observing that
	\begin{equation*} \begin{aligned}
		&\quad \left| \sum_{j = k}^\infty \tl{U}_{\lambda_j} \int_{\tau_0}^{\tau-1} e^{B^2 \lambda_j (\tau - \ol{\tau})} \langle \grave{\chi} Err_U , \tl{U}_{\lambda_j} \rangle_{n, \frac{1}{2B^2}} d \ol{\tau} \right| \\
		&\le  \sum_{j = k}^\infty | \tl{U}_{\lambda_j} | \int_{\tau_0}^{\tau-1} e^{B^2 (\lambda_j -\lambda_k) (\tau - \ol{\tau})} e^{B^2 \lambda_k (\tau - \ol{\tau}) } \left| \langle \grave{\chi}  Err_U , \tl{U}_{\lambda_j} \rangle_{n, \frac{1}{2B^2}} \right| d \ol{\tau}\\
		& \le \left( \sum_{j = k}^\infty | \tl{U}_{\lambda_j} | e^{B^2 (\lambda_j - \lambda_k)} \| \tl{U}_{\lambda_j} \|_{H^1_{n, \frac{1}{2B^2}}}  \right) \int_{\tau_0}^{\tau-1} e^{B^2 \lambda_k (\tau - \ol{\tau} )} \| Err_U \|_{H^{-1}_{n, \frac{1}{2B^2}}} d \ol{\tau} \\
		& \lesssim_{p,q}  \left( \sum_{j = k}^\infty | \tl{U}_{\lambda_j} | e^{B^2 (\lambda_j - \lambda_k)} \sqrt{j}  \right) \int_{\tau_0}^{\tau-1} e^{B^2 \lambda_k (\tau - \ol{\tau} )} \| Err_U \|_{H^{-1}_{n, \frac{1}{2B^2}}} d \ol{\tau} 
		&& ( \ref{eigfuncUH^1GrowthJ})\\
		& \lesssim_{p,q,k, \ul{D}, \Upsilon_U}   \left( \sum_{j = k}^\infty | \tl{U}_{\lambda_j} | e^{-B^2 (j - k)} \sqrt{j}  \right) \int_{\tau_0}^{\tau-1} e^{B^2 \lambda_k (\tau - \ol{\tau} )} e^{B^2 \lambda_k \ol{\tau} - \delta \ol{\tau} } d \ol{\tau} 
		&& (\ref{summH^-1EstErr})\\
		& \lesssim_{p,q, \Gamma}  \gamma^{-2k -2B^2 \lambda_k} \left( \sum_{j = k}^\infty  e^{-B^2 (j - k)} j^{C_{p,q, \Gamma}}  \right) \int_{\tau_0}^{\tau-1} e^{B^2 \lambda_k (\tau - \ol{\tau} )} e^{B^2 \lambda_k \ol{\tau} - \delta \ol{\tau} } d \ol{\tau} 
		&& ( \ref{laguerreGrowthJ} ) \\
		& \lesssim_{p,q,k, \Gamma}  e^{B^2 \lambda_k \tau} \gamma^{-2k -2B^2 \lambda_k} \int_{\tau_0}^{\tau - 1} e^{- \delta \ol{\tau} } d \ol{\tau} \\
	\end{aligned} \end{equation*} 
	The estimate then follows from observing that the integral $$\int_{\tau_0}^{\tau - 1} e^{- \delta \ol{\tau} } d \ol{\tau} $$ can be made arbitrarily small by taking $\tau_0 \gg 1$ sufficiently large.
	
%
%
%
\end{proof}

\begin{lem} \label{fastModeCommU}
	For any $\nu \in (0,1)$ and $\Gamma > 0$, 
	there exists $\tau_0 \gg 1$ sufficiently large depending on $p,q,k, \ul{D}, M, \beta, \Gamma, \nu$ such that
		$$\left| \sum_{j = k }^\infty \tl{U}_{\lambda_j} \int_{\tau_0}^\tau e^{B^2 \lambda_j (\tau - \ol{\tau})} \left \langle B^2 [ \grave{\chi}, \mathcal{D}_U ] \tl{U} + \grave{\chi}_\tau \tl{U}, \tl{U}_{\lambda_j} \right \rangle_{n, \frac{1}{2B^2}} d \ol{\tau} \right|$$
		$$ \le \nu e^{B^2 \lambda_k \tau} ( \gamma^{-2k -2B^2 \lambda_k} + \gamma^{-2 B^2 \lambda_k} )$$
	for all $(\gamma, \tau) \in \Sigma_{\Upsilon_U, \Gamma}$.
\end{lem}
\begin{proof}
	We follow the strategy as in the proof of the previous lemma.
	Begin by splitting the integral
	\begin{equation*} \begin{aligned}
		 & \sum_{j = k }^\infty \tl{U}_{\lambda_j} \int_{\tau_0}^\tau e^{B^2 \lambda_j (\tau - \ol{\tau})} \left \langle B^2 [ \grave{\chi}, \mathcal{D}_U ] \tl{U} + \grave{\chi}_\tau \tl{U}, \tl{U}_{\lambda_j} \right \rangle_{n, \frac{1}{2B^2}} d \ol{\tau}  \\
		\le & \sum_{j = k}^\infty \tl{U}_{\lambda_j} \left( \int_{\tau_0}^{\tau-1} e^{B^2 \lambda_j (\tau - \ol{\tau})} \left \langle B^2 [ \grave{\chi}, \mathcal{D}_U ] \tl{U} + \grave{\chi}_\tau \tl{U}, \tl{U}_{\lambda_j} \right \rangle_{n, \frac{1}{2B^2}} d \ol{\tau} \right. \\
		& + \left. \int_{\tau-1}^\tau e^{B^2 \lambda_j (\tau - \ol{\tau})} \left \langle B^2 [ \grave{\chi}, \mathcal{D}_U ] \tl{U} + \grave{\chi}_\tau \tl{U}, \tl{U}_{\lambda_j} \right \rangle_{n, \frac{1}{2B^2}} d \ol{\tau}	\right)
	\end{aligned} \end{equation*}
	
	The second term can be estimated as follows
	\begin{equation*} \begin{aligned}
		&\quad  \left|  \sum_{j = k}^\infty \tl{U}_{\lambda_j}  \int_{\tau-1}^\tau e^{B^2 \lambda_j (\tau - \ol{\tau})} \left \langle B^2 [ \grave{\chi}, \mathcal{D}_U ] \tl{U} + \grave{\chi}_\tau \tl{U}, \tl{U}_{\lambda_j} \right \rangle_{n, \frac{1}{2B^2}} d \ol{\tau}	\right| \\
		&= \left| \int_{\tau-1}^\tau e^{B^2 \cl{D}_U (\tau - \ol{\tau})} \left(  B^2 [ \grave{\chi}, \mathcal{D}_U ] \tl{U} + \grave{\chi}_\tau \tl{U}  \right) d \ol{\tau} \right.\\
		& \quad \left.- \sum_{j = 0}^{k-1} \tl{U}_{\lambda_j} \int_{\tau-1}^{\tau} e^{B^2 \lambda_j (\tau - \ol{\tau})} \left \langle B^2 [ \grave{\chi}, \mathcal{D}_U ] \tl{U} + \grave{\chi}_\tau \tl{U}, \tl{U}_{\lambda_j} \right \rangle_{n, \frac{1}{2B^2}} d \ol{\tau}	\right| \\
		& \le \left| \int_{\tau-1}^\tau e^{B^2 \cl{D}_U (\tau - \ol{\tau})} \left(  B^2 [ \grave{\chi}, \mathcal{D}_U ] \tl{U} + \grave{\chi}_\tau \tl{U}  \right) d \ol{\tau} \right| \\
		& \quad +  \sum_{j=0}^{k-1} | \tl{U}_{\lambda_j} | \| \tl{U}_{\lambda_j} \|_{H^1_{n , \frac{1}{2B^2}}} \int_{\tau-1}^\tau e^{B^2 \lambda_j (\tau - \ol{\tau} ) } \left \| B^2 [ \grave{\chi}, \mathcal{D}_U ] \tl{U} + \grave{\chi}_\tau \tl{U}  \right \|_{H^{-1}_{n , \frac{1}{2B^2}} } d \ol{\tau} \\
		&\lesssim_{p,q,k}   \left| \int_{\tau-1}^\tau e^{B^2 \cl{D}_U (\tau - \ol{\tau})} \left(  B^2 [ \grave{\chi}, \mathcal{D}_U ] \tl{U} + \grave{\chi}_\tau \tl{U}  \right) d \ol{\tau} \right| \\
		& \quad +  (\gamma^{-2k-2B^2 \lambda_k} + \gamma^{-2B^2 \lambda_k} )  \int_{\tau-1}^\tau  \left \| B^2 [ \grave{\chi}, \mathcal{D}_U ] \tl{U} + \grave{\chi}_\tau \tl{U}  \right \|_{H^{-1}_{n , \frac{1}{2B^2}} } d \ol{\tau} \\
		& \le \frac{1}{4} \nu e^{B^2 \lambda_k \tau} \gamma^{-2k-2B^2 \lambda} 
		&& (\ref{shortTimeEstCommU})\\
		& \quad + C_{p,q,k, \ul{D}} (\gamma^{-2k-2B^2 \lambda_k} + \gamma^{-2B^2 \lambda_k} )  \int_{\tau-1}^\tau e^{B^2 \lambda_k \ol{\tau} - \delta \ol{\tau} } d \ol{\tau} 
		&& (\ref{outerH^-1EstErr}) \\
		& \le \frac{1}{2} \nu e^{B^2 \lambda_k \tau}\left(  \gamma^{-2k-2B^2 \lambda} + \gamma^{-2B^2\lambda} \right)
	\end{aligned} \end{equation*}
	if $\tau_0 \gg 1 $ is sufficiently large depending on $p,q,k, M, \beta, \Gamma, \ul{D}, \nu$.
	
	For the other integral, observe
	\begin{equation*} \begin{aligned}
		 & \quad \left| \sum_{j = k}^\infty \tl{U}_{\lambda_j}  \int_{\tau_0}^{\tau-1} e^{B^2 \lambda_j (\tau - \ol{\tau})} \left \langle B^2 [ \grave{\chi}, \mathcal{D}_U ] \tl{U} + \grave{\chi}_\tau \tl{U}, \tl{U}_{\lambda_j} \right \rangle_{n, \frac{1}{2B^2}} d \ol{\tau} \right|\\
		 &\le  \sum_{j=k}^\infty | \tl{U}_{\lambda_j} |  \int_{\tau_0}^{\tau-1} 
		 e^{B^2 (\lambda_j - \lambda_k)(\tau - \ol{\tau})} e^{B^2 \lambda_k ( \tau - \ol{\tau}) }\\
		 & \qquad \cdot \left| \left \langle B^2 [ \grave{\chi}, \mathcal{D}_U ] \tl{U} + \grave{\chi}_\tau \tl{U}, \tl{U}_{\lambda_j} \right \rangle_{n, \frac{1}{2B^2}} \right|  d\ol{\tau} \\
		 & \le  \left( \sum_{j=k}^\infty  | \tl{U}_{\lambda_j} | e^{B^2 (\lambda_j - \lambda_k)} \| \tl{U}_{\lambda_j} \|_{H^1_{n, \frac{1}{2B^2}} } \right) \\
		& \quad \cdot  \int_{\tau_0}^{\tau - 1} e^{B^2 \lambda_k ( \tau - \ol{\tau}) }
		  \left \| B^2 [ \grave{\chi}, \mathcal{D}_U ] \tl{U} + \grave{\chi}_\tau \tl{U} \right \|_{H^{-1}_{n, \frac{1}{2B^2}}  } d\ol{\tau} \\
		  & \lesssim_{p,q}  \left( \sum_{j=k}^\infty  | \tl{U}_{\lambda_j} | e^{B^2 (\lambda_j - \lambda_k)} \sqrt{j} \right) \\
		& \quad	\cdot \int_{\tau_0}^{\tau - 1} e^{B^2 \lambda_k ( \tau - \ol{\tau}) }
		  \left \| B^2 [ \grave{\chi}, \mathcal{D}_U ] \tl{U} + \grave{\chi}_\tau \tl{U} \right \|_{H^{-1}_{n, \frac{1}{2B^2}}  } d\ol{\tau} 
		  && (\ref{eigfuncUH^1GrowthJ})\\
		  & \lesssim_{p,q,k, \ul{D}, \Gamma}   \gamma^{-2k -2B^2 \lambda_k} \left( \sum_{j=k}^\infty   e^{- B^2 (j - k)} j^{C_{p,q,\Gamma}} \right) \\
		& \quad \cdot \int_{\tau_0}^{\tau - 1} e^{B^2 \lambda_k ( \tau - \ol{\tau}) } e^{B^2 \lambda_k \ol{\tau} - \delta \ol{\tau}} d\ol{\tau} 
		 && (\ref{outerH^-1EstErr}) \\
		 & \lesssim_{p,q,k, \Gamma}  e^{B^2 \lambda_k \tau} \gamma^{-2k -2B^2 \lambda_k} 
		 \int_{\tau_0}^{\tau - 1}  e^{ - \delta \ol{\tau}} d\ol{\tau} 
	\end{aligned} \end{equation*}
	The estimate then follows by observing that 
		$$\int_{\tau_0}^{\tau - 1}  e^{ - \delta \ol{\tau}} d\ol{\tau}  $$
	can be made arbitrarily small by taking $\tau_0 \gg 1 $ sufficiently large.
\end{proof}

\begin{lem} \label{fastModeErrZ}
	For any $\nu \in (0,1)$ and $\Gamma > 0$, there exists 
	$\Upsilon_Z \gg 1$ sufficiently large depending on $p,q,k, \ul{D}, l, \nu$ and
	$\tau_0 \gg 1$ sufficiently large depending on $p,q,k, M, \beta, \ul{D}, \Upsilon_Z, \nu$
	such that
	\begin{gather*}
		\left| \sum_{j = K+1}^\infty \tl{Z}_{h_j} \int_{\tau_0}^\tau e^{B^2 h_j(\tau - \ol{\tau})} \left \langle \grave{\chi}  Err_Z (\tl{Z}, \tl{U} ) (\ol{\tau}), \tl{Z}_{h_j} \right \rangle_{n-2, \frac{1}{2B^2}} d \ol{\tau} \right| \\
		\le \nu e^{B^2 \lambda_k \tau} ( \gamma^{-2k -2B^2 \lambda_k} + \gamma^{-2 B^2 \lambda_k} )
	\end{gather*}
	for all $(\gamma, \tau) \in \Sigma_{\Upsilon_Z, \Gamma}$.
\end{lem}
\begin{proof}
	As for the $U$ equation, split the integral as
	\begin{equation*} \begin{aligned}
		&\sum_{j = K+1}^\infty \tl{Z}_{h_j} \int_{\tau_0}^\tau e^{B^2 h_j (\tau - \ol{\tau})} \langle \grave{\chi} Err_Z (\ol{\tau} ), \tl{Z}_{h_j} \rangle_{n-2, \frac{1}{2B^2}} d \ol{\tau} \\
		=& \sum_{j = K+1}^\infty \tl{Z}_{h_j} \left( \int_{\tau_0}^{\tau -1 } \cdots \quad+ \int_{\tau - 1 }^\tau  \cdots \right)
	\end{aligned} \end{equation*}
	The second term in the sum can be estimated as follows
	\begin{equation*} \begin{aligned}
		&\quad \left|  \int_{\tau -1}^\tau e^{B^2 \cl{D}_Z (\tau - \ol{\tau}) } Err_Z (\ol{\tau}) d \ol{\tau} - \sum_{j =0}^K  \tl{Z}_{h_j} \int_{\tau - 1}^\tau e^{B^2 h_j (\tau - \ol{\tau}) } \langle  Err_Z (\ol{\tau} ), \tl{Z}_{h_j} \rangle_{n-2, \frac{1}{2B^2}} d \ol{\tau} \right|\\
		&\le  \left| \int_{\tau -1}^\tau e^{B^2 \cl{D}_Z (\tau - \ol{\tau}) } Err_Z (\ol{\tau}) d \ol{\tau} \right| \\
		& \quad +  \sum_{j =0}^K  | \tl{Z}_{h_j} | \| \tl{Z}_{h_j} \|_{H^1_{n-2, \frac{1}{2B^2}}} \int_{\tau - 1}^\tau  e^{B^2 h_j (\tau - \ol{\tau}) }  \|  Err_Z ( \ol{\tau} ) \|_{H^{-1}_{n-2, \frac{1}{2B^2}}} d \ol{\tau} \\
		& \le  \left| \int_{\tau -1}^\tau e^{B^2 \cl{D}_Z (\tau - \ol{\tau}) } Err_Z (\ol{\tau}) d \ol{\tau} \right| 
		 + C_{p,q,k, \ul{D}, \Upsilon_Z, \Gamma}  \gamma^2  \int_{\tau -1}^\tau e^{B^2 \lambda_k \ol{\tau}- \delta \ol{\tau}}  d \ol{\tau} 
		 && (\ref{summH^-1EstErr})\\
		&\le \frac{1}{4} \nu e^{B^2 \lambda_k \tau} \left( \gamma^{-2k-2B^2 \lambda_k} + \gamma^{-2B^2 \lambda_k} \right) 
		 +C_{p,q,k, \ul{D}, \Upsilon_Z, \Gamma}  \gamma^2  \int_{\tau -1}^\tau e^{B^2 \lambda_k \ol{\tau}- \delta \ol{\tau}}  d \ol{\tau} 
		 && (\ref{summShortTimeEstErrZ})\\
		 &\le \frac{1}{2} \nu e^{B^2 \lambda_k \tau} \left( \gamma^{-2k-2B^2 \lambda_k} + \gamma^{-2B^2 \lambda_k} \right) 
	\end{aligned} \end{equation*}
	if $\Upsilon_Z \gg 1$ depending on $p,q,k, \ul{D}, l, \nu$ and 
	$\tau_0 \gg 1$ depending on $p,q,k, \Upsilon_Z, M, \beta, \ul{D}, \nu$.
	
	We estimate the other integral as follows:
	\begin{equation*} \begin{aligned}
		& \quad \left| \sum_{j = K+1}^{\infty}  \tl{Z}_{h_j}  \int_{\tau_0}^{\tau - 1} e^{B^2 h_j (\tau - \ol{\tau} )} \langle Err_Z (\ol{\tau} ), \tl{Z}_{h_j} \rangle_{n-2, \frac{1}{2B^2}} d \ol{\tau} \right| \\
		& \le  \sum_{j = K+1}^{\infty}  |\tl{Z}_{h_j}| \| \tl{Z}_{h_j} \|_{H^1_{n-2, \frac{1}{2B^2}}}  \\
		& \quad \cdot \int_{\tau_0}^{\tau - 1} e^{B^2 (h_j-\lambda_k) (\tau - \ol{\tau} )} e^{B^2 \lambda_k (\tau - \ol{\tau})} \| Err_Z (\ol{\tau} ) \|_{H^{-1}_{n-2, \frac{1}{2B^2}}} d \ol{\tau}\\
		& \le \sum_{j = K+1}^{\infty}  |\tl{Z}_{h_j}| \| \tl{Z}_{h_j} \|_{H^1_{n-2, \frac{1}{2B^2}}}  e^{B^2 (h_j-\lambda_k)}	\\
		& \quad \cdot \int_{\tau_0}^{\tau - 1}  e^{B^2 \lambda_k (\tau - \ol{\tau})} \| Err_Z (\ol{\tau} ) \|_{H^{-1}_{n-2, \frac{1}{2B^2}}} d \ol{\tau}\\
		& \lesssim_{p,q,k,\Gamma}  \gamma^2 \left( \sum_{j = K+1}^\infty  j^{ C_{p,q,\Gamma} }  \| \tl{Z}_{h_j} \|_{H^1_{n-2, \frac{1}{2B^2}}}  e^{B^2 (h_j-\lambda_k)} \right) \\
		& \quad \cdot \int_{\tau_0}^{\tau - 1}  e^{B^2 \lambda_k (\tau - \ol{\tau})} \| Err_Z (\ol{\tau} ) \|_{H^{-1}_{n-2, \frac{1}{2B^2}}} d \ol{\tau}
		&& ( \ref{laguerreGrowthJ}) \\
		& \lesssim_{p,q,k}  \gamma^2 \left( \sum_{j = K+1}^\infty  j^{ C}  \sqrt{j}  e^{-B^2 j} \right) \\
		& \quad \cdot \int_{\tau_0}^{\tau - 1}  e^{B^2 \lambda_k (\tau - \ol{\tau})} \| Err_Z (\ol{\tau} ) \|_{H^{-1}_{n-2, \frac{1}{2B^2}}} d \ol{\tau}
		&& (  \ref{eigfuncZH^1GrowthJ} )\\
		& \lesssim_{p,q,k, \ul{D}, \Upsilon_Z}  \gamma^2 \left( \sum_{j = K+1}^\infty  j^{ C }  \sqrt{j}  e^{-B^2 j} \right) 
		\int_{\tau_0}^{\tau - 1}  e^{B^2 \lambda_k (\tau - \ol{\tau})} e^{B^2 \lambda_k \ol{\tau} - \delta \ol{\tau} } d \ol{\tau}
		&& (\ref{summH^-1EstErr})\\
		& \le  e^{B^2 \lambda_k \tau} \gamma^2 \left( \sum_{j = K+1}^\infty  j^{ C }  \sqrt{j}  e^{-B^2 j} \right) 
		\int_{\tau_0}^{\tau - 1}  e^{ - \delta \ol{\tau} } d \ol{\tau}\\
		& \lesssim_{p,q,k}  e^{B^2 \lambda_k \tau} ( \gamma^{-2k -2B^2 \lambda_k} + \gamma^{-2B^2 \lambda} ) \int_{\tau_0}^{\tau - 1}  e^{ - \delta \ol{\tau} } d \ol{\tau}\\
	\end{aligned} \end{equation*}
	The estimate then follows from taking $\tau_0 \gg 1$ sufficiently large.
	
\end{proof}

\begin{lem} \label{fastModeCommZ}
	For any $\nu \in (0,1)$ and $\Gamma > 0$, there exists 
	$\tau_0 \gg 1$ sufficiently large depending on $p,q,k, M, \beta, \Gamma, \ul{D}, \nu$ such that
	\begin{gather*}
		\left| \sum_{j = K+1}^\infty \tl{Z}_{h_j} \int_{\tau_0}^\tau e^{B^2 h_j(\tau - \ol{\tau})} \left \langle B^2 [ \grave{\chi}, \mathcal{D}_Z ] \tl{Z} + \grave{\chi}_\tau \tl{Z}(\ol{\tau}), \tl{Z}_{h_j} \right \rangle_{n-2, \frac{1}{2B^2}} d \ol{\tau} \right| \\
		\le \nu e^{B^2 \lambda_k \tau} ( \gamma^{-2k -2B^2 \lambda_k} + \gamma^{-2 B^2 \lambda_k} )
	\end{gather*}
	for all $(\gamma, \tau) \in \Sigma_{\Upsilon_Z, \Gamma}$.
\end{lem}
\begin{proof}
	The proof is similar to that of the previous lemma.
	Begin by splitting the integral
	\begin{equation*} \begin{aligned}
		&  \sum_{j = K+1}^\infty \tl{Z}_{h_j} \int_{\tau_0}^\tau e^{B^2 h_j(\tau - \ol{\tau})} \left \langle B^2 [ \grave{\chi}, \mathcal{D}_Z ] \tl{Z} + \grave{\chi}_\tau \tl{Z}(\ol{\tau}), \tl{Z}_{h_j} \right \rangle_{n-2, \frac{1}{2B^2}} d \ol{\tau} \\
		=& \sum_{j = K+1}^\infty \tl{Z}_{h_j} \left( \int_{\tau_0}^{\tau - 1} \cdots \quad + \int_{\tau -1}^\tau \cdots \right)
	\end{aligned} \end{equation*}
	
	The second term can be estimated as follows
	\begin{equation*} \begin{aligned}
		& \quad  \left| \sum_{j = K+1}^\infty \tl{Z}_{h_j} \int_{\tau-1}^\tau e^{B^2 h_j(\tau - \ol{\tau})} \left \langle B^2 [ \grave{\chi}, \mathcal{D}_Z ] \tl{Z} + \grave{\chi}_\tau \tl{Z}(\ol{\tau}), \tl{Z}_{h_j} \right \rangle_{n-2, \frac{1}{2B^2}} d \ol{\tau} \right| \\
		& = \left| \int_{\tau-1}^\tau e^{B^2 \cl{D}_Z ( \tau - \ol{\tau})}
		\left( 		B^2 [ \grave{\chi}, \mathcal{D}_Z ] \tl{Z} + \grave{\chi}_\tau \tl{Z}(\ol{\tau})	\right) \right. \\
		& \quad \left. - \sum_{j =0}^K \tl{Z}_{h_j} \int_{\tau-1}^\tau e^{B^2 h_j(\tau - \ol{\tau})} \left \langle B^2 [ \grave{\chi}, \mathcal{D}_Z ] \tl{Z} + \grave{\chi}_\tau \tl{Z}(\ol{\tau}), \tl{Z}_{h_j} \right \rangle_{n-2, \frac{1}{2B^2}} d \ol{\tau} \right| \\
		& \le  \left| \int_{\tau-1}^\tau e^{B^2 \cl{D}_Z ( \tau - \ol{\tau})}
		\left( 		B^2 [ \grave{\chi}, \mathcal{D}_Z ] \tl{Z} + \grave{\chi}_\tau \tl{Z}(\ol{\tau})	\right) \right| \\
		& \quad + \sum_{j =0}^K |\tl{Z}_{h_j}|  \| \tl{Z}_{h_j} \|_{H^1_{n-2, \frac{1}{2B^2}}}
		\int_{\tau-1}^\tau e^{B^2 h_j(\tau - \ol{\tau})} \left \| B^2 [ \grave{\chi}, \mathcal{D}_Z ] \tl{Z} + \grave{\chi}_\tau \tl{Z}(\ol{\tau})   \right \|_{H^{-1}_{n-2, \frac{1}{2B^2}}} d \ol{\tau} \\
		& \le  \frac{1}{4} \nu e^{B^2 \lambda_k \tau} \left( \gamma^{-2k-2B^2 \lambda} + \gamma^{-2B^2 \lambda_k} \right) 
		&& (\ref{shortTimeEstCommZ}) \\
		& \quad + C_{p,q,k, \ul{D}, \Gamma}  \gamma^{-2k-2B^2 \lambda_k}   
		\int_{\tau-1}^\tau e^{B^2 \lambda_k \ol{\tau} - \epsilon \ol{\tau}}  d \ol{\tau} 
		&&(\ref{summH^-1EstErr})\\
		& \le \frac{1}{2} \nu e^{B^2 \lambda_k \tau} \left( \gamma^{-2k-2B^2 \lambda} + \gamma^{-2B^2 \lambda_k} \right) 
	\end{aligned} \end{equation*}
	if $\tau_0 \gg 1$ is sufficiently large depending on $p,q,k, M, \beta, \Gamma, \ul{D}, \nu$.
	
	For the other integral, observe
	\begin{equation*} \begin{aligned}
		& \quad \left| \sum_{j = K+1}^\infty \tl{Z}_{h_j} \int_{\tau_0}^{\tau-1} e^{B^2 h_j(\tau - \ol{\tau})} \left \langle B^2 [ \grave{\chi}, \mathcal{D}_Z ] \tl{Z} + \grave{\chi}_\tau \tl{Z}(\ol{\tau}), \tl{Z}_{h_j} \right \rangle_{n-2, \frac{1}{2B^2}} d \ol{\tau} \right| \\
		& \le  \sum_{j=K+1}^\infty | \tl{Z}_{h_j} | \| \tl{Z}_{h_j} \|_{H^1_{n-2, \frac{1}{2B^2}}} \\
		& \quad \cdot \int_{\tau_0}^{\tau-1}  e^{B^2 (h_j - \lambda_k)( \tau - \ol{\tau} )} e^{B^2 \lambda_k ( \tau - \ol{\tau} )}
		\| 	B^2 [ \grave{\chi}, \mathcal{D}_Z ] \tl{Z} + \grave{\chi}_\tau \tl{Z}	\|_{H^{-1}_{n-2, \frac{1}{2B^2}}} d \ol{\tau} \\
		& \lesssim_{p,q,k,\Gamma} \gamma^2 \left( \sum_{j=K+1}^\infty j^{C} e^{-B^2 j} \right) \\
		& \quad  \cdot \int_{\tau_0}^{\tau-1} e^{B^2 \lambda_k (\tau - \ol{\tau} )} 
		\| 	B^2 [ \grave{\chi}, \mathcal{D}_Z ] \tl{Z} + \grave{\chi}_\tau \tl{Z}	\|_{H^{-1}_{n-2, \frac{1}{2B^2}}} d \ol{\tau} 
		&& (\ref{laguerreGrowthJ} \text{ and } \ref{eigfuncZH^1GrowthJ})\\
		& \lesssim_{p,q,k, \ul{D}}  \gamma^2
		\int_{\tau_0}^{\tau-1} e^{B^2 \lambda_k (\tau - \ol{\tau} )} e^{B^2 \lambda_k \ol{\tau} - \epsilon \ol{\tau} } d \ol{\tau} 
		&& (\ref{summH^-1EstErr})\\
		& \lesssim_{p,q,k, \Gamma} e^{B^2 \lambda_k \tau} \left( \gamma^{-2k - 2B^2 \lambda_k} + \gamma^{-2B^2 \lambda_k} \right)
		\int_{\tau_0}^{\tau-1} e^{- \epsilon \ol{\tau}} d \ol{\tau} 
	\end{aligned} \end{equation*}
	The estimate then follows from taking $\tau_0 \gg 1$ sufficiently large.
\end{proof}

\begin{lem} \label{fastModeN}
	For any $\nu \in (0,1)$ and $\Gamma > 0$, there exists 
	$\Upsilon_Z \gg 1$ is sufficiently large depending on $p,q, k, l, \Upsilon_U, \ul{D}, \nu$
	and $\tau_0 \gg 1$ sufficiently large depending on $p,q,k, M, \beta, \Upsilon_U, \Upsilon_Z, \ul{D}, \Gamma, \nu$ such that
	\begin{gather*}
		\left| \sum_{j = K+1}^\infty \tl{Z}_{h_j} \int_{\tau_0}^\tau e^{B^2 h_j(\tau - \ol{\tau})} \left \langle \grave{\chi} B^2 \cl{N} \tl{U} - B^2 \cl{N} e^{B^2 \lambda_k \ol{\tau} } \tl{U}_{\lambda_k}, \tl{Z}_{h_j} \right \rangle_{n-2, \frac{1}{2B^2}} d \ol{\tau}\right| \\
		\le \left( \nu + C_{p,q, k, \ul{D}, \Gamma} \eta_1^U \right)  e^{B^2 \lambda_k \tau} ( \gamma^{-2k -2B^2 \lambda_k} + \gamma^{-2 B^2 \lambda_k} )
	\end{gather*}
	for all $(\gamma, \tau) \in \Sigma_{\Upsilon_Z, \Gamma}$.
\end{lem}
\begin{proof}
	As above, begin by splitting the integral
	\begin{equation*} \begin{aligned}
		& \sum_{j = K+1}^\infty \tl{Z}_{h_j} \int_{\tau_0}^\tau e^{B^2 h_j (\tau - \ol{\tau})} \left \langle \grave{\chi} B^2 \cl{N} \tl{U} - B^2 \cl{N} e^{B^2 \lambda_k \ol{\tau} } \tl{U}_{\lambda_k}, \tl{Z}_{h_j} \right \rangle_{n-2, \frac{1}{2B^2}} d \ol{\tau} \\
		=& 
		\sum_{j = K+1}^\infty \tl{Z}_{h_j}  \int_{\tau_0}^{\tau-1} e^{B^2 h_j (\tau - \ol{\tau})} \left \langle \grave{\chi} B^2 \cl{N} \tl{U} - B^2 \cl{N} e^{B^2 \lambda_k \ol{\tau} } \tl{U}_{\lambda_k}, \tl{Z}_{h_j} \right \rangle_{n-2, \frac{1}{2B^2}} d \ol{\tau}\\
		&+ \sum_{j = K+1}^\infty \tl{Z}_{h_j} \int_{\tau-1}^\tau e^{B^2 h_j (\tau - \ol{\tau})} \left \langle \grave{\chi} B^2 \cl{N} \tl{U} - B^2 \cl{N} e^{B^2 \lambda_k \ol{\tau} } \tl{U}_{\lambda_k}, \tl{Z}_{h_j} \right \rangle_{n-2, \frac{1}{2B^2}} d \ol{\tau}
	\end{aligned} 	\end{equation*}
	
	The second term can be estimated as follows:
	\begin{equation*} \begin{aligned}
		&\left| \sum_{j = K+1}^\infty \tl{Z}_{h_j} \int_{\tau-1}^\tau e^{B^2 h_j (\tau - \ol{\tau})} \left \langle \grave{\chi} B^2 \cl{N} \tl{U} - B^2 \cl{N} e^{B^2 \lambda_k \ol{\tau} } \tl{U}_{\lambda_k}, \tl{Z}_{h_j} \right \rangle_{n-2, \frac{1}{2B^2}} \right| \\
		=& \left| \int_{\tau-1}^\tau e^{B^2 \cl{D}_Z (\tau - \ol{\tau})} \left(  \grave{\chi} B^2 \cl{N} \tl{U} - B^2 \cl{N} e^{B^2 \lambda_k \ol{\tau} } \tl{U}_{\lambda_k} \right)  d \ol{\tau} \right. \\
		&\left. - \sum_{j = 0}^{K} \tl{Z}_{h_j} \int_{\tau-1}^\tau e^{B^2 h_j (\tau - \ol{\tau})} \left \langle \grave{\chi} B^2 \cl{N} \tl{U} - B^2 \cl{N} e^{B^2 \lambda_k \ol{\tau} } \tl{U}_{\lambda_k}, \tl{Z}_{h_j} \right \rangle_{n-2, \frac{1}{2B^2}} d \ol{\tau} \right|\\
		\le& \left| \int_{\tau-1}^\tau e^{B^2 \cl{D}_Z (\tau - \ol{\tau})} \left( \grave{\chi} B^2 \cl{N} \tl{U} - B^2 \cl{N} e^{B^2 \lambda_k \ol{\tau} } \tl{U}_{\lambda_k}	\right) d \ol{\tau} \right| \\
		& + \sum_{j = 0}^{K} | \tl{Z}_{h_j} | \| \tl{Z}_{h_j} \|_{H^1_{n-2, \frac{1}{2B^2}}} \int_{\tau-1}^\tau e^{B^2 h_j (\tau - \ol{\tau})} \left \| \grave{\chi} B^2 \cl{N} \tl{U} - B^2 \cl{N} e^{B^2 \lambda_k \ol{\tau} } \tl{U}_{\lambda_k} \right \|_{H^{-1}_{n-2, \frac{1}{2B^2}}} d \ol{\tau} \\
		\le & \left| \int_{\tau-1}^\tau e^{B^2 \cl{D}_Z (\tau - \ol{\tau})} \left( \grave{\chi} B^2 \cl{N} \tl{U} - B^2 \cl{N} e^{B^2 \lambda_k \ol{\tau} } \tl{U}_{\lambda_k}	\right) d \ol{\tau} \right| \\
		& + \left( \gamma^{2} + \gamma^{2K+2} \right) 
		\int_{\tau - 1}^{\tau} C_{p,q,k, \Upsilon_U, \ul{D}} e^{B^2 \lambda_k \ol{\tau} - \delta \ol{\tau}} d \ol{\tau}\\
		& + \left( \gamma^{2} + \gamma^{2K+2} \right) 
		\int_{\tau - 1}^{\tau} C_{p,q,k, \ul{D}} \eta_1^U e^{B^2 \lambda_k \ol{\tau} } d \ol{\tau}
		&& (\ref{summH^-1EstErr})\\
		\le & \frac{1}{4}  \nu e^{B^2 \lambda \tau} \left( \gamma^{-2k -2B^2 \lambda} + \gamma^{-2B^2 \lambda} \right) \\
		& + C_{p,q} \eta_1^U e^{B^2 \lambda \tau} \left( \gamma^{-2k -2B^2 \lambda} + \gamma^{-2B^2 \lambda} \right) 	
		&& (\ref{summShortTimeEstN}) \\
		& + C_{p,q,k, \Upsilon_U, \ul{D}, \Gamma} \gamma^{-2k-2B^2 \lambda_k} e^{B^2 \lambda_k \tau} e^{- \delta \tau_0} \\
		& + C_{p,q,k,  \ul{D}, \Gamma} \eta_1^U \gamma^{-2k-2B^2 \lambda_k} e^{B^2 \lambda_k \tau} \\
		\le& \frac{1}{2} \nu  e^{B^2 \lambda \tau} \left( \gamma^{-2k-2B^2 \lambda_k } + \gamma^{-2B^2\lambda_k} \right)\\
		& + C_{p,q,k, \ul{D}, \Gamma} \eta_1^U e^{B^2 \lambda_k \tau} 	\left( \gamma^{-2k-2B^2 \lambda_k } + \gamma^{-2B^2\lambda_k} \right)\\
	\end{aligned} \end{equation*}
	if $\Upsilon_Z \gg 1$ is sufficiently large depending on $p,q, k, l, \Upsilon_U, \ul{D}, \nu$ and
	$\tau_0 \gg 1$ is sufficiently large depending on $p,q, k, M , \beta, \Upsilon_U, \Upsilon_Z$.
	
	By similar logic as in the proof of the above lemmas, the first term can be estimated as follows:
	\begin{equation*} \begin{aligned}
		& \left| \sum_{j = K+1}^\infty \tl{Z}_{h_j}  \int_{\tau_0}^{\tau-1} e^{B^2 h_j (\tau - \ol{\tau})} \left \langle \grave{\chi} B^2 \cl{N} \tl{U} - B^2 \cl{N} e^{B^2 \lambda_k \ol{\tau} } \tl{U}_{\lambda_k}, \tl{Z}_{h_j} \right \rangle_{n-2, \frac{1}{2B^2}} d \ol{\tau} \right| \\
		\le &   \sum_{j = K+1}^\infty | \tl{Z}_{h_j} | \| \tl{Z}_{h_j} \|_{H^1_{n-2, \frac{1}{2B^2}}}
		\int_{\tau_0}^{\tau-1} e^{B^2 h_j ( \tau - \ol{\tau} )} \| \grave{\chi} B^2 \cl{N} \tl{U} - B^2 \cl{N} e^{B^2 \lambda_k \ol{\tau} } \tl{U}_{\lambda_k}  \|_{H^{-1}_{n-2, \frac{1}{2B^2} } } d \ol{\tau} \\
		\le & \sum_{j = K+1}^\infty | \tl{Z}_{h_j} | \| \tl{Z}_{h_j} \|_{H^1_{n-2, \frac{1}{2B^2}}}
		C_{p,q,k, \ul{D}, \Upsilon_U} \int_{\tau_0}^{\tau-1} e^{B^2 h_j ( \tau - \ol{\tau} )}
		e^{B^2 \lambda_k \ol{\tau} - \epsilon \ol{\tau}} d \ol{\tau} \\
		& + \sum_{j = K+1}^\infty | \tl{Z}_{h_j} | \| \tl{Z}_{h_j} \|_{H^1_{n-2, \frac{1}{2B^2}}}
		C_{p,q,k, \ul{D}} \int_{\tau_0}^{\tau-1} e^{B^2 h_j ( \tau - \ol{\tau} )}
		 \eta_1^U e^{B^2 \lambda_k \ol{\tau}} d \ol{\tau}
		 && (\ref{summH^-1EstErr})\\
	\end{aligned} \end{equation*}
	The first term can be bounded by similar logic as in the previous proofs.
	We now bound the second term
	\begin{equation*} \begin{aligned}
		&\quad  \sum_{j = K+1}^\infty | \tl{Z}_{h_j} | \| \tl{Z}_{h_j} \|_{H^1_{n-2, \frac{1}{2B^2}}} \int_{\tau_0}^{\tau-1} e^{B^2 h_j ( \tau - \ol{\tau} ) } \eta_1^U e^{B^2 \lambda_k \ol{\tau} } d \ol{\tau} \\
		& \lesssim_{p,q,k,\Gamma}  \eta_1^U\sum_{j = K+1}^\infty j^{ C } \gamma^2  e^{B^2 h_j  \tau} \int_{\tau_0}^{\tau - 1} e^{B^2 ( \lambda_k - h_j) \ol{\tau} } d \ol{\tau} 
		&& (\ref{eigfuncZH^1GrowthJ} \text{ and } \ref{laguerreGrowthJ})\\
		& \le \eta_1^U \gamma^2 \sum_{j = K+1}^\infty j^{ C }   e^{B^2 h_j  \tau} \frac{1}{B^2 ( \lambda_k - h_j)} e^{B^2 ( \lambda_k - h_j) (\tau - 1) } && ( h_j < \lambda_k \text{ for } j \ge K+1) \\
		& =  \eta_1^U e^{B^2 \lambda_k \tau}  \gamma^2 \sum_{j = K+1}^\infty j^{ C }    \frac{1}{B^2 ( \lambda_k - h_j)} e^{B^2 (- \lambda_k + h_j)  } \\
		& \lesssim_{p,q,k}  \eta_1^U e^{B^2 \lambda_k \tau}  \gamma^2 \sum_{j = K+1}^\infty j^{ C }   \frac{1}{j} e^{-B^2 j  } \\
		& \lesssim  \eta_1^U e^{B^2 \lambda_k \tau}  \gamma^2 \\
		& \lesssim_{p,q,k}  \eta_1^U e^{B^2 \lambda_k \tau}  ( \gamma^{-2k -2B^2 \lambda_k} + \gamma^{-2B^2 \lambda_k} ) \\
	\end{aligned} \end{equation*} 
\end{proof}

%% file: ScalarCurv2.tex
In this section, we investigate the asymptotics of the scalar curvature $R$ at the singularity time $T$.

\subsection{Inner Region}

A formal matched asymptotic argument suggests that
	$$\sup_{x \in S^p \times S^{q+1} } | Rm |(x, t) \sim (T-t)^{-1 - 2 \alpha} \doteqdot \lambda(t)$$
and, if $(M_B = \mathbb{R}^{q+1} \times S^p, g_B)$ denotes the Ricci-flat B{\" o}hm metric (see section \ref{bohmMetric}), then
	$$( S^{p} \times S^{q+1} , \lambda(t) g(t), (  pt , \text{pole}	) ) \rightarrow ( M_B, g_B, 0 \times pt )		\qquad \text{as } t \nearrow T$$
in the $C^\infty$ pointed Cheeger-Gromov topology.	
Hence, there exist locally defined diffeomorphisms
	$$\Phi_t : M_B \dashrightarrow S^p \times S^{q+1}$$
such that the rescaled and reparametrized Ricci flow solutions should satisfy
	$$\tl{g}(t) \doteqdot \lambda(t) \Phi_t^* g(t) \xrightarrow[t \nearrow T]{C_{loc}^\infty(M_B)} g_B$$

Observe that $\tl{g}(t)$ evolves according to
\begin{gather*}
	\partial_t \tl{g} = \frac{\partial_t \lambda}{\lambda} \tl{g} + \lambda \Phi_t^* L_V g - 2 \lambda \Phi_t^* Rc[g] = \frac{\partial_t \lambda}{\lambda} \tl{g} +  L_{\Phi_t^*V} \tl{g} - 2 \lambda  Rc[ \tl{g}] \\
	\implies \frac{1}{\lambda} \partial_t \tl{g} = \frac{\partial_t \lambda }{\lambda^2} \tl{g} +  \frac{1}{\lambda} L_{\Phi_t^*V} \tl{g} - 2  Rc[ \tl{g}]
\end{gather*}
where $\partial_t \Phi_t = V \in \Gamma( \Phi_t^* T( S^p \times S^{q+1} ))$.

After reparametrizing time by setting $\tau = \tau(t)$ such that
	$$d \tau = \lambda(t) dt = (T-t)^{-1-2 \alpha} dt$$
then $\frac{ \partial_t \lambda}{\lambda^2 } = \frac{C_\alpha}{\tau}$  $(C_\alpha > 0)$ and 
this flow becomes
	$$\partial_\tau \tl{g} = \frac{C_\alpha}{\tau} \tl{g} + \frac{1}{\tau } L_{\tl{V}}  \tl{g} - 2 Rc [\tl{g} ]$$
where $ \frac{1}{\lambda} \Phi_t^* V = \frac{1}{\tau} \tl{V} \in \Gamma ( TM_B)$.

To investigate the scalar curvature behavior, first note that
	$$dVol_{\tl{g}} \to dVol_{g_B}	\qquad \text{ suggests } \qquad 		\partial_\tau d Vol_{\tl{g}(\tau)} \to 0$$
and
\begin{equation*} \begin{aligned}
	\partial_\tau dVol_{\tl{g}(\tau)} 
	=& \frac{1}{2} tr_{\tl{g}} ( \partial_\tau \tl{g} ) dVol_{\tl{g}} \\
	=& \frac{1}{2}  tr_{\tl{g}}  \left( \frac{C_\alpha}{\tau} \tl{g} + \frac{1}{\tau} L_{\tl{V}} \tl{g} - 2 Rc[\tl{g}] \right) d Vol_{\tl{g}} \\
	= & \frac{1}{2} \left( \frac{C_\alpha}{\tau} (n+1) + \frac{2}{\tau} div_{\tl{g}}( \tl{V} ) - 2 R[\tl{g}] \right)  d Vol_{\tl{g}}
\end{aligned} \end{equation*}
This suggests that 
\begin{equation*} \begin{aligned}
	 \frac{1}{\lambda} R[ \Phi_t^* g] =& R[\tl{g} ] \\
	 \approx& \frac{1}{\tau} \left( \frac{C_\alpha (n+1)}{2}  + div_{\tl{g}}( \tl{V} ) \right) \\
	\implies R[ \Phi_t^* g] \approx& \frac{\lambda}{\tau}  \left( \frac{C_\alpha (n+1)}{2}  + div_{\tl{g}}(  \tl{V} ) \right) \\
	=& \frac{1}{(T-t)}  \left( \frac{C_\alpha (n+1)}{2}  + div_{\tl{g}}(  \tl{V} ) \right) \\
	\lesssim & \frac{1}{T-t}
\end{aligned} \end{equation*}

\subsection{Parabolic-Inner Overlap}
For $e^{- \alpha_k \tau - \tau/2} \ll \psi \ll e^{-\tau/2}$, write $\psi = \Upsilon e^{- \alpha_k \tau - \tau/2}$ where $1 \ll \Upsilon \ll e^{\alpha_k \tau}$.
For such $\psi$, 
	$$\tl{U} \approx e^{B^2 \lambda_k \tau} \tl{U}_{\lambda_k}(\gamma) \approx \tl{c} e^{B^2 \lambda_k \tau} \gamma^a$$
	$$\tl{Z} \approx e^{B^2 \lambda_k \tau} \tl{Z}_{\lambda_k}(\gamma) \approx \tl{d} e^{B^2 \lambda_k \tau} \gamma^a$$
where
	$$-\tl{d}(a-2)(n+a-1) = -4p B^2 a \tl{c}$$
	$$a = -2k - 2B^2 \lambda_k = \frac{1-n}{2} + \frac{1}{2} \sqrt{ (n-9)(n-1) } < 0$$

We have
	$$\tau = -\log(T-t) 	\qquad \gamma = \psi e^{\tau/2}$$
	$$Z(\gamma, \tau) = z(\psi, t) 		\qquad u(\psi, t) + \tau/2 = U(\gamma, \tau)$$ 
Therefore, for $\Upsilon e^{-\alpha_k \tau - \frac{\tau}{2} } \ll \psi \ll e^{-\tau/2}$,
\begin{gather*}
	z(\psi, t) = Z = B^2 + \tl{Z}(\gamma, \tau) \approx B^2 + \tl{d} e^{B^2 \lambda_k \tau} \gamma^a = B^2 + \tl{d} e^{B^2 \lambda_k \tau + \frac{a}{2} \tau } \psi^a \\
	u(\psi, t) = U - \tau/2 = \log \left( \frac{A}{B} \gamma \right) - \tau/2 + \tl{U} \approx \log \left( \frac{A}{B} \psi \right) + \tl{c} e^{B^2 \lambda_k \tau} \gamma^a \\
	= \log \left( \frac{A}{B} \psi \right) + \tl{c} e^{B^2 \lambda_k \tau + \frac{a}{2} \tau} \psi^a
	\end{gather*}

Set
	$$c \doteqdot \tl{c} e^{B^2 \lambda_k \tau + \frac{a}{2} \tau}$$
	$$d \doteqdot \tl{d}e^{B^2 \lambda_k \tau + \frac{a}{2} \tau}$$	
	
When $p = q = 5$, then $a = -3$ and numerical computations reveal that
\begin{equation*} \begin{aligned}
	R \approx& \frac{20}{3 \psi^{11}} \left( -36 c^3 + 4 c^2 \psi^3 + 6 c \psi^6 + 3 ( e^{- \frac{2c}{\psi^3} } - 1 ) \psi^9 \right) \\
	=& \frac{20}{3 \psi^2} \left( -36 c^3 \psi^{-9} + 4 c^2 \psi^{-6} + 6 c \psi^{-3} + 3 ( e^{- \frac{2c}{\psi^3} } - 1 ) \right) \\
	 \approx & e^{2 \alpha_k \tau + \tau} \frac{20}{3 \Upsilon^2 } \left( -36 \tl{c}^3 \Upsilon^{-9} + 4 \tl{c}^2 \Upsilon^{-6} + 6 \tl{c} \Upsilon^{-3} + 3 ( e^{- 2 \tl{c} \Upsilon^{-3} } - 1 ) \right) \\
	\approx& e^{2 \alpha_k \tau + \tau} \frac{20}{3 \Upsilon^2 } \left( -36 \tl{c}^3 \Upsilon^{-9} + 4 \tl{c}^2 \Upsilon^{-6} + 6 \tl{c} \Upsilon^{-3} + 3 ( -2 \tl{c} \Upsilon^{-3} + 2 \tl{c}^2 \Upsilon^{-6} - \frac{4}{3} \tl{c}^3 \Upsilon^{-9} + ... ) \right) \\
	 =& e^{2 \alpha_k \tau + \tau} \frac{20}{3 \Upsilon^2 } \left( -40 \tl{c}^3 \Upsilon^{-9} + 10 \tl{c}^2 \Upsilon^{-6} \right)
\end{aligned} \end{equation*}

It follows that $R \lesssim e^\tau$ in this region if and only if $e^{2 \alpha_k \tau} \lesssim \Upsilon^8$.
This condition is consistent with $\Upsilon \ll e^{\alpha_k \tau}$.

Moreover, $R \lesssim 1$ in this region if and only if $e^{2 \alpha_k \tau + \tau } \lesssim \Upsilon^8$.
When $p = q = 5$, this condition is consistent with $\Upsilon \ll e^{\alpha \tau}$ if and only if
	$$1 < \frac{8}{9}(- \lambda_k) \iff k \ge 3$$
	
\begin{remark}
We note that the results in this subsection for the parabolic-inner overlap depend only on a numerical computation and not the formal matched asymptotic expansions of the Ricci flow solution.
\end{remark}

\subsection{Parabolic Region}
For $\gamma \sim 1$,
	$$Z( \gamma, \tau) \approx B^2 + e^{B^2 \lambda_k \tau} \tl{Z}_{\lambda_k}(\gamma) $$
	$$U (\gamma, \tau) \approx \log \left( \frac{A}{B} \gamma \right) + e^{B^2 \lambda_k \tau} \tl{U}_{\lambda_k} (\gamma)$$
and
	$$z(\psi, t) \approx B^2 +e^{B^2 \lambda_k \tau} \tl{Z}_{\lambda_k}( \psi e^{-\tau/2} )$$
	$$u(\psi, t) \approx \log \left( \frac{A}{B} \psi \right) + e^{B^2 \lambda_k \tau} \tl{U}_{\lambda_k} (\psi e^{-\tau/2} )$$
For $\gamma \sim 1$, $\tl{U}_{\lambda_k}, \tl{Z}_{\lambda_k}$ and all their derivatives are $O(1)$.\\

Recall that the sectional curvatures are given by
	\begin{equation*} \begin{aligned}
			l = \frac{1-z}{\psi^2} =& \frac{1 - \psi_s^2}{\psi^2} & j = e^{-2u} - zu_\psi^2 =& \frac{1- \phi_s^2}{\phi^2}  \\
			k = -\frac{z_\psi}{2 \psi} =& -\frac{\psi_{ss}}{\psi} & h = -z u_{\psi \psi} - z u_\psi^2 - u_\psi \frac{z_\psi}{2} =& -\frac{ \phi_{ss}}{\phi}  \\
			m = -\frac{z}{\psi} u_\psi =& -\frac{\psi_s \phi_s}{\psi \phi} \\
	\end{aligned} \end{equation*}
and the scalar curvature is given by
	$$R = 2p h + 2qk + p(p-1) j + q(q-1) l + 2pq m$$
	
Let the \textit{perturbed sectional curvatures} refer to the difference of the sectional curvatures from that of the Ricci-flat cone, e.g.
	$$\tl{l} \doteqdot l - \frac{1 - B^2}{\psi^2}$$
For $\gamma \sim 1$, the perturbed sectional curvatures are $O \left(e^{(B^2 \lambda_k +1)\tau} \right)$.
Indeed, 
	\begin{gather*}
		\tl{l} = - \frac{ \tl{z}}{\psi^2} = - \frac{ e^{B^2 \lambda_k \tau} \tl{Z}_{\lambda_k}( \psi e^{\tau/2} )}{\psi^2} = - \frac{ e^{B^2 \lambda_k \tau}}{e^{-\tau}} O(1) \sim e^{(B^2 \lambda_k +1 )\tau} \\
		\tl{k} = - \frac{ \tl{z}_\psi }{2 \psi} = - \frac{ e^{B^2 \lambda_k \tau} e^{\tau/2}  \tl{Z}_{\lambda_k}'( \psi e^{\tau/2})}{2 \psi} \sim e^{(B^2 \lambda_k +1)\tau}
	\end{gather*}
for example.
It follows that
	$$R = 2p \tl{h} + 2q \tl{k} + p(p-1) \tl{j} + q(q-1) \tl{l} + 2pq \tl{m} \ll e^{\tau} = \frac{1}{T-t}$$
which yields the type I scalar curvature bound in this region.

\begin{remark}
Observe that $R$ is bounded in this region if and only if
	$$B^2 \lambda_k + 1 \le 0$$
Recall
	$$B^2 \lambda_k + 1 = - k + \frac{n-1}{4} - \frac{1}{4} \sqrt{(n-9)(n-1)} +1$$
It can be checked, using the monotonicity in $n$ and the behavior at $n=10$, that
	\begin{equation*}
		B^2 \lambda_k + 1 = - k + \frac{n-1}{4} - \frac{1}{4} \sqrt{(n-9)(n-1)} +1 
		\left\{ \begin{array}{cc}
		> 0 & \text{if } k \le 2 \\
		<0 & \text{if } k \ge 3\\
		\end{array} \right. 
	\end{equation*}
\end{remark}

\subsection{Outer Region}
By lemma \ref{outerCurvatureBounds}, the Riemann curvature tensor is $O \left( (T-t)^{-1} \right)$ on the region defined by $\gamma \ge \Gamma$, which in particular implies the type I scalar curvature bound on this region.\\

In summary, the combination of numerical arguments and formal matched asymptotic expansions for the Ricci flow solutions $g_k(t)$ in theorem \ref{mainThmAbridged} given above suggests that the scalar curvature of these solutions blows up no faster than the type I rate, i.e.
	$$\limsup_{t \nearrow T} (T-t) \sup_{x \in M} |R|(x,t) < \infty,$$
despite the fact that the Riemann curvature tensor can blow up at rates given by arbitrarily large powers of $\frac{1}{T-t}$.
We note that these dynamics stand in marked contrast to closed Ricci flows in dimension 3 where the Hamilton-Ivey pinching estimate ~\cite{Hamilton99, Ivey93, Perelman02} implies that scalar curvature controls the Riemann and Ricci curvatures.

%% file: AnalyticFacts2.tex
\subsection{The Weighted Spaces $H^k_{a,b}$}
\begin{definition}
	For $a \in \mathbb{N}$ and $b > 0$, define the weighted spaces $L^2_{a,b}$ as
		$$L^2_{a,b} \doteqdot L^2( \mathbb{R}^{a+1}, e^{-b |x|^2/2} dx )$$
	It is readily seen that $L^2_{a,b}$ is a Hilbert space with respect to the inner product
		$$\langle u, v \rangle_{L^2_{a,b}} \doteqdot \langle u, v \rangle_{a,b} \doteqdot \int_{\mathbb{R}^{a+1}} u(x) v(x) e^{- b |x|^2/2} dx$$
	
	Similarly, denote the associated weighted Sobolev spaces by
		$$H^k_{a,b} \doteqdot H^k( \mathbb{R}^{a+1}, e^{-b |x|^2/2} dx )$$
\end{definition}

Because we will frequently consider rotationally symmetric functions, the radial coordinate will often be denoted as $r = |x|$
and it will be assumed throughout that the $L^2_{a,b}$ inner product is normalized so that
	$$\| u(r) \|_{L^2_{a,b}}^2 = \int_0^\infty u(r)^2 r^a e^{- b r^2/2} dr$$
for rotationally symmetric functions $u$.

\begin{lem} \label{MultCts}
	If $u \in H^1_{a,b}$ then
		$$\| r u \|_{L^2_{a, b}} \lesssim_{a,b} \| u \|_{H^1_{a,b}}$$		
\end{lem}
\begin{proof}
	By density, it suffices to consider $u \in C^\infty_c( \mathbb{R}^{a+1})$.
	For a constant $C$ to be determined, integrate
		$$0 \le ( C u_r - r u)^2$$
	to obtain
	\begin{equation*} \begin{aligned}
		0 \le& \int_0^\infty ( C u_r - r u)^2 r^a e^{- b r^2/2} dr \\
		=&  \int_0^\infty \left( C^2 u_r^2 + r^2 u^2 \right) r^a e^{-b r^2/2} dr -C \int_0^\infty r \partial_r (u^2) r^a e^{-b r^2/2} dr \\
		=& \int_0^\infty \left( C^2 u_r^2 + r^2 u^2 \right) r^a e^{-b r^2/2} dr + C \int_0^\infty \left[ u^2 + ru^2 \left( \frac{a}{r} - br \right) \right] r^a e^{-b r^2/2} dr \\
		\implies& (Cb - 1) \int_0^\infty r^2 u^2 r^a e^{-b r^2/2} dr \le  \int_0^\infty \left[ C^2 u_r^2  + C(a+1) u^2 \right] r^a e^{-b r^2/2} dr \\
	\end{aligned} \end{equation*}
	By taking $C = \frac{2}{b}$ and integrating over $\theta \in \mathbb{S}^a$, the result follows.
\end{proof}

\begin{lem} \label{DivCts}
	Assume $a > 1$.
	If  $u \in H^1_{a,b}$ then
		$$\left\| \frac{1}{r} u \right\|_{L^2_{a, b}} \lesssim_{a,b} \| u \|_{H^1_{a,b}}$$
\end{lem}
\begin{proof}
By density, it suffices to consider $u \in C^\infty_c( \mathbb{R}^{a+1})$.
	For a constant $C$ to be determined, integrate
		$$0 \le \left( C u_r + \frac{1}{r} u \right)^2$$
	to obtain
	\begin{equation*} \begin{aligned}
		0 \le& \int_0^\infty \left( C u_r + \frac{1}{r} u \right)^2 r^a e^{-b r^2/2} dr \\
		= & \int_0^\infty \left[ C^2 u_r^2 + \frac{1}{r^2} u^2 \right] r^a e^{-b r^2/2} dr + \int_0^\infty \frac{C}{r} \partial_r (u^2) r^a e^{-b r^2/2} dr \\
		=& \int_0^\infty \left[ C^2 u_r^2 + \frac{1}{r^2} u^2 \right] r^a e^{-b r^2/2} dr \\
		& - \int_0^\infty \left[ \left( -  \frac{C}{r^2} \right) u^2 + \frac{C}{r} u^2 \left( \frac{a}{r} - b r \right) \right] r^a e^{-br^2/2 } dr + \cancelto{0}{ \left. \frac{C}{r} u^2 r^a e^{- b r^2/2} \right|_0^\infty } \\
		\implies& ( C (a-1) -1) \int_0^\infty \frac{1}{r^2} u^2 r^a e^{-b r^2/2} dr \le \int_0^\infty \left[ C^2 u_r^2 + C b u^2 \right] r^a e^{- b r^2/2} dr\\
	\end{aligned} \end{equation*}
	By taking $C = \frac{2}{a-1}$ and integrating over $\theta \in \mathbb{S}^a$, the result follows.
\end{proof}

\subsection{Exponential Integrals}
\begin{lem} \label{asympsOfExpInt}
	Let $\kappa \in \mathbb{R}$ and $b > 0$ be constants.
	Then the function
		$$F(R) \doteqdot \int_R^\infty r^\kappa e^{ - b r^2/2} dr$$
	satisfies
		$$F(R) = \frac{1}{b} R^{\kappa-1} e^{- b R^2 /2} \big( 1 + o(1) \big) \qquad \text{as } R \nearrow +\infty$$
\end{lem}
\begin{proof}
	In the case that $\kappa = 1$, in fact $F(R) = \frac{1}{b} e^{ - bR^2/2}$.\\
	For $\kappa \ne 1$, integration by parts shows that
		$$F(R) = \int_R^\infty \frac{ r^{\kappa-1} }{-b} \partial_r \left( e^{- b r^2/2} \right) dr = \frac{1}{b} R^{\kappa-1} e^{-b R^2 /2} + \frac{1}{b} \int_R^\infty \frac{ r^{\kappa-2} }{\kappa - 1} e^{-br^2/2} dr $$
	The statement then follows by observing that
		$$\left|  \frac{1}{b} \int_R^\infty \frac{ r^{\kappa-2} }{\kappa - 1} e^{-br^2/2} dr \right| \le \frac{ F(R) }{b | \kappa -1 | R^2 }$$
\end{proof}

\subsection{Laguerre Polynomials}
\begin{lem} \label{laguerreGrowthJ}
	For any $\alpha >-1$ and $M > 0$, there exists $C = C(\alpha, M)$ such that
		$$\left( \frac{ j!}{ \Gamma (j + \alpha +1)} \right)^{1/2} \sup_{x \in (0,M] } | L_j^{(\alpha)} (x) | \le C j^C	\qquad \text{for all } j \in \mathbb{N}$$
\end{lem}
\begin{proof}
	Theorem 8.22.4 of ~\cite{S75} 
	implies that for any $\alpha > -1$
	\begin{equation*} \begin{aligned}
		L_j^{(\alpha)}(x) =& N^{-\alpha /2} \frac{ \Gamma(j + \alpha + 1)}{j!} e^{x/2} x^{-\alpha/2} J_\alpha ( 2 \sqrt{ N x} )  \\
		& +\begin{cases}
			x^{5/4 - \alpha/2} e^{x/2} O(j^{\alpha/2 - 3/4} ), & \text{if } j^{-1} \le x \le M \\
			x^{2} e^{x/2} O(j^\alpha),		&\text{if }0 < x \le j^{-1} \\
		\end{cases}
	\end{aligned} \end{equation*}
	uniformly for $x \in (0, M ]$ where
		$$N = j + (\alpha + 1)/2		\qquad \text{and} \qquad  J_\alpha(x) \text{ is the Bessel function of the first kind.}$$
	The Bessel functions $J_\alpha (z)$ satisfy the bound
		$$| J_\alpha (z) | \le C_\alpha z^\alpha$$
	and Stirling's formula yields that
		$$\frac{ \Gamma ( j + \alpha +1)}{j!}  \sim j^{\alpha/2} \qquad \text{ as } j \nearrow + \infty$$
	It therefore follows that 
	\begin{gather*}
		\left( \frac{j!}{\Gamma(j + \frac{1}{2} \sqrt{(n-9)(n-1)} +1)} \right)^{1/2} \sup_{\gamma \in (0, \Gamma]} \left| L_j^{\left(\frac{1}{2} \sqrt{(n-9)(n-1)} \right)} \left( \frac{ \gamma^2}{4B^2} \right) \right|  
	\end{gather*}
	can be bounded by $C j^{C}$ where $C$ depends only on $\alpha$ and $M$.	
\end{proof}

\subsection{The Linearized Operators} \label{LinearizedOps}
\begin{definition}
	For $a \in \mathbb{N}$ and $b > 0$, let $\mathring{L}^2_{a,b} \subset L^2_{a,b}$
	denote the closed subspace of rotationally symmetric functions
		$$\mathring{L}^2_{a,b} \doteqdot \{ u \in L^2_{a,b} : u(x) = u( |x| ) \}$$
	The associated Sobolev subspaces $\mathring{H}^k_{a,b} \subset H^k_{a,b}$ are analogously defined as
		$$\mathring{H}^k_{a,b} \doteqdot \{ u \in H^k_{a,b} : u(x) = u(|x|) \}$$
	These subspaces are Hilbert spaces with respect to the induced norm.

\end{definition}

\subsubsection{The Operator $\cl{D}_U$}

\begin{definition}
Let $\cl{D}_U : Dom(\cl{D}_U) \to \mathring{L}^2_{n, \frac{1}{2B^2}}$ denote the unbounded linear operator given by
	$$\cl{D}_U  = \frac{\partial^2}{\partial r^2} + \left( \frac{n}{r} - \frac{ r}{2B^2} \right) \frac{\partial}{\partial_r} + \frac{2(n-1)}{r^2}$$
with domain
	$$Dom( \cl{D}_U ) = \{ \tl{U} \in \mathring{H}^1_{n,\frac{1}{2B^2}} : \cl{D}_U \tl{U} \in \mathring{L}^2_{n,\frac{1}{2B^2}} \text{ in the sense of distributions} \}$$
\end{definition}	
	
\begin{prop} \label{opU}
	Assume $n \ge 9$.
	The operator $\cl{D}_U$ can be uniquely extended to a self-adjoint operator on $\mathring{L}^2_{n,\frac{1}{2B^2}}$.
	The self-adjoint extension (also denoted by $\cl{D}_U$) satisfies the following properties:
	\begin{enumerate}
		\item $Dom( \cl{D}_U) \subset \mathring{H}^1_{n,\frac{1}{2B^2}}$
		\item there exists $C > 0$ such that 
			$$\langle \tl{U}, \cl{D}_U \tl{U} \rangle_{n, \frac{1}{2B^2}} \le C  \left\| \tl{U} \right\|_{n, \frac{1}{2B^2}}^2 		\qquad \text{for all } \tl{U} \in Dom( \cl{D}_U)$$
		\item the spectrum $\sigma( \cl{D}_U)$ consists of a countable collection of multiplicity one eigenvalues $\{ \lambda_j \}_{j \in \mathbb{N}}$, where
			$$ B^2 \lambda_j = -j + \frac{n-1}{4} - \frac{1}{4} \sqrt{(n-9)(n-1)}$$
	\end{enumerate}
\end{prop}
\begin{proof}
	Because the proof follows a standard approach in the literature of mathematical physics and is identical to that of lemma 2.3 in ~\cite{HV} and proposition 2.2 in ~\cite{M04Oct}, we do not include the details here.
	For the readers' convenience, we only show where we use the dimension restriction $n \ge 9$. 
	Namely, it follows from the proof of lemma \ref{DivCts} that for any constant $C$
	\begin{equation*} \begin{aligned}
		\langle \tl{U}, \cl{D}_U \tl{U} \rangle_{n, \frac{1}{2B^2}}  
		=& - \int_0^\infty \tl{U}_r^2 r^n e^{-r^2/(4B^2)} dr + \int_0^\infty \frac{2(n-1)}{r^2} \tl{U}^2 r^n e^{- r^2/(4B^2)} dr \\
		\le&  \left[ -\left(  \frac{n-1}{C}  - \frac{1}{C^2} \right) + 2(n-1) \right] \int_0^\infty \frac{1}{r^2} \tl{U}^2 r^n e^{- r^2/(4B^2)} dr \\
		&+ \frac{1}{2B^2 C} \int_0^\infty  \tl{U}^2 r^n e^{- r^2/(4B^2)} dr \\
	\end{aligned} \end{equation*}
	Replacing $C$ with $C^{-1}$ it follows that
	\begin{equation*} \begin{aligned}
		\langle \tl{U}, \cl{D}_U \tl{U} \rangle_{n, \frac{1}{2B^2}}  
		\le & \left( C^2 -C(n-1) + 2(n-1) \right) \left\| \frac{1}{r} \tl{U} \right\|^2_{n, \frac{1}{2B^2} } + \frac{C}{2B^2} \| \tl{U} \|^2_{n, \frac{1}{2B^2} } \\
	\end{aligned} \end{equation*}
	When $n \ge 9$, there exists $C >0$ such that $C^2 - C(n-1) +2(n-1) \le 0$.
	With this choice of $C$, the upper bound
		$$\langle \tl{U}, \cl{D}_U \tl{U} \rangle_{n, \frac{1}{2B^2}}^2 \le \frac{C}{2B^2} \left\| \tl{U} \right\|_{n, \frac{1}{2B^2}}^2 $$
	follows.
\end{proof}

\subsubsection{The Operator $\cl{D}_Z$}
\begin{definition}
Let $\cl{D}_Z : Dom(\cl{D}_Z) \to \mathring{L}^2_{n-2, \frac{1}{2B^2}}$ denote the unbounded linear operator given by
	$$\cl{D}_Z  = \frac{\partial^2}{\partial r^2} + \left( \frac{n-2}{r} - \frac{ r}{2B^2} \right) \frac{\partial}{\partial_r} - \frac{2(n-1)}{r^2}$$
with domain
	$$Dom( \cl{D}_Z ) = \{ \tl{Z} \in \mathring{H}^1_{n-2,\frac{1}{2B^2}} : \cl{D}_Z \tl{Z} \in \mathring{L}^2_{n-2,\frac{1}{2B^2}} \text{ in the sense of distributions} \}$$
\end{definition}	
	
\begin{prop} \label{opZ}
	The operator $\cl{D}_Z$ can be uniquely extended to a self-adjoint operator on $\mathring{L}^2_{n-2,\frac{1}{2B^2}}$.
	The self-adjoint extension (also denoted by $\cl{D}_Z$) satisfies the following properties:
	\begin{enumerate}
		\item $Dom( \cl{D}_Z) \subset \mathring{H}^1_{n-2,\frac{1}{2B^2}}$,
		\item there exists $C > 0$ such that 
			$$\langle \tl{Z}, \cl{D}_Z \tl{Z} \rangle_{n-2, \frac{1}{2B^2}} \le C \left\| \tl{Z} \right\|_{n-2, \frac{1}{2B^2}}	^2	\qquad \text{for all } \tl{Z} \in Dom( \cl{D}_Z)$$
		\item the spectrum $\sigma( \cl{D}_Z)$ consists of a countable collection of multiplicity one eigenvalues $\{ h_j \}_{j \in \mathbb{N}}$, where
			$$B^2 h_j = -j -1$$
	\end{enumerate}
\end{prop}

\begin{remark}
	Due to the sign of the zeroth order coefficient of the $\cl{D}_Z$ operator, no additional dimension restriction is necessary as in the case of proposition \ref{opU}.
\end{remark}

\subsection{Eigenfunctions} \label{Eigenfunctions}
The eigenfunctions of the operators $\cl{D}_U, \cl{D}_Z$ can be written explicitly in terms of generalized Laguerre polynomials.
Here, we collect several facts regarding these eigenfunctions and the generalized Laguerre polynomials that will be used throughout the paper.
The reader may consult ~\cite{S75} for additional information on generalized Laguerre polynomials.

\subsubsection{Eigenfunctions for $\cl{D}_U$}
\begin{prop} \label{eigConstU}
	For any $k \in \mathbb{Z}_{\ge 0}$, an eigenfunction $\tl{U}_{\lambda_k}$ of $\cl{D}_U$ with eigenvalue $\lambda_k$ is given by
		$$\tl{U}_{\lambda_k}(\gamma) \doteqdot c_k \gamma^{-2k-2B^2 \lambda_k} L_k^{\left( \frac{1}{2} \sqrt{(n-9)(n-1)} \right)} \left( \frac{ \gamma^2}{4B^2} \right)$$
	where $L_k^{\left( \frac{1}{2} \sqrt{(n-9)(n-1)} \right)} ( \cdot)$ is a generalized Laguerre polynomial.
	
	In particular, the $\{ \tl{U}_{\lambda_k} (\gamma) \}_{k \in \mathbb{Z}_{\ge 0}}$ are orthonormal in $L^2_{n, \frac{1}{2B^2}}$ if and only if
		$$c_k^2 = \frac{1}{2B^2} \left(\frac{1}{4B^2} \right)^\alpha \frac{ k!}{ \Gamma( k +\frac{1}{2} \sqrt{(n-9)(n-1)} + 1)}$$
	where $\Gamma( \cdot) $ is the gamma function.
\end{prop}
\begin{proof}
	After noting that the standard Kummer function $M(a,b;z)$ can be written as a generalized Laguerre polynomial when $a \in \mathbb{Z}_{\le 0}$, the proof of the first statement proceeds as in the proofs of lemma 2.3 in ~\cite{HV} and proposition 2.2 in ~\cite{M04Oct}.

	The second statement follows from the change of variables $x = \frac{ \gamma^2}{4B^2}$ and the fact that
		$$\int_0^\infty x^\alpha L_k^{(\alpha)} ( x) L_m^{(\alpha)}(x) e^{-x} dx = \frac{ \Gamma(k+ \alpha + 1 )}{k!} \delta_{km}$$
\end{proof}

\begin{prop} \label{eigfuncUAsymps}
	For $\tl{U}_{\lambda_k}$ as in proposition \ref{eigConstU}, $\tl{U}_{\lambda_k}$ has asymptotics
	\begin{equation*} \begin{aligned}
		\tl{U}_{\lambda_k}(\gamma) =& c_k {k+\alpha \choose k} \gamma^{-2k - 2B^2 \lambda_k} \big(1 + o(1) \big)  \qquad& \text{as } \gamma \searrow 0 \\
		\tl{U}_{\lambda_k}(\gamma) =& c_k \frac{1}{k!} \frac{1}{ (-4B^2)^{k}} \gamma^{-2B^2 \lambda_k}\big(1 + o(1) \big)	 \qquad& \text{as } \gamma \nearrow +\infty\\
	\end{aligned} \end{equation*}
\end{prop}

\begin{remark}
	Note that
		$$2k + 2B^2 \lambda_k = \frac{n-1}{2} - \frac{1}{2} \sqrt{(n-9)(n-1)}$$
	depends only on $n = p+q$ and is in fact independent of $k$.
	Nonetheless, we shall write many estimates with exponents $2k + 2B^2 \lambda_k$ to display the unity with proposition \ref{eigConstU}.
\end{remark}

\begin{prop} \label{eigfuncUH^1GrowthJ}
	For the orthonormal $\tl{U}_{\lambda_j}$ as in proposition \ref{eigConstU},
		$$\| \tl{U}_{\lambda_j} \|_{H^1_{n,\frac{1}{2B^2}}} \lesssim_{p,q} \sqrt{|\lambda_j|} \lesssim_{p,q} \sqrt{j}$$
\end{prop}
\begin{proof}
	This proposition follows from an explicit computation with the generalized Laguerre polynomials or by similar logic as in the proof of lemma 2.4 in ~\cite{HV}.
\end{proof}

\begin{prop} \label{PKernelU}
	For $0 \le r < 1$, the Poisson kernel for Laguerre polynomials defined by
	$$P_{\frac{1}{2} \sqrt{(n-9)(n-1)}} \left( \frac{ \gamma^2}{4 B^2} ; \frac{  \ol{\gamma}^2}{4B^2}; r \right) \doteqdot \sum_{j=0}^\infty \frac{	r^j L_j^{\frac{1}{2} \sqrt{ (n-9)(n-1)}} \left( \frac{  \gamma^2}{4B^2} \right) L_j^{\frac{1}{2} \sqrt{ (n-9)(n-1)}}	 \left( \frac{  \ol{\gamma}^2}{4B^2} \right)	}{ \Gamma( \frac{1}{2} \sqrt{ (n-9)(n-1)} + 1 ) {j + \frac{1}{2} \sqrt{ (n-9)(n-1)}\choose j}}$$
	satisfies
	\begin{equation*} \begin{aligned}
		& P_{\frac{1}{2} \sqrt{ (n-9)(n-1)} } \left( \frac{  \gamma^2}{4B^2} ; \frac{  \ol{\gamma}^2}{4B^2}; r \right) \\
		=& (1-r)^{-1} e^{ - \frac{ \gamma^2 + \ol{\gamma}^2 }{4B^2}  \frac{r}{1-r} } \left( - \frac{\gamma^2 \ol{\gamma}^2 r}{16B^4} \right)^{-\frac{1}{4} \sqrt{ (n-9)(n-1) } } J_{\frac{1}{2} \sqrt{ (n-9)(n-1) } } \left( \frac{ \gamma \ol{\gamma} \sqrt{ - r}}{2B^2(1-r)}  \right) \\
		=& (1-r)^{-1} e^{ - \frac{  \gamma^2 + \ol{\gamma}^2  }{4B^2}  \frac{r}{1-r} } \left(  \frac{ \gamma \ol{\gamma} \sqrt{r} }{4B^2 }\right)^{-\frac{1}{2} \sqrt{ (n-9)(n-1) } } I_{\frac{1}{2} \sqrt{ (n-9)(n-1) } } \left( \frac{ \gamma \ol{\gamma} \sqrt{ r}}{2B^2(1-r)}  \right)\\
	\end{aligned} \end{equation*}
	where $J_{\cdot} ( \cdot)$ is the Bessel function of the first kind and $I_{\cdot} (\cdot)$ is the modified Bessel function of the first kind.
\end{prop}
\begin{proof}
	This proposition is the content of theorem 5.1 in ~\cite{S75}.
\end{proof}

\begin{prop} \label{kernelU}
	For $\tau > \tau_0$ and $\tl{U} \in \mathring L^2_{n, \frac{1}{2B^2} }$
	\begin{equation*} \begin{aligned}
	& e^{B^2 \cl{D}_U (\tau - \tau_0)} \tl{U}(\gamma) \\
	=&  \frac{1}{2B^2} \left( \frac{1}{4B^2} \right)^{\frac{1}{2} \sqrt{(n-9)(n-1)}} \left( \gamma e^{-\frac{1}{2} (\tau - \tau_0)} \right)^{-2k-2B^2 \lambda_k} \\
	&\cdot \int_0^\infty \ol{\gamma}^{-2k-2B^2\lambda_k} P_{\frac{1}{2} \sqrt{(n-9)(n-1)}} \left( \frac{  \gamma^2}{4B^2 } ; \frac{  \ol{\gamma}^2}{4B^2}; e^{-(\tau-\tau_0)} \right) \tl{U}(\ol{\gamma}) \ol{\gamma}^n e^{-\ol{\gamma}^2/(4B^2)} d \ol{\gamma}\\
	=& \frac{1}{2B^2} \left( \gamma e^{-(\tau - \tau_0)/2} \right)^{ \frac{1-n}{2} } \frac{ exp \left[ - \frac{1}{4B^2} \frac{ e^{-(\tau - \tau_0)}  }{ 1 - e^{-(\tau - \tau_0)} } \gamma^2 \right]}{1 - e^{- (\tau - \tau_0)} } \\
	 &\cdot \int_0^\infty I_{\frac{1}{2} \sqrt{ (n-9)(n-1) } } \left( \frac{1}{2B^2} \frac{  \gamma \ol{\gamma} e^{- \frac{1}{2}(\tau - \tau_0)}}{1-e^{- (\tau - \tau_0)}}  \right) \tl{U}(\ol{\gamma}) \ol{\gamma}^{\frac{n+1}{2}} e^{ - \frac{1}{4B^2} \left(  \frac{1}{1- e^{-(\tau - \tau_0)} } \right) \ol{\gamma}^2} d \ol{\gamma}
	\end{aligned} \end{equation*} 
	Moreover, setting
	\begin{equation*} \begin{aligned} 
	& T_U(  \gamma, \ol{\gamma} ; \tau - \tau_0) \\
	\doteqdot &exp \left[ - \frac{1}{4B^2} \frac{ e^{-(\tau - \tau_0)} \gamma^2 + \ol{\gamma}^2}{ 1 - e^{-(\tau - \tau_0)} } \right] I_{\frac{1}{2} \sqrt{ (n-9)(n-1) } } \left( \frac{1}{2B^2} \frac{  \gamma \ol{\gamma} e^{- \frac{1}{2}(\tau - \tau_0)}}{1-e^{- (\tau - \tau_0)}}  \right) \gamma^{- \frac{1}{2} \sqrt{ (n-9)(n-1) } } 
	\end{aligned} \end{equation*}
	and
	\begin{equation*} \begin{aligned}
	&H_U( \gamma , \ol{\gamma} ; \tau - \tau_0) \\
	\doteqdot & exp \left[ - \frac{1}{4B^2} \frac{ (e^{- (\tau - \tau_0)/2} \gamma - \ol{\gamma} )^2 }{ 1 - e^{-(\tau - \tau_0)} } \right]  \left( 1 + \frac{1}{2B^2} \frac{ e^{- (\tau -\tau_0)/2} \gamma \ol{\gamma}  }{1 - e^{-(\tau - \tau_0)} } \right)^{- \frac{1}{2} - \frac{1}{2} \sqrt{(n-9)(n-1) } }
	\end{aligned} \end{equation*}
	then
		$$e^{B^2 \cl{D}_U(\tau - \tau_0) } \tl{U} (\gamma) = \frac{1}{2B^2} \gamma^{-2k - 2B^2 \lambda_k } \frac{ e^{\frac{n-1}{4}(\tau - \tau_0) } }{1 - e^{-(\tau - \tau_0)} } \int_0^\infty T_U(\gamma, \ol{\gamma} ; \tau - \tau_0) \tl{U}(\ol{\gamma} ) \ol{\gamma}^{ \frac{n+1}{2} } d \ol{\gamma}$$
	and for $\tau_0 < \tau  <\tau_0 + 1$
	\begin{equation*} \begin{aligned}
		 |e^{ B^2 \cl{D}_U(\tau - \tau_0) } \tl{U} (\gamma) | 
		\lesssim_{p,q} & \gamma^{-2k - 2B^2\lambda_k} ( \tau - \tau_0)^{-1-\frac{1}{2} \sqrt{ (n-9)(n-1) } } \\
		& \cdot \int_0^\infty | \tl{U}(\ol{\gamma} , \tau_0 ) | H_U(\gamma, \ol{\gamma} ; \tau - \tau_0) \ol{\gamma}^{\frac{n+1}{2} + \frac{1}{2} \sqrt{ (n-9)(n-1) } } d \ol{\gamma}
	\end{aligned} \end{equation*}
\end{prop}

\begin{proof}
	The first two equalities follow from
		$$e^{B^2 \cl{D}_U (\tau - \tau_0)} \tl{U} = \sum_{l = 0}^\infty e^{B^2 \lambda_l (\tau - \tau_0)} \langle \tl{U}, \tl{U}_{\lambda_l} \rangle_{ n, \frac{1}{2B^2}} \tl{U}_{\lambda_l}$$
	and propositions \ref{eigConstU} and \ref{PKernelU}.
	Rewriting the second equality in terms of $T_U$ immediately yields the next statement.
	
	Finally, note that the asymptotics of modified Bessel functions imply that 
		$$\left| I_{ \frac{1}{2} \sqrt{ (n-9)(n-1)} } (x) \right| \lesssim_{n} \frac{ x^{ \frac{1}{2} \sqrt{ (n-9)(n-1)}} e^{x} }{( 1 + x)^{ \frac{1}{2} \sqrt{ (n-9)(n-1)} + \frac{1}{2} } } \qquad (\text{for all } x > 0)$$
	From this estimate and the fact that $1 - e^{-x}$ is locally Lipschitz in $x$, the final statement of the proposition follows.
\end{proof}

\begin{lem} \label{maximalFuncBoundU}
		The action of the semigroup $e^{ B^2 \cl{D}_U(\tau- \tau_0)}$ on $\tl{U} \in \mathring L^2_{n, \frac{1}{2B^2} }$ can be estimated by
			$$\left| e^{ B^2 \cl{D}_U (\tau - \tau_0)} \tl{U} (\gamma) \right| \lesssim_{p,q} \left( \gamma e^{- \frac{1}{2} (\tau - \tau_0)} \right)^{ -2k-2B^2\lambda_k} (\cl{M}_U \tl{U})(\gamma)$$
		where $\cl{M}_U \tl{U}$ is the maximal function
		\begin{equation} \label{maximalFuncU}
			\cl{M}_U \tl{U} (\gamma) \doteqdot \sup_{I \ni \gamma} \frac{ \int_I | \tl{U}(\ol{\gamma})| \ol{\gamma}^{ 2k +2B^2 \lambda_k }  \ol{\gamma}^{1 + \sqrt{ (n-9)(n-1)} } e^{- b \ol{\gamma}^2/2} d \ol{\gamma}}{ \int_I \ol{\gamma}^{1 + \sqrt{ (n-9)(n-1)} } e^{- b \ol{\gamma}^2/2} d \ol{\gamma}}
		\end{equation}
		where the supremum is taken over all intervals $I \subset \mathbb{R}_{>0}$ containing $\gamma$.\\
		
		Additionally, if $| \tl{U}(\ol{\gamma}) | \ol{\gamma}^{  2k+ 2B^2 \lambda_k }$ is nonincreasing (nondecreasing), the the supremum in the definition of the maximal function is obtained by the interval $I = [ 0, \gamma ]$ ($I = [\gamma, \infty)$).
\end{lem}

\begin{proof}
	By proposition \ref{kernelU} and nonnegativity of the Poisson kernel 
		$$P_{\frac{1}{2} \sqrt{(n-9)(n-1)}} \left( \frac{ \gamma^2}{4 B^2} ; \frac{  \ol{\gamma}^2}{4B^2};  e^{-(\tau - \tau_0)} \right) \ge 0$$
	it follows that 
	\begin{equation*} \begin{aligned}
		& \left| e^{ B^2 \cl{D}_U (\tau - \tau_0)} \tl{U} (\gamma) \right| \\
		\lesssim_{p,q} & \left( \gamma e^{-\frac{1}{2} (\tau - \tau_0)} \right)^{-2k-2B^2 \lambda_k} \\
		&\cdot \int_0^\infty \ol{\gamma}^{-2k-2B^2\lambda_k} P_{\frac{1}{2} \sqrt{(n-9)(n-1)}} \left( \frac{  \gamma^2}{4B^2 } ; \frac{  \ol{\gamma}^2}{4B^2}; e^{-(\tau-\tau_0)} \right) | \tl{U}(\ol{\gamma}) |\ol{\gamma}^n e^{-b \ol{\gamma}^2/2} d \ol{\gamma}\\
		=& \left( \gamma e^{-\frac{1}{2} (\tau - \tau_0)} \right)^{-2k-2B^2 \lambda_k} \\
		&\cdot \int_0^\infty  P_{\frac{1}{2} \sqrt{(n-9)(n-1)}} \left( \frac{  \gamma^2}{4B^2 } ; \frac{  \ol{\gamma}^2}{4B^2}; e^{-(\tau-\tau_0)} \right) \\
		& \qquad \cdot | \tl{U}(\ol{\gamma}) | \ol{\gamma}^{2k + 2B^2 \lambda_k} \ol{\gamma}^{1 + \sqrt{(n-9)(n-1)} } e^{-b \ol{\gamma}^2/2} d \ol{\gamma}\\
	\end{aligned} \end{equation*}	
	By applying the change of variables,
		$$x = \frac{\ol{\gamma}^2}{4B^2} \qquad y = \frac{\gamma^2}{4B^2} \qquad f(x) = | \tl{U}( \ol{\gamma}) | \ol{\gamma}^{2k + 2B^2 \lambda_k}$$
	we obtain
	\begin{equation*} \begin{aligned}
		& \left| e^{ B^2 \cl{D}_U (\tau - \tau_0)} \tl{U} (\gamma) \right| \\
		\lesssim_{p,q} & ( \gamma e^{-\frac{1}{2} (\tau - \tau_0)} )^{-2k-2B^2 \lambda_k} \\
		&\cdot \int_0^\infty  P_{\frac{1}{2} \sqrt{(n-9)(n-1)}} \left( y ; x; e^{-(\tau-\tau_0)} \right) f(x) x^{\frac{1}{2} \sqrt{(n-9)(n-1)} } e^{-x} d x\\
	\end{aligned} \end{equation*}
	A result of Muckenhoupt ~\cite{M69} then yields that the integral can be bounded by
	$$\sup_{I \ni y} \frac{ \int_{I} f(x) x^{\frac{1}{2} \sqrt{(n-9)(n-1)} } e^{-x} d x }{\int_{I} x^{\frac{1}{2} \sqrt{(n-9)(n-1)} } e^{-x} d x} = \sup_{I \ni \gamma} \frac{ \int_{I} | \tl{U}( \ol{\gamma}) | \ol{\gamma}^{2k + 2B^2 \lambda_k} \ol{\gamma}^{1+ \sqrt{(n-9)(n-1)} } e^{-b \ol{\gamma}^2/2} d \ol{\gamma} }{\int_{I} \ol{\gamma}^{1+ \sqrt{(n-9)(n-1)} } e^{-b \ol{\gamma}^2/2} d\ol{\gamma}}$$
	This completes the proof of the first statement of the lemma.
			
	The second statement in the lemma follows from writing $I = [\gamma_0, \gamma_1]$, differentiating 
		$$ \frac{ \int_{\gamma_0}^{\gamma_1} | \tl{U}(\ol{\gamma})| \ol{\gamma}^{ 2k +2B^2 \lambda_k }  \ol{\gamma}^{1 + \sqrt{ (n-9)(n-1)} } e^{- b \ol{\gamma}^2/2} d \ol{\gamma}}{ \int_{\gamma_0}^{\gamma_1} \ol{\gamma}^{1 + \sqrt{ (n-9)(n-1)} } e^{- b \ol{\gamma}^2/2} d \ol{\gamma}}$$
	with respect to $\gamma_0$ and $\gamma_1$, and using the monotonicity of $ | \tl{U}(\ol{\gamma})| \ol{\gamma}^{ 2k +2B^2 \lambda_k }  $ to deduce that the above quotient of integrals is monotonic in $\gamma_0$ and $\gamma_1$.
\end{proof}

\subsubsection{Eigenfunctions for $\cl{D}_Z$}
Many of the proofs in this subsection are identical to those in the previous subsection and so we omit the proofs here.

\begin{prop} \label{eigConstZ}
	For any $j \in \mathbb{Z}_{\ge 0}$, an eigenfunction $\tl{Z}_{h_j}$ of $\cl{D}_Z$ with eigenvalue $h_j$ is given by
		$$\tl{Z}_{h_j} (\gamma) \doteqdot d_j \gamma^2 L_j^{\left( \frac{n+1}{2} \right)} \left( \frac{\gamma^2}{4 B^2} \right)$$
		
	In particular, the $\{ \tl{Z}_{h_j} \}_{j \in \mathbb{Z}_{\ge 0} }$ are orthonormal in $L^2_{n-2, \frac{1}{2B^2}}$ if and only if
		$$d_j^2 \doteqdot \frac{1}{B} \left( \frac{1}{4B^2} \right)^{\frac{n+2}{2}} \frac{j !}{\Gamma( j + \frac{n+1}{2} + 1)}$$
\end{prop}

\begin{prop} \label{eigfuncZAsymps}
	For $\tl{Z}_{h_j}$ as in proposition \ref{eigConstZ}, $\tl{Z}_{h_j}$ has asymptotics
	\begin{equation*} \begin{aligned}
		\tl{Z}_{h_j} =&  d_j {j + \frac{n+1}{2} \choose j} \gamma^2 \big(1 + o(1) \big) &	\qquad & \text{as } \gamma \searrow 0 \\
		\tl{Z}_{h_j} =&  d_j (-1)^j \frac{1}{j!} \left( \frac{1}{4B^2} \right)^j \gamma^{2j + 2} \big( 1 + o(1) \big) &		& \text{as } \gamma \nearrow \infty \\
	\end{aligned} \end{equation*}
\end{prop}

\begin{prop} \label{ZlambdaAsymps}
	$\tl{Z}_{\lambda_k}$ has asymptotics 
		$$\tl{Z}_{\lambda_k}(\gamma) \approx c_k' \gamma^{-2k - 2B^2 \lambda_k} \qquad \gamma \searrow 0$$
		$$\tl{Z}_{\lambda_k}(\gamma) \approx c_k'' \gamma^{-2B^2 \lambda_k - 2} \qquad \gamma \nearrow +\infty$$
	where
		$$c_k' = \frac{ 4pB^2 (-2k - 2B^2 \lambda_k) }{ (-2k - 2B^2 \lambda_k -2)(n-1-2k - 2B^2 \lambda_k)} c_k {k+\alpha \choose k}$$
		$$c_k'' = \frac{ 8p B^6 \lambda_k}{2B^2 \lambda_k + 1} c_k \frac{1}{k!} \frac{1}{ (-4B^2)^{k}}$$
\end{prop}

\begin{prop} \label{eigfuncZH^1GrowthJ}
	For the orthonormal $\tl{Z}_{h_j}$ as in proposition \ref{eigConstZ},
		$$\| \tl{Z}_{h_j} \|_{H^1_{n-2,\frac{1}{2B^2}}} \lesssim_{p,q} \sqrt{j}$$
\end{prop}

\begin{prop} \label{PKernelZ}
	For $0 \le r < 1$, the Poisson kernel
	$$P_{\frac{n+1}{2} } \left( \frac{ \gamma^2}{4 B^2} ; \frac{  \ol{\gamma}^2}{4B^2}; r \right) \doteqdot \sum_{j=0}^\infty \frac{	r^j L_j^{\frac{n+1}{2} } \left( \frac{  \gamma^2}{4B^2} \right) L_j^{\frac{n+1}{2} }	 \left( \frac{  \ol{\gamma}^2}{4B^2} \right)	}{ \Gamma( \frac{n+1}{2}  + 1 ) {j + \frac{n+1}{2} \choose j}}$$
	satisfies
	$$P_{\frac{n+1}{2}  } \left( \frac{  \gamma^2}{4B^2} ; \frac{  \ol{\gamma}^2}{4B^2}; r \right) =	 (1-r)^{-1} e^{ - \frac{ \gamma^2 + \ol{\gamma}^2 }{4B^2}  \frac{r}{1-r} } \left( - \frac{\gamma^2 \ol{\gamma}^2 r}{16B^4} \right)^{-\frac{n+1}{4}  } J_{\frac{n+1}{2}  } \left( \frac{ \gamma \ol{\gamma} \sqrt{ - r}}{2B^2(1-r)}  \right)$$
	$$= (1-r)^{-1} e^{ - \frac{  \gamma^2 + \ol{\gamma}^2  }{4B^2}  \frac{r}{1-r} } \left(  \frac{ \gamma \ol{\gamma} \sqrt{r} }{4B^2 }\right)^{-\frac{n+1}{2}  } I_{\frac{n+1}{2} } \left( \frac{ \gamma \ol{\gamma} \sqrt{ r}}{2B^2(1-r)}  \right)$$
	where $J_{\cdot} ( \cdot)$ is the Bessel function of the first kind and $I_{\cdot} (\cdot)$ is the modified Bessel function of the first kind.
\end{prop}

\begin{prop} \label{kernelZ}
	For $\tau > \tau_0$ and $\tl{Z} \in \mathring L^2_{n-2, \frac{1}{2B^2} }$
	\begin{equation*} \begin{aligned}
	&e^{B^2 \cl{D}_Z(\tau - \tau_0)} \tl{Z}(\gamma)  \\
	= & \frac{1}{2B^2} \left( \frac{1}{4B^2} \right)^{\frac{n+1}{2} } \left( \gamma e^{ - \frac{1}{2} (\tau - \tau_0)} \right)^2 \\
	& \cdot \int_0^\infty \ol{\gamma}^2 P_{ \frac{n+1}{2} } \left( \frac{ \gamma^2}{4B^2} ; \frac{  \ol{\gamma}^2}{4B^2}; e^{- (\tau - \tau_0)} \right) \tl{Z} ( \ol{\gamma} ) \ol{\gamma}^{n-2} e^{- \ol{\gamma}^2/(4B^2)} d\ol{\gamma} \\
	=& \frac{1}{2B^2} \left( \gamma e^{- \frac{1}{2}(\tau - \tau_0)} \right)^{ \frac{3-n}{2} } 	\frac{ exp \left[ -\frac{1}{4B^2} \frac{e^{- (\tau- \tau_0)} }{ 1 - e^{ - (\tau-\tau_0)}} \gamma^2 \right] }{1 - e^{- (\tau - \tau_0)} } \\
	& \cdot \int_0^\infty I_{ \frac{n+1}{2} }\left( \frac{1}{2B^2}\frac{  \gamma \ol{\gamma} e^{- \frac{1}{2}(\tau - \tau_0)}}{1-e^{- (\tau - \tau_0)}}  \right) \tl{Z}(\ol{\gamma}) \ol{\gamma}^{ \frac{n-1}{2} } 		e^{ - \frac{1}{4B^2} \left( \frac{ 1 }{ 1 - e^{-(\tau- \tau_0)} } \right) \ol{\gamma}^2 } d \ol{\gamma}
	\end{aligned} \end{equation*} 
	
	Moreover, if 
		$$T_Z(  \gamma, \ol{\gamma}, ; \tau - \tau_0) \doteqdot exp \left[ - \frac{1}{4B^2} \frac{ e^{-(\tau - \tau_0)} \gamma^2 + \ol{\gamma}^2}{ 1 - e^{-(\tau - \tau_0)} } \right] I_{ \frac{n+1}{2} } \left( \frac{1}{2B^2} \frac{  \gamma \ol{\gamma} e^{- \frac{1}{2}(\tau - \tau_0)}}{1-e^{- (\tau - \tau_0)}}  \right) \gamma^{- \frac{n+1}{2} } $$

	and
		$$H_Z( \gamma , \ol{\gamma} ; \tau - \tau_0) \doteqdot exp \left[ - \frac{1}{4B^2} \frac{ (e^{- (\tau - \tau_0)/2} \gamma - \ol{\gamma} )^2 }{ 1 - e^{-(\tau - \tau_0)} } \right]  \left( 1 + \frac{1}{2B^2} \frac{ e^{- (\tau -\tau_0)/2} \gamma \ol{\gamma}  }{1 - e^{-(\tau - \tau_0)} } \right)^{- \frac{1}{2} - \frac{n+1}{2}}$$
	then
		$$e^{B^2 \cl{D}_Z ( \tau - \tau_0)} \tl{Z} (\gamma) = \frac{1}{2B^2} \gamma^2 		\frac{ e^{ \frac{n-3}{4} (\tau - \tau_0)} }{ 1 - e^{- (\tau - \tau_0)} } 		\int_0^\infty T_Z (\gamma, \ol{\gamma}; \tau - \tau_0 ) \tl{Z}( \ol{\gamma} ) \ol{\gamma}^{ \frac{n-1}{2} } d \ol{\gamma}$$
	and for $\tau_0 < \tau < \tau_0 +1$
		$$\left| e^{B^2 \cl{D}_Z ( \tau - \tau_0)} \tl{Z} (\gamma) \right| \lesssim_{p,q} 		\gamma^2 (\tau - \tau_0)^{-1 - \frac{n+1}{2} } \int_0^\infty H_Z (\gamma, \ol{\gamma}; \tau - \tau_0) | \tl{Z}(\ol{\gamma}) | \ol{\gamma}^n d \ol{\gamma}$$
		\end{prop}

\begin{lem} \label{maximalFuncBoundZ}
		The action of the semigroup $e^{(\tau - \tau_0)B^2 D_Z}$ on $\tl{Z} \in \mathring L^2_{n-2, \frac{1}{2B^2}}$ can be estimated by
			$$\left| e^{B^2 D_Z (\tau - \tau_0)} \tl{Z}(\gamma) \right| \lesssim_{p,q} \gamma^2 e^{-(\tau - \tau_0)} ( \cl{M}_Z \tl{Z})(\gamma)$$
		where $\cl{M}_Z$ is the maximal function
		\begin{equation} \label{maximalFuncZ}
			(\cl{M}_Z \tl{Z} ) (\gamma) \doteqdot \sup_{I \ni \gamma} \frac{ \int_I | \tl{Z} (\ol{\gamma})| \ol{\gamma}^{-2} \ol{\gamma}^{n+2} e^{- b \ol{\gamma}^2/2} d \ol{\gamma} }{ \int_I \ol{\gamma}^{n+2} e^{- b \ol{\gamma}^2/2} d\ol{\gamma} }
		\end{equation}
			
		Also, if $| \tl{Z}(\ol{\gamma})| \ol{\gamma}^{-2}$ is nonincreasing (nondecreasing) then the supremum in the definition of the maximal function is achieved on $I = [0, \gamma]$ ( $I = [\gamma, \infty)$).
\end{lem}

%% file: AppendixOfConstants.tex
	For reference, we list the constants used throughout the paper.
	This list is ordered such that the each constant depends only on the constants above it.
	
	\begin{itemize}
	\item	$p,q \in \mathbb{N}$ denote the dimensions of the spheres $M = S^p \times S^{q+1}$.
	We restrict our attention to $q \ge 10$ and $p \ge 2$.
	
	\item $n = p+q$
	
	\item $ A = \sqrt{ \frac{p-1}{n-1} }$ and $ B = \sqrt{ \frac{q-1}{n-1} }$ are the constants for which the cone metric
		$$g_{RFC} = ds^2 + A^2 s^2 g_{S^p} + B^2 s^2 g_{S^q}$$
		is Ricci-flat.
	
	\item	$k \in \mathbb{N}$ is the parameter that controls the choice of eigenfunction that dominates the behavior of the solution at the parabolic scale.
	The corresponding eigenvalue $\lambda_k$ satisfies
		$$B^2 \lambda_k = - k + \frac{ n-1}{4} - \frac{1}{4} \sqrt{(n-9)(n-1)}$$
	We reserve ``$k$" for eigenvalues such that $B^2 \lambda_k < 0$ and use ``$j$" for general eigenvalues $\lambda_j$ without a sign restriction.

\begin{remark}
		$$\frac{ n-1}{4} - \frac{1}{4} \sqrt{(n-9)(n-1)} \searrow 1 	\qquad \text{ as } n \nearrow \infty$$
	In other words, $$\frac{ n-1}{4} - \frac{1}{4} \sqrt{(n-9)(n-1)} = 1 + \epsilon_n$$
	Consequently, $B^2 \lambda_k < 0$ implies $k \ge 2$ in which case
		$$B^2 \lambda_k \le -1 + \epsilon_n< 0$$
	Also,
		$$n \ge 10 \implies \frac{ n-1}{4} - \frac{1}{4} \sqrt{(n-9)(n-1)} \le \frac{3}{2}$$
		$$n \ge 10, k \ge 2 \implies B^2 \lambda_k \le - \frac{1}{2}$$
\end{remark}

\begin{remark}
	$$2k + 2B^2 \lambda_k = \frac{n-1}{2} - \frac{1}{2} \sqrt{ (n-9)(n-1)} > 0$$
	depends only $n$.
	In particular, it is independent of $k$.
\end{remark}
	
	\item $\alpha = \alpha_{p,q,k}$ controls the scale where we transition from the ``inner region" to the ``parabolic region" and is defined by
	
	$$\alpha = \alpha_{p,q,k} = \frac{ - B^2 \lambda_k}{2k + 2 \lambda_k B^2 } = \frac{-2B^2 \lambda_k }{n-1 - \sqrt{(n-9)(n-1)}} > 0$$
	
	\item $\ol{\beta} \in \left( 0, \frac{1}{2} \right)$ is used in the assignment of initial data as a term controlling the transition from the ``parabolic region" to the ``outer region."
	
	\item $\beta \in \left( 0 , \max \big( \beta_k, \ol{\beta} \big) \right)$ in general where 
		$$ \beta_k \doteqdot \frac{1}{2} \left( \frac{ 1}{ 1 + \frac{1}{-2B^2 \lambda_k} } \right) \in \left( 0, \frac{1}{2} \right)$$
	 On occasion we allow for $\beta \in \left( 0 , \frac{1}{2} \right)$ more generally.
		
	\item $M > 0 $ is an uniform constant. $M e^{\beta \tau}$ controls the scale where we transition from the ``parabolic region" to the ``outer region."
	
	\item $\ul{\eta} =( \eta_0^U, \eta_1^U, \eta_2^U, \eta_0^Z, \eta_1^Z, \eta_2^Z) \ll 1$ is sufficiently small depending on $p,q,k$.
	It controls the closeness, in a weighted $C^2$ sense, of the profile functions $(\tl{Z}, \tl{U})(\gamma, \tau)$ to the eigenmode solutions $\left( e^{B^2 \lambda_k \tau} \tl{Z}_{\lambda_k}(\gamma), e^{B^2 \lambda_k \tau} \tl{U}_{\lambda_k}(\gamma) \right)$ in the ``parabolic region."
	
	\item $\Gamma \gg 1$ is sufficiently large depending on $p,q,k , \ul{\eta}$.
	
	\item $a,b,c, \kappa, \epsilon$ are parameters used for barriers in the inner-parabolic overlap region. 
	Their values depend on $p,q,k$ and are described in remark \ref{constantsToUse}.
	
	\item $l \in ( 0, k + B^2 \lambda_k )$ depends on $p,q,k,a,b,c, \epsilon$.
	
	\item $\ul{\Upsilon} \gg 1$ is sufficiently large depending on $p,q,k, \ul{\eta}, \Gamma, a,b,c, \kappa, \epsilon, l$.
	
	\item $\Upsilon_{U} \gg 1$ is sufficiently large depending on $p,q,k, \ul{\eta}, \Gamma, \ul{\Upsilon} ,a,b,c, \kappa, \epsilon, l$.
	
	\item $\Upsilon_{Z} \gg 1$ is sufficiently large depending on $p,q,k, \ul{\eta}, \Gamma, \ul{\Upsilon}, \Upsilon_U , a,b,c, \kappa, \epsilon, l$.
	$\ul{\Upsilon}, \Upsilon_U, $ and $\Upsilon_Z$ control the transition from the ``inner region" to the ``parabolic region."
	
	\item $\tau = - \log( T-t)$ denotes a reparametrization of the ``time remaining" to the singularity time $T$.
	$\tau \ge \tau_0 \gg 1$ is sufficiently large depending on $p,q,k, \ul{\eta}, \Gamma, \ul{\Upsilon}, \Upsilon_U, \Upsilon_Z, M, \ol{\beta}, \beta, a,b,c, \kappa, \epsilon, l$.
	
	\end{itemize}
	
	\begin{remark}
		To simplify some of the statements throughout the paper, 
		we assume that any large constant has a universal lower bound and analogously that any small constant has a universal upper bound.
	\end{remark}
	
	\begin{remark}
		Throughout $``A \lesssim B"$ means that there exists a constant $C$ such that $A \le C B$.
		$``A \lesssim_{a,b} B"$ indicates that the constant $C = C(a,b)$ depends on $a,b$.
		$``A \sim B"$ means that $A \lesssim B \lesssim A$.
	\end{remark}